%% file: diploma-thesis.tex
\documentclass[a4paper,10pt,twoside,openright]{memoir}

\usepackage{navigator}
\usepackage[ngerman,american]{babel}

\title{Interpolation between Banach spaces and continuity of Radon-like integral transforms}
\author{Pavel Zorin-Kranich}

\makeindex

\usepackage{amsmath,amssymb,amsthm}
\usepackage{esint}
\theoremstyle{plain}

\newtheorem{theorem}[equation]{Theorem}
\newtheorem{prop}[equation]{Proposition}
\newtheorem{proposition}[equation]{Proposition}
\newtheorem{lemma}[equation]{Lemma}

\newtheorem{corollary}[equation]{Corollary}
\theoremstyle{definition}
\newtheorem{definition}[equation]{Definition}

\usepackage{ifxetex}
\ifxetex
\usepackage{fontspec}
\defaultfontfeatures{Mapping=tex-text}
\usepackage{unicode-math}
\usepackage{hyperref}
\renewcommand{\C}{\mathbb{C}}
\else
\usepackage[T1]{fontenc}
\usepackage{microtype}
\usepackage[charter]{mathdesign}
\usepackage[utf8]{inputenc} 
\usepackage[unicode]{hyperref}
\renewcommand*{\C}{\mathbb{C}}
\fi

\usepackage[msc-links,initials,backrefs,alphabetic]{amsrefs} 

\hypersetup{
plainpages=false,
pdfpagelabels=true,
bookmarksnumbered=false,
bookmarks=true,
pdfpagemode=UseNone,
pdfstartview=FitH,
pdfdisplaydoctitle=true,
pdflang=en-US,
pdfborder={0 0 0},
pdftitle={Interpolation and continuity of Radon transforms},
pdfauthor={Pavel Zorin-Kranich},
pdfsubject={Mathematics Subject Classification (2010): Primary 44A12},
pdfkeywords={real interpolation, complex interpolation, Hardy space, Radon transform}
}

\newcommand*{\Z}{\mathbb{Z}}
\newcommand*{\Q}{\mathbb{Q}}
\newcommand*{\N}{\mathbb{N}}
\newcommand*{\R}{\mathbb{R}}

\newcommand*{\boundary}{\partial}
\newcommand*{\id}{\mathrm{id}}
\newcommand*{\BMO}{\mathrm{BMO}}
\newcommand*{\VMO}{\mathrm{VMO}}
\newcommand*{\Schwartz}{\mathcal{S}}
\newcommand*{\dif}{\text{d}}
\newcommand*{\tdif}[3][]{\frac{\dif^{ #1} #2}{\dif { #3}^{ #1}}}

\newcommand*{\pdif}[3][]{\frac{\partial^{ #1} #2}{\partial { #3}^{ #1}}}

\newcommand*{\inv}{^{-1}}
\newcommand*{\cconv}{\overline{\mathrm{conv}}\,}
\newcommand*{\Union}{\bigcup\limits}
\newcommand*{\Intersection}{\bigcap\limits}
\newcommand*{\union}{\cup}
\newcommand*{\intersection}{\cap}
\newcommand*{\Sum}{\sum\limits}

\newcommand*{\Prod}{\prod\limits}
\def\<{\left\langle}
\def\>{\right\rangle}
\newcommand*{\hol}[1]{\mathcal{H}(#1)}
\newcommand*{\const}{\mathrm{const}}

\newcommand*{\DMO}[1]{\expandafter\DeclareMathOperator\csname #1\endcsname {#1}}
\DMO{conv}
\DMO{supp}
\DMO{lin}
\DMO{dist}
\DMO{tr}

\newcommand{\Frechet}{\text{Fréchet}}

\newcommand{\Fejer}{\text{Fejér}}

\newcommand{\Calderon}{\text{Calderón}}

\usepackage{tikz}
\usetikzlibrary{calc,through}

\newcommand{\Fourier}{\mathcal{F}}
\newcommand{\Radon}{\mathcal{R}}
\newcommand{\Leb}[1][]{\lambda_{#1}}
\newcommand{\GM}{\mathcal{M}} 
\newcommand{\mo}{\text{mo}} 
\newcommand{\Det}{\text{Det}}

\allowdisplaybreaks[3]

\def\clap#1{\hbox to 0pt{\hss#1\hss}}

\def\mathrlap{\mathpalette\mathrlapinternal}

\def\mathrlapinternal#1#2{%
\rlap{$\mathsurround=0pt#1{#2}$}}

\begin{document}

\frontmatter
\pagestyle{empty}
\calccentering{\unitlength}
\begin{adjustwidth*}{\unitlength}{-\unitlength}
\begin{center}
{\huge Interpolation between Banach spaces and\\
continuity of Radon-like integral transforms}\\
\vspace{1cm}
a diploma thesis by\\
\vspace{0.5cm}
{\Large Pavel Zorin-Kranich}\\

\vspace{2cm}
\begin{abstract}
We present the abstract framework and some applications of interpolation theory.
The main new result concerns interpolation between $H^1$ and $L^p$ estimates for analytic families of operators acting on Schwartz functions.
\end{abstract}

\vfill
Eberhard Karls Universität Tübingen\\
Mathematisches Institut\\
\vspace{1cm}
Advisors:\\
Prof.~F.~Ricci (Scuola Normale Superiore di Pisa, Italy)\\
Prof.~R.~Nagel (Universität Tübingen, Germany)\\
\vspace{1cm}
Presented in January 2011
\end{center}
\end{adjustwidth*}

\clearpage
\section*{Copyright information}
Copyright \copyright{}  2010--2013  Pavel Zorin-Kranich.
Permission is granted to copy, distribute and/or modify this document
under the terms of the GNU Free Documentation License, Version 1.3
or any later version published by the Free Software Foundation;
with no Invariant Sections, no Front-Cover Texts, and no Back-Cover Texts.
A copy of the license is embedded into the PDF file.

\section*{Why GNU FDL?}
I chose the GNU FDL licence because it requires ``transparent'' copies of any documents derived from the present one to be made available.
In case of a \LaTeX\ document such as this one this means that the full unobfuscated \LaTeX\ source code must be made available to the public.

\cleardoublepage
\pagestyle{plain}
\tableofcontents
\include{introduction}

\mainmatter
\pagestyle{headings}
\include{chapter_real_interpolation}
\include{chapter_complex_interpolation}
\include{chapter_fractional_integration}
\include{chapter_radon}
\include{chapter_rearrangement}
\include{chapter_hardy}
\include{chapter_k_plane}
\include{chapter_curvature}

\backmatter
\include{back_matter}
\bibliography{pzorin-harmonic}
\printindex

\embeddedfile[GNU FDL 1.3]{fdl-1.3.txt}{fdl-1.3.txt}
\end{document}

%% file: introduction.tex
\chapter{Introduction}
The interpolation theory deals with the question what is a good \emph{method} to define an interpolation space ``between'' two given Banach spaces both contained in a larger topological vector space (e.g.\ two $L^{p}$ spaces inside the space of measurable functions).
The method should have the \emph{interpolation property}: given compatible (i.e.\ agreeing on the intersection) continuous operators on both spaces, one would like them to induce a continuous operator on the interpolation space.
The hope is that these operators are easier to analyze when considered on the boundary spaces.

The applications we have in mind are to $L^{p}$ continuity of integral operators of the form
\[
Tf(y)=\int_{M} K(y,x) f(x),
\]
where $M$ is a manifold and the kernel $K(y, \cdot)$ is a distribution supported on a submanifold of strictly positive codimension, e.g.\ on a line in $\R^{n}$, $n\geq 2$.

The first chapter summarizes the standard results on linear real interpolation by Peetre's $K$-method (1963) and related results for interpolation of estimates for multilinear forms.
The first notable application of the abstract theory is the $L^{p}$ continuity of the Hardy-Littlewood maximal operator, $p>1$.

In chapter~\ref{chap:complex} we discuss \Calderon{}'s complex interpolation method (1964) along with prerequisites from complex analysis, including the characterization of analyticity of vector valued functions on the complex plane.
We then relax the hypothesis of the Stein interpolation theorem (a generalization of the Riesz-Thorin theorem) as to include operators with small initial domain.
This straightforward step is essential in what follows.
We immediately verify that this version is applicable to the complex interpolation space $[L^{p_{0}},L^{p_{1}}]_{\theta}$ over $\R^{n}$.

The third chapter is a brief account of properties of Riesz transforms.
They can be thought of as differential operators of non-integer order, in the sense that they constitute an analytic family of operators and happen to be ordinary differential operators for some integer arguments.

The fourth chapter begins with the basic properties of the classical Radon transform as an operator on the Schwartz space of test functions.
We then find the range of exponents in which the Radon transform is $L^{p}$-continuous.
The proof requires complex interpolation and we address the technical issues which were left implicit in the original literature.
We also connect the Radon transform to a convolution operator on the Heisenberg group.

In chapter~\ref{chap:rearrangement} we clarify in which sense it is possible to transform a bounded measurable subset $T$ of $\R^{n}$ into a ball by means of rearrangement.
The standard reference for this trick seems to be Federer's book, which only provides convergence to some ball in the Hausdorff distance.
Our quantitative argument shows that this ball must have the same measure as $T$.
The Brunn-Minkowski inequality, the main ingredient in the proof of a rearrangement inequality due to Brascamp, Lieb and Luttinger (1974), is an immediate corollary.

The next chapter deals with the Hardy space $H^{1}$, which is a useful substitute for $L^{1}$ in interpolation theory.
We are mostly interested in the atomic structure of $H^{1}$, i.e.\ the fact that every function in $H^{1}$ is a linear combination of functions with particularly nice properties.
We provide the most refined version of this decomposition.
The required modifications to the original proof seem to be known to the experts but have not been written down anywhere.
We mention the recent result of Meda, Sjögren, and Vallarino (2008) clarifying how the atomic decomposition is related to the continuity of operators on $H^{1}$.
The classical proof that $H^{1}$ is the dual of $\VMO$, the space of functions with vanishing mean oscillation, is presented in a simplified form.
Our central result is the Proposition~\ref{prop:h1-lp-interpolation-by-schwartz-functions}, which allows to interpolate between $H^{1}(\R^{n})$ and $L^{p}(\R^{n})$ by means of Schwartz functions.

Chapter~\ref{chap:k-plane} contains applications of rearrangement and interpolation methods to the $k$-plane transform.
We simplify some arguments and carry out an extension to the complex case.

The last chapter deals with continuity of convolution operators with kernels supported on submanifolds of Lie groups.
Here the central lemma regards transport of measure by a smooth map.
We recast it in the language of interpolation theory.

\section*{Acknowledgment}
The work on this thesis started at the Scuola Normale Superiore di Pisa where I have spent the academic year 2009--2010 thanks to an exchange program of the University of Tübingen.
I am grateful to Prof.~Fulvio Ricci for the motivation and the guidance he has provided me with as well as for his patience.
The support of Prof.~Rainer Nagel was invaluable not only in relation to this text.

\section*{Online version}
In the present version I have corrected some typographical errors present in the original and cleaned up the \LaTeX\ code.
I have also taken the liberty to remove some unnecessary fluff and add an abstract in an effort to make the text more useful.

%% file: chapter_real_interpolation.tex
\chapter{Real interpolation}
\label{chap:real}
We review two equivalent real-variable methods for constructing interpolation spaces between an appropriate couple of Banach spaces, mostly following the exposition in \cite{MR928802}.

The Marcinkiewicz interpolation theorem then allows one to transport estimates on operators on the endpoint spaces to interpolation spaces.
It is most useful in conjunction with the knowledge of explicit expressions for the norms of the spaces in question.
These norms will be computed for interpolation spaces between various $L^p$'s.

\section{The \texorpdfstring{$K$}{K}-method}
\index{real interpolation!K-method@$K$-method}
When applied to $L^p$ spaces, the $K$-method ultimately boils down to decomposition of a function in two parts by absolute value.
The abstract approach here is due to Peetre~\cite{MR0178381}.
It will come in handy in the proofs of the reiteration theorems for interpolation.

Let $X_0$ and $X_1$ be Banach spaces contained in a topological vector space.
The $K$-functional is defined by
\[
K(f,t,X_0,X_1) = \inf \{ ||f_0||_{X_0} + t ||f_1||_{X_1}, f=f_0+f_1 \}, \quad \text{for } f \in X_0 + X_1.
\]
For every $0 < \theta < 1$ and $1 \leq q \leq \infty$, the $(\theta,q;K,X_0,X_1)$-norm on $X_0 + X_1$ is defined by
\[
||f||_{\theta,q;K,X_0,X_1} =
\begin{cases}
\left( \int_0^\infty (t^{-\theta} K(f,t,X_0,X_1))^q \frac{\dif t}{t} \right)^{1/q}, & q < \infty,\\
\sup_{t>0} t^{-\theta} K(f,t,X_0,X_1), & q = \infty.
\end{cases}
\]
We will call it just $K$-norm if the supplementary information is clear from the context.

The usefulness of this definition stems from the following interpolation theorem for operators.
\begin{theorem}[Marcinkiewicz]
\index{Marcinkiewicz interpolation theorem}
\label{th:marcinkiewicz}
Let $T : X_0 + X_1 \to Y_0 + Y_1$ be a linear operator such that
\[
||T f||_{Y_j} \leq M_j ||f||_{X_j}, j=0,1.
\]
Then, for every $0 < \theta < 1$ and $1 \leq q \leq \infty$,
\[
||Tf||_{\theta, q; K, Y_0, Y_1} \leq M_0^{1-\theta} M_1^{\theta} ||Tf||_{\theta, q; K, X_0, X_1}.
\]
\end{theorem}
\begin{proof}
By linearity we have that
\begin{multline*}
K(Tf, t, Y_0, Y_1)
\leq
\inf_{f=g+h}
||Tg||_{Y_0} + t ||Th||_{Y_1}\\
\leq
\inf_{f=g+h}
M_0 ||g||_{X_0} + M_1 t ||h||_{X_1}
=
M_0 K(f, t M_1/M_0, Y_0, Y_1).
\end{multline*}
Inserting this into the definition of the $(\theta, q; K, Y_0, Y_1)$-norm yields
\begin{align*}
||Tf||_{\theta, q; K, Y_0, Y_1}
&\leq
\left( \int_0^\infty (t^{-\theta} M_0 K(f,t M_1/M_0,X_0,X_1))^q \frac{\dif t}{t} \right)^{1/q}\\
&=
M_0 (M_1/M_0)^{-\theta} \left( \int_0^\infty (t^{-\theta} K(f,t,X_0,X_1))^q \frac{\dif t}{t} \right)^{1/q}\\
&=
M_0^{1-\theta} M_1^{\theta} ||f||_{\theta, q; K, X_0, X_1}.
\qedhere
\end{align*}
\end{proof}
If $Y_0$ and $Y_1$ are ordered (say, Banach function spaces), then the assumptions of the theorem may be weakened as to include subadditive operators $T$.
This stronger version will be useful in the proof of the Hardy-Littlewood maximal inequality.

\section{The \texorpdfstring{$J$}{J}-method and the equivalence theorem}
The $J$-method is modeled on dyadic decomposition by absolute value.
\index{real interpolation!J-method@$J$-method}
Let $X_0$ and $X_1$ be Banach spaces contained in a topological vector space and define the $J$-functional by
\[
J(f,t,X_0,X_1) = \max \{ ||f||_{X_0}, t ||f_1||_{X_1}\}, \quad \text{for } f \in X_0 \intersection X_1.
\]
For every $0 < \theta < 1$ and $1 \leq q \leq \infty$, the $(\theta,q;J,X_0,X_1)$-norm (or just $J$-norm) on $X_0 + X_1$ is defined by
\begin{align*}
||f||_{\theta,q;J,X_0,X_1} &= \inf_u
\begin{cases}
\left( \int_0^\infty (t^{-\theta} J(u(t),t,X_0,X_1))^q \frac{\dif t}{t} \right)^{1/q}, & q < \infty,\\
\sup_{t>0} t^{-\theta} J(u(t),t,X_0,X_1), & q = \infty,
\end{cases}
\end{align*}
where the infimum is taken over measurable functions $u : (0,\infty) \to X_0 \intersection X_1$ such that $\int_0^\infty u(t) \dif t/t = f$ with convergence in $X_0 + X_1$.

We now show that the $K$- and the $J$-norm are equivalent.
This fact furnishes powerful estimates needed to prove the reiteration theorems.
The estimates below for $f \in X_{0} \cap X_{1}$ follow immediately from the definitions.
\begin{align}
K(f,t,X_0,X_1) &\leq ||f||_{X_0} \leq J(f,s,X_0,X_1) &\text{for all }t,s, \label{ineq:K-J}\\
K(f,t,X_0,X_1) &\leq t ||f||_{X_1} \leq t/s J(f,s,X_0,X_1) &\text{for all }t,s, \label{ineq:K-tsJ}\\
J(f,t,X_0,X_1) &\leq J(f,s,X_0,X_1) &\text{for } t\leq s, \label{ineq:J-J}\\
J(f,t,X_0,X_1) &\leq t/s J(f,s,X_0,X_1) &\text{for } t\geq s. \label{ineq:J-tsJ}
\end{align}
The other ingredients in the proof are Hardy's inequalities and dyadic versions of the $K$- and the $J$-norm.

\begin{lemma}[{Hardy's inequalities}]
\index{Hardy's inequalities}
Let $\lambda > 0$, $1 \leq q < \infty$ and $f$ be a measurable function on $[0,\infty)$.
Then
\begin{align}
  \left( \int_0^\infty \left( t^\lambda \int_t^\infty f(s) \frac{\dif s}{s} \right)^q \frac{\dif t}{t} \right)^{1/q}
  &\leq
  \frac1\lambda \left(\int_0^\infty (t^\lambda f(t))^q \frac{\dif t}{t} \right)^{1/q},
  \label{ineq:hardy-positive}\\
  \left( \int_0^\infty \left( t^{-\lambda} \int_0^t f(s) \dif s \right)^q \frac{\dif t}{t} \right)^{1/q}
  &\leq
  \frac1\lambda \left(\int_0^\infty (t^{1-\lambda} f(t))^q \frac{\dif t}{t} \right)^{1/q}.
  \label{ineq:hardy-negative}
\end{align}
\end{lemma}
\begin{proof}
To show (\ref{ineq:hardy-positive}) in the case $q>1$, write $f(s) s\inv = s^{-(\lambda + 1)/q'} (s^{(\lambda + 1)/q'} f(s)s\inv)$.
Applying the Hölder inequality to the inner integral we obtain
\begin{align*}
\int_t^\infty f(s) \frac{\dif s}{s}
&\leq
\left( \int_t^\infty s^{-\lambda-1} \dif s \right)^{1/q'}
\left( \int_t^\infty [ s^{(\lambda+1)/q'} f(s) s^{-1} ]^q \dif s \right)^{1/q}\\
&=
\lambda^{-1/q'} t^{-\lambda/q'}
\left( \int_t^\infty s^{(\lambda+1)(q-1)-q} f(s)^q \dif s \right)^{1/q}.
\end{align*}
The left-hand side of (\ref{ineq:hardy-positive}) may therefore be estimated by
\begin{align*}
\dots &\leq
\left( \int_0^\infty t^{q \lambda} \lambda^{-q/q'} t^{-q \lambda/q'} \int_t^\infty s^{(\lambda+1)(q-1)-q} f(s)^q \dif s \frac{\dif t}{t} \right)^{1/q}\\
&=
\lambda^{-1/q'} \left( \int_{t=0}^\infty t^{\lambda-1} \int_{s=t}^\infty s^{(\lambda+1)(q-1)-q} f(s)^q \dif s \dif t \right)^{1/q}\\
&=
\lambda^{-1/q'} \left( \int_{s=0}^\infty \int_{t=0}^s t^{\lambda-1} \dif t s^{(\lambda+1)(q-1)-q} f(s)^q \dif s \right)^{1/q}\\
&=
\lambda^{-1} \left( \int_{s=0}^\infty s^\lambda s^{(\lambda+1)(q-1)-q} f(s)^q \dif s \right)^{1/q}\\
&=
\frac1\lambda \left( \int_{s=0}^\infty [s^\lambda f(s)]^q \frac{\dif s}{s} \right)^{1/q}.
\end{align*}
The case $q=1$ is similar but the Hölder inequality is not needed.
The proof of (\ref{ineq:hardy-negative}) is analogous if we decompose $f$ as $f(s) = s^{(\lambda-1)/q'} (s^{(1-\lambda)/q'} f(s))$.
\end{proof}

We are now ready to estimate the $K$-norm with the $J$-norm.
Let $u$ be as above.
Then
\[
K(f,t,X_0,X_1)
\leq
\int_0^\infty K(u(s),t,X_0,X_1) \dif s/s
\leq
\int_0^\infty \min\{1, t/s\} J(u(s),s,X_0,X_1) \dif s/s
\]
by (\ref{ineq:K-J}) and (\ref{ineq:K-tsJ}).
In the case $q < \infty$ this implies that
\begin{align*}
||f||_{\theta,q;K,X_0,X_1}
&\leq
\left( \int_0^\infty \left(t^{-\theta} \int_0^t J(u(s),s,X_0,X_1) \dif s/s \right)^q \frac{\dif t}{t} \right)^{1/q}\\
& \quad {} +
\left( \int_0^\infty \left(t^{-\theta} \int_t^\infty t/s J(u(s),s,X_0,X_1) \dif s/s \right)^q \frac{\dif t}{t} \right)^{1/q}\\
&\leq
\left( \frac{1}{\theta} + \frac{1}{1 - \theta} \right) \left( \int_0^\infty (t^{-\theta} J(u(t),t,X_0,X_1))^q \frac{\dif t}{t} \right)^{1/q}
\end{align*}
by the Hardy inequalities (\ref{ineq:hardy-negative}) and (\ref{ineq:hardy-positive}) for the former and the latter term, respectively.
On the other hand, in the case $q=\infty$ we have that
\begin{align*}
||f||_{\theta,q;K,X_0,X_1}
&=
\sup_{t>0} t^{-\theta} K(f,t,X_0,X_1)\\
&\leq
\sup_{t>0} t^{-\theta} \left[
\int_0^t J(u(s),s,X_0,X_1) \dif s/s
+
\int_t^\infty \frac{t}{s} J(u(s),s,X_0,X_1) \dif s/s
\right]\\
&\leq
\sup_r r^{-\theta} J(u(r),r,X_0,X_1) \sup_{t>0} t^{-\theta} \left[
\int_0^t s^\theta \dif s/s
+
\int_t^\infty \frac{t}{s} s^\theta \dif s/s
\right]\\
&\leq
\left( \frac{1}{\theta} + \frac{1}{1 - \theta} \right)
\sup_r r^{-\theta} J(u(r),r,X_0,X_1).
\end{align*}
Taking the infimum over $u$ yields the claimed estimate in both cases.

In order to obtain the converse we consider the dyadic versions of the $K$- and the $J$-norm.
Let $\lambda^{\theta, q}$ denote the space of sequences $(a_\nu)_{\nu=-\infty}^\infty$ such that
\begin{equation}
\label{eq:lambda-norm}
||(a_\nu)_\nu||_{\theta, q} =
\begin{cases}
\left( \sum_{\nu=-\infty}^\infty (2^{-\theta \nu} a_\nu)^q \right)^{1/q},& q < \infty,\\
\sup_\nu 2^{-\theta \nu} a_\nu,& q = \infty
\end{cases}
\end{equation}
is finite and define the dyadic $K$- and $J$-norms by
\begin{align*}
||f||_{\theta,q;K,X_0,X_1}^d
&=
|| (K(f,2^\nu,X_0,X_1))_\nu ||_{\theta, q}
\quad\text{and}\\
||f||_{\theta,q;J,X_0,X_1}^d
&=
\inf_{f_\nu} || (J(f_\nu,2^\nu,X_0,X_1))_\nu ||_{\theta, q},
\end{align*}
the infimum this time being taken over decompositions $f = \sum_\nu f_\nu$ with $f_\nu \in X_0 \intersection X_1$ and convergence in $X_0 + X_1$.
The equivalence
\[
||f||_{\theta,q;K,X_0,X_1} \leq C ||f||_{\theta,q;K,X_0,X_1}^d \leq C' ||f||_{\theta,q;K,X_0,X_1}
\]
is clear from the fact that $K(f,t)$ is monotonous in its second argument.
Furthermore, restricting the infimum on the left-hand side to piecewise constant functions, we obtain
\[
||f||_{\theta,q;J,X_0,X_1} \leq ||f||_{\theta,q;J,X_0,X_1}^d.
\]
In order to estimate the dyadic $J$-norm and thus complete the proof of equivalence of the norms we need to establish the existence of a particular decomposition $f = \sum_\nu f_\nu$.

\begin{lemma}[Fundamental lemma of interpolation theory]
\label{lem:fundamentalJK}
Let $f \in X_0 + X_1$ be such that $K(f,t) \to 0$ as $t\to 0$ and $K(f,t)/t \to 0$ as $t\to\infty$.
Then for every $\epsilon>0$ there exists a decomposition $f = \sum_{\nu = -\infty}^\infty f_\nu$ with convergence in $X_0 + X_1$ such that $J(f_\nu, 2^\nu) \leq 3(1+\epsilon) K(f, 2^\nu)$ for every $\nu \in \Z$.
\end{lemma}
\begin{proof}
By definition of the $K$-functional there exist $f_{j,\nu}$, $j=0,1$ such that
\[
||f_{0,\nu}||_{X_0} + 2^\nu ||f_{1,\nu}||_{X_1} \leq (1+\epsilon) K(f, 2^\nu).
\]
By the assumptions $||f_{0,\nu}||_{X_0} \to 0$ as $\nu \to -\infty$ and $||f_{1,\nu}||_{X_1} \to 0$ as $\nu \to \infty$.
Let
\[
f_\nu := f_{0,\nu} - f_{0,\nu-1} = f_{1,\nu-1} - f_{1,\nu} \in X_0 \intersection X_1.
\]
Then
\[
f - \sum_{\nu=-N}^N f_\nu = f_{1,N} + f_{0,-N-1} \to 0 \text{ in } X_0 + X_1 \text{ as } N \to\infty
\]
and
\begin{align*}
J(f_\nu, 2^\nu)
&\leq
\max \{ ||f_{0,\nu}||_{X_0} + ||f_{0,\nu-1}||_{X_0}, 2^\nu ||f_{1,\nu}||_{X_1} + 2 \cdot 2^{\nu-1} ||f_{1,\nu-1}||_{X_1} \}\\
&\leq
(1+\epsilon) K(f,2^\nu) + 2(1+\epsilon) K(f,2^{\nu-1})
\leq
3 (1+\epsilon) K(f,2^\nu).
\qedhere
\end{align*}
\end{proof}
This result allows us to conclude the proof of the equivalence theorem.
\begin{theorem}
For every $0 < \theta < 1$ and $1 \leq q \leq \infty$, the norms $|| \cdot ||_{\theta, q; K, X_0, X_1}^d$, $|| \cdot ||_{\theta, q; J, X_0, X_1}^d$, $|| \cdot ||_{\theta, q; K, X_0, X_1}$ and $|| \cdot ||_{\theta, q; J, X_0, X_1}$ on $X_0 + X_1$ are equivalent.
\end{theorem}
The space $[X_0,X_1]_{\theta, q} \subseteq X_0 + X_1$ defined by any of these norms is called a \emph{real interpolation space} between $X_0$ and $X_1$.
Note that
\[
X_{0}\cap X_{1} \hookrightarrow [X_0,X_1]_{\theta, q} \hookrightarrow X_0 + X_1,
\]
where the continuity of the former inclusion becomes evident if one considers the dyadic $J$-norm, and the continuity of the latter inclusion can be seen using the $K$-norm.
\begin{proof}
We have already shown that
\[
||f||_{\theta, q; K, X_0, X_1}^d \leq  C ||f||_{\theta, q; K, X_0, X_1} \leq C' ||f||_{\theta, q; J, X_0, X_1} \leq C'' ||f||_{\theta, q; J, X_0, X_1}^d.
\]
and need only estimate the dyadic $J$- by the dyadic $K$-norm.

Clearly, if $||f||_{\theta, q; K, X_0, X_1}^d < \infty$, then the monotonicity of $K$ in $t$ implies the decay conditions on $K(f,t,X_0,X_1)$ required in Lemma~\ref{lem:fundamentalJK}.
With the decomposition $f = \sum_\nu f_\nu$ provided by the same lemma, we have that
\begin{multline*}
||f||_{\theta, q; J, X_0, X_1}^d
\leq
|| (J(f_\nu,2^\nu,X_0,X_1))_\nu ||_{\theta, q}\\
\leq
C || (K(f,2^\nu,X_0,X_1))_\nu ||_{\theta, q}
=
C ||f||_{\theta, q; K, X_0, X_1}^d.
\qedhere
\end{multline*}
\end{proof}

Henceforth we will omit the superscript ${}^d$ from the notation and use the dyadic and continuous versions of the $K$- and the $J$-norm interchangeably.

\section{The reiteration theorem, case \texorpdfstring{$\theta_{0}<\theta_{1}$}{\unichar{"03B8}\unichar{"2080}<\unichar{"03B8}\unichar{"2081}}}
The basic tool for explicit computation of norms on the interpolation spaces is the reiteration theorem which characterizes the interpolation spaces between interpolation spaces.
Here we switch to the more concise methods in \cite{MR0482275}.
\begin{theorem}[Reiteration theorem]
\index{real interpolation!reiteration theorem!1@$\theta_0 < \theta_1$}
\label{th:real-reiteration}
Let $0 \leq \theta_0 < \theta_1 \leq 1$, $0 < \eta < 1$, $\theta = (1-\eta) \theta_0 + \eta \theta_1$, $1 \leq q \leq \infty$ and $Y_0$, $Y_1$ be Banach spaces such that
\begin{align*}
[ X_0, X_1 ]_{\theta_0 , 1} \subseteq Y_0 \subseteq [ X_0, X_1 ]_{\theta_0 , \infty}& \text{ if } \theta_0 > 0, \quad Y_0 = X_0 \text{ if } \theta_0 = 0,\\\
[ X_0, X_1 ]_{\theta_1 , 1} \subseteq Y_1 \subseteq [ X_0, X_1 ]_{\theta_1 , \infty}& \text{ if } \theta_1 < 1, \quad Y_1 = X_1 \text{ if } \theta_0 = 1,
\end{align*}
each with continuous inclusion.
Then, up to norm equivalence,
\[
[Y_0, Y_1]_{\eta, q} = [X_0, X_1]_{\theta, q}.
\]
\end{theorem}
We give the proof only for $q < \infty$, the case $q=\infty$ is similar but easier.
\begin{proof}
Assume first that $0 < \theta_0 < \theta_1 < 1$.

Let $f \in [Y_0, Y_1]_{\eta, q}$ and $f = f_0 + f_1$ with $f_0 \in Y_0$, $f_1 \in Y_1$.
Then
\[
K(f,t,X_0,X_1)
\leq
K(f_0,t,X_0,X_1) + K(f_1,t,X_0,X_1)
\leq
C t^{\theta_0} ||f_0||_{Y_0} + C t^{\theta_1} ||f_1||_{Y_1}
\]
because $Y_j \subseteq [ X_0, X_1 ]_{\theta_j , \infty}$.
Taking the infimum over decompositions $f=f_0+f_1$ yields
\[
K(f,t,X_0,X_1) \leq C t^{\theta_0} K(f,t^{\theta_1 - \theta_0},Y_0,Y_1).
\]
Inserting this estimate into the definition of the $K$-norm and using the relation $\frac{\theta-\theta_0}{\theta_1 - \theta_0} = \eta$ we obtain
\begin{align*}
||f||_{\theta, q; K, X_0, X_1}
&=
\left( \int_0^\infty (t^{-\theta} K(f,t,X_0,X_1))^q \frac{\dif t}{t} \right)^{1/q}\\
&\leq
C \left( \int_0^\infty (t^{-\theta+\theta_0} K(f,t^{\theta_1 - \theta_0},Y_0,Y_1))^q \frac{\dif t}{t} \right)^{1/q}\\
&=
C \left( \int_0^\infty (s^{-\frac{\theta-\theta_0}{\theta_1 - \theta_0}} K(f,s,Y_0,Y_1))^q \frac{\dif s}{s} \right)^{1/q}
&(s = t^{\theta_{1}-\theta_{0}})\\
&=
C ||f||_{\eta, q; K, Y_0, Y_1}.
\end{align*}
For the converse observe first that for every $u \in X_0 \intersection X_1$ we have that
\begin{multline*}
J(u,s^{\theta_1 - \theta_0}, Y_0, Y_1)
=
s^{-\theta_0} \max_j s^{\theta_j} ||u||_{Y_j}
\leq
s^{-\theta_0} \max_j s^{\theta_j} ||u||_{\theta_j, 1; J, X_0, X_1}\\
\leq
s^{-\theta_0} \max_j s^{\theta_j} C J(u,s,X_0,X_1) s^{-\theta_j}
=
C s^{-\theta_0} J(u,s,X_0,X_1)
\end{multline*}
because $[ X_0, X_1 ]_{\theta_j , 1} \subseteq Y_j$ and by the dyadic characterization of the $J$-norm.
Using the variable change from the first part of the proof, decomposing $f=\int_0^\infty u(s) \dif s/s$ and inserting the last result we obtain
\begin{align*}
||f||&_{\eta, q; K, Y_0, Y_1}\\
&=
C \left( \int_0^\infty (t^{-\theta+\theta_0} K(f,t^{\theta_1 - \theta_0},Y_0,Y_1))^q \frac{\dif t}{t} \right)^{1/q}\\
&\leq
C \left( \int_0^\infty \left(t^{-\theta+\theta_0} \int_0^\infty K(u(s),t^{\theta_1 - \theta_0},Y_0,Y_1) \frac{\dif s}{s} \right)^q \frac{\dif t}{t} \right)^{1/q}\\
&\leq
C \left( \int_0^\infty \left(t^{-\theta+\theta_0} \int_0^\infty \min\{ 1, \frac{t^{\theta_1 - \theta_0}}{s^{\theta_1 - \theta_0}}\} J(u(s),s^{\theta_1 - \theta_0},Y_0,Y_1) \frac{\dif s}{s} \right)^q \frac{\dif t}{t} \right)^{1/q}\\
&\leq
C \left( \int_0^\infty \left(t^{-\theta} \int_0^\infty \min\{ (t/s)^{\theta_0}, (t/s)^{\theta_1}\}\} J(u(s),s,Y_0,Y_1) \frac{\dif s}{s} \right)^q \frac{\dif t}{t} \right)^{1/q}.
\end{align*}
We now split the inner integral as $\int_{0}^{\infty} = \int_{0}^{t} + \int_{t}^{\infty}$, use the Minkowski inequality and the Hardy inequalities (\ref{ineq:hardy-positive}) and (\ref{ineq:hardy-negative}).
This yields
\[
||f||_{\eta, q; K, Y_0, Y_1}
\leq
C \left( \int_0^\infty \left(t^{-\theta} J(u(t),t,Y_0,Y_1) \right)^q \frac{\dif t}{t} \right)^{1/q}.
\]
Taking the infimum over $u(s)$ we obtain
\[
||f||_{\eta, q; K, Y_0, Y_1}
\leq
C ||f||_{\theta, q; J, X_0, X_1}.
\]
The cases $\theta_0 = 0$ and $\theta_1 = 1$ can be treated similarly, only the estimates for $K(f,t,X_0,X_1)$ and $J(u,s^{\theta_1 - \theta_0}, Y_0, Y_1)$ become easier.
\end{proof}

\section{Lorentz spaces}
The reiteration theorem still leaves some work to be done.
Namely we have to compute the interpolation spaces between the (hopefully easier to handle) endpoints spaces.
Here we do so for the spaces $L^1$ and $L^\infty$.
The customary notation is as follows.
\begin{definition}
\index{Lorentz space}
Let $1 < p < \infty$ and $1\leq q\leq \infty$.
The corresponding \emph{Lorentz space} is
\[
L^{p,q} = [ L^1, L^\infty ]_{\theta,q},
\]
where
\[
\frac1p = 1-\theta  \quad\left( = \frac{1-\theta}{1} + \frac{1-\theta}{\infty}\right).
\]
The norm on $L^{p,q}$ is denoted by $||\cdot||_{p,q} = || \cdot ||_{\theta, q; L^1, L^\infty}$.
\end{definition}
To obtain an explicit expression for the norm on $L^{p,q}$ we now calculate
\[
K(f,t,L^1,L^\infty).
\]
This quantity is most conveniently expressed in terms of the non-increasing rearrangement of $f$.
\begin{definition}
Let $f$ be a measurable function on a measure space $(\Omega, \mu)$.
Its \emph{distribution function} is defined by
\[
\lambda_f(s) = \mu \{ |f|>s \}, \quad s>0
\]
and its \emph{non-increasing rearrangement} by
\index{rearrangement!non-increasing}
\[
f^*(t) = \inf \{ s : \lambda_f(s) \leq t \}, \quad t>0.
\]
\end{definition}
We establish some basic properties of $\lambda_f$ and $f^*$ first.

Both operations are monotonous in the sense that $|f| \leq |g|$ a.e.\ implies
$\lambda_f(s) \leq \lambda_g(s)$ for all $s$ since $\{ |f|>s \} \subseteq \{ |g|>s \}$
and $f^*(t) \leq g^*(t)$ for all $t$ since $\{ s : \lambda_f(s) \leq t \} \supseteq \{ s : \lambda_g(s) \leq t \}$.

Moreover, for every $f$, both $\lambda_f$ and $f^*$ are monotonously decreasing.
Together with continuity of $\mu$ from below (monotonous convergence theorem) this implies that
\begin{align*}
\lambda_f(f^*(t))
&=
\mu\{ |f| > \inf \{ s, \lambda_f(s) \leq t \} \}\\
&=
\mu \Union_{s : \lambda_f(s) \leq t} \{ |f| > s \}\\
&=
\lim_{s \searrow \inf \{ s, \lambda_f(s) \leq t \}} \mu \{ |f| > s \}\\
&\leq
\lim_{s \searrow \inf \{ s, \lambda_f(s) \leq t \}} t\\
&=
t
\end{align*}
and, Lebesgue measure on $(0,\infty)$ being denoted by $\Leb[1]$,
\begin{align*}
\lambda_{f*}(s_0)
&=
\Leb[1] \{ |f^*| > s_0 \}\\
&=
\sup \{ t : f^*(t) > s_0 \}\\
&=
\sup \{ t : \inf\{ s : \lambda_f(s) \leq t \} > s_0 \}\\
&=
\sup \{ t : \exists\epsilon>0 : \lambda_f(s_0 + \epsilon) > t \}\\
&=
\sup_{\epsilon>0} \lambda_f(s_0 + \epsilon)\\
&=
\lim_{\epsilon \searrow 0} \mu\{ |f| > s_0 + \epsilon \}\\
&=
\mu\{ |f| > s_0 \}\\
&=
\lambda_f(s_0).
\end{align*}
Hence, by definition of the Lebesgue integral,
\begin{multline}
\label{eq:rearrangement-Lp-norm}
||f||_p^{p}
=
\int |\{ |f|^p > t\}| \dif t
=
p \int s^{p-1} |\{ |f| > s\}| \dif s\\
=
p \int s^{p-1} \lambda_f (s) \dif s
=
||f^*||_p^{p}
\end{multline}
for every $1 \leq p < \infty$.

We now return to the problem of calculation of $K(f,t,L^1,L^\infty)$.
\begin{proposition}
For every $f \in L^1 + L^\infty$,
\[
K(f,t,L^1,L^\infty) = \int_0^t f^*.
\]
\end{proposition}
\begin{proof}
Observe that $K(f,t,L^1,L^\infty)$ only depends on $|f|$.
Thus we may without loss of generality assume that $f \geq 0$.
Let
\[
Q_a (x) :=
\begin{cases}
x-a, & \text{if } x>a,\\
0, & \text{if } x\leq a.
\end{cases}
\]
Then, for fixed $a$,
\[
|| Q_a \circ f ||_1 = \min_{||h||_\infty = a} || f - h ||_1,
\]
because $|Q_a \circ f| \leq |f-h|$ pointwise a.e.\ for every $h$ with $||h||_\infty = a$.
Therefore
\[
K(f,t,L^1,L^\infty) = \inf_{a \geq 0} || Q_a \circ f ||_1 + ta.
\]
Since $\lambda_{Q_a \circ f}(s) = \lambda_f(s+a)$ we have by definition $(Q_a \circ f)^*(t) = Q_a \circ f^*(t)$.
By (\ref{eq:rearrangement-Lp-norm}) we also have that
\[
|| Q_a \circ f ||_1 = || (Q_a \circ f)^* ||_1 = || Q_a \circ f^* ||_1,
\]
so that
\[
K(f,t,L^1,L^\infty) = \inf_{a \geq 0} || Q_a \circ f^* ||_1 + ta.
\]
The infimum is attained for $a = f^*(t)$.
Indeed, since $f^*$ is non-increasing
\[
|| Q_a \circ f^* ||_1 + ta
=
\int_0^{\sup\{s : f^*(s) > a\}} (f^* - a) + ta
=
\int_0^t f^*
+
\int_t^{\sup\{s : f^*(s) > a\}} (f^* - a).
\]
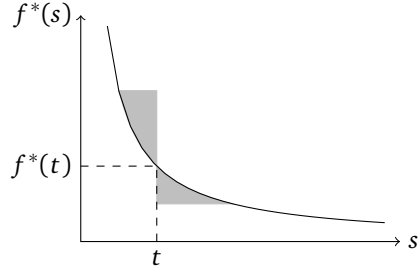
\begin{figure}
\centering
\begin{tikzpicture}
\draw[<->] (0,3) node[left] {$f^*(s)$} -- (0,0) -- (4.2,0) node[right] {$s$};

\draw[fill,color=lightgray] plot[domain=0.5:2] (\x,{1 / \x}) -- (1,.5) -- (1,2) -- cycle;

\draw plot[domain=0.35:4] (\x,{1 / \x});
\draw[style=dashed] (1,0) node[below] {$t$} -- (1,1) -- (0,1) node[left] {$f^*(t)$};
\end{tikzpicture}
\caption{The optimal value for $a$ is $f^*(t)$, for other values $||Q_a \circ f^*||_1 + a$ increases by the area of one the shaded regions}
\label{fig:K-L1-Linf}
\end{figure}
If $a < f^*(t)$ then $\sup\{s : f^*(s) > a\} \geq t$ and the latter integral is non-negative.
Otherwise $\sup\{s : f^*(s) > a\} \leq t$ and the latter integral is non-negative again since the integrand is negative, see Figure~\ref{fig:K-L1-Linf}.
Summarizing, we have that
\[
K(f,t,L^1,L^\infty)
=
K(|f|,t,L^1,L^\infty)
=
|| Q_{f^*(t)} \circ f^* ||_1 + t f^*(t)
=
\int_0^t f^*.
\qedhere
\]
\end{proof}
This proposition provides us with a direct expression for the norm on $L^{p,q}$.
If we are willed to sacrifice the triangle inequality and use a quasinorm, we obtain an even simpler characterization of $L^{p,q}$.
\begin{proposition}
For every $1 < p < \infty$ and $1\leq q \leq \infty$ the quasinorm
\[
|| f ||_{p,q}^* :=
\begin{cases}
\left( \frac{q}{p} \int_0^\infty  ( t^{1/p} f^*(t) )^{q} \frac{\dif t}{t} \right)^{1/q}, & q<\infty,\\
\sup_{t>0} t^{1/p} f^*(t), & q=\infty.
\end{cases}
\]
is equivalent to $|| \cdot ||_{p,q}$.
\end{proposition}
\begin{proof}
Consider first the case $q<\infty$.
Since $f^*$ is non-increasing, we have that
\begin{equation}
\label{eq:f-int-f}
t f^*(t) \leq \int_0^t f^*(s) \dif s,
\end{equation}
so that
\begin{multline*}
|| f ||_{p,q}^*
\leq
\left( \frac{q}{p} \int_0^\infty  ( t^{1/p-1} \int_0^t f^*(s) \dif s )^{q} \frac{\dif t}{t} \right)^{1/q}\\
=
\frac{q^{1/q}}{p^{1/q}} \left( \int_0^\infty  ( t^{-\theta} K(f,t,L^1,L^\infty) )^{q} \frac{\dif t}{t} \right)^{1/q}
=
\frac{q^{1/q}}{p^{1/q}} ||f||_{p,q}.
\end{multline*}
For the converse we use the Hardy inequality (\ref{ineq:hardy-negative}) and obtain
\begin{align*}
||f||_{p,q}
&=
\left( \int_0^\infty  \left( t^{-\theta} \int_0^t f^*(s) \dif s \right)^{q} \frac{\dif t}{t} \right)^{1/q}\\
&\leq
\frac{1}{\theta} \left( \int_0^\infty  ( t^{1-\theta} f^*(t) )^{q} \frac{\dif t}{t} \right)^{1/q}\\
&=
\frac{1}{\theta} \frac{q^{-1/q}}{p^{-1/q}} ||f||_{p,q}^*.
\end{align*}
Now consider the case $q=\infty$.
By (\ref{eq:f-int-f}) we have that
\[
||f||_{p,\infty}^*
=
\sup_t t^{1/p} f^*(t)
\leq
\sup_t t^{1/p-1} \int_0^t f^*(s) \dif s
=
||f||_{p,\infty}.
\]
For the converse observe that
\[
t^{1/p-1} \int_0^t f^*(s) \dif s
\leq
t^{1/p-1} \left( \int_0^t s^{-1/p} \dif s \right) \left( \sup_s s^{1/p} f^*(s) \right)
=
p' ||f||_{p,\infty}^*.
\]
Taking the supremum over $t$ yields the claim.
\end{proof}
\index{Lorentz space}
The latter quasinorm $|| \cdot ||_{p,q}^*$ also defines non-trivial spaces for $p=1$.
They are called \emph{Lorentz spaces} as well and denoted by $L^{1,q}$.
Furthermore, conventionally $L^{\infty,\infty} = L^\infty$.

With these definitions and by (\ref{eq:rearrangement-Lp-norm}) we have that $L^{p,p} = L^p$ with equal norms for every $1\leq p \leq \infty$.

Next we are going to establish the inclusion relations between the Lorentz spaces which are valid independently of the underlying measure space.

\begin{lemma}
\label{lemma:non-incr-sup-estimate}
Let $f : (0,\infty) \to (0,\infty)$ be a non-increasing function, $1 \leq p \leq q \leq \infty$, $a>0$ and $g(r) = f(r) r^{a}$.
Then $||g||_{q} \leq C_{p,q,a} ||g||_p$, where the norms are taken with respect to the dilation-invariant measure $\dif r/r$.
\end{lemma}
\begin{proof}
By $\log$-convexity of the function $1/q \mapsto ||g||_q$, it is sufficient to estimate $||g||_\infty$ in terms of $||g||_p$ since then, if $1/q = (1-\theta)/p$, it follows that
\[
||g||_{q}
\leq
||g||_p^{1-\theta} ||g||_\infty^\theta
\leq C ||g||_p.
\]
For every $r_0 \in (0,\infty)$ we have that $f(r_0) = g(r_0) r_0^{-a}$, and the monotonicity of $f$ implies
\begin{align*}
f(r) &\geq g(r_0) r_0^{-a} & \text{for all } r<r_0,\\
g(r) &\geq g(r_0) (r/r_0)^{a} & \text{for all } r<r_0.
\end{align*}
Inserting this into the definition of the $L^p$ norm yields
\begin{multline*}
||g||_p
\geq
\left( \int_0^{r_0} g(r)^p \frac{\dif r}{r} \right)^{\frac1p}
\geq
g(r_0) \left( \int_0^{r_0} \left(\frac{r}{r_0}\right)^{ap} \frac{\dif r}{r} \right)^{\frac1p}\\
=
g(r_0) \left( \int_0^1 s^{ap} \frac{\dif s}{s} \right)^{\frac1p}
= (ap)^{-1/p} g(r_0).
\qedhere
\end{multline*}
\end{proof}
Since $f^*$ is non-increasing and by the $||\cdot||_{p,q}^*$-characterization of $L^{p,q}$ we obtain
\begin{corollary}
\label{cor:lorentz-scale}
Let $1 \leq p<\infty$ and $1 \leq q_0 \leq q_1 \leq \infty$.
Then
\[
L^{p, q_0} \subset L^{p, q_1}
\]
with continuous inclusion.
\end{corollary}

\begin{theorem}
\index{real interpolation!between Lorentz spaces}
\label{th:lorentz-interpolation}
Let $1 \leq p_0 < p_1 < \infty$ and $1 \leq q_0, q_1 \leq \infty$ or $1 \leq p_0 < p_1 = \infty$ and $1 \leq q_0 \leq \infty$, $q_1 = \infty$.
If
\[
\frac{1}{p} = \frac{1-\theta}{p_0} + \frac{\theta}{p_1},
\]
where $0 < \theta < 1$, then
\[
[ L^{p_0, q_0}, L^{p_1, q_1} ]_{\theta, q} = L^{p,q}.
\]
\end{theorem}
\begin{proof}
Assume for the moment that
\begin{equation}
\label{eq:interpolation-weak-L1-Linfty}
[L^{1,\infty},L^\infty]_{\theta,q} = [L^{1},L^\infty]_{\theta,q} \text{ for every } \theta \text{ and } q.
\end{equation}
Then also $[L^{1,r},L^\infty]_{\theta,q} = [L^{1},L^\infty]_{\theta,q}$ for every $1\leq r\leq\infty$ and we can conclude by Corollary~\ref{cor:lorentz-scale} and the Reiteration Theorem~\ref{th:real-reiteration}.

We now prove (\ref{eq:interpolation-weak-L1-Linfty}).
The inclusion ``$\supseteq$'' is clear because $L^{1,\infty} \supseteq L^1$ by Corollary~\ref{cor:lorentz-scale}.
To obtain the converse recall that $K(f,t,L^1,L^\infty) = \int_{s=0}^t f^*(s) \dif s$.
Furthermore, whenever $f = g+h$ with $g \in L^{1,\infty}$, $h \in L^\infty$,
\[
f^*(s) \leq g^*(s) + ||h||_\infty \leq (||g||_{1,\infty} + s ||h||_\infty)/s.
\]
Taking the infimum over all decompositions $f=g+h$ we find that
\[
f^*(s) \leq K(f,s,L^{1,\infty}, L^\infty)/s.
\]
Therefore, by definition of the $(\theta,q;K,L^1,L^\infty)$-norm and the Hardy inequality (\ref{ineq:hardy-negative}),
\begin{align*}
||f||_{\theta,q;K,L^1,L^\infty}
&\leq
\left( \int_0^\infty \left(t^{-\theta} \int_0^t s\inv K(f,s,L^{1,\infty}, L^\infty) \dif s \right)^q \frac{\dif t}{t} \right)^{1/q}\\
&\leq
\frac1\theta
\left( \int_0^\infty \left(t^{1-\theta} t\inv K(f,t,L^{1,\infty}, L^\infty) \right)^q \frac{\dif t}{t} \right)^{1/q}\\
&=
\frac1\theta
||f||_{\theta,q;K,L^{1,\infty},L^\infty}.
\qedhere
\end{align*}
\end{proof}

\section{The reiteration theorem, case \texorpdfstring{$\theta_{0}=\theta_{1}$}{\unichar{"03B8}\unichar{"2080}=\unichar{"03B8}\unichar{"2081}}}
To describe the interpolation spaces between Lorentz spaces with $p_0 = p_1$ we need a second reiteration theorem.
The main ingredient in the proof is going to be the already established reiteration theorem~\ref{th:real-reiteration}.

Recall the definition (\ref{eq:lambda-norm}) of the sequence space  $\lambda^{\theta, q}$.
It is clear that $(a_\nu)_\nu \in \lambda^{\theta, q}$ if and only $(2^{-\theta \nu} a_\nu)_\nu \in \ell^q$.
Since for all $1 \leq q_0 < q_1 \leq \infty$ and $0 < \eta < 1$ we have that $[\ell^{q_0}, \ell^{q_1}]_{\eta, q} = \ell^{q,q} = \ell^q$ with $1/q = (1-\eta)/q_0 + \eta/q_1$ by Theorem~\ref{th:lorentz-interpolation}, we also have that
\begin{equation}
\label{eq:fine-real-interpolation-of-sequences}
[\lambda^{\theta, q_0}, \lambda^{\theta, q_1}]_{\eta, q} = \lambda^{\theta, q}
\end{equation}
under the same assumptions.
Together with the dyadic characterizations of the real interpolation spaces this yields the following reiteration theorem.

\begin{theorem}
\index{real interpolation!reiteration theorem!2@$\theta_0 = \theta_1$}
Let $0 < \theta < 1$, $1 \leq q_0, q_1 \leq \infty$, $0 < \eta < 1$ and $1/q = (1-\eta)/q_0 + \eta/q_1$.
Then
\[
[[X_0, X_1]_{\theta, q_0}, [X_0, X_1]_{\theta, q_1}]_{\eta,q} = [X_0, X_1]_{\theta, q}.
\]
\end{theorem}
\begin{proof}
Let $f \in [[X_0, X_1]_{\theta, q_0}, [X_0, X_1]_{\theta, q_1}]_{\eta,q}$.
Then
\begin{align*}
||f||_{\theta, q; K, X_0, X_1}
&=
|| (K(f,2^\nu, X_0, X_1))_\nu ||_{\theta, q}\\
&\leq
C \inf_{f = \sum_\mu u_\mu} || ( \sum_\mu K(u_\mu,2^\nu, X_0, X_1))_\nu ||_{\eta, q; J, \lambda^{\theta, q_0}, \lambda^{\theta, q_1}}\\
\intertext{because of (\ref{eq:fine-real-interpolation-of-sequences}) and by subadditivity of $K$ in the $f$ argument}
&\leq
C \inf_{f = \sum_\mu u_\mu} || (J( (K(u_\mu,2^\nu, X_0, X_1))_\nu, 2^\mu,  \lambda^{\theta, q_0}, \lambda^{\theta, q_1}) )_\mu ||_{\eta, q}\\
\intertext{by definition of $(\eta, q; J, \lambda^{\theta, q_0}, \lambda^{\theta, q_1})$-norm}
&=
C \inf_{f = \sum_\mu u_\mu} || (J( u_\mu, 2^\mu, [X_0, X_1]_{\theta, q_0}, [X_0, X_1]_{\theta, q_1}) )_\mu ||_{\eta, q}\\
&=
C || f ||_{\eta, q; J, [X_0, X_1]_{\theta, q_0}, [X_0, X_1]_{\theta, q_1}}.
\end{align*}

For the converse suppose that $f \in [X_0, X_1]_{\theta, q}$.
Then we have that
\begin{multline*}
K(f,t, [X_0,X_1]_{\theta,q_0}, [X_0,X_1]_{\theta,q_1})\\
\leq
\inf_{f= \sum_\nu u_\nu; J(u_\nu, 2^\nu, X_0, X_1) = a_{0\nu} + a_{1\nu}}
|| \sum_\nu \frac{a_{0\nu}}{a_{0\nu} + a_{1\nu}} u_\nu ||_{\theta, q_0; J, X_0, X_1}\\
+
t || \sum_\nu \frac{a_{1\nu}}{a_{0\nu} + a_{1\nu}} u_\nu ||_{\theta, q_1; J, X_0, X_1}
\end{multline*}
by definition of $K$ and because $\sum_\nu \frac{a_{0\nu}}{a_{0\nu} + a_{1\nu}} u_\nu + \sum_\nu \frac{a_{1\nu}}{a_{0\nu} + a_{1\nu}} u_\nu = f$,
\[
\dots\leq
\inf || (J(\frac{a_{0\nu}}{a_{0\nu} + a_{1\nu}} u_\nu, 2^\nu, X_0, X_1))_\nu ||_{\theta, q_0}
+ t || (J(\frac{a_{1\nu}}{a_{0\nu} + a_{1\nu}} u_\nu, 2^\nu, X_0, X_1))_\nu ||_{\theta, q_1}
\]
by definition of the $(\theta, q_j; J, X_0, X_1)$-norm,
\begin{align*}
\dots &=
\inf || (a_{0\nu} )_\nu ||_{\theta, q_0}
+ t || (a_{1\nu})_\nu ||_{\theta, q_1}\\
&=
\inf_{f= \sum_\nu u_\nu}
K( (J(u_\nu, 2^\nu, X_0, X_1))_\nu, t, \lambda^{\theta, q_0}, \lambda^{\theta, q_1} ).
\end{align*}
Inserting this result into the definition of the $(\eta, q; K, [X_0, X_1]_{\theta, q_0}, [X_0, X_1]_{\theta, q_1})$-norm we obtain
\begin{align*}
|| f ||&_{\eta, q; K, [X_0, X_1]_{\theta, q_0}, [X_0, X_1]_{\theta, q_1}}\\
&=
|| (K( f, 2^\mu, [X_0, X_1]_{\theta, q_0}, [X_0, X_1]_{\theta, q_1}) )_\mu ||_{\eta, q}\\
&\leq
C || ( \inf_{f = \sum_\nu u_\nu} K( (J(u_\nu,2^\nu, X_0, X_1))_\nu, 2^\mu,  \lambda^{\theta, q_0}, \lambda^{\theta, q_1}) )_\mu ||_{\eta, q}\\
&\leq
C \inf_{f = \sum_\nu u_\nu} || (J(u_\nu,2^\nu, X_0, X_1))_\nu ||_{\eta, q; K, \lambda^{\theta, q_0}, \lambda^{\theta, q_1}}\\
&=
C \inf_{f = \sum_\nu u_\nu} || (J(u_\nu,2^\nu, X_0, X_1))_\nu ||_{\theta, q}\\
&=
C || f ||_{\theta, q; J, X_0, X_1}.
\qedhere
\end{align*}
\end{proof}
\begin{corollary}
\label{cor:lorentz-interpolation-fine}
Let $1 < p < \infty$ and $1 \leq q_0 < q < q_1 \leq \infty$.
Then
\[
[ L^{p, q_0}, L^{p, q_1} ]_{\eta, q} = L^{p,q},
\]
where $\eta$ is given by
\[
\frac1q = \frac{1-\eta}{q_0} + \frac{\eta}{q_1}.
\]
\end{corollary}

\section{The Hardy-Littlewood maximal function}
Let $f$ be a measurable function on $\R^n$.
The (non-centered) \emph{Hardy-Littlewood maximal function} is defined by
\index{Hardy-Littlewood maximal!function}
\[
Mf(x) = \sup_{Q:x\in Q} \fint_Q |f|,
\]
the supremum being taken over all open cubes containing $x$ with edges parallel to the axes.
By the monotone convergence theorem we may restrict the above supremum to cubes with rational coordinates, so that $Mf$ is always measurable.
Clearly, $M$ is subadditive and
\[
||Mf||_\infty \leq ||f||_\infty.
\]
In this section we will obtain a Lorentz space estimate for $Mf$, $f\in L^1$ and conclude that $M$ is continuous on $L^p$, $p>1$.

We need an alternative characterization of $L^{p, \infty}$ in terms of the distribution function first.
\begin{proposition}
\label{prop:Lpinfty-lambda}
For a measurable function $f$
\[
\sup_t t^{1/p} f^*(t) = \sup_s s \lambda_f(s)^{1/p}.
\]
In particular, $f \mapsto \sup_s s \lambda_f(s)^{1/p}$ is an equivalent norm on $L^{p,\infty}$.
\end{proposition}
\begin{proof}
Assume first that $\sup_t t^{1/p} f^*(t) = C < \infty$.
Then, for every $0<s<\infty$,
\begin{multline*}
\lambda_f(s) = \lambda_{f^*}(s) = |\{ f^* > s \}| = \sup\{ t : f^*(t) > s \}\\
\leq
\sup \{ t : C t^{-1/p} > s \} = \sup \{ t : t < C^p s^{-p} \} = C^p s^{-p},
\end{multline*}
whence $\sup_s s \lambda_f(s)^{1/p} \leq C$.
For the converse assume the latter inequality.
Then
\[
f^*(t) = \sup\{ s : \lambda_f(s) \geq t \}
\leq
\sup\{ s : C^p s^{-p} \geq t \} = C t^{-1/p},
\]
which proves the claim.
\end{proof}

The remaining part of the argument is combinatorial in nature.
\begin{lemma}[{Vitali covering lemma}]
\index{Vitali covering lemma}
\label{lemma:vitali}
Let $\{ Q_j \}_{j \in J}$ be a collection of cubes in $\R^n$ such that $\union_{j\in J} Q_j$ is bounded.
Then there exists a countable subset $J' \subset J$ such that $\{ Q_j \}_{j \in J'}$ are pairwise disjoint and $\union_{j \in J'} 5 Q_j \supseteq \union_{j \in J} Q_j$ ($5 Q$ is the cube with the same center as $Q$ and whose edge length is five times bigger).
\end{lemma}
\begin{proof}
Choose $j_1 \in J$ such that $Q_{j_1}$ has almost (up to a factor of $1+\epsilon$ with choice of $\epsilon$ depending on the dimension of the ambient space) maximal measure and then, inductively, choose $j_{k+1}$ such that $Q_{j_{k+1}}$ has almost maximal measure among the cubes disjoint from each of $Q_{j_1}, \dots, Q_{j_k}$ indefinitely or until no such cube exists.

Then $J'=\{j_1, \dots \}$ has the required properties: $\{ Q_j \}_{j \in J'}$ are pairwise disjoint by construction.
Since all cubes are contained in a bounded set, this implies $|Q_{j_k}| \to 0$.
Hence every $Q_{j \in J\setminus J'}$ intersects a cube $Q_{j' \in J'}$ of measure at least $|Q_j|/(1+\epsilon)$, because otherwise $j \in J'$.
Therefore $Q_j \subseteq 5 Q_{j'}$ if $\epsilon>0$ is small enough.
\end{proof}

\begin{lemma}
\label{lem:hardy-littlewood-maximal-inequality-L1}
Let $f \in L^1(\R^n)$.
Then for every $s>0$, $\lambda_{Mf}(s) = \Leb[1]\{ |Mf| > s \} \leq C ||f||_1 / s$.
\end{lemma}
\begin{proof}
By definition of $Mf$ we have
\[
\{ |Mf| > s \}
=
\Union_{Q : \fint_Q |f| > s} Q
=
\Union_m \Union_{Q : \fint_Q |f| > s, |Q| > 2^{-m}} Q
=:
\Union_m A_m.
\]
Since $A_m$ grows with $m$, it is sufficient to obtain a uniform estimate on $|A_m|$.

The set $A_m$ is bounded because otherwise there exists an infinite disjoint collection of cubes $Q$ such that $|Q|$ and $\fint_Q |f|$ are bounded from below, which contradicts $f \in L^1$.
Hence Lemma~\ref{lemma:vitali} applies and there exists a disjoint collection $\{ Q_j : \fint_{Q_j} |f| > s, |Q_j| > 2^{-m} \}$ such that
\[
|A_m|
\leq
5^n \left| \union_j Q_j \right|
\leq
5^n \frac1s \int_{\union_j Q_j} |f|
\leq 
5^n ||f||_1 / s.
\qedhere
\]
\end{proof}

By Lemma~\ref{prop:Lpinfty-lambda}, this implies $||Mf||_{1,\infty} \leq C ||f||_1$.
An application of the Marcin\-kie\-wicz interpolation theorem~\ref{th:marcinkiewicz} yields

\begin{theorem}
\index{Hardy-Littlewood maximal!inequality}
\label{th:hardy-littlewood-maximal-inequality}
For $1<p\leq\infty$ the maximal operator $M$ is bounded on $L^p$ and
\[
||Mf||_p \leq Cp' ||f||_p.
\]
\end{theorem}

Proposition~\ref{prop:Lpinfty-lambda} is also useful for the calculation of the $L^{p,\infty}$-norm of highly symmetrical functions, as the example below shows.
\begin{proposition}
\label{prop:sphere-Lp-Rn-Lpinfty}
Let $1 < p < \infty$ and $g \in L^p(S^{n-1})$.
Then
\[
f(x) := |x|^{-n/p} g(x/|x|) \in L^{p,\infty}(\R^n)
\]
and $||f||_{p, \infty} = C ||g||_p$.
\end{proposition}
\begin{proof}
The scaling behavior of $f$ implies
\[
\{x : |f(x)| > s \} = \{x=s^{-p/n} y : |f(y)| > 1\} = s^{-p/n} \{ y : |f(y)| > 1 \}.
\]
Hence
\begin{align*}
||f||_{p, \infty}
&=
\sup_{s>0} s \Leb \{ x : |f(x)| > s \}^{1/p}\\
&=
\sup_{s>0} \Leb \{ y : |f(y)| > 1 \}^{1/p}\\
&=
C \left( \int_{S^{n-1}} \left( |g(\omega)|^{p/n} \right)^n \dif\omega \right)^{1/p}\\
&=
C ||g||_p.
\qedhere
\end{align*}
\end{proof}

A frequently used special case occurs when $g=1$.
Then $f(x) = |x|^{-n/p}$ and $f \in L^{p,\infty}(\R^n)$.
This observation is most useful in conjunction with the weak-type Young inequality (Proposition~\ref{prop:weak-type-young}).

This readily implies the boundedness of the following operator.
\begin{proposition}
\label{prop:Rn-Lp1-sphere-Lp}
Let $f \in \Schwartz(\R^n)$, $1 < p < \infty$ and
\[
U_p f(\omega) = \int_\R |t|^{n/p-1} f(t \omega) \dif t, \quad \omega\in S^{n-1}.
\]
Then $||U_p f||_p \leq C ||f||_{p,1}$.
\end{proposition}
\begin{proof}
For every $g \in L^{p'}(S^{n-1})$,
\begin{align*}
\int_{S^{n-1}} g(\omega) U_p f(\omega) \dif\omega
&=
C \int_{\R^n} g(x/|x|) |x|^{n/p-1} f(x) |x|^{-n+1} \dif x\\
&\leq
C ||f||_{p,1} || |x|^{-n/p'} g(x/|x|) ||_{p', \infty}\\
&=
C ||f||_{p,1} || g ||_{p'}.
\qedhere
\end{align*}
\end{proof}

\section{Interpolation between dual spaces}
\label{sec:real-int-dual}
Suppose that $X_0 \intersection X_1$ is dense both in $X_0$ and in $X_1$.
Then it makes sense to define the intersection of the dual spaces $X_0'$ and $X_1'$ by
\[
\phi_0 \in X_0' = \phi_1 \in X_1' \text{ if and only if } \phi_0|_{X_0 \intersection X_1} = \phi_1|_{X_0 \intersection X_1}.
\]
Denseness of $X_0 \intersection X_1$ is necessary and sufficient to ensure that each $\phi_0$ equals at most one $\phi_1$ and vice versa.
For the rest of the section let $X_0$, $X_1$, $0 < \theta < 1$ and $1\leq q \leq \infty$ be fixed.
We will be concerned with the interpolation space $[X_0', X_1']_{\theta,q'} \subset X_0' + X_1'$.

\begin{proposition}
\label{prop:sum-of-duals}
There exists a natural isometric isomorphism
\[
\Phi : X_0' + X_1' \to (X_0 \intersection X_1)'
\]
given by $\Phi(\phi_0 + \phi_1)(x):=\phi_0(x) + \phi_1(x)$ for every $x \in X_0 \intersection X_1$.
\end{proposition}
\begin{proof}
We will first show that $\Phi$ is a bijection and then that it is also isometric.

\emph{Injectivity.}
Assume that $\phi_0(x) + \phi_1(x) = \tilde\phi_0(x) + \tilde\phi_1(x)$ for every $x \in X_0 \intersection X_1$.
By definition this is equivalent to $(\phi_0 - \tilde\phi_0) = - (\phi_1 - \tilde\phi_1) \in X_0' \intersection X_1'$.
Hence
\[
\phi_0 + \phi_1
=
(\phi_0 - (\phi_0 - \tilde\phi_0)) + (\phi_1 - (\phi_1 - \tilde\phi_1))
=
\tilde\phi_0 + \tilde\phi_1.
\]

\emph{Surjectivity.}
Let $\phi \in (X_0 \intersection X_1)'$.
Then $\psi(x,x):=\phi(x)$ is a continuous linear form on the space $\{ (x,x) : x \in X_0 \intersection X_1 \} \subseteq X_0 \oplus X_1$ endowed with the norm $||(x_0,x_1)||_{X_0 \oplus X_1} = \max\{||x_0||, ||x_1||\}$ and $||\psi||=||\phi||$.
Therefore $\psi$ extends to a linear form $\tilde\psi$ on $X_0 \oplus X_1$ with $||\tilde\psi||=||\phi||$.
Let
\[
\phi_0(x_0) := \tilde\psi(x_0,0) \text{ and } \phi_1(x_1) := \tilde\psi(0,x_1).
\]
Then $\phi(x) = \psi(x,x) = \phi_0(x) + \phi_1(x)$.

\emph{Contractivity.}
We have that
\begin{align*}
||\Phi(\phi_0 + \phi_1)||
&=
\sup_{||(x,x)||_{X_0 \oplus X_1} \leq 1} |\phi_0(x) + \phi_1(x)|\\
&=
\inf_{\phi_0 + \phi_1 = \tilde\phi_0 + \tilde\phi_1} \sup_{||(x,x)||_{X_0 \oplus X_1} \leq 1} |\tilde\phi_0(x) + \tilde\phi_1(x)|\\
&\leq
\inf_{\phi_0 + \phi_1 = \tilde\phi_0 + \tilde\phi_1} \sup_{||(x_0,x_1)||_{X_0 \oplus X_1} \leq 1} |\tilde\phi_0(x_0)| + |\tilde\phi_1(x_1)|\\
&=
\inf_{\phi_0 + \phi_1 = \tilde\phi_0 + \tilde\phi_1} ||\tilde\phi_0||_{X_0'} + ||\tilde\phi_1||_{X_1'}\\
&=
|| \phi_0 + \phi_1 ||_{X_0' + X_1'}.
\end{align*}

\emph{Contractivity of the inverse.}
Let $\phi \in (X_0 \intersection X_1)'$ and $\phi_0$, $\phi_1$ be the linear forms constructed in the proof of surjectivity.
Then
\begin{multline*}
|| \phi_0 + \phi_1 ||_{X_0' + X_1'}
\leq
||\phi_0||_{X_0'} + ||\phi_1||_{X_1'}
=
\sup_{||x_0||_{X_0} \leq 1, ||x_1||_{X_1} \leq 1} |\phi_0(x_0) + \phi_1(x_1)|\\
=
\sup_{||(x_0, x_1)||_{X_0 \oplus X_1} \leq 1} |\tilde\psi(x_0,x_1)|
= ||\tilde\psi|| = ||\phi||.
\qedhere
\end{multline*}
\end{proof}
Henceforth we shall identify $\phi$ and $\Phi\phi$.
The preceding proposition implies the following relationship between the $K$- and the $J$-functional.
\begin{multline}
\label{eq:K-sup-Jinv}
K(\phi, t, X_0', X_1')
= ||\phi||_{X_0' + t X_1'}
= ||\phi||_{X_0' + (t\inv X_1)'}
= ||\phi||_{(X_0 \intersection (t\inv X_1))'}\\
= \sup_{f \in X_0 \intersection (t\inv X_1)} \phi(f) / ||f||_{X_0 \intersection (t\inv X_1)}
= \sup_{f \in X_0 \intersection X_1} \phi(f) / J(f,t\inv,X_0, X_1).
\end{multline}

\begin{proposition}
\label{prop:intersection-dense-in-interpolation-space}
If $q < \infty$, then $X_0 \intersection X_1$ is dense in $[X_0, X_1]_{\theta, q}$.
\end{proposition}
\begin{proof}
Let $f \in [X_0, X_1]_{\theta, q}$.
There exists a decomposition $f = \sum_\nu f_\nu$ such that
\[
\left( \sum_{\nu=-\infty}^\infty (2^{-\theta \nu} J(f_\nu,2^\nu,X_0,X_1))^{q} \right)^{1/q} < \infty.
\]
Let $h_N = f - \sum_{|\nu|<N} f_\nu = \sum_{|\nu| \geq N} f_\nu$.
Then
\[
||h_N||_{\theta, q; J, X_0, X_1} \leq \left( \sum_{|\nu| \geq N} (2^{-\theta \nu} J(f_\nu,2^\nu,X_0,X_1))^{q} \right)^{1/q} \to 0.
\]
The assertion follows because $\sum_{|\nu|<N} f_\nu \in X_0 \intersection X_1$.
\end{proof}

\begin{proposition}
\index{real interpolation!between dual spaces}
\label{prop:interpolation-dual}
If $q < \infty$, then
\[
[X_0', X_1']_{\theta, q'} = [X_0, X_1]_{\theta, q}'.
\]
\end{proposition}
\begin{proof}
Let $\phi \in [X_0', X_1']_{\theta, q'}$ and consider an arbitrary decomposition $\phi = \sum_\nu \phi_\nu$  with $\phi_\nu \in X_0' \intersection X_1'$ and convergence in $X_0' + X_1'$.

For every $f \in X_0 \intersection X_1 \subset [X_0, X_1]_{\theta, q}$ set
\[
\<f, \phi\> := \sum_\nu \< f, \phi_\nu \>.
\]
This does not depend on the decomposition because $X_0' + X_1' = (X_0 \intersection X_1)'$ by Proposition~\ref{prop:sum-of-duals}.
The bilinearity of this pairing is clear, so that we only need to show continuity.

Consider a decomposition $f = f_0 + f_1 \in X_0 + X_1$ and observe that
\[
|\< f, \phi_\nu \>| \leq |\< f_0, \phi_\nu \>| + |\< f_1, \phi_\nu \>| \leq ||f_0||_{X_0} ||\phi_\nu||_{X_0'} + t ||f_1||_{X_1} t\inv ||\phi_\nu||_{X_1'}
\]
for every $t>0$.
Taking the infimum over decompositions $f=f_0 + f_1$ and setting $t=2^\nu$ we obtain that
\[
|\< f, \phi_\nu \>| \leq K(f,2^\nu,X_0,X_1) J(\phi_\nu,2^{-\nu}, X_0',X_1'),
\]
whence
\begin{align*}
|\<f, \phi\>|
&\leq
\sum_\nu |\< f, \phi_\nu \>|\\
&\leq
\sum_\nu 2^{-\theta\nu} K(f,2^\nu,X_0,X_1) 2^{\theta\nu} J(\phi_\nu,2^{-\nu}, X_0',X_1')\\
&\leq
|| f ||_{\theta, q; K, X_0, X_1} \left( \sum_{\nu=-\infty}^\infty (2^{-\theta \nu} J(\phi_\nu,2^\nu,X_0',X_1'))^{q'} \right)^{1/q'}
\end{align*}
by the Hölder inequality.
Taking the infimum over decompositions $\phi = \sum_\nu \phi_\nu$ we obtain
\[
|\<f, \phi\>| \leq || f ||_{\theta, q; K, X_0, X_1} || \phi ||_{\theta, q'; J, X_0', X_1'}.
\]
By Proposition~\ref{prop:intersection-dense-in-interpolation-space} the intersection $X_0 \intersection X_1$ is dense in $[X_0, X_1]_{\theta, q}$.
Therefore $\phi$ admits a unique extension to a linear form on $[X_0, X_1]_{\theta, q}$.

Conversely, let $\phi \in [X_0, X_1]_{\theta, q}'$.
Since the inclusion $X_0 \intersection X_1 \hookrightarrow [X_0, X_1]_{\theta, q}$ is continuous $\phi$ restricts to a continuous linear form on $X_0 \intersection X_1$.
Proposition~\ref{prop:sum-of-duals} allows us to regard $\phi$ as an element of $X_0' + X_1'$.
By (\ref{eq:K-sup-Jinv}) for every $\nu$ and $\epsilon_{-\nu}>0$ there exists an $f_\nu \in X_0 \intersection X_1$ such that
\[
K(\phi,2^{-\nu},X_0',X_1') - \epsilon_{-\nu} \leq |\phi(f_\nu)| / J(f_\nu, 2^\nu, X_0, X_1).
\]
Let $(\alpha_\nu)_\nu \in \lambda^{1-\theta, q}$ be an arbitrary positive sequence.
Then
\begin{align*}
\sum_\nu 2^{-\nu} & \alpha_\nu (K(\phi,2^{\nu},X_0',X_1') - \epsilon_{\nu})\\
&=
\sum_\nu 2^{\nu} \alpha_{-\nu} (K(\phi,2^{-\nu},X_0',X_1') - \epsilon_{-\nu})\\
&\leq
\sum_\nu 2^{\nu} \alpha_{-\nu} |\phi(f_\nu)| / J(f_\nu, 2^\nu, X_0, X_1)\\
&=
\phi \left( \sum_\nu 2^{\nu} \alpha_{-\nu} f_\nu / J(f_\nu, 2^\nu, X_0, X_1) \right)\\
&\leq
||\phi||_{[X_0, X_1]_{\theta, q}'} \left|\left| \sum_\nu 2^{\nu} \alpha_{-\nu} f_\nu / J(f_\nu, 2^\nu, X_0, X_1) \right|\right|_{[X_0, X_1]_{\theta, q}}\\
&\leq
||\phi||_{[X_0, X_1]_{\theta, q}'} \left|\left| ( 2^{\nu} \alpha_{-\nu} J(f_\nu, 2^\nu, X_0, X_1) / J(f_\nu, 2^\nu, X_0, X_1))_\nu \right|\right|_{\lambda^{\theta, q}}\\
&=
||\phi||_{[X_0, X_1]_{\theta, q}'} \left|\left| ( 2^{\nu} \alpha_{-\nu} )_\nu \right|\right|_{\lambda^{\theta, q}}\\
&=
||\phi||_{[X_0, X_1]_{\theta, q}'} \left|\left| ( \alpha_{\nu} )_\nu \right|\right|_{\lambda^{1-\theta, q}}.
\end{align*}
Since $(\lambda^{1-\theta, q})' = \lambda^{\theta, q'}$ with the dual pairing $\< (a_\nu)_\nu, (b_\nu)_\nu \> = \sum_\nu 2^\nu a_\nu b_\nu$, this implies that
\[
|| (K(\phi,2^{\nu},X_0',X_1') - \epsilon_{\nu})_\nu ||_{\lambda^{\theta, q'}} \leq ||\phi||_{[X_0, X_1]_{\theta, q}'}.
\]
Letting $\epsilon_\nu \to 0$ we see that
\[
||\phi||_{\theta, q'; X_0', X_1'} \leq ||\phi||_{[X_0, X_1]_{\theta, q}'}.
\qedhere
\]
\end{proof}

This immediately implies the following version of the Marcinkiewicz interpolation theorem~\ref{th:marcinkiewicz} for bilinear forms.
\begin{corollary}
\label{cor:marcinkiewicz-bilinear}
Let $A : (X_0 + X_1) \times (Y_0 + Y_1) \to \C$ be a bilinear form such that
\[
|A(x,y)| \leq C || x ||_{X_0} || y ||_{Y_0} \text{ and } |A(x,y)| \leq C || x ||_{X_1} || y ||_{Y_1}.
\]
If $0 < \theta < 1$ and $1 \leq q < \infty$, then
\[
|A(x,y)| \leq C || x ||_{\theta, q'; X_0, X_1} || y ||_{\theta, q; Y_0, Y_1}.
\]
\end{corollary}

With the interpolation theorem for bilinear forms at hand we are able to obtain an interpolation result for multilinear forms as well.
The following version will be useful later on.
\begin{proposition}
\label{prop:interpolation-multilinear}
Let $A(f_0, \dots, f_n)$ be a multilinear form that satisfies
\[
|A(f_0, \dots, f_n)| \leq C ||f_j||_{1,1} \Prod_{j' \neq j} ||f_{j'}||_{n/k,1}
\]
for every $0 \leq j \leq n$.
Then
\[
|A(f_0, \dots, f_n)| \leq C \Prod_{j} ||f_{j}||_{(n+1)/(k+1),n+1}.
\]
\end{proposition}
\begin{proof}
Let $f_2, \dots, f_n$ be fixed.
Then by the assumption
\[
A(\cdot, \cdot, f_2, \dots, f_n) : L^{1,1} \times L^{n/k,1} \to \C, L^{n/k,1} \times L^{1,1} \to \C.
\]
Since the space of simple functions is dense both in $L^{n/k,1}$ and in $L^{1,1}$, we can apply the Marcinkiewicz interpolation theorem for bilinear forms (Corollary~\ref{cor:marcinkiewicz-bilinear}) with $q=1$.
Together with Theorem~\ref{th:lorentz-interpolation} it implies
\[
A(\cdot, \cdot, f_2, \dots, f_n) : L^{(n+1)/(k+1),\infty} \times L^{p,1} \to \C,
\]
where
\[
\frac1p = k+1 - 1 \cdot \frac{k+1}{n+1} - (n-1) \cdot \frac{k}{n}.
\]
Therefore and by symmetry we have
\[
A(f_0, \dots, f_n) \leq C ||f_{0}||_{(n+1)/(k+1),\infty} ||f_j||_{p,1} \Prod_{j' \geq 1, j' \neq j} ||f_{j'}||_{n/k,1}
\]
for every $j \geq 1$.
Now let $f_0, f_3, \dots, f_n$ be fixed and interpolate in $f_1$ and $f_2$.
Continuing in the same way we obtain
\[
A(f_0, \dots, f_n) \leq C ||f_{0}||_{(n+1)/(k+1),\infty} \dots ||f_{m-1}||_{(n+1)/(k+1),\infty} ||f_j||_{p,1} \Prod_{j' \geq m, j' \neq j} ||f_{j}||_{n/k,1},
\]
for every $m$ and every $j\geq m$ with
\[
\frac1p = k+1 - m \cdot \frac{k+1}{n+1} - (n-m) \cdot \frac{k}{n}.
\]
Taking $m=n$ we obtain
\[
A(f_0, \dots, f_n) \leq C ||f_j||_{(n+1)/(k+1),1} \Prod_{j' \neq j} ||f_{j'}||_{(n+1)/(k+1),\infty}
\]
for every $j$.
Now we repeat the procedure using Corollary~\ref{cor:lorentz-interpolation-fine} to interpolate in the minor exponent.
\end{proof}

Another consequence of the interpolation theorem for dual spaces is the weak-type Young inequality.
\begin{proposition}
\label{prop:weak-type-young}
Let $1 < p,q < \infty$ be such that
\[
0 < \frac{1}{r} = \frac1q + \frac1p - 1 < 1.
\]
and $f \in L^{p,\infty}(\R^n)$ and $g \in L^q(\R^n)$.
Then
\[
||f * g||_{r} \leq C ||f * g||_{r,q} \leq C ||f||_{p,\infty} ||g||_q.
\]
\end{proposition}
\begin{proof}
The first inequality follows from Corollary~\ref{cor:lorentz-scale} because $q \leq r$.

For the second inequality observe that
\[
||f * g||_{\infty} \leq C ||f||_{\infty} ||g||_1
\text{ and }
||f * g||_{1} \leq C ||f||_{1} ||g||_1,
\]
by duality and Fubini's theorem, respectively.
Since $f*g$ is linear in $f$, the Marcinkiewicz interpolation theorem~\ref{th:marcinkiewicz} implies that
\[
||f * g||_{p,\infty} \leq C ||f||_{p,\infty} ||g||_1.
\]
Moreover, by duality (Proposition~\ref{prop:interpolation-dual}) we have that
\[
||f * g||_{\infty} \leq C ||f||_{p,\infty} ||g||_{p',1}.
\]
Since $f*g$ is also linear in $g$, we can once again apply the Marcinkiewicz interpolation theorem~\ref{th:marcinkiewicz} and obtain
\[
||f * g||_{r,s} \leq C ||f||_{p,\infty} ||g||_{q,s}
\]
for every $1 \leq s \leq \infty$.
Setting $s=q$ yields the claim.
\end{proof}

Since $|\cdot|^{-n/p} \in L^{p,\infty}(\R^n)$ by Proposition~\ref{prop:sphere-Lp-Rn-Lpinfty}, we immediately obtain the Hardy-Littlewood-Sobolev theorem on fractional integration.
\begin{corollary}
\label{cor:hardy-littlewood-sobolev}
Let $1 < q,r < \infty$ and $0 < \alpha < n$ satisfy $\frac{1}{r} + 1 = \frac1q + \frac{\alpha}{n}$.
Then
\[
||f||_{r} \leq C ||g||_q,
\]
whenever $g \in L^q(\R^n)$ and
\[
f(x) = \int_{\R^n} \frac{g(y)}{|x-y|^\alpha} \dif y.
\]
\end{corollary}

\section{Supplement: Maximal inequality in large dimension}
The constants obtained in Lemma~\ref{lem:hardy-littlewood-maximal-inequality-L1} grow exponentially in $n$.
Here we briefly discuss an estimate with better asymptotic behavior for the centered maximal function associated to the standard ball in $\R^n$ due to Stein and Strömberg \cite{MR727348}.

Their method makes use of the heat semigroup
\[
T_t f = f * h_t, \quad h_t(y) = (4\pi t)^{-n/2} e^{-|y|^2/(4t)}.
\]
The operators $T_t$ are positive complete contractions on $L^1$ (i.e.\ $T_t f\geq 0$ whenever $f\geq 0$, $||T_t f||_1 \leq ||f||_1$, and $||T_t f||_\infty \leq ||f||_\infty$).
Operators satisfying the latter two estimates are also sometimes called \emph{Dunford-Schwartz operators}.
Moreover the semigroup $(T_t)_{t>0}$ is strongly continuous, i.e.~the map $t\mapsto T_t f$ is $L^1$-norm continuous for every $f \in L^1$.

These properties of $(T_t)_{t>0}$ ensure that the Hopf mean ergodic theorem applies.
The classical reference is the book of Dunford and Schwartz \cite[VIII.7]{MR1009162}.
The proof of the Hopf lemma presented here (due to Garcia) may be found e.g.~in Krengel's book \cite{MR797411}.
\begin{lemma}[Hopf]
\index{Hopf's lemma}
\label{lem:Hopf}
Let $(\Omega, \mu)$ be a measure space and $T$ be a positive complete contraction on $L^1(\Omega,\mu) + L^\infty(\Omega, \mu)$.
Let $S_k := \sum_{j=0}^{k-1} T^j$ denote the sums of iterates of $T$ and $M_n f := \sup_{1 \leq k \leq n} S_k f$ be the associated maxima.
Then
\[
\int_{\{M_n f \geq 0\}} f \dif\mu \geq 0
\]
for every $n$ and every $f \in L^1 - L^\infty_+$, where $L^\infty_+$ denotes the positive cone of $L^\infty$.
\end{lemma}
\begin{proof}
By definition of $M_n$ and positivity of $T$ we have
\[
S_k f = f + TS_{k-1}f \leq f + T(M_n f)^+
\]
for every $2 \leq k \leq n$ (Here $f^+$ denotes the positive part of $f$).
The corresponding estimate for $S_1 f = f$ is trivial.
Taken together these inequalities imply
\[
M_n f \leq f + T (M_n f)^+.
\]
By positivity of $T$ we have $(M_n f)^+ \in L^1$ and therefore
\begin{multline*}
\int_{\{M_n f \geq 0\}} f \dif\mu
\geq
\int_{\{M_n f \geq 0\}} M_n f - T(M_n f)^+ \dif\mu
=\\
\int (M_n f)^+ \dif\mu - \int_{\{M_n f \geq 0\}} T(M_n f)^+ \dif\mu
\geq
\int (M_n f)^+ \dif\mu - \int T(M_n f)^+ \dif\mu\\
\geq
0.
\qedhere
\end{multline*}
\end{proof}
Since $f(x) \leq 0$ whenever $M_nf(x) = 0$, Hopf's Lemma~\ref{lem:Hopf} immediately implies that
\[
\int_{\{M_n f > 0\}} f \dif\mu \geq 0
\]
(note that the domain of integration has changed).
\begin{theorem}
\index{maximal ergodic theorem!discrete case}
\label{thm:max-ergodic-discrete}
Let $T$ be as above, $A_k := \frac1k S_k$ be the weighted averages of iterates of $T$ and $A_n^*[T] f = A_n^* f := \sup_{1\leq k\leq n} A_k f$.
Then, for every $f \in L^1$ and $\lambda > 0$,
\[
\lambda \mu\{ A_n^* f \geq \lambda \} \leq \int_{\{A_n^* f > \lambda\}} f \dif\mu.
\]
In particular, the maximal operator $A^*[T] f = A^* f := \sup_{k \in \N} A_k f$ satisfies
\[
\lambda \mu\{ A^* f \geq \lambda \} \leq ||f||_1.
\]
\end{theorem}
\begin{proof}
\newcommand{\eins}{\mbox{$1 \hspace{-2.25pt} \mathrm{l}$}}
The latter assertion follows from the former by the monotone convergence theorem.
Furthermore we only need to verify the former with $\{A_n^* f > \lambda\}$ in place of $\{A_n^* f \geq \lambda\}$.
For this end we consider the function $g = f - \lambda \eins$.
By $L^\infty$-contractivity of $T$ we have
\[
T^j g = T^j f - \lambda T^j \eins \geq T^j f - \lambda \eins.
\]
Taking the appropriate sums and suprema we obtain
\[
A_n^* g \geq A_n^* f - \lambda \eins.
\]
This implies
\[
\{g>0\} = \{f>\lambda\} \subseteq \{A_n^* f > \lambda\} \subseteq \{A_n^* g > 0\} = \{M_n g > 0\}.
\]
In particular,
\[
\int_{\{ A_n^* f > \lambda \}} (f - \lambda \eins) \dif\mu
=
\int_{\{ A_n^* f > \lambda \}} g \dif\mu
\geq
\int_{\{M_n g > 0\}} g \dif\mu
\geq
0
\]
by the remark following Hopf's Lemma~\ref{lem:Hopf}.
\end{proof}

\begin{theorem}
\index{maximal ergodic theorem!continuous case}
Let $(T_t)_{t>0}$ be a strongly continuous semigroup of positive complete contractions, $B_s := \frac1s \int_0^s T_t f \dif t$ be the weighted averages of the semigroup and $B^* f := \sup_{s>0} B_s f$.
Then, for every $\lambda > 0$,
\[
\lambda \mu\{ B^* f > \lambda \} \leq ||f||_1.
\]
\end{theorem}
Note that the supremum in the definition of $B^*$ has to be taken with respect to the Banach lattice structure on $L^1$, because the parameter varies over an uncountable set.
\begin{proof}
By strong continuity we can reduce to a countable supremum in the definition of $B^*$, i.e.
\[
B^* f = \sup_{s>0, s \in \Q} B_s f,
\]
the supremum now being equivalent to the pointwise supremum.
Again by strong continuity,
\[
B_s f = \lim_{k\to \infty} \frac1{\lceil s k! \rceil} \sum_{j=0}^{\lceil s k! \rceil - 1} T\left(\frac{j}{k!}\right) f
\]
for every rational $s$.
Passing to a subsequence by means of a diagonal argument we may assume pointwise convergence almost everywhere.
Since the expression following the limit symbol is bounded by $A^*[T(1/k!)] f$, we obtain
\[
B^* f \leq \liminf_{i \to \infty} A^*[T(1/k_i!)] f 
\]
pointwise almost everywhere.
Therefore
\[
\{ B^* f > \lambda \} \subseteq \Union_N \Intersection_{i=N}^\infty \{ A^*[T(1/k_i!)] f > \lambda \}.
\]
Since the union in question is increasing and by Theorem~\ref{thm:max-ergodic-discrete} the claim follows.
\end{proof}
In order to obtain the Hardy-Littlewood maximal inequality we estimate $M f$ by $B^{*} f$.
\begin{lemma}
\label{lem:ball-gaussian}
There exists a constant $C$ such that
\[
\Leb[n](B(0,1))\inv \leq C n \frac{1}{1/n} \int_0^{1/n} h_t(1) \dif t
\]
for every $n \geq 3$, where $B(0,1) \subset \R^n$ is the standard unit ball.
\end{lemma}
The statement of the lemma includes an abuse of the notation because $h_t$ is defined on $\R^n$.
It is justified by radial symmetry.

Before giving the proof we infer the maximal inequality.
Observe that it suffices to consider non-negative functions.
By the change of the variable $t' = r^2 t$ we obtain
\[
\Leb[n](B(0,r))\inv
\leq
C n \frac{1}{r^n/n} \int_0^{1/n} h_t(1) \dif t
=
C n \frac{1}{r^2/n} \int_0^{r^2/n} h_{t'}(r) \dif t'.
\]
Since $h_t$ is non-increasing this implies that
\[
M f(x)
=
\sup_r (f * \Leb[n](B)\inv \chi_B )(x)
\leq
Cn B^* f,
\]
and we obtain $\Leb[n]\{ Mf > \lambda \} \leq ( Cn/\lambda ) ||f||_1$.

\begin{proof}[Proof of Lemma~\ref{lem:ball-gaussian}]
Change of variable $s=1/(4t)$ yields
\begin{multline*}
\int_0^{1/n} (4\pi t)^{-n/2} e^{-1/(4t)} \dif t
=
\frac14 \pi^{-n/2} \int_{n/4}^\infty s^{n/2-2} e^{-s} \dif s\\
=
\frac14 \pi^{-n/2} \left( \Gamma(n/2-1) - \int_0^{n/4} s^{n/2-2} e^{-s} \dif s \right).
\end{multline*}
In order to deal with the latter integral observe that there exists an $a$ such that $1 < e/2 < a < \frac{1}{2-2 \log 2}$.
For such an $a$ we have
\[
\int_0^{n/(4a)} s^{n/2-2} e^{-s} \dif s
\leq
\int_0^{n/(4a)} s^{n/2-2} \dif s
=
\frac{1}{n/2-1} \left( \frac{n}{4a} \right)^{n/2-1}
=
O\left( \frac{n^{n/2-2}}{(4a)^{n/2}} \right)
\]
and
\begin{multline*}
\int_{n/(4a)}^{n/4} s^{n/2-2} e^{-s} \dif s
\leq
e^{-n/4a} \int_0^{n/4} s^{n/2-2} \dif s\\
=
\frac{e^{-n/4a}}{n/2-1} \left( \frac{n}{4} \right)^{n/2-1}
=
O\left( \frac{e^{-n/4a} n^{n/2-2}}{4^{n/2}} \right).
\end{multline*}
By the Stirling formula both these quantities are $o(\Gamma(n/2-1))$, so that
\begin{multline*}
\int_0^{1/n} (4\pi t)^{-n/2} e^{-1/(4t)} \dif t
\geq
C \pi^{-n/2} \Gamma(n/2-1)\\
=
C \pi^{-n/2} \Gamma(n/2)/(n/2-1)
\geq
C n^{-2} \Leb[n](B(0,1))\inv.
\qedhere
\end{multline*}
\end{proof}


%% file: chapter_complex_interpolation.tex
\chapter{Complex interpolation}
\label{chap:complex}
In this chapter we review a complex interpolation method due to Calderón \cite{MR0167830} and calculate the corresponding interpolation spaces between various $L^p$ spaces.
The main advantage of the complex method is the possibility to interpolate estimates for an analytic family of operators, as opposed to a single operator in the real case.

The basic tool is the three lines lemma which is a maximum principle for holomorphic functions on an unbounded strip in $\C$.
The strongest variant of the three lines lemma is proved using the principle of harmonic majoration for subharmonic functions.

\section{Harmonic majoration}
We will need to majorize the logarithm of the modulus of a holomorphic function $f$ by a harmonic function on a bounded domain $\Omega$ given the majoration on $\boundary\Omega$.
We split our considerations in two parts.
First we show that $\log|f|$ is subharmonic and then we prove the principle of harmonic majoration for subharmonic functions.
\begin{definition}
\index{subharmonic function}
A function $f$ is said to be subharmonic in $\Omega \subset \C$, if for every closed ball $\bar B (z,r) \subset \Omega$,
\[
f(z) \leq \frac{1}{2\pi} \int_{-\pi}^\pi f(z + r e^{i\theta}) \dif\theta.
\]
\end{definition}

We begin with the observation that if $f$ has no zeroes, then $\log|f|$ is in fact harmonic (and we are done by the maximum principle for harmonic functions).
\begin{proposition}
\label{prop:simply-conn-log}
Let $\Omega \subset \C$ be simply connected.
Then for every holomorphic function $f$ which does not vanish identically on $\Omega$ there exists a holomorphic function $g$ such that $f=e^g$.
\end{proposition}
\begin{proof}
Since $\Omega$ is simply connected and $f$ does not vanish, it can be lifted to a holomorphic function taking values in the universal covering $\C^*$ of $\C \setminus \{0\}$.
The logarithm is a holomorphic function on $\C^*$, so that the composition of the logarithm with the lift is the holomorphic function one is looking for.
\end{proof}

The following auxiliary lemma is a nice application of the Cauchy integral theorem to a definite integral.
\begin{lemma}[{\cite[15.17]{MR924157}}]
The following identity holds
\label{lemma:int-ln-1-e-itheta}
\[
\int_{\theta = 0}^{2 \pi} \log |1 - e^{i \theta}| \dif \theta = 0.
\]
\end{lemma}
\begin{proof}
By Proposition~\ref{prop:simply-conn-log} applied to $\Omega = \{ \Re z < 1 \}$, there exists a $g \in \hol\Omega$ such that
\[
e^{g(z)} = 1 - z.
\]
Clearly, one can assume that $g(0)=0$, so that $g(z) / z$ extends to a holomorphic function on $\Omega$.
Now consider the paths in the picture and use the Cauchy theorem to infer
\marginpar{\resizebox{\marginparwidth}{!}{
\begin{tikzpicture}[scale=1.5]
\fill (0,0) circle (1pt) node[below] {$0$};
\fill (1,0) circle (1pt) node[below] {$1$};

\begin{scope} 
\clip (1,0) circle (0.5) -- (1.1,1.1) -- (1.1,-1.1) -- (-1.1,-1.1) -- (-1.1,1.1) -- (1.1,1.1) -- cycle;
\draw (0,0) circle (1);
\end{scope}
\path (-1,0) node[anchor=east] {$\gamma_\delta$};

\begin{scope} 
\clip (0,0) circle (1);
\draw (1,0) circle (0.5);
\end{scope}
\path (0.5,0) node[anchor=east] {$\gamma'_\delta$};

\draw[->] (1,0) -- node[anchor=south west] {$\delta$} +(120:0.5);
\end{tikzpicture}
}}
\begin{align*}
\int_{\theta = 0}^{2 \pi} \log |1 - e^{i \theta}| \dif \theta
&=
\lim_{\delta \to 0} \int_{\theta = \delta}^{2 \pi-\delta} \Re g(e^{i \theta}) \dif \theta\\
&=
\lim_{\delta \to 0} \Re \left( \int_{\gamma_\delta} \frac{g(z)}{i z} \dif z \right)\\
&=
\lim_{\delta \to 0} \Re \left( \int_{\gamma'_\delta} \frac{g(z)}{i z} \dif z \right)\\
&=
0,
\end{align*}
since the expression in the parentheses is of order $\delta \log(\delta)$.
\end{proof}

Next we prove a version of the principle that the value of a harmonic function at a point is equal to its mean value on a sphere centered at this point (as we have already observed, $\log|f|$ is harmonic if $f$ does not vanish anywhere).
\begin{proposition}[{Jensen's formula, \cite[15.18]{MR924157}}]
\index{Jensen's formula}
\label{prop:jensen-formula}
Let $\Omega = B(0,R)$, $f \in \hol\Omega$, $f(0) \neq 0$, $0 < r < R$, and $\alpha_1, \dots, \alpha_N$ be the zeroes of $f$ in $\bar B(0,r)$ with multiplicity.
Then
\begin{equation}
\label{eq:jensens-formula}
|f(0)| \Prod_{n=1}^N \frac{r}{|\alpha_n|} = \exp \left( \frac{1}{2\pi} \int_{-\pi}^\pi \log |f(r e^{i\theta})| \dif \theta \right).
\end{equation}
\end{proposition}
\begin{proof}
Assume that the $\alpha_n$'s are ordered in such a way that $|\alpha_1|, \dots, |\alpha_m| < r$, $|\alpha_{m+1}|=\dots=|\alpha_N|=r$ and let
\[
g(z) := f(z)
\Prod_{n=1}^m \frac{r^2 - \bar\alpha_n z}{r(\alpha_n - z)}
\Prod_{n=m+1}^N \frac{\alpha_n}{\alpha_n - z}.
\]
Then $g$ extends to a non-vanishing holomorphic function on $B(0,r + \epsilon)$ for some $\epsilon$.
For $z = r e^{i\theta}$ the terms in the definition of $g$ satisfy
\[
\left| \frac{r^2 - \bar\alpha_n z}{r(\alpha_n - r z)} \right| = 1,
\quad
\log \left| \frac{\alpha_n}{\alpha_n - z} \right| = - \log | 1 - e^{i(\theta - \theta_n)} |,
\]
where $\theta_n$ denotes the argument of $\alpha_n$.
By Proposition~\ref{prop:simply-conn-log}, the function $\log |g|$ is the real part of a holomorphic function and thus harmonic, so that, by definition of $g$ and Lemma~\ref{lemma:int-ln-1-e-itheta},
\begin{align*}
\log \left| f(0) \Prod_{n=1}^N \frac{r}{|\alpha_n|} \right|
&= \log | g(0) |\\
&=
\frac{1}{2 \pi} \int_{\theta = 0}^{2\pi} \log | g(r e^{i \theta}) | \dif\theta\\
&=
 \frac{1}{2 \pi} \int_{\theta = 0}^{2\pi} \left[
\log | f(r e^{i \theta}) |
- \Sum_{n=m+1}^N \log \left| 1 - e^{i (\theta-\theta_n)} \right|
\right] \dif\theta\\
&=  \frac{1}{2 \pi} \int_{\theta = 0}^{2\pi} \log | f(r e^{i \theta}) | \dif\theta.
\qedhere
\end{align*}
\end{proof}

With Jensen's formula at hand, we readily obtain subharmonicity of $\log |f|$.

\begin{corollary}
\label{cor:hol-log-subharm}
If $f$ is holomorphic in $\Omega$, then $\log |f|$ is subharmonic in $\Omega$.
\end{corollary}
\begin{proof}
By a translation of the complex plane, the problem reduces to verifying
\[
\log |f(0)| \leq \frac{1}{2\pi} \int_{-\pi}^\pi \log |f(r e^{i\theta})| \dif \theta
\]
for $f \in \hol{B(0,R)}$ and every $0<r<R$.
If $f(0)=0$, there is nothing to prove.
Otherwise, by Jensen's formula~(\ref{eq:jensens-formula}), the assertion is equivalent to
\[
\log |f(0)| \leq
\log \left( |f(0)| \Prod_{n=1}^N \frac{r}{|\alpha_n|} \right),
\]
where $\alpha_n$'s are the zeroes of $f$ inside the ball of radius $r$ with multiplicity, so that $\log (r/|\alpha_n|) > 0$.
\end{proof}

At last, we show the principle of harmonic majoration.
\begin{proposition}
\index{harmonic majoration}
\label{prop:harmonic-majoration}
Let $\Omega \subset \C$ be a domain with compact closure, $f$ an upper semicontinuous function on $\bar\Omega$ which is subharmonic in $\Omega$ and $u$ a continuous function on $\bar \Omega$ which is harmonic in $\Omega$ such that $f \leq u$ on $\boundary \Omega$.
Then $f \leq u$ in $\Omega$.
\end{proposition}
\begin{proof}
The function $f-u$ is upper semicontinuous on $\bar \Omega$ and subharmonic in $\Omega$, so that without loss of generality we may assume $u = 0$.

Now assume that, on the contrary to the assertion, $m = \sup_\Omega f > 0$.
Since $f$ is upper semicontinuous on $\bar \Omega$, it assumes its supremum on a compact subset $E \subset \bar \Omega$.
By the hypothesis, $E \cap \boundary \Omega = \emptyset$.
Let $z \in \boundary E$.
Since $z$ also lies in the interior of $\Omega$, there exists an $r>0$ such that the circle of radius $r$ around $z$ is contained in $\bar\Omega$ but not in $E$.
But then
\[
\frac{1}{2\pi} \int_{0}^{2 \pi} f(z + r e^{i \theta}) \dif\theta < m = f(z),
\]
as the decomposition into the parts inside and outside of $E$ shows.
This contradicts the assumption that $f$ is subharmonic.
\end{proof}

\section{The three lines lemma}
The three lines lemma is a version of the maximum principle for the strip $S := \{0 < \Re z < 1\} \subset \C$.
The unboundedness of $S$ makes an additional qualitative hypothesis necessary for the maximum principle to hold.
\begin{definition}
\index{admissible growth}
A function $g$ on $\R$ is said to have \emph{admissible growth} if $g(r) = O(e^{a |r|})$ with $a < \pi$.
A function $f$ on $\bar S$ is said to have \emph{admissible growth} if $M(r) := \log \sup_{\Im z = r} |f(z)|$ has admissible growth.
\end{definition}
To establish a maximum principle for functions of admissible growth on $\bar S$ we will use an explicit solution formula for the Dirichlet boundary problem
\[
\Delta u = 0 \text{ in } S,
\quad
u(j + i y) = a_j(y) \text{ for } j=0,1, \, y \in \R
\]
We reduce this problem to the Dirichlet problem on the unit disc.
By symmetry, it is sufficient to consider the case $a_1 = 0$ and to find a formula for $u(x)$ with $x$ real.

A conformal mapping from the unit disc $B(0,1)$ onto $S$ is given by
\[
h(z) = \frac{1}{i \pi} \log\left( i \frac{1+z}{1-z} \right),
\]
cf.~\cite[Lemma 3.1]{MR928802}.
Since $h$ is conformal, $u \circ h\inv$ is harmonic in $B(0,1)$ and solves a Dirichlet boundary problem with initial data supported on the lower unit half-circle, which is mapped by $h$ onto the line $i \R$.
The desired relation is
\[
u(x) = \frac{1}{2\pi} \int_{-\pi}^0 a_0(-i h(e^{i \phi})) P_{h\inv(x)}(\phi) \dif \phi
=
\int_{-\infty}^\infty a_0(y) \omega(x,-y) \dif y
\]
with some kernel $\omega$.
Here, the first integral is the solution formula for the unit disk, where $P_z(\phi) = \Re \frac{e^{i \phi} + z}{e^{i \phi} - z}$ is the Poisson kernel, while the second integral represents the solution formula we are looking for.
A necessary and sufficient condition for equality to hold is
\[
-2 \pi \omega(x,-y) \tdif{}{\phi} (-i h(e^{i \phi})) = \Re \frac{e^{i \phi} + h\inv(x)}{e^{i \phi} - h\inv(x)},
\]
where $\phi$ is given by $y = -i h(e^{i \phi})$.
A calculation shows that
\[
\omega(x,y)
= \frac12 \frac{\sin(\pi x)}{\cosh(\pi y) - \cos(\pi x)}
= \Re \frac{i}{1 - e^{-iz}}.
\]
It is the kernel of the harmonic measure on $\boundary S$ in the sense that
if $a_0, a_1 : \R \to \C$ are continuous functions of admissible growth, then
\[
u (x+iy)
=
\int_{-\infty}^{+\infty} \omega(x, y-t) a_0(t) \dif t
+
\int_{-\infty}^{+\infty} \omega(1-x, y-t) a_1(t) \dif t
\]
is harmonic in $S$, since $z \mapsto \omega(x,y)$ is harmonic and the growth condition on the $a_j$'s allows to differentiate under the integral sign, since $\omega(x,y) \lesssim \tan\frac{x \pi}{2} e^{-\pi |y|}$ for $|y|>1$.
Also, $u$ extends to a continuous function on $\bar S$ with $u(j+iy)=a_j(y)$ for $j=0,1$, because $\int_{y \in \R} \omega(x,y) \dif y = 1-x$ (see below) and the measure $\omega(x,y) \dif y$ is concentrated at $y=0$ for $x \to 0$.

\begin{lemma}[{three lines lemma, \cite[p. 210]{MR0057936}}]
\index{three lines lemma}
\label{lemma:three-lines}
Let $f \in \hol{ S } \cap C(\bar S)$, $a_0, a_1 \in C(\R)$ be functions of admissible growth and assume
\[
\log |f(j+iy)| \leq a_j(y), \quad (j=0,1,\, -\infty < y < +\infty).
\]
Then, for every $\theta \in (0,1)$,
\[
\log |f(\theta)|
\leq
\int_{-\infty}^{+\infty} \omega(\theta, y) a_0(y) \dif y
+
\int_{-\infty}^{+\infty} \omega(1-\theta, y) a_1(y) \dif y.
\]
\end{lemma}
If $f$ does not have admissible growth, the proposition fails as the example $f(z) = \exp(-i \exp(i \pi z))$ shows.
Indeed, $\exp(i \pi z) \in \R$ on $\boundary S$, so that $\log |f|_{\boundary S}| \equiv 0$, while $f(1/2 + iy) = \exp(\exp(-\pi y)) \to \infty$ as $y \to -\infty$.
\begin{proof}
Let $a < a' < \pi$ and fix an $\epsilon > 0$.
Consider
\begin{multline*}
u_T (x+iy)
=
\int_{-T}^{+T} \omega(x, y-t) a_0(t) \dif t
+
\int_{-T}^{+T} \omega(1-x, y-t) a_1(t) \dif t\\
+
\epsilon \cosh (a' y) \cos (a'(x-1/2))
\end{multline*}
This is a harmonic function because
\[
\cosh (a' y) \cos (a'(x-1/2)) = 2 \Re e^{-i a' (z-1/2)} + 2 \Re e^{i a' (z-1/2)}.
\]
The first two summands extend to $a_j$ on $\{ \Re z = j, |\Im z| < T \}$ for $j=0,1$, while
the last one grows faster than $\log |f|$ and the first two as $y=T\to\infty$.
Thus, for $T$ large enough, $u_T$ majorizes $\log |f|$ on the boundary of $S_T = S \cap \{ |\Im z| < T \}$ and, by the principle of harmonic majoration (Proposition \ref{prop:harmonic-majoration}), on $S_T$, since $\log |f|$ is subharmonic by Corollary~\ref{cor:hol-log-subharm}.

By Fatou's lemma,
\begin{align*}
\log |f(x+iy)|
&\leq \limsup_{T\to\infty} u_T(x+iy)\\
&= \limsup_{T\to\infty} \int_{-T}^{+T} \omega(x, y-t) a_0(t) \dif t
+
\int_{-T}^{+T} \omega(1-x, y-t) a_1(t) \dif t\\
&\quad+
\epsilon \cosh (a' y) \cos (a'(x-1/2))\\
&\leq \int_{-\infty}^{+\infty} \omega(x, y-t) a_0(t) \dif t
+
\int_{-\infty}^{+\infty} \omega(1-x, y-t) a_1(t) \dif t\\
&\quad+
\epsilon \cosh (a' y) \cos (a'(x-1/2)).
\end{align*}

Let $\epsilon \to 0$, set $z=\theta$ and observe that $\omega(x,\cdot)$ is an even function to obtain the claim.
\end{proof}
To see that this lemma generalizes the Hadamard three line theorem we calculate
\begin{align*}
\int_{y=-\infty}^\infty \omega(x,y) \dif y
&=
\frac12 \int_{y=-\infty}^\infty \frac{\sin(\pi x)}{\cosh(\pi y) - \cos(\pi x)} \dif y\\
&=
\int_{u=0}^\infty \frac{\sin(\pi x)}{u+1/u - 2 \cos(\pi x)} \frac{\dif u}{\pi u}
& (u=e^{\pi y})\\
&=
\frac{\sin(\pi x)}{\pi} \int_{u=0}^\infty \frac{\dif u}{u^2 - 2u \cos(\pi x) + 1}\\
&=
\frac{1}{\pi} \int_{s=-\cot \pi x}^\infty \frac{\dif s}{s^2 + 1}
& \left( s = \frac{u - \cos \pi x}{\sin \pi x} \right)\\
&=
\frac{1}{\pi} \left. \arctan s \right|_{s=-\cot \pi x}^\infty\\
&=
1-x.
\end{align*}
Therefore, if $f$ is a function of admissible growth which is bounded by $M_j$ on $\{ \Re z = j \}$, $j=0,1$, then $|f(\theta)| \leq M_0^{1-\theta} M_1^\theta$ for all $0<\theta<1$.

We shall frequently need a vector-valued version of the three lines lemma.
Recall the definition of $X_{0}' \intersection X_{1}'$ from Section~\ref{sec:real-int-dual}.

\begin{corollary}
\label{cor:three-lines-vector}
Let $V \subset (X_{0}+X_{1})'$ be a subspace such that the $X_{0}+X_{1}$- and the $V'$-norm on $X_{0}+X_{1}$ are equivalent (subspaces with this property are sometimes called \emph{(norm-)determining}).

Let $f \in \hol{ S, X_{0}+X_{1} }$ be $\sigma(X_{0}+X_{1},V)$-continuous on $\bar S$ and $a_0, a_1 \in C(\R)$ be functions of admissible growth such that
\[
\log ||f(j+iy)||_{X_{j}} \leq a_j(y), \quad (j=0,1,\, -\infty < y < +\infty).
\]
Then, for every $\theta \in (0,1)$,
\[
\log ||f(\theta)||_{X_{0}+X_{1}}
\leq C+
\int_{-\infty}^{+\infty} \omega(\theta, y) a_0(y) \dif y
+
\int_{-\infty}^{+\infty} \omega(1-\theta, y) a_1(y) \dif y.
\]
\end{corollary}
\begin{proof}
This follows at once considering $g \circ f$ with $g \in V$, $||g||_{V} \leq 1$ and observing that the estimates thus obtained are uniform in $\phi$.
\end{proof}

\section{Analyticity of vector-valued functions}
We are going to use analytic functions with values in locally convex vector spaces extensively.
We thus need a tangible characterization of analyticity.
Our main tool is the Cauchy integral formula, and thus we start by giving an account of an appropriate notion of integral.
The standard reference for this material is the book of Hille and Phillips \cite{MR0089373}, but we follow the more structured approach in a set of lecture notes by P.~Garrett \cite{garrett-page}.

\begin{definition}
\index{convex envelope property}
A topological vector space $E$ is said to have the \emph{convex envelope property} if the closed convex hull of every compact subset of $E$ is compact.
\end{definition}
First we verify that the spaces we are interested in have this property.
\begin{definition}
\index{totally bounded subset}
A subset of a topological vector space is called \emph{totally bounded} if, for every neighborhood of the origin $U$, it can be covered by finitely many translates of $U$.

A subset of a metric space is called \emph{totally bounded} if, for every $\epsilon>0$, it can be covered by finitely many balls of radius $\epsilon$.
\end{definition}
We remark that every compact subset of a topological vector space is totally bounded.
\begin{proposition}
Let $K$ be a totally bounded subset of a topological vector space.
Then $\conv K$ is totally bounded.
\end{proposition}
\begin{proof}
Let $U$ be an arbitrary neighborhood of the origin.
Then by joint continuity of addition and local convexity there exists a convex neighborhood of the origin $\tilde U$ such that $\tilde U + \tilde U \subset U$.
Since $K$ is totally bounded, we have that $K \subset F + \tilde U$ for some finite set $F$.
Since the set $\conv F$ is the continuous image of the standard simplex in $\R^{|F|}$, it is compact and therefore totally bounded, so that $\conv F \subset \tilde F + \tilde U$ with a finite set $\tilde F$.
Therefore
\[
\conv K \subset \conv (F+\tilde U) \subset \conv F + \tilde U
\subset \tilde F + \tilde U + \tilde U \subset \tilde F + U.
\qedhere
\]
\end{proof}
Observe that the notions of total boundedness in the sense of topological vector spaces and in the sense of metric spaces coincide for a subset of a metric vector space.
Hence we immediately obtain the following.
\begin{corollary}
\label{cor:frechet-conv-env}
Let $K$ be a compact subset of a \Frechet\ space.
Then $\cconv K$ is also compact.
In other words, every \Frechet\ space has the convex envelope property.
\end{corollary}

\index{quasi-complete topological vector space}
The latter result extends to a more general class of locally convex spaces via a construction along the lines of the Banach-Alaoglu theorem.
We say that a subset of a topological vector space is \emph{bounded} if its image under every continuous linear form is bounded (for a subset of a locally convex vector space this is equivalent to the boundedness of images under arbitrary continuous seminorms by the uniform boundedness principle).
A topological vector space is called \emph{quasi-complete} if all its bounded closed sets are complete.

\begin{theorem}
\label{thm:quascomp-convenv}
Let $E$ be a quasi-complete locally convex vector space.
Then $E$ has the convex envelope property.
\end{theorem}
\begin{proof}
We may assume that the topology is given by a family of seminorms $a$.
Our aim is to show that the closed convex hull of a compact set $A \subset E$ is compact.

For every seminorm $a$ let $E_{a}$ be the Banach space completion of $E/a\inv(0)$ with respect to $a$.
We furnish $\tilde E := \Prod_{a} E_{a}$ with the product topology.
Consider the canonical operators
\[
\iota : E \to \tilde E \text{ and } \pi_{a} : \tilde E \to E_{a}.
\]
The former operator is injective since the totality of seminorms separates the points of $E$ and therefore a homeomorphism onto its image.
Furthermore each $\pi_{a} \circ \iota$ is continuous, so that $\pi_{a}(\iota(A))$ is compact.
Since $E_{a}$ is a \Frechet\ space, Corollary~\ref{cor:frechet-conv-env} implies that $C_{a} := \cconv \pi_{a}(\iota(A))$ is compact.
By the Tychonov theorem $\tilde C := \Prod_{a} C_{a}$ is compact; $\tilde C$ is also convex and contains $\iota(A)$, so that $\cconv \iota(A)$ is compact.

It therefore suffices to show that $\cconv \iota(A) = \iota (\cconv A)$.
The inclusion ``$\supseteq$'' follows from continuity of $\iota$.
For the converse, observe that $\iota (\cconv A)$ is convex, contained in the bounded set $\tilde C$ and closed in $\iota(E)$, hence complete by quasi-completeness of $E$ and therefore closed in $\tilde E$.
\end{proof}

We now give some examples of quasi-complete spaces.
Evidently a complete topological vector space, and in particular a Banach space, is quasi-complete.

\begin{proposition}
Let $V_{0} \subset V_{1} \subset \dots$ be a strictly ascending chain of locally convex spaces such that each $V_{i}$ is a closed subspace of $V_{i+1}$.
We furnish $V := \union_{i} V_{i}$ with the colimit topology, i.e.\ we define open sets as those whose intersection with every $V_{i}$ is open in the respective topology.

Assume that every space $V_{i}$ is quasi-complete.
Then $V$ is quasi-complete.
\end{proposition}
\begin{proof}
With the above definition $V$ is a topological vector space and every $V_{i}$ a closed subspace thereof.
We claim that a subset $A \subset V$ is bounded if and only if it is contained in some $V_{i}$ and is bounded as a subset thereof.
The ``if'' part of the assertion is clear, so let us turn to the ``only if'' part.

If $A$ is not contained in any of $V_{i}$'s, we can find a subsequence of natural numbers $n(i)$ and $x_{i} \in A \intersection V_{n(i)} \setminus V_{n(i-1)}$.
By the Hahn-Banach theorem there exists a sequence of functionals $\lambda_{i} \in V_{n(i)}'$ such that $\lambda_{i}(x_{i}) = i$ and $\lambda_{i+1}|_{V_{n(i)}} = \lambda_{i}$.
By definition of the colimit topology they are restrictions of a unique $\lambda \in V'$, and $\lambda(A)$ is unbounded.

Hence the only closed bounded subsets of $V$ are the closed bounded subsets of the spaces $V_{i}$ that are complete by the hypothesis.
\end{proof}

Thus we see that the space of compactly supported smooth functions on $\R^{n}$ is quasi-complete, since it is the strict colimit of \Frechet\ spaces.
Another interesting space is the space of continuous linear operators.
\begin{proposition}
Let $X$ and $Y$ be \Frechet\ spaces.
Then the space $L(X,Y)$ is quasi-complete w.r.t.\ the strong operator topology.
\end{proposition}
\begin{proof}
Let $A \subset L(X,Y)$ be closed and bounded in the strong operator topology.
By the uniform boundedness principle (see \cite[Theorem III.4.2]{MR1741419} for a sufficiently general version) $A$ is equicontinuous.
The operator defined as a pointwise limit of a strongly Cauchy net in $A$ is therefore continuous and is the strong limit of the net.
\end{proof}
More in general, this result remains true if $X$ is replaced by a barreled locally convex space and $Y$ by an arbitrary locally convex space.
Note that a similar argument for the space $L(X,Y)$ equipped with the weak operator topology only works if $Y$ is complete w.r.t.\ the weak topology, which is the case for instance if $Y$ is a reflexive Banach space.

\begin{definition}
Let $\Omega$ be compact, $\mu : C(\Omega, \R) \to \R$ be a measure (i.e.\ a positive linear form) and $E$ be a topological vector space.
The \emph{Gelfand-Pettis} integral of a continuous function $f : \Omega \to E$ is a vector $\mu(f) \in E$ such that
\[
\<\xi, \mu(f)\> = \mu (\<\xi, f\>) \text{ for every } \xi \in E'.
\]
\end{definition}
Clearly, the Gelfand-Pettis integral is unique if $E'$ separates points, e.g.\ if $E$ is locally convex.
Next we provide a sufficient condition for its existence.

\begin{proposition}
Let $E$ be a topological vector space which has the convex envelope property.
Let also $\Omega$ be compact and $\mu : C(\Omega, \R) \to \R$ be a measure.
Then every continuous function $f : \Omega \to E$ admits a Gelfand-Pettis integral.
\end{proposition}
\begin{proof}
Without loss of generality we may assume $\mu(\Omega)=1$.
By the convex envelope property the set $K := \cconv f(\Omega)$ is compact.

Let $\{\lambda_{1}, \dots, \lambda_{n}\} \subset E'$ be finite.
Then $\Lambda = (\lambda_{1}, \dots, \lambda_{n}) : E \to \R^{n}$ is a continuous linear operator.
Assume that $(\mu(\lambda_{1}\circ f), \dots, \mu(\lambda_{n}\circ f)) \not\in \Lambda(K)$.
Then there exists a linear form $\phi$ on $\R^{n}$ which separates the former point from the latter convex compact set.
In particular $\mu(\phi\circ\Lambda\circ f)$ is separated from $\phi\circ\Lambda(K) \supset \conv (\phi\circ\Lambda\circ f(\Omega))$, which contradicts the positivity and normalization of $\mu$.

Therefore the closed subset
\[
K_{\Lambda} = \{x \in K : \<\lambda_{i},x\> = \mu(\<\lambda_{i},f\>) \text{ for } i=1,\dots,n\}
\]
is non-empty.
Moreover these sets are compact and enjoy the finite intersection property (since the intersection of finitely many such sets has the same form).
Their intersection is hence non-empty, and every point in the intersection is a Gelfand-Pettis integral of $f$.
\end{proof}

We are now ready to give the basic characterization of analytic vector-valued functions.
\begin{theorem}
\label{thm:hol-lcvs}
Let $E$ be a locally convex quasi-complete vector space, $\Omega \subset \C$ be an open set and $f : \Omega \to E$ be a function.
Then the following properties are equivalent.
\begin{enumerate}
\item\label{thm:hol-lcvs:an}
$f$ is locally analytic, i.e.\ the sum of a power series in a neighborhood of each point in $\Omega$.
\item\label{thm:hol-lcvs:sh}
$f$ is strongly holomorphic, i.e.\ differentiable w.r.t.\ the topology of $E$.
\item\label{thm:hol-lcvs:wh}
$f$ is weakly holomorphic, i.e.\ $\<\lambda,f\>$ is differentiable for every $\lambda \in E'$.
\end{enumerate}
\end{theorem}
\begin{proof}
The implications (\ref{thm:hol-lcvs:an})$\implies$(\ref{thm:hol-lcvs:sh})$\implies$(\ref{thm:hol-lcvs:wh}) are clear, so that we concentrate on (\ref{thm:hol-lcvs:wh})$\implies$(\ref{thm:hol-lcvs:an}).

In order to simplify the notation we assume that $f(0)=0$ and prove the assertion in a neighborhood of zero.
Choose $r>0$ in such a way that $\bar B_{2r}(0) \subset \Omega$.

Since $f$ is weakly holomorphic, for every $\lambda \in E'$ the function $z \mapsto \<\lambda,f(z)\>/z$ extends to a continuous function on $\bar B_{2r}$.
Hence the set $\{ f(z)/z, z \in B_{2r}\}$ is weakly bounded.
By the uniform boundedness principle it is bounded for every seminorm on $E$, so that $f$ is in fact strongly continuous at $0$, and by translation invariance strongly continuous on $\Omega$.

By the Cauchy integral formula applied to $\<\lambda, f\>$ and the definition of the Gelfand-Pettis integral we see that
\[
f(z) = \frac{1}{2\pi i} \int_{\gamma} \frac{f(\zeta)}{\zeta-z} \dif\zeta
\]
for $z\in B_{r}(0)$, where $\gamma$ is the boundary of $B_{2r}(0)$ with the positive orientation.
But then we have the estimate $|z/\zeta|\leq 1/2$ and therefore
\[
\frac1{\zeta - z} = \frac1\zeta \frac1{1-z/\zeta} = \frac1\zeta \sum_{n=0}^{\infty} \frac{z^{n}}{\zeta^{n}}
\]
with uniform convergence in $z$ and $\zeta$.
Inserting this into the previous formula yields
\[
f(z) = \frac{1}{2\pi i} \int_{\gamma} \sum_{n=0}^{\infty} \frac{z^{n} f(\zeta)}{\zeta^{n+1}} \dif\zeta.
\]
Since $f$ is continuous, $f(\gamma)$ is compact.
By continuity of scalar multiplication and Theorem~\ref{thm:quascomp-convenv} the closed absolute convex hull of $f(\gamma)$ is also compact and therefore bounded.
Since the Gelfand-Pettis integral of a function is contained in the closed convex hull of its image times the measure of the integration domain, we can interchange integration and summation and obtain
\[
f(z) = \sum_{n=0}^{\infty} \frac{z^{n}}{2\pi i} \int_{\gamma} \frac{f(\zeta)}{\zeta^{n+1}} \dif\zeta,
\]
where the integrals are Gelfand-Pettis integrals and convergence is uniform on $B_{r}$.
\end{proof}

\section{Intermediate spaces}
\index{complex interpolation!space $\mathcal{F}$}
Let $X_0, X_1$ be complex Banach spaces which are both contained in some topological vector space, $D \subseteq X_0 + X_1$ be a locally convex topological vector space (not necessarily carrying the subspace topology) and $\mathcal{F}(X_{0},X_{1},D)$ be the set of all bounded analytic functions $f_z : \bar S \to D$ such that $f_{j+i \cdot} \in C_{0}(\R, X_j)$ for $j=0,1$ and
\begin{equation}
\label{eq:F-norm}
||f_z||_{\mathcal{F}(X_{0},X_{1})} := \sup_{j=0,1;\ y\in\R} ||f_{j+iy}||_{X_j} < \infty.
\end{equation}
Note that $f_{z}$ stands both for the analytic function and for its value at $z \in \bar S$, depending on the context.

The space $\mathcal{F}(X_{0},X_{1})$ with the norm (\ref{eq:F-norm}) is a Banach space (here and below the omission of $D$ indicates that $D = X_{0} + X_{1}$).
For every $0 < \theta < 1$ we let $\mathcal{N}_{\theta}(X_{0},X_{1},D)$ be the subspace of $\mathcal{F}(X_{0},X_{1},D)$ which consists of the functions vanishing at $z = \theta$.
The space $\mathcal{N}_{\theta}$ is closed since convergence in $\mathcal{F}(X_{0},X_{1})$ implies pointwise convergence on $S$ by the three lines lemma (Corollary~\ref{cor:three-lines-vector}).

\index{complex interpolation!interpolation space}
The \emph{interpolation space} $[X_0, X_1]_{\theta}$ is defined as $\mathcal{F}(X_{0},X_{1}) / \mathcal{N}_{\theta}(X_{0},X_{1})$.
As a quotient of a Banach space by a closed subspace the space it is itself a Banach space.
The equivalence class of a function $f_{z}$ is canonically identified with the value $f_{\theta}$.
Under this identification the norm on the interpolation space is given by
\[
||f||_{[X_0, X_1]_{\theta}} = \inf_{f_z \in \mathcal{F}(X_{0}, X_{1}): f=f_\theta} ||f_z||_{\mathcal{F}(X_{0},X_{1})}.
\]

\begin{proposition}
\label{prop:F-X0-X1-D-dense}
Let $D \subset X_{0} \intersection X_{1}$ be a dense subspace.
Then $\mathcal{F}(X_{0},X_{1},D)$ is dense in $\mathcal{F}(X_{0},X_{1})$.
Moreover, the space $\mathcal{F}_{0}(X_{0},X_{1},D)$ of functions of the form
\begin{equation}
\label{eq:F-X0-X1-D-finite}
\sum_{k} h_{k} x_{k},
\end{equation}
where $h_{k}$ are complex-valued analytic functions, $x_{k} \in D$ and the sum is finite, is dense in $\mathcal{F}(X_{0},X_{1})$.
\end{proposition}
\begin{proof}
Take an arbitrary $f_{z} \in \mathcal{F}(X_{0}, X_{1})$ with norm $1$.
Given an $\epsilon>0$ we can find a $\delta>0$ such that $||f_{z} - e^{\delta z^{2}} f_{z}||_{\mathcal{F}(X_{0}, X_{1})} < \epsilon$.
Let $T>0$ be so large that $|e^{\delta z^{2}}| < \epsilon$ whenever $|\Im z| > T$.
Then there exists an even smaller $\delta'>0$ such that $|e^{\delta' z^{2}}|>1-\epsilon$ whenever $|\Im z|<T$.
Let $T'>0$ be such that $|e^{\delta' z^{2}}| < \epsilon$ whenever $|\Im z| > T'$.
If we now find an analytic function $g_{z}$ uniformly bounded by $1+O(\epsilon)$ which approximates $e^{\delta z^{2}} f_{z}$ up to $O(\epsilon)$ inside the region $|\Im z| < T'$, then $||e^{\delta' z^{2}} g_{z} - e^{\delta z^{2}} f_{z}||_{\mathcal{F}(X_{0}, X_{1})} = O(\epsilon)$.
For this end let
\[
g_{z} := \sum_{k=-\infty}^{\infty} e^{\delta (z + ikT')^{2}} f_{z + ikT'}.
\]
This is a periodic function and its uniform boundedness and proximity to $e^{\delta z^{2}} f_{z}$ follow from the super-exponential decay of $e^{\delta z^{2}} f_{z}$.
Combined, these three properties also imply $||e^{\delta' z^{2}} g_{z} - g_{z}||_{\mathcal{F}(X_{0}, X_{1})} = O(\epsilon)$.
Consider the Fourier coefficients
\[
\hat g_{k}(z) = \frac{1}{2mT'} \int_{-mT'}^{mT'} g_{z+it} e^{-k(z+it)/T'} \dif t.
\]
By periodicity this integral is independent of $m$ and $\Im z$.
Letting $m \to \infty$ and using the Cauchy integral theorem and boundedness of $g_{z}$ we see that it is also independent of $\Re z$.
Therefore the Fourier coefficients lie in $X_{0} \intersection X_{1}$.
Since $g_{z}$ is analytic, the \Fejer{} sums of the Fourier series converge uniformly.
But the \Fejer{} sums are of the form (\ref{eq:F-X0-X1-D-finite}).
It now suffices to approximate every $x_{k}$ by an element of $D$.
\end{proof}

We now proceed to the announced generalization of the Riesz-Thorin interpolation theorem.
\begin{theorem}[{Stein \cite{MR0082586}}]
\index{Stein interpolation theorem}
\label{th:stein}
Let $D$ be a dense subspace of $X_0 \intersection X_1$ and $V$ be a subspace of $(Y_{0}+Y_{1})'$ such that the $Y_{0}+Y_{1}$- and the $V'$-norm on $Y_{0}+Y_{1}$ are equivalent.
Consider a family of linear operators
\[
T_z : D \to Y_0 + Y_1, \qquad (z \in \bar S)
\]
which is analytic on $S$, is $\sigma(Y_{0}+Y_{1},V)$-continuous and has admissible growth in the sense that for every $f \in D$ and $g \in V$ the function $\<T_{z}f,g\>$ is continuous on $\bar S$ and has admissible growth uniformly for bounded $f$ and $g$.
Assume in addition that for every $y \in \R$ and $f \in D$
\[
||T_{iy} f||_{Y_0} \leq M_0 ||f||_{X_0}, \quad ||T_{1+iy} f||_{Y_1} \leq M_1 ||f||_{X_1}.
\]

Then the operator
\[
T : \mathcal{F}(X_{0},X_{1},D) \to \mathcal{F}(Y_{0},Y_{1}),\, f_{z} \mapsto T_{z} f_{z}
\]
extends to a bounded operator from $\mathcal{F}(X_{0},X_{1})$ to $\mathcal{F}(Y_{0},Y_{1})$.

If in addition there exists an absolute constant $C$ such that
\[
\inf_{f_{z} \in \mathcal{F}(X_{0}, X_{1}, D) : f_{\theta}=f} ||f_{z}||_{\mathcal{F}(X_{0},X_{1})}
\leq
C ||f||_{[X_{0}, X_{1}]_{\theta}}
\]
for every $f \in D$, then $T$ maps $\mathcal{N}_{\theta}(X_{0},X_{1})$ into $\mathcal{N}_{\theta}(Y_{0},Y_{1})$ and therefore induces an operator from $[X_{0}, X_{1}]_{\theta}$ to $[Y_{0}, Y_{1}]_{\theta}$.
\end{theorem}
\begin{proof}
Let $f \in \tilde X_{\theta}$ and $f_z \in \mathcal{F}(X_{0},X_{1},D)$ be a function with $f_\theta = f$.
Then $T_z f_z$ is an analytic function of admissible growth with values in $Y_0 + Y_1$ and by the assumptions
\[
|| T_{j+iy} f_{j+iy} ||_{Y_j}
\leq
M_j ||f_{j+iy}||_{X_j}
\]
By the three lines lemma (Corollary~\ref{cor:three-lines-vector}), the function $T_{z} f_{z}$ is bounded by an element of $C_{0}(\bar S)$, so that $T_{z} f_{z} \in \mathcal{F}(Y_{0},Y_{1})$ and
\[
||T_z f_z||_{\mathcal{F}(Y_{0},Y_{1})} \leq \max(M_{0},M_{1}) ||f_z||_{\mathcal{F}(X_{0},X_{1})},
\]
so that $T$ is a bounded operator on $\mathcal{F}(X_{0},X_{1},D)$.
By Proposition~\ref{prop:F-X0-X1-D-dense} the latter space is dense in $\mathcal{F}(X_{0},X_{1})$.

Assume now that the additional condition is satisfied.
Let $f_{z} \in \mathcal{N}_{\theta}(X_{0},X_{1})$.
For $\epsilon>0$ let $g_{z} \in \mathcal{F}_{0}(X_{0},X_{1},D)$ be a function with $||f_{z}-g_{z}||_{\mathcal{F}(X_{0},X_{1})} < \epsilon$ given by Proposition~\ref{prop:F-X0-X1-D-dense}.
Then $f_{\theta} - g_{\theta} = -g_{\theta} \in D$ and there exists a function $h_{z} \in \mathcal{F}(X_{0},X_{1},D)$ such that $h_{\theta} = g_{\theta}$ and $||h_{z}||_{\mathcal{F}(X_{0},X_{1})} < C \epsilon$.
Therefore
\begin{multline*}
||(Tf_{*})_{\theta}||_{Y_{0}+Y_{1}}
\leq
||(T(f_{*}-g_{*}))_{\theta}||_{Y_{0}+Y_{1}} + ||T_{\theta} g_{\theta}||_{Y_{0}+Y_{1}}\\
\leq
C ||T(g_{*}-f_{*})||_{\mathcal{F}(Y_{0},Y_{1})} + ||T_{\theta} h_{\theta}||_{Y_{0}+Y_{1}}
\leq
C \epsilon + C ||T h_{*}||_{\mathcal{F}(Y_{0},Y_{1})}
\leq
C \epsilon.
\end{multline*}
Since $\epsilon$ was arbitrary, this shows that $Tf_{*} \in \mathcal{N}_{\theta}(Y_{0},Y_{1})$.
Thus $T$ induces an continuous linear operator from $[X_{0}, X_{1}]_{\theta}$ to $[Y_{0},Y_{1}]_{\theta}$.
\end{proof}
In general it is not clear whether $T : [X_{0},X_{1}]_{\theta} \to [Y_{0},Y_{1}]_{\theta}$ is well-defined.
In the sequel we shall answer this question affirmatively in the cases $X_{j} = L^{p_{j}}(\R^{n})$, $1 \leq p_{j} \leq \infty$, $D=\Schwartz(\R^{n})$ and $X_{0} = H^{1}(\R^{n})$, $X_{1} = L^{p}(\R^{n})$, $1 < p < \infty$, $D=\Schwartz(\R^{n})$.

For the moment observe that $T(\mathcal{N}_{\theta}(X_{0},X_{1})) \subset \mathcal{N}_{\theta}(Y_{0},Y_{1})$ is always the case if $T_{z}$ is constant.
This is because convergence in $\mathcal{F}(X_{0},X_{1},D)$ implies convergence of the values at $\theta$ in the topology of $X_{0} + X_{1}$ on $D$ and because $T_{\theta} : D \subset X_{0}+X_{1} \to Y_{0}+Y_{1}$ is bounded (due to the density of $D$ in $X_{0} \cap X_{1}$).

For future reference we cite here a description of the complex interpolation between dual spaces.
\begin{theorem}[{\cite[Corollary 4.5.2]{MR0482275}}]
\label{thm:complex-interpolation-dual}
Let $X_{0}$ and $X_{1}$ be Banach spaces at least one of which is reflexive and $0<\theta<1$.
Then
\[
[X_{0}', X_{1}']_{\theta} = [X_{0}, X_{1}]_{\theta}'
\]
\end{theorem}

\section{Interpolation between \texorpdfstring{$L^{p}$}{Lp} spaces}
The motivating example for the complex interpolation method are the $L^p$ spaces.
Indeed,
\[
L^{p_\theta} = [ L^{p_0}, L^{p_1} ]_\theta,
\text{ where } \frac{1}{p_\theta} = \frac{1-\theta}{p_0} + \frac{\theta}{p_1}.
\]
We will prove a more general version of this result.

For a measure space $(\Omega,\mu)$ and a Banach space $X$ we denote by $L^{p}_{0}(\Omega,X)$ the closure of the space of simple functions in the Bochner space $L^{p}(\Omega,X)$.
We shall usually omit $\Omega$ if there is no ambiguity.
Note that $L^{p}_{0}(X) = L^{p}(X)$ if $1 \leq p < \infty$.
Moreover, if $\Omega$ is a topological space, then $C_{0}(\Omega,X) \subset L^{\infty}_{0}(\Omega, X)$.

\begin{theorem}
\label{thm:bochner-complex-interpolation}
Let $X_{j}$, $j=0,1$ be Banach spaces such that $X_{0} \intersection X_{1}$ is dense in both $X_{j}$'s, $0 < \theta < 1$ and $Z_{j} \subset X_{j}'$ be such that $[Z_{0},Z_{1}]_{\theta}$ encodes, by duality $\<f_{z},x\> = \<f_{\theta},x\>$, the $[X_{0}, X_{1}]_{\theta}$-norm on $X_{0} \cap X_{1}$.
Let also
\[
1 \leq p_{0} < p_{1} \leq \infty
\text{ and }
\frac{1}{p_\theta} = \frac{1-\theta}{p_0} + \frac{\theta}{p_1}.
\]
Write $Y_{j} := L^{p_{j}}_{0}(X_{j})$ for $j=0,1$ and $Y_{\theta} := L^{p_{\theta}}_{0}([X_{0},X_{1}]_{\theta})$.
Then we have that
\[
[Y_{0}, Y_{1}]_{\theta} = Y_{\theta}
\]
with equal (not merely equivalent) norms.
\end{theorem}
\begin{proof}
Observe first that the space $SF(X_{0} \intersection X_{1})$ of simple functions with values in $X_{0} \intersection X_{1}$ is dense in $Y_{0} \intersection Y_{1}$ provided with the norm $\max\{||\cdot||_{Y_{0}}, ||\cdot||_{Y_{1}}\}$.
Indeed, let $f \in Y_{0} \intersection Y_{1}$.
By definition there exist simple functions $f_{j}$ with values in $X_{j}$ which approximate $f$ in $Y_{j}$ for $j=0,1$, respectively.
Hence we may assume without loss of generality that $\mu(\Omega) < \infty$, $f_{j} = 1_{\Omega} x_{j}$ with $x_{j} \in X_{j}$ and $||f-f_{j}||_{Y_{j}} < \epsilon$.
But then by the Chebyshev inequality we have that
\[
\mu(||f-1_{\Omega} x_{j}||_{X_{j}} > \epsilon (2/\mu(\Omega))^{1/p_{j}}) < \mu(\Omega)/2, \text{ for } j=0,1
\]
in the case $p_{1} < \infty$, and therefore there exists some $\omega \in \Omega$ such that $||f(\omega)-x_{j}||_{X_{j}} \leq \epsilon (2/\mu(\Omega))^{1/p_{j}}$, and thus we obtain
\[
||f-1_{\Omega} f(\omega)||_{Y_{j}} < (1 + 2^{1/p_{j}}) \epsilon, \text{ for } j=0,1,
\]
which is an approximation of $f$ by a $X_{0} \intersection X_{1}$-valued simple function.
The case $p_{1} = \infty$ is similar but easier.

Therefore, by Proposition~\ref{prop:F-X0-X1-D-dense}, the space $SF(X_{0} \intersection X_{1})$ is dense in $[Y_{0}, Y_{1}]_{\theta}$.
It thus suffices to verify that the $[Y_{0}, Y_{1}]_{\theta}$- and the $Y_{\theta}$-norm coincide for every $f = \sum_{k} a_{k} \chi_{k} x_{k}  \in SF(X_{0} \intersection X_{1})$.
In this decomposition $\chi_{k}$ are characteristic functions with disjoint support, $a_{k}$ are positive real numbers and $||x_{k}||_{[X_{0},X_{1}]_{\theta}} = 1$ for all $k$.

For each $k$ let $(x_{k})_{z} \in \mathcal{F}(X_{0},X_{1})$ be an analytic representative of $x_{k}$ with norm at most $1+\epsilon$.
Then
\[
f_{z} = e^{\epsilon(z^{2}-\theta^{2})} \sum_{k} a_{k}^{p_{\theta}/p_{z}} \chi_{k} (x_{k})_{z}
\]
is an analytic function with $f_{\theta} = f$ and $||f_{z}||_{\mathcal{F}(Y_{0},Y_{1})} \leq (1+\epsilon)e^{\epsilon(1-\theta^{2})} ||f||_{Y_{\theta}}$.
This proves the existence of a canonical contractive surjection
\[
Y_{\theta} \hookrightarrow [Y_{0}, Y_{1}]_{\theta}.
\]
For the converse we use the spaces $Z_{j}$.
Without loss of generality we may assume that $||f||_{Y_{\theta}}=1$.
Let $f_{z} \in \mathcal{F}(Y_{0},Y_{1})$ be an arbitrary function satisfying $f_{\theta}=f$.

For every $k$ let $x'_{k} \in [Z_{0}, Z_{1}]_{\theta}$ be such that $\<x_{k}, x'_{k}\> = 1$ and $||x'_{k}||_{\theta; Z_{0}, Z_{1}} < 1+\epsilon$.
Let $(x'_{k})_{z} \in \mathcal{F}(Z_{0},Z_{1})$ be an analytic representative of $x_{k}$ with norm bounded by $1+\epsilon$.
Then
\[
g_{z} = e^{\epsilon(z^{2}-\theta^{2})} \sum_{k} a_{k}^{p_{\theta}/p_{z}'} \chi_{k} (x'_{k})_{z}
\]
is a function in $\mathcal{F}(L^{p_{0}'}(Z_{0}), L^{p_{1}'}(Z_{1}))$ and the three lines lemma~\ref{lemma:three-lines} shows that
\begin{multline*}
||f||_{Y_{\theta}} = 1 = \< f_{z}, g_{z}\>
\leq ||f_{z}||_{\mathcal{F}(Y_{0},Y_{1})} ||g_{z}||_{\mathcal{F}(L^{p_{0}'}(Z_{0}), L^{p_{1}'}(Z_{1}))}\\
\leq ||f_{z}||_{\mathcal{F}(Y_{0},Y_{1})} (1+\epsilon)e^{\epsilon(1-\theta^{2})}.
\end{multline*}
Letting $\epsilon \to 0$ and taking the infimum on the right-hand side we obtain the contractivity of the embedding $[Y_{0}, Y_{1}]_{\theta} \hookrightarrow Y_{\theta}$.
\end{proof}

Thanks to Theorem~\ref{thm:complex-interpolation-dual} we can take $Z_{j} = X_{j}'$ whenever at least one of $X_{j}$ is reflexive.
The next example justifies the use of proper subspaces of $X_{j}'$.

The mixed-norm $L^P$ spaces are given by
\[
L^{P_j}_{0}(\Prod \Omega_k) = L^{P_{j,n}}_{0}(\Omega_n, L^{P_{j,n-1}}_{0}(\Omega_{n-1},\dots,L^{P_{j,1}}_{0}(\Omega_1)\dots))
\]
for some $1 \leq P_{j,k} \leq \infty$, $j=0,1$.
See \cite{MR0126155} for their properties; they are very similar to those of the usual $L^p$-spaces.
For $0 < \theta < 1$ define intermediate exponents $P_\theta$ and conjugate exponents $P_\theta'$ componentwise, e.g.\ $P_{\theta,k}\inv = (1-\theta) P_{0,k}\inv + \theta P_{1,k}\inv$.

\begin{corollary}
\index{complex interpolation!between mixed-norm Lp spaces@between mixed-norm $L^p$ spaces}
\label{cor:complex-interpolation-LP}
If for every $k$ we have $P_{0,k} \neq P_{1,k}$ then
\[
L^{P_\theta}_{0} = [ L^{P_0}_{0}, L^{P_1}_{0} ]_\theta.
\]
\end{corollary}
\begin{proof}
We induct on $n$ and use Theorem~\ref{thm:bochner-complex-interpolation} in each step.
When $n=1$ we use $X_{0}=X_{1}=Z_{0}=Z_{1}=\C$, while for larger $n$ we consider
\[
X_{j} := L^{(p_{j,1}, \dots, p_{j,n-1})}_{0}(\Prod_{k=1}^{n-1} \Omega_k)
\text{ and }
Z_{j} := L^{(p'_{j,1}, \dots, p'_{j,n-1})}_{0}(\Prod_{k=1}^{n-1} \Omega_k).
\qedhere
\]
\end{proof}

In order to obtain a continuous operator between interpolation spaces in Theorem~\ref{th:stein} we need the following stronger version of Proposition~\ref{prop:F-X0-X1-D-dense}.

\begin{proposition}
\label{prop:lp-interpolation-by-schwartz-functions}
Let $1 \leq p_{0} < p_{1} \leq \infty$ and consider $\Schwartz(\R^{n})$ with its usual \Frechet\ space structure.
Let $f \in \Schwartz(\R^{n})$ be such that $||f||_{p_{\theta}} = 1$.
Then, for every $\epsilon>0$, there exists a function $f_{z} \in \mathcal{F}(L^{p_{0}}(\R^{n}),L^{p_{1}}(\R^{n}),\Schwartz(\R^{n}))$ such that $||f_{z}||_{\mathcal{F}(L^{p_{0}},L^{p_{1}})} \leq 1 + \epsilon$ and $f_{\theta} = f$.
\end{proposition}
This result tells us that the $[L^{p_{0}}(\R^{n}), L^{p_{1}}(\R^{n})]_{\theta}$-norm of a Schwartz function can be calculated considering only functions in $\mathcal{F}(L^{p_{0}}(\R^{n}), L^{p_{1}}(\R^{n}), \Schwartz(\R^{n}))$.
\begin{proof}
It suffices to find a uniformly bounded analytic function which satisfies the conclusion, since the decay can always be obtained at an arbitrarily small cost in $\epsilon$ by multiplication with $e^{\delta (z^{2} - \theta^{2})}$ with $\delta$ small enough.

Fix some $\delta > 0$ and let $\{ \tilde\phi_{j} \}_{j\in\N}$ be a smooth partition of identity on $\C$ by non-negative radial functions such that
\begin{align*}
\supp \tilde\phi_{0}(z) &\subset \{ |z| < \delta (1+\delta)^{3} \},\\
\supp \tilde\phi_{j}(z) &\subset \{ \delta (1+\delta)^{2j} < |z| < \delta (1+\delta)^{2j+3} \} &\text{if } j > 0
\end{align*}
and let $\phi_{j}(z) := z \tilde\phi_{z}$.
Then
\[
f = \sum_{j} \phi_{j} \circ f,
\]
where only finitely many summands are non-zero.

Let $\tilde\chi_{j}$ be the characteristic function of $\supp \tilde\phi_{j}$.
Set $\chi_{j}(z) := z \tilde\chi_{j}$ and $m_{j} := \sup |\chi_{j}|$.
Define
\[
f_{z} := \phi_{0} \circ f + \sum_{j > 0} m_{j}^{p_{\theta} / p_{z}} \frac{\phi_{j} \circ f}{m_{j}}.
\]
This is a linear combination of Schwartz functions with bounded coefficients which are analytic on $S$ and continuous on $\bar S$.
In the case $p_{z} \in [1,\infty)$ we use the pointwise estimate
\begin{align*}
\left| m_{j}^{p_{\theta}/p_{z}} \frac{\phi_{j} \circ f}{m_{j}} + m_{j+1}^{p_{\theta}/p_{z}} \frac{\phi_{j+1} \circ f}{m_{j+1}} \right|^{p_{z}}
&\leq
\left| m_{j+1}^{p_{\theta}/p_{z}} \frac{\phi_{j} \circ f}{m_{j}} + m_{j+1}^{p_{\theta}/p_{z}} \frac{\phi_{j+1} \circ f}{m_{j}} \right|^{p_{z}}\\
&\leq
m_{j}^{p_{\theta}} (1+\delta)^{4 p_{\theta}} \left| \frac{\chi_{j} \circ f}{m_{j}} \right|^{p_{z}}
\end{align*}
when $f(x) \in A_{j} := \supp\phi_{j} \setminus \supp\phi_{j-1}$, $j>0$, and obtain
\begin{align*}
||f_{z}||_{p_{z}}^{p_{z}}
&\leq
\int_{|f|<\delta (1+\delta)^{3}} |f|^{p_{z}} + (1+\delta)^{4 p_{\theta}} \sum_{j>0} \int_{f \in A_{j}} m_{j}^{p_{\theta}} \left| \frac{\chi_{j} \circ f}{m_{j}} \right|^{p_{z}}\\
&\leq
\epsilon/2 + (1+\delta)^{4 p_{\theta}} \max\{1, (1+\delta)^{4 (p_{z} - p_{\theta})}\} \sum_{j>0} \int_{f \in A_{j}} m_{j}^{p_{\theta}} \left| \frac{\chi_{j} \circ f}{m_{j}} \right|^{p_{\theta}}\\
&\leq
\epsilon/2 + \max\{(1+\delta)^{4 p_{\theta}}, (1+\delta)^{4 p_{z}}\} || f ||_{p_{\theta}}^{p_{\theta}}\\
&\leq
1 + \epsilon
\end{align*}
for sufficiently small $\delta$.
An analogous estimate holds when $\Im z \neq 0$.
In the case $p_{z} = \infty$ we use the even simpler pointwise estimate
\[
\left| \frac{\phi_{j} \circ f}{m_{j}} + \frac{\phi_{j+1} \circ f}{m_{j+1}} \right|
\leq
\left| \frac{\chi_{j} \circ f}{m_{j}} \right|
\]
whenever $f(x) \in A_{j}$ and conclude
\[
||f_{z}||_{\infty}
\leq ||\phi_{0} \circ f||_{\infty} + \max_{j>0} || \frac{\chi_{j} \circ f}{m_{j}} ||_{\infty}
\leq \epsilon + 1
\]
for $\delta$ small enough.
\end{proof}


%% file: chapter_fractional_integration.tex
\chapter{Fractional integration}
With an interpolation theorem for analytic families of operators at our disposition, we are now interested in obtaining such families.
In this chapter we extend a semigroup $T^n$ of differential operators on $\Schwartz(\R^n)$ parameterized by a natural number to a family of operators $T^z$ parameterized by a complex variable following mainly Stein's exposition in \cite{MR0290095}.
There is a slight obstacle in the way.

\begin{theorem}[\cite{MR0124611}]
\label{th:local-operators-are-differential}
Let $\Omega \subset \R^n$ be open and $D : C^\infty_c(\Omega) \to C^\infty_c(\Omega)$ be a linear operator which is local, i.e.\ $D$ satisfies
\[
\supp (Df) \subseteq \supp (f)
\quad
\text{for all } f \in C^\infty_c(\Omega).
\]
Then $D$ is a finite order differential operator in a neighborhood of each point of $\Omega$.
\end{theorem}

The proof is deferred to the end of the chapter.
This result shows that one cannot expect the operators $T^z$ to be local.
On the contrary, the differential operators are exceptional in this respect, so that one may expect them to arise at some exceptional points of $T^z$.

\section{Riesz potentials}
For $z$ with $\Re z > -n$ define a tempered distribution $g_z \in \Schwartz'(\R^n)$ by
\[
g_z (\phi) = \int_{\R^n} |y|^{z} \phi(y) \dif y = \int_{r=0}^\infty r^{z+n-1} \Phi(r) \dif r,
\]
where
\[
\Phi(r) = \Omega_n \int_{S^{n-1}} \phi(r \sigma) \dif \sigma
\]
and $\Omega_n = \frac{2 \pi^{n/2}}{\Gamma(n/2)}$ is the area of the unit sphere in $\R^{n}$.

In order to extend the family $g_{z}$ to $\{\Re z \leq -n\}$, rewrite it in the form
\begin{align*}
g_z (\phi)
&= \int_{r=1}^\infty r^{z+n-1} \Phi(r) \dif r
+ \int_0^1 r^{z+n-1} \left( \Phi(r) - \sum_{k=0}^{m} \frac{r^{2k} \Phi^{(2k)}(0)}{(2k)!} \right) \dif r\\
&\quad + \Sum_{k=0}^{m} \frac{\Phi^{(2k)}(0)}{(z + n + 2k) (2k)!}.
\end{align*}
The first term on the right is a uniform limit of Riemann sums for bounded $|z|$ and thus analytic on $\C$.
Since the odd derivatives of $\Phi$ vanish at zero, the expression in the parentheses in the second term is bounded by $(2m+2)! r^{2m+2} \sup_r |\Phi^{(2m+2)}(r)|$.
If $\Re z > -n-2m-2$ and $|z|$ bounded, it is a uniform limit of Riemann sums as well.

The last term, on the other hand, is meromorphic and has simple poles at $-n-2m,-n-2m+2,\dots,-n$.
Since the right-hand side does not depend on $m$ in its domain of definition, $g$ extends to a weakly analytic, meromorphic family of distributions on $\C \setminus -n-2\N$.

To remove the singularities, consider
\[
h_z = g_z \Gamma\left(\frac{z+n}{2}\right)\inv.
\]
Since $\Gamma((\cdot+n)/2)$ does not vanish anywhere, $h_z$ is defined for all $z \not\in -n-2\N$.
On the other hand, $\Gamma((\cdot+n)/2)$ has simple poles at $-n-2\N$, so that the quotient extends to a linear form on $\Schwartz(\R)$ at each of these points.
To see that this linear form is continuous, i.e.\ a tempered distribution, we will compute it explicitly.

Let $L$ denote the Laplace operator.
In polar coordinates $L=r^{1-n} \pdif{}{r} r^{n-1} \pdif{}{r} + r^{-2} \Delta_{S^{n-1}}$.
Averaging over $S^{n-1}$ we obtain
\begin{align*}
L^k \phi (0)
&= \Omega_n\inv (L^k \Phi(|\cdot|)) (0)\\
&= \Omega_n\inv \frac{\Phi^{(2k)}(0)}{(2k)!} \Prod_{j=1}^k (2j)(2j+n-2)\\
&= \frac{\Gamma(n/2)}{2 \pi^{n/2}} \frac{\Phi^{(2k)}(0)}{(2k)!} 2^k k! \frac{(2k+n-2)!!}{(n-2)!!}\\
&= \frac{\Gamma(k+n/2)}{2 \pi^{n/2}} 2^{2k} k! \frac{\Phi^{(2k)}(0)}{(2k)!},
\end{align*}
where in the last line we have used the relation
\[
m!! =
\begin{cases}
2^{m/2} \Gamma(m/2+1) & \text{for } m\in\N \text{ even},\\
\sqrt{2/\pi} 2^{m/2} \Gamma(m/2+1) & \text{for } m\in\N \text{ odd}.
\end{cases}
\]
The residuum of $\Gamma((\cdot+n)/2)$ at $-n-2k$ is $2 (-1)^k/k!$.
Thus
\begin{equation}
\begin{split}
h_{-n-2k}(\phi)
&= \left(\frac{2 (-1)^k}{k!}\right)\inv \frac{\Phi^{(2k)}(0)}{(2k)!}\\
&= \frac{k!}{2 (-1)^k} \left( \frac{\Gamma(k+n/2)}{2 \pi^{n/2}} 2^{2k} k! \right)\inv L^k \phi (0)\\
&= \pi^{n/2} 2^{-2k} \Gamma(k+n/2)\inv ((-L)^k \phi)(0),
\end{split}
\label{eq:laplace-h}
\end{equation}
which is a tempered distribution.
We have shown that $h_z$ is a weakly entire distribution-valued function in the sense that $h_z(\phi)$ is entire for every $\phi \in \Schwartz(\R^n)$.

Note that while $h_z$ is a scalar multiple of a power of the Laplacian for $z \in -n-2\N$, it is clearly not a local operator for any other $z$.

\begin{definition}
\index{Riesz potential}
Let $f$ be a fixed function.
Define the \emph{Riesz potential} of order $\gamma$ by
\[
I^\gamma f(x)
= 2^{-\gamma} \pi^{-n/2} \frac{\Gamma(n/2 - \gamma/2)}{\Gamma(\gamma/2)} \int_{\R^n} |y|^{\gamma-n} f(x+y) \dif y
\]
whenever the integral converges and extend the mapping $\gamma \mapsto I^\gamma f(x)$ by analytic continuation to the maximal domain on which the extension is unambiguous.
\end{definition}
For $f\in\Schwartz(\R^n)$ the Riesz potential $I^{\gamma}f(x)$ coincides with
\[
2^{-\gamma} \pi^{-n/2} \Gamma\left(\frac{n - \gamma}{2}\right) h_{\gamma-n} f(x + \cdot).
\]
By the preceding discussion, this is a meromorphic function with simple poles at $\gamma = n + 2\N$, so that for Schwartz functions, Riesz potentials of all orders except $n + 2\N$ are well defined.
Furthermore, Riesz potentials of Schwartz functions are smooth as functions of $x$.

By (\ref{eq:laplace-h}) the powers of the Laplacian coincide with Riesz potentials on the space of Schwartz functions
\begin{equation}
(-L)^k = I^{-2 k} \quad \text{for }k\in\N, f\in\Schwartz(\R^{n}).
\label{eq:riesz-laplacian}
\end{equation}

\begin{proposition}
\label{prop:riesz-of-smooth}
Let $\phi$ be a smooth function on $\R^n$ satisfying the estimate
\[
|\phi(x)| \leq C |x|^{\gamma-n}
\]
with $0 < \gamma < n$.
Then $I^\alpha \phi$ is defined at least when $\Re \alpha < n - \gamma$.
\end{proposition}
\begin{proof}
Clearly, $I^\alpha \phi(y)$ is defined for $0 < \Re \alpha < n - \gamma$ and every fixed $y$.

Using a bump function, decompose $\phi = \phi_1 + \phi_2$, where $\phi_1$ is a smooth function with compact support and $\phi_2$ is identically zero on $B(y,2|y|)$.
Then $I^\alpha \phi_1$ is a holomorphic function on $\Re\alpha < n$, while $\phi_2$ satisfies the estimate
\[
\left| \int \phi_2(x+y) |x|^{\alpha - n} \dif x \right|
\leq
C \int_{|x|>2|y|} |x+y|^{\gamma-n} |x|^{\Re \alpha-n} \dif x
\leq
C
<
\infty
\]
uniformly for $\Re \alpha < C_1 < n - \gamma$.
Therefore, we may derive under the integration sign, so that $I^\alpha \phi_2(y)$ has an analytic continuation to the region $\Re\alpha < n - \gamma$.
The sum of the continuations of $I^\alpha \phi_1$ and $I^\alpha \phi_2$ is a continuation of $I^\alpha \phi$ and is independent of the choice of the decomposition $\phi = \phi_1 + \phi_2$.
\end{proof}

\section{Composition of Riesz potentials}
The Riesz potentials convolve Schwartz functions with some distributions, an operation which may not yield a Schwartz function, so that we might not be able to apply a Riesz potential to the result.
To evade this problem for the moment, we are going to find a subspace of $\Schwartz$ invariant under the action of $I^\gamma$.

We calculate the Fourier transform of $h_z$ for certain values of $z$ first.
In a preliminary step, we deduce some formulae for the $\Gamma$ function.
If $\Re z > 0$ and by the change of variable $s = t |\xi|^2/2$ we have
\[
\Gamma\left(-\frac{z}{2}\right)
=
\int_{s=0}^\infty s^{-1-z/2} e^{-s} \dif s
=
\int_{t=0}^\infty t^{-1-z/2} e^{-t |\xi|^2 / 2} \left( \frac{|\xi^2|}{2} \right)^{-z/2} \dif t,
\]
while in the case $\Re z > -n$ the change of variable $s = |x|^2/(2 t)$ yields
\begin{align*}
\left(\frac{|x|^2}{2}\right)^{-n/2-z/2} \Gamma\left(\frac{z+n}{2}\right)
&=
\int_{s=0}^\infty \left(\frac{|x|^2}{2}\right)^{-n/2-z/2} s^{-1+n/2+z/2} e^{-s} \dif s\\
&=
\int_{t=0}^\infty t^{-1-z/2} t^{-n/2} e^{-|x|^2 / 2 t} \dif t
\end{align*}
The following calculation is based on the fact that the Fourier transform of a Gaussian is once again a Gaussian, the Plancherel theorem and the two preceding formulae.
For $-n < \Re z < 0$, we have that
\begin{align*}
\Gamma\left(\frac{z+n}{2}\right)& \Gamma\left(-\frac{z}{2}\right)
h_z (\Fourier \phi)\\
&=
\Gamma\left(-\frac{z}{2}\right) \int \Fourier \phi(\xi) |\xi|^z \dif \xi\\
&=
\int \Fourier \phi(\xi) |\xi|^z \int_{t=0}^\infty t^{-1-z/2} e^{-t |\xi|^2 / 2} \left( \frac{|\xi^2|}{2} \right)^{-z/2} \dif t \dif \xi
\\
&= 2^{z/2}
\int_{t=0}^\infty t^{-1-z/2} \int \Fourier \phi(\xi) e^{-t |\xi|^2 / 2} \dif \xi \dif t\\
&= 2^{z/2}
\int_{t=0}^\infty t^{-1-z/2} (2\pi)^{n/2} t^{-n/2} \int \phi(x) e^{-|x|^2 / 2 t} \dif x \dif t\\
&= 2^{z+n/2} \Gamma\left(\frac{z+n}{2}\right)
(2\pi)^{n/2} \int \phi(x) |x|^{-n-z} \dif x
\\
&= 2^{z+n} \pi^{n/2}
\Gamma\left(\frac{z+n}{2}\right) \Gamma\left(-\frac{z}{2}\right)
h_{-n-z}(\phi)
\end{align*}
By definition of the Fourier transform of a distribution we obtain that
\begin{equation}
\Fourier h_z = 2^{z+n} \pi^{n/2} h_{-n-z}
\label{eq:fourier-hz}
\end{equation}
for $-n < \Re z < 0$.
However, the functions on left- and right-hand side are entire, so that we have equality on $\C$ by analytic continuation.

\begin{definition}
Let $\Schwartz^*$ be the space of Schwartz functions which are orthogonal to all polynomials.
\end{definition}
The Fourier transform of $\Schwartz^*$ is the space $\Schwartz_0$ of all Schwartz functions such that all their derivatives vanish at origin.

\begin{proposition}
Every Riesz potential $I^\gamma$ maps the space $\Schwartz^*$ into itself.
\end{proposition}
\begin{proof}
Let $f \in \Schwartz^*$, $g \in \Schwartz$ and $\Re z < 0$.
Then
\begin{equation}
\label{eq:FT-h-z-convolution-f}
\begin{split}
\Gamma(-z/2) 2^{-z} \pi^{n/2} h_{z} (f*g)
&=
\Gamma(-z/2) \Fourier h_{-n-z} (f * g)\\
&=
\Gamma(-z/2) h_{-n-z} ( \Fourier (f * g) )\\
&=
\Gamma(-z/2) h_{-n-z} ( \Fourier f \cdot \Fourier g )\\
&=
\int_{\R^n} |\xi|^{-n-z} \Fourier f(\xi) \Fourier g(\xi) \dif \xi.
\end{split}
\end{equation}
Since $\Fourier f \in \Schwartz_0$ and the latter space is preserved under multiplication by $|\xi|^{-n-z}$, the right-hand side is finite for all $z \in \C$.
By analytic continuation, equality holds for all $z \in \C \setminus 2\N$.
Since $\Fourier$ is a continuous mapping of $\Schwartz$ into itself, the map $\psi_z : g \mapsto h_{z} (f * g)$ is a tempered distribution.
Its Fourier transform $\Fourier \psi_z$ is a $\Schwartz_0$-function, so that $\psi_z \in \Schwartz^*$ by the bijectivity of Fourier transform on the space of tempered distributions.
Furthermore,
$\psi_z(g) = h_z (f * g) = (h_z * f)(g)$.
For $z \not\in 2\N$, this shows that $I^{z+n} f = 2^{-z-n} \pi^{-n/2} \Gamma(-z/2) (h_z * f) \in \Schwartz^*$.

For $z \in 2\N$, $(h_z * f)(x)=0$ for every $x$, because in such a case $h_z$ is a polynomial and $f$ is orthogonal to all polynomials.
Therefore
\[
F(z,x) = 2^{-z} \pi^{n/2} \Gamma(-z/2) (h_z * f)(x)
\]
extends to an entire function of $z$.
On the other hand, $F(z,x)$ is the inverse Fourier transform of $|\xi|^{-z-n} \Fourier f$ for $z \not\in 2\N$, and the latter are Schwartz functions with Schwartz seminorms bounded locally uniformly in $z$, so that $F(2k,\cdot)$ is in fact pointwise equal to the inverse Fourier transform of $|\xi|^{-2k-n} \Fourier f$, so it is in particular a $\Schwartz_0$-function.
\end{proof}
As a side product of the proof we have, for every $f \in \Schwartz^*$ and $g \in \Schwartz$,
\[
(2 \pi)^n \int (I^{\gamma} f)(x) g(x) \dif x
=
\int_{\R^n} |\xi|^{-\gamma} \Fourier f(\xi) \Fourier g(\xi) \dif \xi,
\]
which gives the following explicit formula for the Fourier transform of $I^{\gamma} f$:
\begin{equation}
\label{eq:fourier-of-riesz}
\Fourier (I^{\gamma} f)(\xi) = |\xi|^{-\gamma} \Fourier f(\xi).
\end{equation}
\begin{proposition}
\label{prop:I-alpha-beta-star}
For every $f \in \Schwartz^*$ and every $\alpha, \beta \in \C$,
\[
I^{\alpha} I^{\beta} f = I^{\alpha + \beta} f.
\]
\end{proposition}
\begin{proof}
By definition of $I^\gamma$ and (\ref{eq:fourier-of-riesz}), for $f \in \Schwartz^*$, $g \in \Schwartz$,
\begin{align*}
\int (I^{\alpha} I^{\beta} f)(x) g(x) \dif x
&=
(2 \pi)^{-n} \int \Fourier(I^{\alpha} I^{\beta} f)(\xi) \Fourier g(\xi) \dif \xi\\
&=
(2 \pi)^{-n} \int |\xi|^{-\alpha} \Fourier(I^{\beta} f)(\xi) \Fourier g(\xi) \dif \xi\\
&=
(2 \pi)^{-n} \int |\xi|^{-\alpha} |\xi|^{-\beta} \Fourier(f)(\xi) \Fourier g(\xi) \dif \xi\\
&=
(2 \pi)^{-n} \int \Fourier(I^{\alpha+\beta} f)(\xi) \Fourier g(\xi) \dif \xi\\
&=
\int ((I^{\alpha+\beta} f)(x) g(x) \dif x,
\end{align*}
so that $I^{\alpha} I^{\beta} f = I^{\alpha+\beta} f$ as distributions and thus also as functions.
\end{proof}

\begin{proposition}
\label{prop:I-alpha-beta}
For every $f \in \Schwartz$ and every $\alpha, \beta \in \C$ such that
$\Re \alpha > 0$, $\Re \beta > 0$, $\Re (\alpha + \beta) < n$,
\[
I^{\alpha} I^{\beta} f = I^{\alpha + \beta} f.
\]
\end{proposition}
\begin{proof}
In the range $\Re \gamma > 0$, $I^\gamma$ is defined by an integral, so that
\begin{align*}
(I^\alpha I^\beta f)(x)
&=
C \int_{\R^n} |x-y|^{\alpha-n} \int_{\R^n} |y-z|^{\beta-n} f(z) \dif z \dif y\\
&=
C \int_{\R^n} f(z) \int_{\R^n} |x-y|^{\alpha-n} |y-z|^{\beta-n} \dif y \dif z\\
&=
C \int_{\R^n} f(z) \int_{\R^n} |(x-z)-y|^{\alpha-n} |y|^{\beta-n} \dif y \dif z\\
&=
C \int_{\R^n} f(z) |x-z|^{\alpha+\beta-n} \int_{\R^n} \left| \frac{x-z}{|x-z|}- y \right|^{\alpha-n} \left| y \right|^{\beta-n} \dif y \dif z,
\end{align*}
where the last equality follows by a change of variable.
By transitivity of the action of $O(n)$ on the unit sphere $S^{n-1}$, the $y$-integral may be written as
\[
\int_{\R^n} \left| e- y \right|^{\alpha-n} \left| y \right|^{\beta-n} \dif y
\]
with a fixed unit vector $e$.
The assumption $\Re(\alpha+\beta) < n$ implies that this integral is finite, so that
\begin{align*}
(I^\alpha I^\beta f)(x)
&=
C \int_{\R^n} f(z) |x-z|^{\alpha+\beta-n} \dif z\\
&=
C (I^{\alpha+\beta} f)(x).
\end{align*}
By Proposition~\ref{prop:I-alpha-beta-star} the constant is $1$.
\end{proof}

\section{Inverse of a Riesz potential}
Proposition~\ref{prop:I-alpha-beta-star} gives the inversion formula $(I^\gamma)\inv = I^{-\gamma}$ on the space $\Schwartz^*(\R^n)$.
On the full space $\Schwartz(\R^n)$ there may be some integrability issues which prevent one from defining the composition $I^{-\gamma} I^\gamma$.
Restricting to a subset of possible $\gamma$ resolves the problem.
\begin{proposition}
\label{prop:I-pm-gamma}
For every $f \in \Schwartz$ and every $\gamma$ such that $0 < \gamma < n$,
\[
I^{-\gamma} I^{\gamma} f = f.
\]
\end{proposition}
\begin{proof}
By Proposition~\ref{prop:I-alpha-beta}, we have
\[
I^{\alpha} I^{\gamma} f (x) = I^{\alpha + \gamma} f (x)
\]
whenever $0 < \Re \alpha < n-\gamma$.
The right-hand side is an entire function of $\alpha$ which is equal to $f(x)$ for $\alpha = -\gamma$.
It is therefore sufficient to show that the left-hand side has an analytic extension to the region $\Re \alpha < n - \gamma$.

Let $\phi = I^\gamma f$.
Since $f \in \Schwartz(\R^n)$, $\phi$ is smooth and we have $f(y) \leq C (1 + |y|)^{-2n}$.
Now,
\begin{align*}
|\phi(x)|
&=
C \left| \int_{\R^n} f(y) |x-y|^{\gamma-n} \dif y \right|\\
&\leq
C \int_{|y| \leq |x|/2} \frac{|x-y|^{\gamma-n}}{(1 + |y|)^{2n}} \dif y
+
C \int_{|y| > |x|/2} (1 + |y|)^{-2n} |x-y|^{\gamma-n} \dif y\\
&\leq
C \int_{|y| \leq |x|/2} \frac{|x/2|^{\gamma-n}}{(1 + |y|)^{2n}} \dif y
+
C \int_{|y| > |x|/2} \frac{(1 + |y|)^{-2n-\gamma+n}}{(1 + |x/2|)^{-\gamma+n}} |x - y|^{\gamma-n} \dif y\\
&\leq
C |x|^{\gamma-n}
+
C |x|^{\gamma-n} \int_{\R^n} (1 + |y|)^{-2n-\gamma+n} |x - y|^{\gamma-n} \dif y\\
&\leq
C |x|^{\gamma-n}.
\end{align*}
By Proposition~\ref{prop:riesz-of-smooth}, $I^{-\gamma} \phi(x)$ is well defined, and by analyticity of the continuation it coincides with $f(x)$.
\end{proof}

\section{Supplement: Local operators are differential}
We present here the proof of Theorem~\ref{th:local-operators-are-differential} which may be found in \cite[Theorem II.1.4]{MR1790156}.
In this section, we use the symbol $||\cdot||_m$ to denote the seminorms
\[
||f||_m = \sup_{|\alpha| \leq m, x \in \R^n} |D^\alpha f (x)|.
\]
The operator $D$ extends to an operator on $C^\infty(\Omega)$ in a natural fashion since every smooth function $f$ agrees with a smooth function with compact support $f_x$ in a neighborhood $U_x$ of each point $x$, and $\left. D f_x \right|_{U_x}$ does not depend on the choice of $f_x$ by the locality of $D$.
\begin{lemma}
\label{lem:approx-by-locally-vanishing-functions}
Let $f \in C^\infty(\R^n)$ be such that $D^\alpha f (x) = 0$ for some $x$ and all $|\alpha| \leq m$.
Then there is a family of functions $f_\delta$, $\delta>0$ such that $f_\delta \equiv 0$ on a neighborhood of $x$ for each $\delta$ and $||f - f_\delta||_m \to 0$ as $\delta \to 0$.
\end{lemma}
\begin{proof}
Without loss of generality, assume $x=0$.
Let $\phi$ be a bump function which is $1$ in a neighborhood of the identity and $0$ outside of the unit ball in $\R^n$.
Let
\[
f_\delta(y) = f(y) (1 - \phi(y/\delta)).
\]
Then $f_\delta \equiv 0$ in a neighborhood of $0$ and
\[
D^\alpha (f - f_\delta)(y)
=
\sum_{\beta \leq \alpha} {\alpha \choose \beta}
(D^{\beta} f)(y) (D^{\alpha - \beta} \phi)(y/\delta) \delta^{-|\alpha|+|\beta|},
\]
so that, for $|\alpha| \leq m$,
\begin{align*}
\sup_{y \in \R^n} |D^\alpha (f-f_\delta)(y)|
&\leq
C \sup_{y \in \R^n}
\sum_{\beta \leq \alpha}
\left| (D^{\beta} f)(y) (D^{\alpha - \beta} \phi)(y/\delta) \delta^{-|\alpha|+|\beta|} \right|\\
&\leq
C \sup_{|y| \leq \delta}
\sum_{\beta \leq \alpha}
\left| (D^{\beta} f)(y) \delta^{-|\alpha|+|\beta|} \right|\\
&=
O(\delta)
\end{align*}
as $\delta \to 0$, because $(D^\beta f)(y) = O(|y|^{m-|\beta|+1})$ by Taylor's formula.
\end{proof}

\begin{lemma}
\label{lem:Cm-C0-continuous-operators-are-differential}
Let $\Omega \subset \R^n$ be open and $D : C^\infty_c(\Omega) \to C^\infty_c(\Omega)$ be a local linear operator satisfying
\begin{equation}
\label{eq:Cm-C0-continuous}
|| D f ||_0 \leq C || f ||_m
\end{equation}
for all $f \in C^\infty_c(\Omega)$.
Then $D$ is a differential operator of order $m$ on $\Omega$, i.e.
\[
D f(y) = \sum_{|\alpha| \leq m} a_\alpha(y) (D^\alpha f)(y)
\]
for some smooth functions $a_\alpha$ and all $f \in C^\infty(\Omega)$ (not necessarily with compact support).
\end{lemma}
\begin{proof}
Let
\[
P_{\alpha,a} = \Prod_{i=1}^n \frac{1}{\alpha_i !} (x_i - a_i)^{\alpha_i}.
\]
Then
\[
D^\beta P_{\alpha,x} = \delta_\alpha^\beta,
\]
and, for every $f \in C^\infty_c(\Omega)$ and $a \in \Omega$,
\[
g(x) = f(x) - \sum_{|\alpha| \leq m} (D^\alpha f)(a) P_{\alpha,a}
\]
is a smooth function on $\Omega$ and $(D^\alpha g)(a) = 0$ for all $|\alpha| \leq m$.
Therefore, by Lemma~\ref{lem:approx-by-locally-vanishing-functions}, $g$ may be approximated in $C^m$-norm by some functions $g_\delta$ vanishing in some neighborhoods of $a$.
By locality of $D$, $(D g_\delta)(a) = 0$, while by (\ref{eq:Cm-C0-continuous}), $(D g_\delta)(a) \to (D g)(a)$ as $\delta \to 0$, so that $(D g)(a) = 0$ as well.
By linearity of $D$, we have
\begin{equation}
\label{eq:representation-as-differential-operator}
(D f)(a) = \sum_{|\alpha| \leq m} (D^\alpha f)(a) (D P_{\alpha,a})(a).
\end{equation}
Now, $P_{\alpha, a}$ is a polynomial in $x$ with coefficients which are polynomials in $a$, say
\[
P_{\alpha, a} = \sum_i p_i(x) q_i(a).
\]
By linearity of $D$, we have
\[
(D P_{\alpha,a})(a)
=
\sum_i (D p_i)(a) q_i(a),
\]
which is a smooth function of $a$.
Therefore, (\ref{eq:representation-as-differential-operator}) is the representation of $D$ in the required form.
\end{proof}

\begin{lemma}
\label{lem:local-operators-are-locally-Cm-C0-continuous}
Let $\Omega \subset \R^n$ be open and $D$ be a local operator on $C^\infty_c(\Omega)$.
Then, for every $x \in \Omega$, there exists a relatively compact neighborhood $U$ of $x$ and a natural number $m$ such that
\[
||D f||_0 \leq ||f||_m \quad \text{for all } f \in C^\infty_c(U \setminus \{x\}).
\]
\end{lemma}
\begin{proof}
Suppose that the contrary is the case, i.e.\ that for every relatively compact neighborhood $U$ of $x$, every $m$ and every $M > 0$ there exists a function $u \in C^\infty_c(U \setminus \{x\})$ such that $||D u||_0 > M ||u||_m$.

Let $U_0$ be some relatively compact neighborhood and define a sequence of functions $\{ u_k \}$ and relatively compact neighborhoods $\{ U_k \}$ inductively in such a way that
\[
u_k \in C^\infty_c(U_k \setminus \{x\}), \quad ||D u_k||_0 > 2^{2 k} ||u_k||_k,
\]
\[
U_{k+1} = U_k \setminus \supp u_k.
\]
Then the supports of $u_k$ are mutually disjoint, so that
\[
u = \sum_k \frac{2^{-k} u_k}{||u_k||_k}
\]
is well-defined and compactly supported.
Furthermore, the series converges in every $C^m$ norm, so that $u$ is smooth.
On the other hand,
\[
|| (D u) |_{\supp u_k} ||_0
=
\frac{2^{-k}}{||u_k||_k} || (D u_k) |_{\supp u_k} ||_0
>
2^k,
\]
so that $D u$ is unbounded on $U_0$, in contradiction to it being a continuous function on the compact set $\bar U_0$.
\end{proof}

\begin{proof}[Proof of Theorem~\ref{th:local-operators-are-differential}]
Let $x \in \Omega$.
By Lemma~\ref{lem:local-operators-are-locally-Cm-C0-continuous}, there exists an $m$ and a neighborhood $U$ of $x$ such that
\[
|| D f ||_0 \leq C || f ||_m
\]
for all $f \in C^\infty_c(U \setminus \{ x \})$.
By Lemma~\ref{lem:Cm-C0-continuous-operators-are-differential},
\begin{equation}
\label{eq:D-is-a-differential-operator}
D f(y) = \sum_{|\alpha| \leq m} a_\alpha(y) (D^\alpha f)(y)
\end{equation}
for all $y \in U \setminus \{x\}$, some smooth functions $a_\alpha$ and all $f \in C^\infty(U \setminus \{ x \})$.
In particular, this is true for $f \in C^\infty_c(\Omega)$.
Furthermore, the formula (\ref{eq:D-is-a-differential-operator}) remains valid at $x$, since both the left- and the right-hand side are continuous functions of $y$.
Therefore, $D$ is a differential operator of order $m$ in $U$.
\end{proof}


%% file: chapter_radon.tex
\chapter{The Radon transform}
\label{chap:radon}
\index{Radon transform}
The Radon transform is among the simplest integral operators with singular kernel one can possibly imagine.
It was originally introduced in \cite{radon1917} for two dimensions with an indication of possible generalizations to more dimensions and non-Euclidean spaces.

This transform is probably best known in connection with computer tomography.
However the $k$-plane transform used there differs from the Radon transform in the codimension of the support of the integration kernel.
We refer to chapter~\ref{chap:k-plane} for the discussion of this application.

The Radon transform of a Schwartz function $f \in \Schwartz(\R^n)$ is defined by
\[
\Radon f (\sigma, t) = \int_{\<\sigma, x\> = t} f(x) \dif\Leb[n-1](x),\quad
\sigma \in S^{n-1}, t \in \R.
\]
It is a function of the affine hyperplane $\{ x : \<\sigma, x\> = t \}$.

We start by proving two of its basic properties, mostly following the exposition in \cite{MR1723736}.
The first one says that $\Radon f$ vanishes for $|t|>R$ only if $f$ itself vanishes identically for $|x|>R$, suggesting that the transform is invertible.
The second is the actual inversion formula.

Afterwards we will turn to the investigation of boundedness of the Radon transform as an operator from $L^p(\R^n)$ to $L^q(S^{n-1},L^r(\R))$.
The norms of functions in the latter space are denoted by $|| \cdot ||_{q;r}$, and we are interested in estimates of the form
\begin{equation}
\label{eq:radon-p-qr-estimate}
|| \Radon f ||_{q; r} \leq C_{q,r,p} || f ||_p.
\end{equation}

\section{The support theorem}
If the support of $f$ is contained in a ball of radius $R$ centered at the origin, then the Radon transform clearly vanishes for $|t|>R$.
In this section we prove the converse statement, beginning with the special case of radial functions.

\begin{prop}
\label{prop:support-radial}
Let $f \in \Schwartz(\R^n)$ be a radial function, such that for every $|t|>R$ and $\sigma \in S^{n-1}$, $\Radon f(\sigma, t)=0$.
Then $\supp f \subseteq B(0,R)$.
\end{prop}
\begin{proof}
By the assumptions, $f(x)=F(|x|)$ for some even function $F$ and
\begin{align*}
\Radon f(\sigma, t)
&= \int_{\<x,\sigma\>=t} f(x) \Leb[n-1](x)\\
&= \int_{y \in \R^{n-1}} F(\sqrt{t^2 + |y|^2}) \dif y\\
&= \Omega_{n-1} \int_{\rho=0}^\infty F(\sqrt{t^2 + \rho^2}) \rho^{n-2} \dif\rho\\
&= \Omega_{n-1} \int_{r=t}^\infty F(r) (r^2 - t^2)^{\frac{n-3}{2}} r \dif r \quad (r^2 = \rho^2 + t^2)
\end{align*}
Let $s > R$, multiply this equality by $t (t^2 - s^2)^{(n-3)/2}$ and integrate for $t \in (s, \infty)$.
Observe that the left-hand side vanishes identically, so that
\begin{align*}
0 &=
\int_{t=s}^\infty t (t^2 - s^2)^{\frac{n-3}{2}} \int_{r=t}^\infty F(r) (r^2 - t^2)^{\frac{n-3}{2}} r \dif r \dif t\\
&=
\int_{r=s}^\infty F(r) r \int_{t=s}^r (t^2 - s^2)^{\frac{n-3}{2}} (r^2 - t^2)^{\frac{n-3}{2}} t \dif t \dif r.
\end{align*}
With the substitution $2 t^2 = a^2 + s^2 - \tau (a^2 - s^2)$, the inner integral becomes
\begin{multline*}
\int_{\tau=-1}^1
\left[(1-\tau)(a^2 - s^2)/2\right]^{\frac{n-3}{2}}
\left[(1+\tau)(a^2 - s^2)/2\right]^{\frac{n-3}{2}}
\frac{(a^2-s^2) \dif\tau}{2}\\
=
C (a^2 - s^2)^{n-2},
\end{multline*}
and inserting it back into the previous equation, one obtains
\[
\int_{r=s}^\infty F(r) r (a^2 - s^2)^{n-2} \dif r = 0.
\]
Applying the differential operator $\left(\tdif{}{(s^2)}\right)^{n-1}$, one gets $F(s) = 0$.
\end{proof}

Let us now see how the general case reduces to the one already treated.
First, observe that it is not sufficient to average $f$ over the spheres centered at $0$, since every odd function would vanish under this procedure.
This limitation can be dealt with if one takes into consideration the spheres centered at points different from the origin.

Let $x_0 \in \R^n$ and define the spherically averaged function centered at $x_0$ by
\[
f_{x_0} (x) = \int_{S^{n-1}} f (x_0 + \sigma |x-x_0|) \dif \sigma.
\]
If $f$ was rapidly decreasing at infinity, then it is rapidly decreasing too.
Its Radon transform is given by
\begin{align*}
\Radon f_{x_0} (\sigma, t)
&=
\int_{\<\sigma, x\> = t} f_{x_0} (x) \dif\Leb[n-1](x)\\
&=
\int_{\<\sigma, x\> = t} \int_{k \in O(n)} f (x_0 + k(x-x_0)) \dif k \dif\Leb[n-1](x)\\
&=
\int_{k \in O(n)} \int_{\<\sigma, x\> = t} f (x_0 + k(x-x_0)) \dif\Leb[n-1](x) \dif k\\
&=
\int_{k \in O(n)} \int_{\mathrlap{\<k \sigma, k x + x_0 - k x_0\> = t + \<k \sigma, x_0 - k x_0\>}} f (x_0 + k(x-x_0)) \dif\Leb[n-1](x_0 + k(x-x_0)) \dif k\\
&=
\int_{k \in O(n)} \Radon f (k \sigma, t + \<k \sigma, x_0 - k x_0\>) \dif k,
\end{align*}
or, written in a tidier way,
\[
\Radon f_{x_0} (\sigma, t + \<\sigma, x_0 \>) = 
\int_{k \in O(n)} \Radon f (k \sigma, t + \<k \sigma, x_0\>) \dif k.
\]
If $\Radon f$ was supported in $\{|t| \leq R\}$, then $||t| - |x_0||>R$ implies $\Radon f_{x_0} (\sigma, t + \<\sigma, x_0 \>) = 0$.
Visually, this says the Radon transform gets averaged over the hyperplanes which have the same distance to $x_0$ and $\Radon f_{x_0}$ vanishes for hyperplanes outside the ball of radius $R + |x_0|$ around $x_0$.
By Prop.~\ref{prop:support-radial}, the averages $f_{x_0}$ vanish outside the ball too.
Put in clear way, we have $\int_S f = 0$ whenever a sphere $S$ does not intersect the ball $\bar B(0,R)$.
\begin{lemma}
\label{lem:spheric-integrals-vanish}
Assume that, for a rapidly decreasing function $f$, whenever a sphere $S$ does not intersect the ball $\bar B(0,R)$, $\int_S f = 0$.
Then, for every such $S$, also $\int_S f x_j = 0$ for every $j$.
\end{lemma}
\begin{proof}
Let $x$, $\tilde R$ be such that $B(x, \tilde R) \supset \bar B(0, R)$.
Then
\[
\int_{B(x, \tilde R)} f = \int_{R^n} f - \Omega_{n} \int_{\rho=\tilde R}^\infty \rho^{n-1} \int_{\boundary B(x,\rho)} f = \int_{R^n} f = \const.
\]
In particular,
\begin{align*}
0
&= \tdif{}{x_j} \int_{B(x, \tilde R)} f
= \int_{B(x, \tilde R)} \partial_j f
= \int_{B(x, \tilde R)} \nabla_y (f(y) \partial_j) \dif y
= \int_{B(x, \tilde R)} \nabla_y (f(y) \partial_j) \dif y\\
&= \int_{\boundary B(x, \tilde R)} \< f(y) \partial_j, y \> \dif y
= \int_{\boundary B(x, \tilde R)} f(y) (x_j + y_j) \dif y
= \int_{\boundary B(x, \tilde R)} f(y) y_j \dif y,
\end{align*}
the last two equalities being valid by the divergence theorem and by the hypothesis (since $x_j$ is constant).
\end{proof}
This allows us to conclude easily.
\begin{theorem}
\label{thm:support}
Let $f \in \Schwartz(\R^n)$ be such that $\Radon f$ vanishes for every $t>R$ and $\sigma \in S^{n-1}$.
Then $\supp f \subseteq B(0,R)$.
\end{theorem}
\begin{proof}
By the preceding considerations, we may apply Lemma~\ref{lem:spheric-integrals-vanish}.
Since $f(x) x_j$ still satisfies its conditions, we may reiterate.
Thus, by induction on the degree of $P$, we get
\[
\int_{\boundary B(0, \tilde R)} f(x) P(x) \dif x = 0
\]
for every polynomial $P$ and every $\tilde R > R$.
By the Stone-Weierstraß theorem, the polynomials are dense in continuous functions on the sphere, so that this implies $f \equiv 0$ on $\boundary B(0, \tilde R)$.
Since $\tilde R>R$ is arbitrary, the assertion follows.
\end{proof}

\section{The inversion formula}
A basic relation between the Radon and the Fourier transform is
\begin{equation}
\label{eq:radon-fourier}
\begin{split}
\Fourier f(\sigma s)
&=
\int_{\R^n} f(x) e^{-i \<\sigma s, x\>} \dif x\\
&=
\int_{t=-\infty}^\infty \int_{\<\sigma, x\> = t} f(x) e^{-i \<\sigma s, x\>} \dif\Leb[n-1](x) \dif t\\
&=
\int_{t=-\infty}^\infty \int_{\<\sigma, x\> = t} f(x) e^{-i s t} \dif\Leb[n-1](x) \dif t\\
&=
\int_{t=-\infty}^\infty \Radon f (\sigma, t) e^{-i s t} \dif t\\
&=
\Fourier_1 \Radon f (\sigma, s),
\end{split}
\end{equation}
where the symbol ``$\Fourier_1$'' indicates the Fourier transform in the coordinate $s \in \R^1$ for a fixed $\sigma \in S^{n-1}$.

Therefore the formal inverse of the Radon transform is $\Radon\inv = \Fourier\inv \Fourier_1$.
On the other hand, formally up to constants,
\begin{multline*}
\int \Radon f \cdot g
= \int \Fourier_1\inv \Fourier f \cdot g
= \int_{S^{n-1}} \int_\R \Fourier f \cdot \Fourier_1 g\\
= \int_{\R^{n}} \Fourier f \cdot |\xi|^{-n+1} \Fourier_1 g
= \int_{\R^{n}} f \cdot \Fourier\inv |\xi|^{-n+1} \Fourier_1 g.
\end{multline*}
This suggests that the inverse of the Radon transform may also be written as the composition of its adjoint and a Riesz potential.
To calculate the adjoint explicitly consider the duality relation in question.
A calculation using Fubini's theorem shows that
\begin{align*}
\int_{S^{n-1}} \int_\R \Radon f(\sigma, t) g(\sigma,t) \dif t \dif \sigma
&= \int_{S^{n-1}} \int_\R \int_{\<\sigma,x\>=t} f(x) \dif\Leb[n-1](x) g(\sigma,t) \dif t \dif \sigma\\
&= \int_{S^{n-1}} \int_{\R^n} f(x) g(\sigma,\<\sigma,x\>) \dif x \dif \sigma\\
&= \int_{\R^n} f(x) \Radon^* g(x) \dif x,
\end{align*}
where
\[
\Radon^* g(x) = \int_{S^{n-1}} g(\sigma, \<\sigma,x\>) \dif \sigma
\]
is the integral of $g$ over all hyperplanes passing through $x$.

We will now compose the Radon transform and its formal adjoint and apply the inversion formula for the Riesz potential.
To simplify the expression, it will be convenient to use group-theoretic notation.
Let $M(n)$ be the group of isometries on $\R^n$, the action of $g \in M(n)$ being written multiplicatively, $x=g\cdot 0$ and $\xi \in G_{n,n-1}$ a fixed hyperplane through $0$.
\begin{align*}
\Radon^* \Radon f(g \cdot 0)
&=
\int_{k \in O(n)} \Radon f(g k \cdot \xi) \dif k\\
&=
\int_{k \in O(n)} \int_{y \in \xi} f(g k \cdot y) \dif\Leb[n-1](y) \dif k\\
&=
\int_{y \in \xi} \int_{k \in O(n)} f(g k \cdot y) \dif k \dif\Leb[n-1](y)\\
&=
\int_{y \in \xi} \int_{\sigma \in S^{n-1}} f(g \cdot 0 + \sigma |y|) \dif \sigma \dif\Leb[n-1](y)\\
&=
\Omega_{n-1} \int_{r=0}^\infty r^{n-2} \int_{\sigma \in S^{n-1}} f(x + \sigma r) \dif \sigma \dif r\\
&=
\Omega_{n-1} \Omega_n\inv \int_{\R^n} |y|\inv f(x + y) \dif y\\
&=
\frac{\Omega_{n-1}}{\Omega_n} g_{-1} * f(x)\\
&=
\Gamma(n/2) 2^{n-1} \pi^{n/2-1} I^{n-1} f
\end{align*}
By Proposition~\ref{prop:I-pm-gamma}, this gives the inversion formula
\[
f = \Gamma(n/2)\inv 2^{-n+1} \pi^{-n/2+1} I^{-n+1} \Radon^* \Radon f.
\]
In case of odd $n$ this specializes to
\[
f = \Gamma(n/2)\inv 2^{-n+1} \pi^{-n/2+1} (-L)^{(n-1)/2} \Radon^* \Radon f
\]
by (\ref{eq:riesz-laplacian}).

\section{\texorpdfstring{$L^p$}{Lp} discontinuity}
The examples in this section impose some restrictions on the possible values of parameters in (\ref{eq:radon-p-qr-estimate}).
The given functions are not of Schwartz class, but each of them may be approximated simultaneously pointwise and in $L^p$ by a monotonously increasing sequence $(f_k)$ of Schwartz functions.
The Radon transforms of $f_k$ converge monotonously as well, and by the monotonous convergence theorem, the lower bounds below are satisfied up to an arbitrary error by some $f_k$.

\paragraph*{Restriction on $p$}
For $p \geq n/(n-1) > 1$, the function
\[
f(x) = (2 + |x|)^{-n/p} \ln\inv (2 + |x|)
\]
is in $L^p$, as
\begin{align*}
||f||_p
&=
C \int_{r=0}^\infty r^{n-1} (2+r)^{-n} \ln^{-p} (2+r)
\leq
\int_{r=2}^\infty r^{-1} \ln^{-p} (r)\\
&=
\int_{r=2}^\infty \left( \frac{\ln^{-p+1}(r)}{1-p} \right)'
=
\frac{\ln^{1-p}(2)}{p-1},
\end{align*}
but for every $\sigma \in S^{n-1}$, $t \in \R$,
\begin{align*}
\int_{\<x,\sigma\>=t} f(x) \dif\Leb[n-1](x)
&\sim \int_{r=0}^\infty \frac{r^{n-2} (2+\sqrt{r^2+t^2})^{-n/p}}{\ln(2+\sqrt{r^2+t^2})}\\
&\geq \int_{r=0}^\infty \frac{r^{n-2} (r+2+|t|)^{-n/p}}{\ln(r+2+|t|)}\\
&\geq \int_{r=3 + |t|}^\infty \frac{(r-2-|t|)^{n-2} r^{-n/p}}{\ln(r)}\\
&\gtrsim \int_{r=3 + |t|}^\infty \frac{r^{n-2-n/p}}{\ln(r)}\\
&\geq \int_{r=3 + |t|}^\infty \frac{r^{-1}}{\ln(r)}\\
&= \lim_{R \to \infty} \left[ \ln\ln(r) \right]_{r=3+|t|}^R
= \infty,
\end{align*}
i.e.\ the Radon transform is identically $+\infty$.
Thus estimates of the form (\ref{eq:radon-p-qr-estimate}) are only possible if
\begin{equation}
\label{eq:radon-estimate-p-restriction}
p < \frac{n}{n-1}.
\end{equation}

\paragraph*{Restriction on $r$}
We consider the characteristic function $\chi_R$ of the ball of radius $R$ around the origin in $\R^n$.
Then $||\chi_R||_p = C R^{n/p}$ and
\begin{align*}
|| \Radon \chi_R ||_{q,r}
&= \left( \int_{\sigma \in S^{n-1}} \left( \int_{t\in\R} \left(\int_{\<x,\sigma\>=t} \chi_R(x) \dif\Leb[n-1](x) \right)^r \dif t \right)^{q/r} \dif \sigma \right)^{1/q}\\
&= C \left( \int_{t=-R}^R \left( (R^2-t^2)^{(n-1)/2} \right)^r \dif t \right)^{1/r}\\
&= C R^{n-1+1/r} \left( \int_{s=-1}^1 \left( (1-s^2)^{(n-1)/2} \right)^r \dif s \right)^{1/r}
= C R^{n-1+1/r},
\end{align*}
so that the mixed norm estimate can only hold if $R^{n-1+1/r} < C R^{n/p}$ for all $R$, which is the case if and only if
\begin{equation}
\label{eq:radon-estimate-r-restriction}
n-1+1/r = n/p.
\end{equation}
We remark that we might have used the dilates of arbitrary functions in this argument.
The characteristic functions of balls, however, yield the most explicit formulae.

\paragraph*{Restriction on $q$}
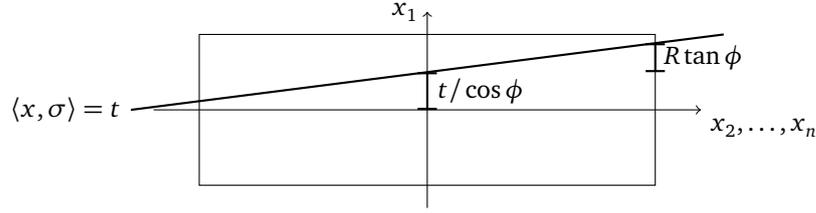
\begin{figure}
\begin{center}
\begin{tikzpicture}[xscale=3]
\draw[->] (0,-1.3) -- (0,1.3) node[left] {$x_1$};
\draw[->] (-1.2,0) -- (1.2,0) node[anchor=north west] {$x_2, \dots, x_n$};

\draw (-1,-1) -- (-1,1) -- (1,1) -- (1,-1) -- cycle;

\coordinate (A1) at (1.3,1);
\coordinate (A2) at (-1.3,0);
\draw[thick] (A1) -- (A2) node[left] {$\<x, \sigma\> = t$};
\coordinate (I1) at (intersection of A1--A2 and 0,0--0,1);
\coordinate (I2) at (intersection of A1--A2 and 1,0--1,1);
\draw[|-|,thick] (0,0) -- node[right] {$t/\cos\phi$} (I1);
\draw[|-|,thick] (I2 |- I1) -- node[right] {$R \tan\phi$} (I2);
\end{tikzpicture}
\end{center}
\caption{A plane passing through the cylinder of height $1$ and radius $R$}
\label{fig:restriction-on-q}
\end{figure}
Finally, consider the characteristic function $\chi_R$ of the cylinder $|x_1| \leq 1$, $x_2^2 + \dots + x_n^2 \leq R^2$.
Then $||\chi_R||_p = C R^{(n-1)/p}$, while the Radon transform is of order $R^{n-1}$ if $\sigma$ is near $e_1$ in the sense that the angle $\phi$ between them satisfies $\tan\phi \leq 1/(2R)$ and $t$ is small in the sense $|t|/\cos\phi \leq 1/2$, as in Figure~\ref{fig:restriction-on-q}.
For $R$ big enough, one may consider the more restrictive conditions $\phi \leq 1/(4 R)$, $|t| \leq 1/4$, and the Radon transform satisfies the lower bound
\begin{align*}
|| \Radon \chi_R ||_{q,r}
&\geq \left( \int_{\{\phi \leq 1/(4 R)\}} \left( \int_{\{|t| \leq 1/4 \}} \left(\int_{\<x,\sigma\>=t} \chi_R(x) \dif\Leb[n-1](x) \right)^r \dif t \right)^{q/r} \dif \sigma \right)^{1/q}\\
&\geq C \mu(\{\phi \leq 1/(4 R)\})^{1/q} \Leb[1](\{|t| \leq 1/4\})^{1/r} R^{n-1}
\geq C R^{(n-1)(1-1/q)},
\end{align*}
where $\mu$ is the measure on $S^{n-1}$.
The estimate (\ref{eq:radon-p-qr-estimate}) implies that
\[
R^{(n-1)(1-1/q)} \leq C R^{(n-1)/p} \text{ as } R\to\infty,
\]
so that $1-1/q \leq 1/p$ or, equivalently,
\begin{equation}
\label{eq:radon-estimate-q-restriction}
q \leq p'.
\end{equation}
By Hölder's inequality and because the measure of $S^{n-1}$ is finite, (\ref{eq:radon-p-qr-estimate}) with any given $q$ implies the same inequality for all smaller $q$'s as well (as long as $q \geq 1$).
Therefore, an estimate on $q$ from below would have disproved (\ref{eq:radon-p-qr-estimate}).

\section{\texorpdfstring{$L^p$}{Lp} estimates by complex interpolation}
\label{sec:radon-int}
By Fubini's theorem,
\begin{equation*}
||\Radon f(\sigma, \cdot)||_1 \leq ||f||_1
\end{equation*}
for every $\sigma$, which immediately shows the estimate
\begin{equation}
\label{eq:radon-L1}
|| \Radon f ||_{\infty; 1} \leq || f ||_1.
\end{equation}
The strategy for showing (\ref{eq:radon-p-qr-estimate}) will be interpolation between the end point (\ref{eq:radon-L1}) and a second estimate.
To obtain that additional estimate, we will in turn embed the Radon transform into a holomorphic family of operators.
For $f \in D:=\Schwartz(\R^{n})$ define
\[
\Radon_z f = h_z * \Radon f,
\]
where the convolution is meant to be taken in the second variable of $\Radon f(\cdot, \cdot)$.

We begin with the end-point estimates on two vertical lines in $\C$.
Let us start with $\{\Re z = 0\}$.
For every $y \in \R$, we have that $||g_{iy}||_{\infty} = 1$, and therefore
\[
|\Radon_{iy} f(\sigma, t)| \leq ||\Radon f(\sigma,\cdot)||_{1} ||h_{iy}||_{\infty} \leq ||f||_{1} |\Gamma((1 + iy)/2)|\inv
\]
for every $\sigma$ and $t$, so that
\[
||\Radon_{iy} f||_{\infty} \leq |\Gamma((1 + iy)/2)|\inv ||f||_{1}.
\]
Proceed with $\{\Re z = -(n+1)/2\}$.
By (\ref{eq:radon-fourier}) we have $\Fourier(\Radon f(\sigma,\cdot))(t) = \Fourier f(\sigma t)$.
The Plancherel theorem on $\R$ and (\ref{eq:fourier-hz}) imply
\begin{align*}
\int_{S^{n-1}} \int_\R |\Radon_z f|^2
&= \int_{S^{n-1}} \int_\R |\Fourier \Radon_z f|^2\\
&= \int_{S^{n-1}} \int_\R |\Fourier (\Radon f(\sigma,\cdot))(t)|^2 |\Fourier h_z(t)|^2 \dif t \dif \sigma\\
&=
\left| \frac{2^{z+1} \pi^{1/2}}{\Gamma(-z/2)} \right|^2
\int_{S^{n-1}} \int_\R |\Fourier f(\sigma t)|^2 |t|^{2(-\Re z-1)} \dif t \dif \sigma\\
&=
C |\Gamma(-z/2)|^{-2} ||\Fourier f||_{2}^2\\
&= C |\Gamma(-z/2)|^{-2} ||f||_{2}^2,
\end{align*}
i.e.
\[
||\Radon_{-(n+1)/2+iy} f||_2 \leq |\Gamma(((n+1)/2 - iy)/2)|^{-1} ||f||_2.
\]
We now verify the technical conditions on $\Radon_{z}$.
Fix an $f\in D$.
Recall that in one dimension
\begin{align*}
h_z (\phi)
&= \Gamma\left(\frac{z+1}{2}\right)\inv \int_{|r|=1}^\infty r^{z} \Phi(r) \dif r\\
&\qquad+ \Gamma\left(\frac{z+1}{2}\right)\inv \int_0^1 r^{z} \left( \Phi(r) - \sum_{k=0}^{m} \frac{r^{2k} \Phi^{(2k)}(0)}{(2k)!} \right) \dif r\\
&\qquad + \Gamma\left(\frac{z+1}{2}\right)\inv \Sum_{k=0}^{m} \frac{\Phi^{(2k)}(0)}{(z + 1 + 2k) (2k)!}\\
&=: h'_{z}(\phi) + h''_{z}(\phi) + h'''_{z}(\phi),
\end{align*}
where $\Phi(r) = \phi(r)+\phi(-r)$.
The second term defines an analytic function on $\{ -1-2m-2 < \Re z \}$.
Since $\Radon$ is a continuous operator between Schwartz spaces, $\Radon f(\sigma,\cdot)$ is bounded in $\Schwartz(\R^{1})$ uniformly in $\sigma$, and $h''_{z} * \Radon f$ is holomorphic as an $L^{\infty}(S^{n-1} \times \R)$-valued function on the open half-plane $\{ -1-2m-2 < \Re z \}$.
Furthermore, $h'''_{z} * \Radon f$ is an entire $L^{\infty}(S^{n-1} \times \R)$-valued function.

The function $h'_{z} * \Radon f$ on the other hand is merely pointwise entire.
It is analytic with values in $L^{\infty}(S^{n-1} \times \R)$ on the open half-plane $\{ \Re z < 0 \}$, but not necessarily differentiable in $L^{\infty}(S^{n-1} \times \R)$ at $z = iy$, since $r^{iy} \ln r$ is unbounded on $(1,\infty)$.

Nevertheless, since $m$ was arbitrary and by Theorem~\ref{thm:quascomp-convenv} the family $\Radon_{z}$ is strongly analytic on $\{\Re z < 0\}$.

Moreover, $h'_{z} * \Radon f$ is locally uniformly bounded on $\{ \Re z \leq 0 \}$.
Therefore for every $v \in L^{1}(S^{n-1} \times \R)$ the dominated convergence theorem implies that the function $\< h'_{z} * \Radon f, v\>$ is continuous on $\{ \Re z \leq 0 \}$.
Hence $\<\Radon_z f, v\>$ is continuous on $\{ \Re z \leq 0\}$.
To see that this function has admissible growth it suffices to use the Stirling formula, which implies that the inverse of the $\Gamma$ function has admissible growth.

Finally, by proposition~\ref{prop:sum-of-duals} we have that $L^{2} + L^{\infty} = (L^{2} \cap L^{1})'$.
Now we can apply the Stein interpolation theorem~\ref{th:stein} with $D=\Schwartz(\R^{n})$, $V := L^{2} \cap L^{1}(S^{n-1} \times \R)$ and strip $S$ replaced by $\{ -(n+1)/2 \leq \Re z \leq 0\}$.
By Proposition~\ref{prop:lp-interpolation-by-schwartz-functions} we obtain for every $0 < \theta < 1$ the estimate
\begin{equation}
||\Radon_z f||_{L^{p'}} \leq C ||f||_{L^p},
\label{ineq:radon-not-mixed}
\end{equation}
where $z=-(n+1) (1-\theta) / 2$ and $p\inv = (1-\theta) 2\inv + \theta 1\inv = 1/2 + \theta/2 = 1 + z/(n+1)$.
Observe that $h_{-1}$ is a constant multiple of the Dirac $\delta(0)$, so that $\Radon_{-1} = C \Radon$.
Therefore, if $\theta$ in chosen in such a way that $z=-1$, we obtain an estimate for $\Radon$, namely
\[
||\Radon f||_{n+1; n+1} = ||\Radon f||_{n+1} \leq C ||f||_{(n+1)/n}.
\]
Combining this result with (\ref{eq:radon-L1}) and applying the Interpolation Theorem~\ref{th:stein} and the characterization of the intermediate mixed norm spaces (Corollary~\ref{cor:complex-interpolation-LP}) again, this time with $P_\theta\inv = (1-\theta) 1\inv + \theta n/(n+1) =: p\inv$, $Q_{\theta,1}\inv = \theta (n+1)\inv = 1/p'$, $Q_{\theta,2}\inv = (1-\theta) 1\inv + \theta (n+1)\inv = r\inv = n p\inv - n +1$, we obtain
\[
||\Radon f||_{p'; r} \leq C ||f||_{p}
\]
for every $1 \leq p \leq (n+1)/n$.

In case $n \geq 2$, we can extend the result to $(n+1)/n < p < n/(n-1)$.
Set $z = -(n+1)/p'$.
Then $-(n+1)/2 \leq -1-1/n < z < -1 < 0$, and thus the inequality~(\ref{ineq:radon-not-mixed}) holds.
Furthermore, in dimension $1$,
\[
\Fourier(h_{-z-2} * h_z) = C h_{z+1} h_{-1-z} = C
\]
by (\ref{eq:fourier-hz}) and because $-1-z > -1$, $z+1 > -1$, so that the distributions $h_{z+1}$, $h_{-1-z}$ are in fact functions.
Therefore
\[
\Radon f = C h_{-z-2} * \Radon_z f.
\]
Since $h_{-z-2} \in L^{1/(z+2),\infty}$, the weak-type Young inequality (Proposition~\ref{prop:weak-type-young}) implies
\[
||\Radon f||_{p'; r} \leq ||\Radon_z f||_{p'; p'} ||h_{-z-2}||_{1/(z+2),\infty} \lesssim ||\Radon_z f||_{p'; p'} \lesssim ||f||_p,
\]
where $r$ is given by the condition
\[
\frac1r + 1 = \frac1{p'} + \frac1{1/(z+2)} \implies \frac1r = 1 -\frac{n}{p'},
\]
in accordance with the restriction (\ref{eq:radon-estimate-r-restriction}).
Observe that at the end point $p=\frac{n}{n-1}$ we have $r=\infty$ and the weak-type Young inequality does not hold, just as one would expect from (\ref{eq:radon-estimate-p-restriction}).

Since the measure on $S^{n-1}$ is finite, the Hölder inequality allows to extend the range of exponents in (\ref{eq:radon-p-qr-estimate}) replacing the $q$ by any number between, and including, $1$ and $q$.
This way, we see that (\ref{eq:radon-estimate-q-restriction}) is the optimal.
Let us summarize the results of this section.
\begin{theorem}
\label{th:radon-lp-lqr-estimate}
Let $n \geq 2$.
Then an estimate of the type
\[
|| \Radon f ||_{q; r} \leq C_{q,r,p} || f ||_p
\]
holds if and only if $1 \leq p < n/(n-1)$, $q \leq p'$ and $n-1+1/r = n/p$.
\end{theorem}

\section{A Lorentz space estimate at the critical point}
The use of the weak-type Young inequality in the previous section suggests the possibility that the Radon transform is a continuous operator between some Lorentz spaces for the critical exponent $p=n/(n-1)$.
In this section, we show that this is indeed the case.
In particular,
\[
||\Radon f||_{n; \infty} \leq C || f ||_{p,1}.
\]
\begin{lemma}
\label{lemma:interpolation-Lr-Hoelder}
Let $0 < \alpha < 1 \leq r < \infty$.
Then for every $f \in C^{0,\alpha}(\R) \cap L^r(\R)$,
\[
||f||_\infty \lesssim ||f||_r^{r \alpha/(1+r \alpha)} ||f||_{C^{0,\alpha}}^{1/(1+r \alpha)}.
\]
\end{lemma}
\begin{proof}
Let $x$ be such that $|f(x)| > 0$.
By Hölder continuity of $f$,
\[
|y-x|^\alpha ||f||_{C^{0,\alpha}} < |f(x)|/2
\implies
|f(x) - f(y)| < |f(x)|/2
\implies
|f(y)| > |f(x)|/2,
\]
i.e.\ $|f(y)| > |f(x)|/2$ on a line segment of length $C (|f(x)| / ||f||_{C^{0,\alpha}})^{1/\alpha}$ centered at $x$.
By the Chebyshev inequality we have
\[
||f||_r^r \geq (|f(x)|/2)^r |\{ |f| > |f(x)|/2 \}| > C |f(x)|^{r+1/\alpha} ||f||_{C^{0,\alpha}}^{-1/\alpha}.
\]
Since $x$ is arbitrary,
\[
||f||_\infty \lesssim ||f||_r^{r\alpha/(r\alpha+1)} ||f||_{C^{0,\alpha}}^{1/(r\alpha+1)}.
\qedhere
\]
\end{proof}

\begin{definition}
The \emph{finite difference} operator is defined by
\[
\Delta_t u(x) = u(x+t) - u(x)
\]
\end{definition}

\begin{lemma}
\label{lemma:finite-difference}
Let $\alpha > 0$, $m>\alpha$ be an integer, $g\in C(\R)$ such that $\Fourier g \in L^1(\R)$.
Then
\[
\sup |\Delta_t^m g| \leq C(m,\alpha) \left( \int_{s \in \R} |\Fourier g (s)|^2 |s|^{1+2\alpha} \dif s\right)^{\frac12} |t|^\alpha
\]
\end{lemma}
\begin{proof}
By the Fourier inversion formula and as $\Fourier(\Delta_t g)(\xi) = \Fourier g (\xi) (e^{i \xi t} - 1)$,
\begin{align*}
|\Delta_t^m g(x)|
&\leq \int_{\xi \in \R} |\Fourier g(\xi)| |e^{i \xi t} - 1|^m \dif \xi\\
&\leq 
\left( \int_{\xi \in \R} |\Fourier g(\xi)|^2 |\xi|^{1+2\alpha} \dif\xi \right)^{\frac12}
\left( \int_{\xi \in \R} |e^{i \xi t} - 1|^{2m} |\xi|^{-1-2\alpha} \dif \xi \right)^{\frac12}\\
&\leq
\left( \int_{\xi \in \R} |\Fourier g (\xi)|^2 |\xi|^{1+2\alpha} \dif\xi \right)^{\frac12} |t|^\alpha
\left( \int_{s \in \R} |e^{i s} - 1|^{2 m} |s|^{-1-2\alpha} \dif s \right)^{\frac12}
\end{align*}
The integrand in the last pair of parenthesis is independent of $t$, bounded by $|s|^{2 m} |s|^{-1-2\alpha}$ for $|s|<1$ and by $|s|^{-1-2\alpha}$ elsewhere and thus finite.
\end{proof}

\begin{lemma}
\label{lemma:fin-dif-L-pp-inf-inf-norm}
Assume $n \geq 3$, let $n/(n-1) < p \leq 2$ and $m$ be an integer with $m>(n-p')/2$.
Then, for every $f \in \Schwartz(\R^n)$,
\[
\left( \int_{S^{n-1}} (\sup_{x,t\neq 0 \in \R} |(\Delta_t^m \Radon f)(\sigma,x)|^{p'} |t|^{p'-n}) \dif\sigma \right)^{1/p'} \leq C ||f||_p.
\]
\end{lemma}
\begin{proof}
Let $\eta := n - p' > 0$,
\[
(T_z f)(\sigma, x, t) := |t|^{\eta z /(n-\eta)} (\Delta_t^m \Radon_z f)(\sigma,x).
\]
At the end point $\Re z = (\eta - n)/2 \leq -1$,
the Lemma~\ref{lemma:finite-difference} with $\alpha = \eta/2$ applies because for every $\sigma$, $\Radon_z f$ is a convolution of a tempered distribution and a Schwartz function and thus continuous, while by (\ref{eq:FT-h-z-convolution-f}),
\[
(\Fourier \Radon_z f)(\sigma, \xi) = 2^{z+1} \pi^{1/2} \Gamma(-z/2)\inv |\xi|^{-z-1} (\Fourier f)(\sigma \xi),
\]
which is in $L^1(\R)$ as a function of $\xi$.
The lemma implies
\begin{align*}
\int_{S^{n-1}} & \sup_{x,t\neq 0 \in \R} \left| T_z f(\sigma, x, t) \right|^2 \dif\sigma\\
&=
\int_{S^{n-1}} \sup_{x,t\neq 0 \in \R} \left| \frac{(\Delta_t^m \Radon_z f)(\sigma,x)}{t^{\eta/2}} \right|^2 \dif\sigma\\
&\leq C
\int_{S^{n-1}}
\int_{\xi \in \R} \left| (\Fourier \Radon_z f)(\sigma,\xi) \right|^2 |\xi|^{1+\eta}
\dif\xi \dif\sigma\\
&= C
\int_{S^{n-1}}
\int_{\xi \in \R} \left| 2^{z+1} \pi^{1/2} \Gamma(-z/2)\inv (\Fourier f)(\sigma \xi) \right|^2 |\xi|^{n-1}
\dif\xi \dif\sigma\\
&= C
|2^{2 z} \Gamma(-z/2)^{-2}|
\int_{\xi \in \R^n} \left| (\Fourier f)(\xi) \right|^2 \dif\xi\\
&= C
|\Gamma(-z/2)|^{-2} ||f||_2^2,
\end{align*}
i.e.\ $T_z : L^2(\R^n) \to L^{2}(S^{n-1},L^\infty(\R \times (\R \setminus \{0\})))$ is bounded.
By the Fubini theorem, $\Re z = 0$ implies
\[
\sup_{\sigma \in S^{n-1}} \sup_{x,t\neq 0 \in \R} \left| T_z f(\sigma, x, t) \right|
\leq
|\Gamma((z+1)/2)|\inv ||f||_1,
\]
i.e.\ $T_z : L^1(\R^n) \to L^{\infty}(S^{n-1},L^\infty(\R \times (\R \setminus \{0\})))$ is bounded.
By Theorem~\ref{th:stein} with $-1 = (1-\theta) \frac{\eta - n}{2}$, $T_{-1}$ extends to a bounded operator from $L^p(\R^n)$ to
\[
L^{p'}(S^{n-1},L^\infty(\R \times (\R \setminus \{0\}))).
\]
This proves the claim since
\[
(T_{-1} f)(\sigma, x, t) = C |t|^{-(n-p')/p'} (\Delta_t^m \Radon f)(\sigma,x).
\qedhere
\]
\end{proof}
For every $0 < \beta < 1$, Lemma~\ref{lemma:interpolation-Lr-Hoelder} and the Hölder inequality imply
\begin{multline*}
\int_{S^{n-1}} \left( \sup_{t\in\R} |\Radon f(\sigma, t)|^n \right) \dif\sigma\\
\leq
C \int_{S^{n-1}} \left( \int_{t\in\R} |\Radon f(\sigma, t)|^r \dif t \right)^{\frac{n \alpha}{1+r \alpha}} \left( \sup_{s\neq t\in\R} \frac{|\Radon f(\sigma, t) - \Radon f(\sigma, s)|}{|t-s|^\alpha} \right)^{\frac{n}{1+r \alpha}} \dif\sigma\\
\leq
C \left( \int_{S^{n-1}} \left( \int_{t\in\R} |\Radon f(\sigma, t)|^r \dif t \right)^{n \alpha/(1+r \alpha)/\beta} \dif\sigma \right)^\beta\\
\left( \int_{S^{n-1}} \left( \sup_{s\neq t\in\R} \frac{|\Radon f(\sigma, t) - \Radon f(\sigma, s)|}{|t-s|^\alpha} \right)^{n/(1+r \alpha)/(1-\beta)} \dif\sigma \right)^{1-\beta}\\
\leq
C || \Radon f||_{n\alpha r/(1+r \alpha)/\beta; r}^{n \alpha r/(1+r\alpha)}
|| f ||_{p_1}^{(1-\beta) p_1'}\\
\leq
C || f||_{p_0}^{n \alpha r/(1+r\alpha)}
|| f ||_{p_1}^{(1-\beta) p_1'},
\end{multline*}
the second to last inequality being valid by Lemma~\ref{lemma:fin-dif-L-pp-inf-inf-norm} if
\[
\frac{n}{(1+\alpha r)(1 - \beta)} = p_1',
\quad
\frac{n \alpha}{(1+\alpha r)(1 - \beta)} = n - p_1',
\quad
\frac{n}{n-1} < p_1 \leq 2,
\]
and the last inequality by Theorem~\ref{th:radon-lp-lqr-estimate} if
\[
\frac{n \alpha r}{(1+r\alpha)\beta} \leq p_0',
\quad
\frac1r = \frac{n}{p_0} - n + 1,
\quad
1 \leq p_0 < \frac{n}{n-1}.
\]
Elementary manipulations show that these restrictions reduce to
\[
1 \leq p_0 < \frac{n}{n-1} < p_1 \leq 2,
\]
and there are explicit expressions for $\alpha$, $\beta$, $r$.
With those, the inequality assumes the form
\begin{equation}
\label{eq:radon-critical-estimate}
|| \Radon f ||_{n; \infty}
\leq C
|| f ||_{p_0}^{1-\gamma} || f ||_{p_1}^\gamma,
\end{equation}
where $\gamma$ satisfies
\[
\frac{1 - \gamma}{p_0} + \frac{\gamma}{p_1} = \frac{n-1}{n}.
\]

The space $L^{p,1}$ is the smallest rearrangement-invariant Banach function space in which $||\chi_E|| = |E|^{1/q}$ for every measurable set $E$, cf.~\cite[Section 2.5]{MR928802}.
In that space, we can estimate the norm of simple functions \emph{from below} by the corresponding linear combination of norms of the characteristic functions,
\[
||\sum_k a_k \chi_k||_{p,1} \geq \sum_k |a_k| || \chi_k ||_{p,1}.
\]
Since simple functions are dense in $L^{p,1}$, this reduces the question of boundedness of an operator $T$ from $L^{p,1}$ to some Banach space $X$ to the question whether
\[
|| T \chi_E ||_X \leq C |E|^{1/p}
\]
for every measurable set $E$.
In case of the Radon transform, we have by (\ref{eq:radon-critical-estimate}), formally,
\[
|| \Radon \chi_E ||_{n; \infty}
\leq C
|| \chi_E ||_{p_0}^{1-\gamma} || \chi_E ||_{p_1}^\gamma
= C |E|^{1/p}
\]
with $p=n/(n-1)$.
But, for every measurable set $E$, $\chi_E$ may be approximated by Schwartz functions simultaneously in $L^{p_0}$, $L^{p_1}$ and $L^{p,1}$, so that we still obtain a dense subset $D \subset L^{p,1}$ such that $\Radon : D \to L^n(L^\infty)$ is a bounded operator.

\section{Radon transform as a convolution operator}
\label{subsec:radon-heisenberg}
In this section we consider a convolution operator on the Heisenberg group
\[
H_n = \{ (x_1, \dots, x_n, y_1, \dots, y_n, t) \}
\]
and provide an alternative view point on the boundedness of that operator \cite[{Lemma 2.6}]{MR1021141}.

Let $f$ be a measurable function on $\R^{2n+1} = H_n$ and $\mu$ be the $2n$-dimensional Lebesgue measure on the hyperplane $\{ t = 0\}$.
The convolution $f*\mu$ may be written in terms of the Radon transform as
\begin{align*}
f * \mu (x, y, t)
&= \int_{\R^n} \int_{\R^n} f(x-r,y-s,t-\frac12 (xs - yr)) \dif r \dif s\\
&= ||v||\inv \int_{\pi(x,y,t)} f.
\end{align*}
Here, $\pi(x,y,t)$ is the plane through $(x,y,t)$ which is orthogonal to
\begin{multline*}
v=(-1,0,\dots,0,y_1/2)
\wedge
\dots
\wedge
(0,\dots,0,-1,0,\dots,0,y_n/2)
\wedge\\
(0,\dots,0,-1,0,\dots,0,-x_1/2)
\wedge
\dots
\wedge
(0,\dots,0,-1,-x_n/2)\\
=
(y_1/2,\dots,y_n/2,-x_1/2,\dots,-x_n/2,1),
\end{multline*}
and $||v||=\left( 1 + x^2/4 + y^2/4 \right)^{1/2}$.
We will use this description of the convolution operator $f \mapsto f*\mu$ to obtain an $L^p$ estimate.

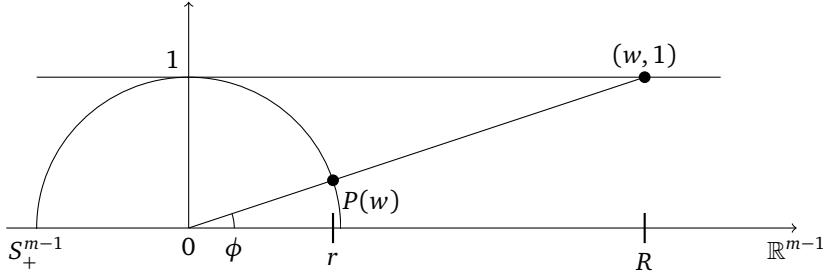
\begin{figure}
\centering
\begin{tikzpicture}[scale=2]
\coordinate (C) at (0,0); 
\coordinate (H) at (4,0); 
\coordinate (R) at (0,1); 
\coordinate (L) at (3,1); 

\draw[->] (-1.2,0) -- (H) node[below] {$\R^{m-1}$};
\draw[->] (0,0) -- (0,1.5);
\node[below] at (0,0) {$0$};
\node[anchor=south east] at (0,1) {$1$};

\draw (L) +(.5,0) -- (-1,1);
\begin{scope}
\clip (-1,0) -- (1,0) -- (1,1.1) -- (-1,1.1) -- cycle;
\node (unit circle) [draw, circle through=(R)] at (C) {};
\end{scope}
\node[below] at (-1,0) {$S^{m-1}_+$};

\draw (C) -- (L) node[above] {$(w,1)$};
\draw[fill] (L) circle (1pt);

\coordinate (I) at (intersection of unit circle and C--L);
\node[anchor=north west] (I1) at (I) {$P(w)$};
\draw[fill] (I) circle (1pt);

\draw[thick] ($(C)!(I)!(H)$) +(0,0.1) -- +(0,-0.1) node[anchor=north] {$r$};
\draw[thick] ($(C)!(L)!(H)$) +(0,0.1) -- +(0,-0.1) node[anchor=north] {$R$};

\begin{scope}
\clip (C) -- (L) -- (H) -- cycle;
\draw (C) circle (0.3);
\end{scope}
\node[below] at (0.3,0) {$\phi$};
\end{tikzpicture}
\caption{Centrographic projection.}
\label{fig:centrographic}
\end{figure}
\index{centrographic projection}
First we make some remarks regarding the centrographic projection map $P$ (cf.\ Figure~\ref{fig:centrographic}).
It is a bijective map from the plane $\R^{m-1}$ onto the upper unit half-sphere $S^{m-1}_+$.
We now compute the density of the image measure of the Lebesgue measure on $\R^{m-1}$ with respect to the standard measure on the sphere.
At a point given by angle $\phi$, there are three effects to take care of:
\begin{itemize}
\item the radius is reduced by a factor of $\sin\phi$ since $r=\cos\phi$ and $R=\frac{\cos \phi}{\sin \phi}$, which gives a factor of $\sin^{2-m}\phi$ in the density function (there are $m-2$ additional dimensions hidden in the figure).
\item The inclination of the ray (the inclined line in the figure) with respect to the plane is $\phi$, while the sphere is orthogonal to the ray, which gives an additional factor of $\sin\inv\phi$.
\item The length of the ray is $1/\sin\phi$, while the distance from the origin to the sphere is $1$, which gives one more factor of $\sin\inv\phi$.
\end{itemize}
So the density function, i.e.\ $(\det\nabla P)\inv$, is $\sin^{-m}\phi$.

Returning to the convolution operator, the map
\[
F:(x,y,t) \mapsto (v/||v||, (x,y,t) \cdot v/||v||)
\]
is bijective between $H_n$ and $S^{2n}_+ \times \R$.
Its first component is the composition of a restriction to a hyperplane of codimension $1$, a dilation by $1/2$ and the centrographic projection $P$ with $m=2n+1$.
The second component contains all the dependence on $t$.
Therefore,
\[
\det \nabla F
=
\left(\pdif{}{t} ((x,y,t) \cdot v/||v||) \right) (2^{-2n}) \det \nabla P
=
(||v||\inv) (2^{-2n}) (\sin^{2n+1}\phi).
\]
Since in our case $\sin\phi = ||v||\inv$, we have
\[
\det \nabla F
=
2^{-2n} ||v||^{-2n-2}.
\]
This readily implies
\[
\int_{H_n} (f * \mu)^{2n + 2}
=
2^{2n}
\int_{S^{2n}_+} \int_\R \left( \Radon f(\sigma, s) \right)^{2n+2} \dif \sigma \dif s.
\]
By Theorem~\ref{th:radon-lp-lqr-estimate}, the latter integral may be estimated by
\[
C ||f||_{(m+1)/m}.
\]


%% file: chapter_rearrangement.tex
\chapter{Rearrangement inequalities}
\label{chap:rearrangement}
Certain estimates are particularly easy to prove for symmetric objects like radially symmetric functions.
This happens if the problem at hand reduces to the one-dimensional case, where tools like the Hardy inequalities are available.
The operation known as the Steiner symmetrization then serves to amplify the results thus obtained to more general functions.

\index{Steiner symmetrization}
For a measurable set $T \subset \R^n$ and $u \in S^{n-1}$, the \emph{Steiner symmetrization} of $T$ with respect to the direction $u$ is the set $T^{*u}$ with the property that for every $v \perp u$, $T^{*u} \cap (\lin (u) + v)$ is a segment centered at $v$ such that its $1$-dimensional Lebesgue measure is the same as that of $T \cap (\lin (u) + v)$ (or $0$ if the latter set is not measurable).

By Fubini's Theorem the Steiner symmetrization preserves measure.
The central idea is that we can, up to an arbitrarily small error, transform any set into a ball with the same measure by repeated applications of the Steiner symmetrization.
\begin{proposition}[Symmetrization principle]
\index{Steiner symmetrization!symmetrization principle}
\label{prop:steiner-approx}
Let $\{T_i\}_{i \in I}$ be a countable collection of bounded subsets of $\R^n$ and $\rho_i$ be the radius of a ball with Lebesgue measure $|T_i|$.
Then there exists a sequence of directions $(u_j)_{j}$ such that $T_i^{*u_1 \dots *u_n} \subset B(0,\rho_i+o_i(1))$ and $\Leb(S_i^{*u_1 \dots *u_n} \Delta B(0,\rho_i)) \to 0$ for all $i \in I$.
\end{proposition}
We present here a quantitative version of the standard proof.
\begin{proof}
The first conclusion clearly implies the second.
Moreover, since the Steiner symmetrization of a set bounded by a constant is bounded by the same constant, the result follows from the special case $|I|=1$ by a diagonal argument.
From now on we omit the index $i$ and define
\[
R := \inf \{ s : T^{*u_1 \dots *u_n} \subset B(0,s) \text{ for some } u_1, \dots, u_n \}.
\]
It suffices to show that $R=\rho$.
Assume for a contradiction that $R>\rho$.

Let $\delta > 0$.
By definition of $R$ we may assume that $T$ is contained in $B(0,R+\delta)$.
Let $e_1, \dots, e_n$ be the canonical frame.
Replacing $T$ by $T^{*e_1 \dots *e_n}$ we may also assume that if $y \in T$ and $x \in \R^{n}$ satisfy $|x_j| \leq |y_j|$ for all $j=1,\dots,n$, then $x \in T$.
This new set is still contained in $B(0, R+\delta)$.

Let $\epsilon$ be such that
\[
(R+\rho)^2/4 + (R+\rho) n \epsilon + n \epsilon^2 \leq R^2.
\]
Assume that no point of the sphere $S(0,R)$ is $\epsilon$-separated from $T$.
By the choice of $\epsilon$, for every $x \in B(0,(R+\rho)/2)$ there exists a $y \in S(0,R)$ with $|x_j| + \epsilon \leq |y_j|$ for all $j$, since
\[
\sum_j (|x_j|+\epsilon)^2 \leq (R+\rho)^2/4 + (R+\rho) n \epsilon + n \epsilon^2 \leq R^2.
\]
On the other hand, the intersection of the cube with side length $2\epsilon$ centered at $y$ with the set $T$ is non-empty by the assumption.
Therefore $x \in T$, so that $T \supset B(0,(R+\rho)/2)$, which contradicts the fact that the Steiner symmetrization preserves measure.

Therefore there exists a point $y_0 \in S(0,R)$ such that $B(y_0, \epsilon) \cap T = \emptyset$.
By compactness of $S(0,R)$, there exist $y_1, \dots, y_m$ different from $y_0$ such that the balls $B(y_j,\epsilon/2)$, $j=0,\dots,m$ cover $S(0,R)$.
Observe that $m$ may be chosen independently of $\delta$.

Define $u_j:=\frac{y_0-y_j}{|y_0-y_j|}$ for $j=1,\dots,m$.
We claim that, provided that $\delta$ is suitably small, $T^{*u_1 \dots *u_m}$ is separated from $S(0,R)$ by some positive distance.
But then $T^{*u_1 \dots *u_m} \subset B(0,R')$ for some $R' < R$, contradicting the minimality of $R$.

If we could choose $\delta=0$, the claim would follow by the argument in \cite[2.10.31]{MR0257325}.
Here we need a more quantitative version.
We look at the decay of the radius of balls separating $S(0,R)$ from $T$ as Steiner symmetrization is applied in terms of their initial radius $r$, $R$ and the error $\delta$.
Assume $p \in S(0,R)$, $B(p,r) \cap T = \emptyset$, $u \in S^{n-1}$ and write $\phi$ for the angle between the lines spanned by $p$ and $u$.

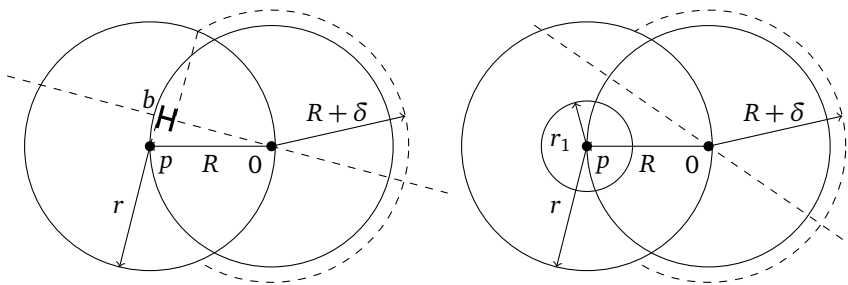
\begin{figure}
\begin{center}
\begin{tikzpicture}[scale=0.8]
\coordinate (C) at (0,0);
\coordinate (P) at (-2,0);
\coordinate (H) at (-2.5,-2);
\coordinate (Rd) at (2.2,0.5);
\draw[fill] (P) circle (0.075) node[anchor=north west] {$p$};
\draw[fill] (C) circle (0.075) node[anchor=north east] {$0$};

\node (R) [draw, circle through=(P)] at (C) {};
\node (D) [circle through=(Rd)] at (C) {};
\node (S) [draw, circle through=(H)] at (P) {};
\coordinate (I1) at (intersection 1 of S and D);
\coordinate (I2) at (intersection 2 of S and D);
\begin{scope}
\clip (I2) -- (-3,3) -- (3,3) -- (3,-3) -- (-3,-3) -- (I1) -- cycle;
\node (D') [draw, dashed, circle through=(Rd)] at (C) {};
\end{scope}

\coordinate (A1) at (165:4.5);
\coordinate (A2) at (-15:3);
\draw[dashed] (A1) -- (A2);

\draw[dashed] (I2) -- ($(A1)!(I2)!(A2)$);
\draw[dashed] (P) -- ($(A1)!(P)!(A2)$);
\draw[|-|,very thick] ($(A1)!(P)!(A2)$) -- node[anchor=south east] {$b$} ($(A1)!(I2)!(A2)$);

\draw[->] (C) -- node[below] {$R$} (P);
\draw[->] (P) -- node[left] {$r$} (H);
\draw[->] (C) -- node[above] {$R+\delta$} (Rd);
\end{tikzpicture}
\begin{tikzpicture}[scale=0.8]
\coordinate (C) at (0,0);
\coordinate (P) at (-2,0);
\coordinate (H) at (-2.5,-2);
\coordinate (Rd) at (2.2,0.5);
\draw[fill] (P) circle (0.075) node[anchor=north west] {$p$};
\draw[fill] (C) circle (0.075) node[anchor=north east] {$0$};

\node (R) [draw, circle through=(P)] at (C) {};
\node (D) [circle through=(Rd)] at (C) {};
\node (S) [draw, circle through=(H)] at (P) {};
\coordinate (I1) at (intersection 1 of S and D);
\coordinate (I2) at (intersection 2 of S and D);
\begin{scope}
\clip (I2) -- (-3,3) -- (3,3) -- (3,-3) -- (-3,-3) -- (I1) -- cycle;
\node (D') [draw, dashed, circle through=(Rd)] at (C) {};
\end{scope}

\coordinate (A1) at (145:3.5);
\coordinate (A2) at (-35:3);
\draw[dashed] (A1) -- (A2);

\draw (P) circle (0.75);
\draw[->] (P) -- node[anchor=north east] {$r_1$} +(105:0.75);

\draw[->] (C) -- node[below] {$R$} (P);
\draw[->] (P) -- node[left] {$r$} (H);
\draw[->] (C) -- node[above] {$R+\delta$} (Rd);
\end{tikzpicture}
\end{center}
\caption{Separation of $p$ from $T^{*u}$ for $\phi \approx \frac\pi2$ (left) and $\phi$ small (right)}
\label{fig:st-sym-rad-decay}
\end{figure}
Then on the one hand, if $\phi$ is almost $\pi/2$ (say, $\cos \phi \leq \frac{r}{32 R}$, $\sin\phi \geq \frac12$), then the projection of $p$ onto $u^\perp$ is separated from the projection of $T$ at least by
\begin{align*}
b &=\frac{r^2 - 2 R \delta - \delta^2}{2 R}\sin\phi - \frac{1}{2 R} \cos\phi \sqrt{(r^2 - \delta)(4R^2 + 4R\delta + \delta^2 - r^2)}\\
&\geq \frac{r^2}{4R} \sin\phi - 2r \cos\phi \geq \frac{r^2}{8 R},
\end{align*}
assuming $\delta$ is small enough, so that $p$ is separated from $T^{*u}$ by at least the same distance.

On the other hand, as long as $\phi$ is bounded away from $\pi/2$, $B(p,r)$ contains a ball of radius bounded from below by some $r_1$ (independent of $\phi$) which does not intersect $u^\perp$.
If the thickness of the rind $B(0,R+\delta) \setminus B(0,R)$ does not exceed, say, $r_1/8$, then $B(p,r_1/4) \cap T^{*u} = \emptyset$.

These bounds are uniform in $p$ and $u$, so that for a given $r>0$ there is a $\delta>0$ and $r_2>0$ such that if $B(p,r) \cap T = \emptyset$, then $B(p,r_2) \cap T^{*u} = \emptyset$.

This may be reiterated a finite number of times, since smaller $\delta$ do not worsen the estimate for $r_2$.
Furthermore, if $p_m$ is the mirror image of $p$ with respect to the hyperplane $u^\perp$, then also $B(p_m,r_2) \cap T^{*u} = \emptyset$.
Because each point of $S(0,R)$ arises as a mirror image of some point in $S(0,R) \cap B(y_0,\epsilon/2)$ and the latter set is separated from $T$, we obtain that $S(0,R)$ is separated from $T^{*u_1 \dots *u_m}$.
\end{proof}

The first application of the symmetrization principle is the following proof of the \emph{Brunn-Minkowski inequality} found in \cite[p.~361]{MR1898210}.
\begin{corollary}
\index{Brunn-Minkowski inequality}
For any non-empty Lebesgue measurable sets $T_1, T_2 \subset \R^n$, we have that
\begin{equation}
\Leb (T_1 + T_2) ^{1/n} \geq \Leb (T_1)^{1/n} + \Leb (T_2)^{1/n}.
\label{ineq:brunn-minkowski}
\end{equation}
\end{corollary}
Note that the assertion is clear if both $T_1$ and $T_2$ are balls.
\begin{proof}
If one of the sets $T_j$ has infinite measure, there is nothing to prove.
If both $T_1$ and $T_2$ have finite measure, they may be approximated in $L^1$ from below by bounded sets, so that we may assume that they are bounded.
Furthermore, passing to subsets if needed, the sets may be assumed to be Borel measurable, so that their intersections with affine subspaces are still measurable.

To be able to apply Steiner symmetrization, observe that $(T_1 + T_2)^{*u} \supseteq T_1^{*u} + T_2^{*u}$ for every $u$.
This is essentially a one-dimensional fact.
If $T_j \subset \R$ are non-empty and bounded, then there exist $x_j \in T_j$ such that
\[
\Leb(T_1 \cap (-\infty, x_1)) > (1-\epsilon) \Leb(T_1)
\text{ and }
\Leb(T_2 \cap (x_2, +\infty)) > (1-\epsilon) \Leb(T_2).
\]
It follows immediately that
\begin{align*}
\Leb(T_1 + T_2)
&\geq
\Leb(T_1 \cap (-\infty, x_1) + x_2) + \Leb(x_1 + T_2 \cap (x_2, +\infty))\\
&\geq
(1-\epsilon) (\Leb(T_1) + \Leb(T_2)).
\end{align*}
Since such $x_j$ exist for all $\epsilon$, the observation follows in one dimension.
In several dimensions, $T_j$ may be decomposed into one-dimensional slices parallel to $u$ to obtain the same result.

By Proposition~\ref{prop:steiner-approx} there exists a sequence of directions $u_j$ such that $T_j^{*u_1\dots*u_m} \subset B(0,\rho_j + o(1))$.
Since the Steiner symmetrization preserves measure and by the above,
\begin{align*}
\Leb (T_1 + T_2)
&= \Leb (T_1 + T_2)^{*u_1\dots*u_m}\\
&\geq \Leb (T_1^{*u_1} + T_2^{*u_1})^{*u_2\dots*u_m}
\geq \dots\\
&\geq \Leb (T_1^{*u_1\dots*u_m} + T_2^{*u_1\dots*u_m}).
\end{align*}
Let $\chi_{i,m}$ be the characteristic function of $T_i^{*u_1\dots*u_m}$ and $\chi_i$ be the charactertic function of $B(0,\rho_i)$.
Then the functions $\chi_{i,m}$ are uniformly bounded and converge to $\chi_i$ in $L^1$, so that $\chi_{1,m} * \chi_{2,m} \to \chi_1 * \chi_2$ in $L^1$ by the Young inequality.
But
\begin{align*}
\Leb (T_1^{*u_1\dots*u_m} + T_2^{*u_1\dots*u_m})
&\geq \Leb\{\chi_{1,m} * \chi_{2,m} > 0\}\\
&\geq \Leb\{\chi_1 * \chi_2 > 0\} + o_{m}(1)\\
&= \Leb B(0,\rho_1 + \rho_2) + o_{m}(1).
\end{align*}
Passing to the limit as $m \to \infty$ we obtain
\begin{align*}
\Leb (T_1 + T_2)^{1/n}
&\geq \Leb B(0,\rho_1 + \rho_2)^{1/n}\\
&= \Leb B(0,\rho_1)^{1/n} + \Leb B(0,\rho_2)^{1/n}\\
&= \Leb(T_1)^{1/n} + \Leb(T_2)^{1/n}.
\qedhere
\end{align*}
\end{proof}

The Brunn-Minkowski inequality readily implies a form of concavity for the area of the sections of a convex set.
\begin{proposition}
\label{prop:convex-set-concave-cut-area}
Let $C \subset \R^{n+1}$ be a convex set and $\phi$ be a fixed linear functional on $\R^{n+1}$.
Denote the area of a $\phi$-slice of $C$ by $S(t) = \Leb[n](C\cap \{\phi = t\})$.
Then $S(t)^{1/n}$ is a concave function of $t$ in the interval where it is greater than $0$.
\end{proposition}
\begin{proof}
Assume $S(t_0) > 0$, $S(t_1) > 0$.
By convexity of $C$, for every $0 < \theta < 1$,
\[
C \cap \{ \phi = t_\theta \}
\supseteq
(1 - \theta) (C \cap \{ \phi = t_0 \})
+
\theta (C \cap \{ \phi = t_1 \}),
\]
where $t_\theta = (1-\theta) t_0 + \theta t_1$.
Applying (\ref{ineq:brunn-minkowski}), one sees that
\begin{align*}
S(t_\theta)^{1/n}
&\geq
\Leb[n]((1-\theta)(C\cap \{\phi = t_0\}))^{1/n}
+
\Leb[n](\theta(C\cap \{\phi = t_1\}))^{1/n}
\\ &=
(1-\theta) S(t_0)^{1/n}
+
\theta S(t_1)^{1/n}.
\qedhere
\end{align*}
\end{proof}

\begin{corollary}
\label{cor:balanced-decreasing}
If $C \subset \R^{n+1}$ is convex and balanced in the sense that $C=-C$, then $S(t)$ is monotonously decreasing for $t \geq 0$.
\end{corollary}
\begin{proof}
$S(t) = S(-t)$ and $S^{1/n}$ is concave by Proposition~\ref{prop:convex-set-concave-cut-area}.
\end{proof}

Let us see how this result may be used to obtain a rearrangement inequality for characteristic functions of sets.
\begin{proposition}
\label{prop:rearr-char}
Let $l_j$ and $f_j$, $j=1, \dots, n$ be some linear functions on $\R^m$ and characteristic functions of intervals $(b_j - c_j, b_j + c_j) \subset \R$, respectively, and $K$ be a characteristic function of a balanced convex set in $\R^m$.
Then
\[
\int_{\R^m} \Prod_{j=1}^n f_j(l_j(z)) K(z) \dif z \leq \int_{\R^m} \Prod_{j=1}^n f_j^{**}(l_j(z)) K(z) \dif z,
\]
where $f_j^{**}$ denotes the non-increasing radial rearrangement of $f_j$.
\end{proposition}
\begin{proof}
Define $f_j(x | t) = f_j(x + b_j t)$.
Then $f_j(\cdot | 0) = f_j$ and $f_j(\cdot | 1) = f_j^{**}$.
It is therefore natural to study
\[
I(t)
= \int_{\R^m} \Prod_{j=1}^n f_j(l_j(z) | t) K(z) \dif z
= \Leb[m] (K \intersection \Intersection_j S_j(t)),
\]
where $S_j(t) = \{f_j(\cdot|t) = 1\} = \{ z | b_j - c_j \leq l_j(z) + b_j t \leq b_j + c_j\}$ and $K$ stands also for the set of which $K$ is the characteristic function.
The set
\[
C = \{ z \in \R^{m+1} | -c_j \leq l_j(z) - b_j z_{m+1} \leq c_j, \, j=1,\dots,n, \, (z_1, \dots, z_m) \in K\}
\]
is given by $n+1$ conditions each of which defines a balanced convex set, so that it is balanced and convex.
Furthermore,
\[
C \cap \{ z_{m+1} = 1-t \} = K \intersection \Intersection_j S_j(t)
\]
regarded as subsets of $\R^m$.
By Corollary~\ref{cor:balanced-decreasing}, $\Leb[m] (C \cap \{ z_{m+1} = 1-t \})$ grows with $t$.
This implies $I(0) \leq I(1)$.
\end{proof}

This result extends to arbitrary functions and yields the rearrangement inequality due to Brascamp, Lieb and Luttinger \cite{MR0346109}.
It contains the Hardy-Littlewood and the Riesz rearrangement inequalities as special cases.
\begin{theorem}
\index{Brascamp-Lieb-Luttinger rearrangement inequality}
\label{thm:steiner-sym-increases-integrals}
Let $f_j$, $j=1, \dots, n$ be positive functions on $\R$, $K$ a Steiner convex function on $\R^m$ and $l_j$ some linear functions.
Then
\[
\int_{\R^m} \Prod_{j=1}^n f_j(l_j(z)) K(z) \dif z \leq \int_{\R^m} \Prod_{j=1}^n f_j^{**}(l_j(z)) K(z) \dif z.
\]
\end{theorem}
\begin{proof}
We use the identity
\begin{equation}
\label{eq:decomposition-characteristic}
g = \int_{s=0}^\infty \chi_{\{g > s\}} \dif s
\implies
g^{**} = \int_{s=0}^\infty \chi_{\{g > s\}}^{**} \dif s
\text{ with } g=f_j,K.
\end{equation}
By monotonous approximation it is sufficient to prove the theorem in the case that $K$ is the characteristic function of a balanced convex set and $f_j$ are characteristic functions of some sets $A_j$.
Again by monotonous approximation $A_j$ may be assumed to be bounded, then open, then finite unions of intervals.

Write $f_j = \sum_{k} f_{j,k}$, where $f_{j,k}$ are characteristic functions of $(b_{j,k} - c_{j,k}, b_{j,k} + c_{j,k})$ and $b_{j, k} + c_{j, k} < b_{j, k+1} - c_{j, k+1}$.
The proof of Proposition~\ref{prop:rearr-char} applied to each combination of functions $f_{j,k}(x | t) = f_{j,k}(x + b_{j,k} t)$ individually shows that
\[
I(t) = \int_{\R^m} \Prod_{j=1}^n \sum_k f_{j,k}(l_j(z) | t) K(z) \dif z
\]
is monotonously increasing in $t$.
Since $f_{j,k}(\cdot | 1) = f_{j,k}^{**}$, there exists a smallest $t_0$ such that the supports of some $f_{j,k}(\cdot | t_0)$ and $f_{j,k+1}(\cdot | t_0)$ touch each other (see Figure~\ref{fig:intervals-touch}).
\begin{figure}
\begin{center}
\begin{tikzpicture}[scale=0.9]
\draw[->] (0,0) -- (0,2.5) node[left] {$t$};

\draw[->] (-5.5,0) -- (6,0) node[right] {$z_j$};
\draw[very thick,|-|] (-5,0) -- node[below] {$f_{j,1}$} (-3,0);
\draw[very thick,|-|] (-0.5,0) -- node[below] {$f_{j,2}$} (0.5,0);
\draw[very thick,|-|] (2,0) -- node[below] {$f_{j,3}$} (3,0);
\draw[very thick,|-|] (4,0) -- node[below] {$f_{j,4}$} (5,0);

\draw[dashed] (-4,0) -- (0,2);
\draw[dashed] (2.5,0) -- (0,2);
\draw[dashed] (4.5,0) -- (0,2);

\draw[very thick,|-|] (-3,1) -- (-1,1);
\draw[very thick,|-|] (-0.5,1) -- (0.5,1);
\draw[very thick,|-|] (0.75,1) -- (1.75,1);
\draw[very thick,|-|] (1.75,1) -- (2.75,1) node[right] {time $t_0$};

\draw[very thick,|-|] (-1,2) -- (1,2);
\draw[very thick,|-|] (-0.5,2) -- (0.5,2);
\end{tikzpicture}
\end{center}
\caption{Reduction of the number of intervals}
\label{fig:intervals-touch}
\end{figure}
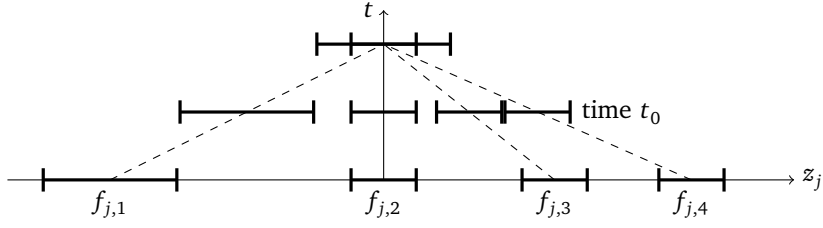
At this point, replace $f_j$ by $\sum_k f_{j,k} (\cdot | t_0)$ and reiterate.
Since the number of the pieces $f_{j,k}$ decreases with each step, at some point one arrives in the setting of Proposition~\ref{prop:rearr-char}.
\end{proof}


%% file: chapter_hardy.tex
\chapter{The Hardy space \texorpdfstring{$H^1$}{H1}}
\label{chap:hardy}
It turns out that $L^1$ can be replaced by the (strictly smaller) Hardy space $H^1$ for purposes of complex interpolation which facilitates proving endpoint estimates.

There exist various equivalent definitions of the Hardy space $H^1$.
We follow here Stein's exposition in \cite{MR1232192} with modifications due to Meda, Sjögren and Vallarino \cite{MR2656522}.
The most convenient characterization for the task of proving the boundedness of an operator \emph{from} the Hardy space is by atomic decomposition.
The latter depends on a parameter $1<q\leq\infty$.
The flavor obtained if $q<\infty$ is particularly adapted to the task of identifying $\BMO$ as the dual space of $H^1$, while $q=\infty$ yields a formally stronger characterization which is useful for interpolation.
Unfortunately, the equivalence of spaces obtained for various $q$ is not evident (although one inclusion is, for every pair of $q$'s).
A detour through a maximal characterization of $H^1$, which is given here as the definition, is needed.

The choice of the maximal function below is fairly arbitrary and we use the one which is tailored to our needs.
Let $b>1$ and $\Phi_1$ be the set of smooth functions with support in $B(0,1)$ and $C^1$-seminorm bounded by a constant.
The corresponding \emph{non-tangential grand maximal function} is defined by
\index{grand maximal function}
\[
\GM f(x) = \GM_{\Phi_1, b} f(x) = \sup_{\phi\in\Phi_1, t>0, |y-x|<bt} |a*\phi_t(y)|,
\]
where $\phi_t(y)=t^{-n} \phi(y/t)$.
\begin{definition}
\index{Hardy space@Hardy space $H^{1}$}
The Hardy space $H^1(\R^n)$ is the space of functions $f$ on $\R^n$ such that
\[
||f||_{H^1} := || \GM f ||_1 < \infty.
\]
\end{definition}
We remark that $\GM f$ is dominated by
\[
\GM_{\Phi} f(x) = \sup_{\phi\in\Phi, t>0} |a*\phi_t(x)|,
\]
where $\Phi$ denotes the set of smooth functions with $C^1$-seminorm bounded by the same constant as before but this time with support in $B(0,1+b)$.
This is evident since $\tau_{y} \phi_t = (\tau_{y/t} \phi)_t$.

\section{Atomic decomposition}
\begin{definition}
Let $1 < q \leq \infty$.
A \emph{$(1,q)$-atom} is a function $a$ on $\R^n$ such that
\begin{enumerate}
\item $\supp a \subset B$ for some ball $B \subset \R^n$,
\item $||a||_q \leq |B|^{-1+1/q}$,
\item $\int_B a = 0$.
\end{enumerate}
\end{definition}

Abusing the notation we sometimes write $\supp a$ for a ball satisfying the above conditions.

\begin{definition}
The space $H^{1,q}_{\text{at}}$ consists of all functions $f \in L^1$ which admit a \emph{$(1,q)$-atomic decomposition}, i.e.
\index{Hardy space@Hardy space $H^{1}$!atomic decomposition}
\begin{equation}
\label{eq:atomic}
f = \sum_j \lambda_j a_j,
\end{equation}
where $a_j$ are $(1,q)$-atoms, $( \lambda_j )_j$ is an absolutely summable sequence and the sum is taken in $L^1$ sense.
The norm on this space is given by $||f||_{H^{1,q}_{\text{at}}} = \inf \sum_j |\lambda_j|$, where the infimum is taken over all decompositions (\ref{eq:atomic}).
\end{definition}

Since every $(1,q)$-atom is also a $(1,\infty)$ atom, we see that $H^{1,\infty}_{\text{at}} \subset H^{1,q}_{\text{at}}$ with contractive inclusion for every $1<q<\infty$.

Next we relate the spaces $H^{1,q}_{\text{at}}$ and $H^1$.
\begin{proposition}
\label{prop:GM-of-atom}
We have that $||\GM_\Phi f||_1 \leq C_q ||f||_{H^{1,q}_{\text{at}}}$ for every $f \in H^{1,q}_{\text{at}}$.
In particular, $||\GM_\Phi a||_1 \leq C_q$ for every $(1,q)$-atom $a$ and there is a continuous inclusion $H^{1,q}_{\text{at}} \subset H^1$.
\end{proposition}
\begin{proof}
Let $f \in H^{1,q}_{\text{at}}$, $f = \sum_j \lambda_j a_j$ be an atomic decomposition, $t>0$ and $\phi \in \Phi$.
Then $\phi_t \in L^\infty$, so that
\[
\phi_t * f = \sum_j \lambda_j (\phi_t * a_j) \text{ pointwise everywhere},
\]
because $||a_j||_1 \leq 1$ for every $j$ by the Hölder inequality and the atomic decomposition converges in $L^1$.
Therefore
\begin{align*}
\GM_\Phi f(x)
&=
\sup_{t>0} \sup_{\phi \in \Phi_1} |\phi_t * f(x)|\\
&\leq
\sup_{t>0} \sup_{\phi \in \Phi_1} \sum_j |\lambda_j| |\phi_t * a_j(x)|\\
&\leq
\sum_j |\lambda_j| \sup_{t>0} \sup_{\phi \in \Phi_1} |\phi_t * a_j(x)|\\
&=
\sum_j |\lambda_j| \GM_\Phi a_j(x).
\end{align*}
Taking first the $L^1$ norm and then the infimum over all atomic decompositions we see that $||\GM_\Phi f||_1 \leq ||f||_{H^{1,q}_{\text{at}}} \sup_a ||\GM_\Phi a||_1$, where the supremum is taken over all $(1,q)$-atoms.

We will now give a uniform estimate for $||\GM_\Phi a||_1$.
By translation invariance we can assume that $\supp a \subset B(0,R)$.
We split
\[
||\GM_\Phi a||_1
=
\int_{B(0,2R)} \GM_\Phi a(x) \dif x
+
\int_{\R^n \setminus B(0,2R)} \GM_\Phi a(x) \dif x
=
I_1 + I_2
\]
and estimate the two terms separately.
In the first term we estimate $\GM_\Phi a$ in terms of the Hardy-Littlewood maximal function as
\begin{equation}
\label{eq:GM-M}
\begin{split}
\GM_\Phi a(x)
&\leq
\sup_{\phi \in \Phi, t>0} \int_{B(x,t(1+b))} |a(y) \phi_t(x-y)| \dif y\\
&\leq
C \sup_{t>0} t^{-n} \int_{B(x,t(1+b))} |a(y)| \dif y\\
&\leq
C Ma(x).
\end{split}
\end{equation}
By the Hölder inequality and the Hardy-Littlewood maximal inequality (Theorem~\ref{th:hardy-littlewood-maximal-inequality}),
\[
I_1
\leq
C R^{n/q'} ||\GM_\Phi a||_q
\leq
C R^{n/q'} ||M a||_q
\leq
C R^{n/q'} ||a||_q
\leq
C.
\]
In the second integral we use $x > 2R$ and $\supp a \subset B(0,R)$ to exclude integrals of products of functions with disjoint support from definition of $\GM_\Phi$ to infer
\begin{align*}
\GM_\Phi a(x)
&=
\sup_{\phi\in\Phi, t>0} |a*\phi_t(x)|\\
&=
\sup_{\phi\in\Phi, t>(|x|-R)/(1+b)} \left| \int_{B(0,R)} a(y) \phi_t(x-y) \dif y \right|\\
&=
\sup_{\phi\in\Phi, t>(|x|-R)/(1+b)} t^{-n}
\left| \int_{B(0,R)} a(y) \left(\phi\left(\frac{x-y}{t}\right) - \phi\left(\frac{x}{t}\right) \right) \dif y \right|\\
&\leq
\sup_{\phi\in\Phi, t>(|x|-R)/(1+b)} t^{-n}
\left| \int_{B(0,R)} a(y) \left( \frac{R}{t} \sup |\nabla \phi| \right) \dif y \right|\\
&\leq
C \sup_{t>(|x|-R)/(1+b)} R t^{-n-1}
\int_{B(0,R)} |a(y)| \dif y\\
&\leq
C R ((|x|-R)/(1+b))^{-n-1}
||a(y)||_q |B(0,R)|^{1/q'}\\
&\leq
C R |x|^{-n-1}.
\end{align*}
Inserting this into the integral yields
\[
I_2
\leq
C R \int_{\R^n \setminus B(0,2R)} |x|^{-n-1} \dif x
=
C.
\qedhere
\]
\end{proof}

Hence we have $H^{1,\infty}_{\text{at}} \subset H^{1,q}_{\text{at}} \subset H^1$ for every $1<q<\infty$.
The next proposition shows that these inclusions are in fact equalities.
The required properties of the Whitney decomposition are summarized in Section~\ref{sec:whitney}.

\begin{proposition}
\label{prop:infinite-atomic}
There exists a canonic continuous inclusion $H^1 \subset H^{1,\infty}_{\text{at}}$.

More precisely, there exist constants $c_{s} < c_{s}' < 1$ and $N \in \N$ such that for every $f \in H^{1}$ there exists an atomic decomposition
\[
f = \sum_{m \in \Z,j} b_{j}^{(m)} = \sum_{m,j} \lambda_{j}^{(m)} a_{j}^{(m)}
\]
such that the following holds.
\begin{enumerate}
\item Every $b_{j}^{(m)}$ is a function supported in a ball with center $x_{j}^{(m)} \in O^{(m)} := \{ \GM f > 2^m \}$ and radius $c_{s} d_{j}^{m}$, where $d_{j}^{(m)} = \dist(x_{j}^{(m)}, \complement O^{(m)})$,
\item every $b_{j}^{(m)}$ is bounded by $C 2^{m}$, where $C$ is a universal constant,
\item the functions $a_{j}^{(m)} := \left( \lambda_{j}^{(m)} \right)\inv b_{j}^{(m)}$ are $(1,\infty)$-atoms,
\item the coefficients are given by $\lambda_{j}^{(m)} := C' 2^{m} (d_{j}^{(m)})^{n}$, where $C'$ is a universal constant, and satisfy
\[
\sum_{m,j} \lambda_{j}^{(m)} \leq C ||\GM f||_{1},
\]
where $C$ is a universal constant,
\item \label{atomic:bdd-int}
for a fixed $m$ the balls $B(x_{j}^{(m)}, c_{s}' d_{j}^{m})$ enjoy a bounded intersection property with constant $N$ which does not depend on $f$ and $m$.
\end{enumerate}

Moreover, if $f$ is continuous, then every atom $a_{j}^{(m)}$ is continuous.
If $f\in L^{p}$ for some $1 < p < \infty$, then the decomposition converges in $L^{p}$.
\end{proposition}
\begin{proof}
The function $\GM f$ is the supremum of continuous functions and thus lower semicontinuous.
The sets $O^{(m)}$ are therefore open.
For each $O^{(m)}$ let $\{ x_j \}_{j \in J^{(m)}}$ be the corresponding Whitney decomposition given by Proposition~\ref{prop:whitney} for some fixed $c$ and $c'$, so that in particular $O^{(m)} = \union_{j \in J^{(m)}} B(x_j, c' d_j^{(m)})$, where $d_j^{(m)} = \dist(x_j, \complement O^{(m)})$.
The superscript ${}^{(m)}$ will be reserved for objects associated with the decomposition of $O^{(m)}$.

\paragraph*{Partitions of unity}
Let $c' < c'' < 1$, take a smooth cut-off function $\psi$ with values in $[0,1]$ supported on $B(0,c'')$ and identically $1$ on $B(0,c')$.
Denote by $\tilde\psi_j^{(m)}$ the dilate of this function supported on $B_j^{(m)} := B(x_j, c'' d_j^{(m)})$.
Then $|\nabla \tilde\psi_j^{(m)}| \leq C (d_j^{(m)})\inv$.

For fixed $m$, the balls $B_j^{(m)}$ have the bounded intersection property with constant $N'$ by (\ref{whitney:bounded-intersection}).
Therefore, and by covering property (\ref{whitney:covering}),
\[
\Psi^{(m)} := \sum_k \tilde\psi_k^{(m)}
\]
satisfies $1 \leq \Psi^{(m)} \leq N'$ on $O^{(m)}$.
Thus the partition of unity
\[
\psi_j^{(m)} := \frac{\tilde\psi_j^{(m)}}{\Psi^{(m)}}
\]
is well-defined.
Furthermore, if $J_0$ denotes the subset of $J$ for which $\tilde\psi_{k \in J_0}^{(m)}$ do not vanish, then
\begin{equation}
\label{eq:partition-of-unity-C1-estimate}
\begin{split}
|\nabla \psi_j^{(m)}|
&= \frac{1}{\Psi^2} | \Psi \nabla \tilde\psi_j^{(m)} - \tilde\psi_j^{(m)} \sum_{j' \in J_0} \nabla \tilde\psi_{j'}^{(m)} |\\
&\leq C (d_j^{(m)})\inv + C \sum_{j' \in J_0} (d_{j'}^{(m)})\inv\\
&\leq C (d_j^{(m)})\inv.
\end{split}
\end{equation}
The last inequality is justified by the bounded intersection property (\ref{whitney:bounded-intersection}), which ensures $|J_0|\leq N'$, and the comparability of the radii (\ref{whitney:comparable-radii}).

We also need some properties connecting the partitions of unity $(m)$ and $(m+1)$.
Since $\supp \psi_k^{(m+1)} \subset O^{(m+1)} \subset O^{(m)}$ for every $k$, we have the splitting formula
\begin{equation}
\label{eq:partitions-of-unity-expansion}
\psi_k^{(m+1)} = \sum_j \psi_j^{(m)} \psi_k^{(m+1)}.
\end{equation}
The estimate (\ref{eq:partition-of-unity-C1-estimate}) readily implies
\[
|\nabla \psi_j^{(m)} \psi_k^{(m+1)}|
\leq
C (d_j^{(m)})\inv + C (d_k^{(m+1)})\inv.
\]
Since the left-hand side is non-zero only if $B(x_j, c'' d_j^{(m)})$ and $B(x_k, c'' d_k^{(m+1)})$ intersect, (\ref{whitney:comparable-radii}) applies and we conclude that
\begin{equation}
\label{eq:partition-of-unity-product-C1-estimate}
|\nabla \psi_j^{(m)} \psi_k^{(m+1)}|
\leq
C (d_k^{(m+1)})\inv.
\end{equation}
Now we proceed as follows.
In the first step we decompose $f$ dyadically according to its magnitude using the fact that $\sum_j \psi_j^{(m)}$ is the characteristic function of $O^{(m)}$.
In the second step we decompose further into functions with support in balls provided by the Whitney decomposition.
This is helpful because the size condition on an atom becomes less severe as the supporting ball shrinks.
While doing so we have to pay attention to the cancellation property, i.e.\ that every term in the resulting sum must have mean zero.
This is what the weighted means take care of.

\paragraph*{Decomposition by size}
Consider the weighted mean
\[
c_j^{(m)} = \frac{ \int f \psi_j^{(m)} }{ \int \psi_j^{(m)} }.
\]
The denominator may be estimated from below by $|B(x_j^{(m)}, c d_j^{(m)})| = C (d_j^{(m)})^n$.
To estimate the numerator, observe that $\psi_j^{(m)}(x) = \phi(\frac{x - x_j^{(m)}}{d_j^{(m)}})$ for some smooth test function $\phi$ such that $|| \phi ||_{C^1} \leq C$ by (\ref{eq:partition-of-unity-C1-estimate}).
Recall that $\phi_t (x) = t^{-n} \phi(x/t)$.
Thus
\[
\left| \int f \psi_j^{(m)} \right|
=
(d_j^{(m)})^n \left| f * \phi_{d_j^{(m)}} (x_j^{(m)}) \right|
\leq
C (d_j^{(m)})^n \GM f(y),
\]
given that $\dist(y, x_j^{(m)}) < b d_j^{(m)}$, where $b$ is the non-centrality constant of $\GM$.
By the assumption $b>1$ and by definition of $d_j^{(m)}$ there exists a $y \not\in O^{(m)}$ that satisfies this condition.
Since then $\GM f(y) \leq 2^m$, we have
\begin{equation}
\label{eq:atomic-mean-c}
c_j^{(m)}
\leq
\frac{ C (d_j^{(m)})^n 2^m }{ (d_j^{(m)})^n }
=
C 2^m.
\end{equation}
Since $\{ \psi_j^{(m)} \}_j$ is a partition of unity on $O^{(m)}$, we can decompose
\begin{align*}
f
&=
f \chi_{\complement O^{(m)}} + \sum_j f \psi_j^{(m)}\\
&=
\left( f \chi_{\complement O^{(m)}} + \sum_j c_j^{(m)} \psi_j^{(m)} \right) + \sum_j (f - c_j^{(m)}) \psi_j^{(m)}\\
&=
g^{(m)} + \sum_j (f - c_j^{(m)}) \psi_j^{(m)}.
\end{align*}
By the choice of $c_j$ every summand in the second term has mean zero.
Furthermore, the second summand is supported on $O^{(m)}$.
By the Chebyshev inequality, $|O^{(m)}| \leq 2^{-m}$, and therefore
\[
g^{(m)} \to f  \text{ pointwise a.e.\ as } m \to +\infty.
\]
On the other hand, we have the bound
\[
|g^{(m)}| \leq C 2^m \to 0  \text{ uniformly as } m \to -\infty.
\]
The combination of both asymptotic expressions yields
\begin{equation}
\label{eq:two-sided-atomic-series}
f
=
\lim_{M \to +\infty} g^{(M)} - g^{(-M)}
=
\lim_{M \to +\infty} \sum_{m=-M}^{M-1} (g^{(m+1)} - g^{(m)}) \text{ pointwise a.e.}
\end{equation}
The terms of this series are given by
\[
g^{(m+1)} - g^{(m)}
=
\sum_j (f - c_j^{(m)}) \psi_j^{(m)}
-
\sum_k (f - c_k^{(m+1)}) \psi_k^{(m+1)}.
\]
In this sum the terms $f \psi_j^{(m+1)}$ and $f \psi_j^{(m)}$ cancel out whenever $\GM f > 2^{m+1}$, the remaining terms are bounded by $C 2^m$ and the mean value is zero by the definition of $c_j^{(m)}$.

\paragraph*{Decomposition by support}
Now consider the weighted mean
\[
d_{j,k}^{(m)} = \frac{ \int \left[ (f - c_k^{(m+1)}) \psi_j^{(m)} \right] \psi_k^{(m+1)} }{ \int \psi_k^{(m+1)} }.
\]
Analogously to the estimate for $c_j^{(m)}$ but using (\ref{eq:partition-of-unity-product-C1-estimate}) instead of (\ref{eq:partition-of-unity-C1-estimate}) we obtain
\begin{equation}
\label{eq:atomic-mean-d}
d_{j,k}^{(m)} \leq C 2^m.
\end{equation}

Since $c_k^{(m+1)}$ are weighted means themselves, we have that
\[
\sum_j d_{j,k}^{(m)}
=
C \int \left[ (f - c_k^{(m+1)}) \sum_j \psi_j^{(m)} \right] \psi_k^{(m+1)}
=
C \int \left[ (f - c_k^{(m+1)}) \right] \psi_k^{(m+1)}
=
0
\]
by (\ref{eq:partitions-of-unity-expansion}) and the definition of $c_k^{(m+1)}$, respectively.
This fact and (\ref{eq:partitions-of-unity-expansion}) allow us to expand $g^{(m+1)} - g^{(m)}$ as
\begin{align*}
&
\sum_j (f - c_j^{(m)}) \psi_j^{(m)}
-
\sum_k (f - c_k^{(m+1)}) \psi_k^{(m+1)}\\
&=
\sum_j (f - c_j^{(m)}) \psi_j^{(m)}
-
\sum_{k} \left( (f - c_k^{(m+1)}) \left(\sum_j \psi_j^{(m)}\right) - \sum_j d_{j,k}^{(m)} \right) \psi_k^{(m+1)}\\
&=
\sum_j \left[
(f - c_j^{(m)}) \psi_j^{(m)}
-
\sum_{k} \left( (f - c_k^{(m+1)}) \psi_j^{(m)} - d_{j,k}^{(m)} \right) \psi_k^{(m+1)}
\right]\\
&=
\sum_j \left[
f \psi_j^{(m)} \left( 1 - \sum_k \psi_k^{(m+1)} \right)
- c_j^{(m)} \psi_j^{(m)}
+
\sum_{k} \left( c_k^{(m+1)} \psi_j^{(m)} + d_{j,k}^{(m)} \right) \psi_k^{(m+1)}
\right]
\end{align*}
The term in the square brackets (let us call it $b_j^{(m)}$) has mean zero as may be seen in the next to last line by the definitions of $c_j^{(m)}$ and $d_{j,k}^{(m)}$.

Note that $\psi_j^{(m)} \psi_k^{(m+1)} \neq 0$ only if $B(x_j^{(m)}, c'' d_j^{(m)}) \cap B(x_k^{(m+1)}, c'' d_k^{(m+1)}) \neq \emptyset$.
This is also a necessary condition for $d_{j,k}^{(m)}\neq 0$.
Property (\ref{whitney:comparable-radii}) implies that $d_k^{(m+1)} < \frac{1+c''}{1-c''} d_j^{(m)}$ and thus $\supp \psi_k^{(m+1)} \subset B(x_j^{(m)}, c''(1+2\frac{1+c''}{1-c''}) d_j^{(m)})$.

Hence $b_j^{(m)}$ is supported in $B(x_j^{(m)}, c_{s} d_j^{(m)})$, where $c_{s} := c''(1+2\frac{1+c''}{1-c''})$.
The latter constant becomes less than one if we chose $c''<\sqrt{5}-2$.
Then (\ref{whitney:bounded-intersection}) ensures that the sets $B(x_j^{(m)}, c_{s}' d_j^{(m)})$ enjoy the bounded intersection property with a constant $N$ which is independent of $f$ and $m$, whichever $c_{s} < c_{s}' < 1$ we chose.

Since $\left( 1 - \sum_k \psi_k^{(m+1)} \right)$ is the characteristic function of $\complement O^{(m+1)}$, by (\ref{eq:atomic-mean-c}) and by (\ref{eq:atomic-mean-d}) we have $|b_j^{(m)}| \leq C 2^m$ a.e.
The series (\ref{eq:two-sided-atomic-series}) becomes
\[
f
=
\sum_{m,j} \lambda_j^{(m)} a_j^{(m)}
=
\sum_{m,j} C 2^m |d_j^{(m)}|^n \frac{b_j^{(m)}}{C 2^m |d_j^{(m)}|^n}.
\]
The latter fractions are $(1, \infty)$-atoms.
Since, for every $m$, the balls $B(x_j^{(m)}, c d_j^{(m)})$ are disjoint by \ref{whitney:disjoint} and contained in $O^{(m)}$, we have that
\begin{align*}
\sum_{m,j} |\lambda_j^{(m)}| &\leq C \sum_m 2^m |O^{(m)}|\\
&= C \int_{\R^n} \left( \sum_{m \in \Z : 2^m < \GM f(x)} 2^m \right) \dif x\\
&\leq 2 C \int_{\R^n} \GM f(x) \dif x\\
&= C || \GM f ||_1.
\end{align*}

We now turn to the additional assertions.
\paragraph*{$L^{p}$ convergence}
Suppose that $f \in L^{p}$.
By the bounded intersection property of the supports we have that $\sum_{j} |b_{j}^{(m)}| \leq CN \chi_{O^{(m)}}$.
This implies the pointwise estimate
\[
\sum_{m,j} |\lambda_{j}^{(m)}| |a_{j}^{(m)}| \leq 2CN \GM f.
\]
By the non-centered Hardy-Littlewood maximal inequality (Theorem~\ref{th:hardy-littlewood-maximal-inequality}), the latter function is in $L^{p}$, and convergence of the atomic decomposition in $L^{p}$ follows from the dominated convergence theorem.

\paragraph*{Continuity of atoms}
Now let $f$ be continuous.
We will show that under this assumption every atom $a_j^{(m)}$ is continuous, or, equivalently, that every function
\begin{equation}
\label{eq:bjm}
b_j^{(m)}
=
(f - c_j^{(m)}) \psi_j^{(m)}
-
\sum_{k} \left( (f - c_k^{(m+1)}) \psi_j^{(m)} - d_{j,k}^{(m)} \right) \psi_k^{(m+1)}
\end{equation}
is.
The continuity of the first term is clear.
In the second term the number of non-zero summands in a neighborhood of every interior point of $O^{(m+1)}$ is bounded by a constant $N'$ provided by (\ref{whitney:bounded-intersection}) applied to the balls $\{B(x_{k}^{(m+1)}, c''' d_{k}^{(m+1)})\}_{k}$ with $c'' < c''' < 1$.
Hence the term is continuous inside $O^{(m+1)}$.
Since the term vanishes outside $O^{(m+1)}$, it suffices to show that
\[
h_k = \left( (f - c_k^{(m+1)}) \psi_j^{(m)} - d_{j,k}^{(m)} \right) \psi_k^{(m+1)} \to 0
\text{ as } k \to \infty
\]
uniformly for fixed $m$ and $j$, because by the bounded intersection property this implies that the sum of these terms declines to zero towards the boundary of $O^{(m+1)}$.

Since $c_k^{(m+1)} \leq C 2^m = C$ and $\psi_j^{(m)}$ is a smooth function with compact support,
\[
\tilde f_k = (f - c_k^{(m+1)}) \psi_j^{(m)}
\]
is a uniformly equicontinuous family of functions (parameterized by $k$).
Since the diameter of the support of $\psi_k^{(m+1)}$ shrinks to zero as $k \to \infty$ and
\[
d_{j,k}^{(m)} = \frac{ \int \tilde f_k \psi_k^{(m+1)} }{ \int \psi_k^{(m+1)} },
\]
we indeed have that $h_k \to 0$ as $k \to \infty$.
\end{proof}

\section{Finite atomic decomposition and extension of operators}
It is not true, in general, that a linear operator $T$ which is defined on the subspace of $H^1$ generated by the $(1,\infty)$-atoms and uniformly bounded on the $(1,\infty)$-atoms is continuous with respect to the $H^1$ norm \cite{MR2163588}.

We discuss two possibilities to treat this non-extension problem.
One of them \cite{MR2399059} consists in choosing an even smaller initial domain of definition, while the second \cite[Theorem 1.21]{MR0447954} requires an additional $L^q$-continuity hypothesis.

Denote by $F^{1,\infty}$ the vector space algebraically generated by $(1,\infty)$-atoms.
A natural norm on $F^{1,\infty}$ is given by $||f||_{F^{1,\infty}} = \inf_{f = \sum'_j \lambda_j a_j} \sum'_j |\lambda_j|$, where the infimum is taken other all \emph{finite} atomic decompositions.

To say that $||T a|| \leq C$ for every $(1,\infty)$-atom $a$ is equivalent to saying that $T$ is continuous in the $F^{1,\infty}$ norm.
However, as the following example of Y.~Meyer \cite[5.6]{MR794574} shows, this norm is not equivalent to the $H^1$ norm.

Let $\{B_{j}\}_{j}$ be a dense countable collection of disjoint open balls contained in $B_{1}(0)$ and let $f$ be a function such that $f=1$ on half of each ball $B_{j}$, $f=-1$ on the other half and $f=0$ everywhere else.
Then $||f||_{H^{1}} \leq \sum |B_{j}|$, since the restriction of $f$ to each $B_{j}$ is a multiple of an atom, but $||f||_{F^{1,\infty}} = |B_{1}(0)|$ is a constant.
To see that observe that $|B_{1}(0)|\inv f$ is an atom.
On the other hand, given any finite atomic decomposition $f=\sum_{k} \lambda_{k} a_{k}$ with $\supp a_{k} = A_{k}$ one can write
\[
\sum_{k} |\lambda_{k}| = \int \sum_{k} |\lambda_{k}| \chi_{A_{k}} |A_{k}|\inv.
\]
The latter function is piecewise constant and bounded from below by $\sum_{k} \lambda_{k} a_{k} = f$.
Since $f=1$ on a dense subset of $B_{1}(0)$, we obtain $\sum_{k} |\lambda_{k}| \geq |B_{1}(0)|$.
Letting $\sum_{j} |B_{j}|$ go to zero we see that the $H^{1}$- and the $F^{1,\infty}$-norm are not equivalent on $F^{1,\infty}$.

Nevertheless the two norms do coincide on the space of finite linear combinations of \emph{continuous} $(1,\infty)$-atoms.
\begin{proposition}
\label{prop:finite-atomic}
Let $f \in F^{1,\infty} \intersection C(\R^n)$ be such that $|| \GM f ||_1 = 1$.
Then there exists a finite decomposition $f = \sum_j \lambda_j a_j$, where $a_j$ are $(1,\infty)$-atoms and $\sum_j |\lambda_j| < C$.
In particular, the $F^{1,\infty}$- and the $H^1$-norm are equivalent on $F^{1,\infty} \intersection C(\R^n)$.
\end{proposition}
\begin{proof}
Recall that we can chose the decomposition
\[
f = \sum_m \sum_j b_j^{(m)} = \sum_m \sum_j \lambda_j^{(m)} a_j^{(m)}
\]
provided by Proposition~\ref{prop:infinite-atomic} in such a way that $\supp b_j^{(m)} \subset O^{(m)}$ and that the bounded intersection property (\ref{whitney:bounded-intersection}) applies, so that each sum $\sum_j b_j^{(m)}$ is pointwise finite with a bound on the number of non-zero summands which is uniform in $f$ and $m$.
Taking into account the estimate $|b_j^{(m)}| < C 2^m$ we see that
\begin{equation}
\label{eq:atom-sum-pointwise}
\sum_j |b_j^{(m)}| \leq C 2^m
\quad\text{and}\quad
\sum_m \sum_j |b_j^{(m)}| \leq C \GM f \quad \text{pointwise a.e.}
\end{equation}

Assume without loss of generality that $\supp f \subset B(0,R)$ and recall from the proof of Proposition~\ref{prop:GM-of-atom} that $\GM f(x) \leq CR^{-n}$ for $|x|>2R$, i.e.\ $O^{(m)} \subset B(0,2R)$ for $m > m'$, where $m'$ is the integer part of $\log_2 CR^{-n}$.
Hence
\[
\supp \sum_{m>m'} \sum_j b_j^{(m)} \subset B(0,2R),
\]
so that $g^{(m'+1)} = \sum_{m \leq m'} \sum_j b_j^{(m)} = f - \sum_{m>m'} \sum_j b_j^{(m)}$ is supported in $B(0,2R)$ as well.
Since $|g^{(m'+1)}| \leq C 2^{m'} \leq C R^{-n}$, the function $g^{(m'+1)}$ is a bounded multiple of an atom.
Since $f$ is uniformly continuous, $(f-c_j^{(m'+1)}) \psi_j^{(m'+1)} \to 0$ uniformly as $j \to \infty$.
Analogously to the continuity of $b_{j}^{(m)}$ this implies that $g^{(m'+1)}$ is continuous.

Since $f$ is bounded, $\GM f$ is bounded as well, so that there exists an $m''$ such that $O^{(m)} = \emptyset$ and $b_j^{(m)}=0$ for $m > m''$.
For a given $\epsilon > 0$ let $\delta > 0$ be such that $f$ varies by at most $\epsilon$ on every ball of radius $\delta$.
Let also $\mathcal{J} = \{ (m,j) : m' < m \leq m'' \text{ and } c''' d_j^{(m)} > \delta \}$.
By (\ref{whitney:disjoint}) and because the sets $O^{(m)}$ have finite measure, the index set $\mathcal{J}$ is finite, so that
\[
f = g^{(m'+1)} + \sum_{(m,j) \in \mathcal{J}} b_j^{(m)} + \text{leftover}
\]
is a finite decomposition.
Furthermore, the ``leftover'' is bounded by
\[
(N'+1)N(m''-m') \epsilon,
\]
where $N'+1$ is the bound on the number of summands in (\ref{eq:bjm}) which do not vanish at a given point.
By Proposition~\ref{prop:infinite-atomic} we have that $b_j^{(m)} = \lambda_j^{(m)} a_j^{(m)}$, where $a_j^{(m)}$ are atoms and $\sum_m \sum_j |\lambda_j^{(m)}| < C$.
Since the ``leftover'' can be an arbitrarily small multiple of a $(1,\infty)$-atom, the conclusion follows.
\end{proof}

We observe that $F^{1,\infty} \intersection C(\R^n)$ is dense in $H^1$.
To show this we use the $(1,q)$-atomic characterization of $H^1$ with $1<q<\infty$.
It is sufficient to approximate a $(1,q)$-atom $a$ by continuous functions with compact support in the $H^{1,q}_{\text{at}}$ norm.
Let $(\phi_\epsilon)_\epsilon$ be an approximated identity associated with a function with compact support.
Then $\phi_\epsilon * a \to a$ in $L^q$ as $\epsilon \to 0$ and $\supp \phi_\epsilon * a$ is contained in an arbitrarily small neighborhood of $\supp a$.
This clearly implies $\phi_\epsilon * a \to a$ in $H^1$.
Since $\phi_\epsilon * a$ is a continuous function with compact support for every $\epsilon$, the conclusion follows.

Therefore a linear operator $T$ defined on $F^{1,\infty} \intersection C(\R^n)$ and bounded uniformly on atoms (i.e.\ such that $||Ta|| < C$ for every continuous $(1,\infty)$-atom $a$) extends in a unique way to a continuous operator on $H^1$.

We now turn our attention to the case of an operator defined on $F^{1,\infty}$ and continuous with respect to an $L^q$ norm.

\begin{theorem}
\label{th:H1-X-cont}
Let $1 < q < \infty$, $D \subset H^1 \intersection L^q$ be a subspace which contains $F^{1,\infty}$, $X$ be a Banach space and $T : D \to X$ be an operator such that $||T a||_X \leq C$ for every continuous atom $a$ and $||T f||_X \leq C ||f||_q$ for every $f \in D$.
Then $T$ admits a unique extension to a continuous operator from $H^1$ to $X$ and this extension is continuous with respect to the $L^q$ norm on $H^1 \intersection L^q$.
\end{theorem}
\begin{proof}
The uniqueness of the extension follows from the fact that $F^{1,\infty}$ is dense in $H^1$.

A characteristic function of a bounded set can be easily approximated in $L^q$ by multiples of $(1,\infty)$-atoms.
Therefore the space $F^{1,\infty}$ is dense in $L^q$.
Extending $T$ to an operator on $H^1 \intersection L^q$ by $L^q$-continuity we may assume $D = H^1 \intersection L^q$.
It suffices to show that $T$ is continuous in the $H^1$ norm.

Let $f \in H^1 \intersection L^q$.
By (\ref{eq:GM-M}) and the Hardy-Littlewood maximal inequality (Theorem~\ref{th:hardy-littlewood-maximal-inequality}) we see that $||\GM f||_q \leq C ||f||_q$.
Let $f = \sum_{j} \lambda_j a_j$ be the atomic decomposition given by Proposition~\ref{prop:infinite-atomic}.
By the additional conclusion of that proposition this series also converges in $L^q$.
Hence, given $\epsilon > 0$, there exists a (finite) partial sum $\bar f = \sum'_j \lambda_j a_j$ such that $||f-\bar f||_{H^1} \leq \epsilon/2$ and $||f-\bar f||_q \leq \epsilon/2$.
Since the $H^{1,\infty}_{\text{at}}$- and the $H^{1,q}_{\text{at}}$-norm are equivalent and using the approximated identity as above we can approximate each $a_j$ by a continuous function $\tilde a_j$ simultaneously in $H^1$ and $L^q$.
By doing so we obtain an $\tilde f = \sum'_j \lambda_j \tilde a_j \in F^{1,\infty} \intersection C(\R^n)$ such that $||f-\tilde f||_{H^1} \leq \epsilon$ and $||f-\tilde f||_q \leq \epsilon$.
Therefore, since the $H^1$- and the $F^{1,\infty}$-norm are equivalent on $F^{1,\infty} \cap C(\R^n)$,
\begin{align*}
||T f||_X
&\leq
||T \tilde f||_X + || T (\tilde f - f) ||_X\\
&\leq
||T||_{F^{1,\infty} \to X} ||\tilde f||_{F^{1,\infty}} + ||T||_{L^{q} \to X} || \tilde f - f ||_{q}\\
&\leq
C ||\tilde f||_{H^1} + C \epsilon\\
&\leq
C (||f||_{H^1} + || \tilde f - f ||_{H^1}) + C \epsilon\\
&\leq
C ||f||_{H^1} + C \epsilon.
\end{align*}
Since $\epsilon$ is arbitrary and the constants do not depend on $\epsilon$, the operator $T$ is continuous on $D$ with respect to the $H^1$ norm.
\end{proof}

A useful variant of the preceding result regards operators with values in function spaces.
\begin{theorem}
\label{th:H1-Lp-cont}
Let $(\Omega,\mu)$ be a $\sigma$-finite measure space, $1 \leq p,r \leq \infty$, $1 < q < \infty$, $D \subset H^1 \intersection L^q(\R^n)$ be a subspace which contains $F^{1,\infty}$, and $T : D \to L^p \intersection L^r(\Omega)$ be an operator such that $||T a||_p \leq C$ for every continuous atom $a$ and $||T f||_r \leq C ||f||_q$ for every $f \in D$.
Then $T$ admits a unique extension to a continuous operator from $H^1(\R^n)$ to $L^p(\Omega)$ and this extension is continuous as an operator from $H^1 \intersection L^q(\R^n)$ endowed with the $L^q$ norm to $L^r(\Omega)$.
\end{theorem}
\begin{proof}
As before, we may assume that $D = H^1 \intersection L^q(\R^n)$ and that
\[
T : (F^{1,\infty} \intersection C(\R^n), ||\cdot||_{H^1}) \to L^p(\Omega)
\]
is continuous.
We only need to show that $T$ is continuous with respect to the $H^1$ norm on the whole space $D$.

Let $f \in D$.
By the proof of Theorem~\ref{th:H1-X-cont} there exists a sequence $(f_j)_j \subset F^{1,\infty} \intersection C(\R^n)$ which converges to $f$ both in the $H^1$- and in the $L^q$-norm.
By the continuity in the $L^q$ norm we have that $Tf_j \to Tf$ as $j \to \infty$ in $L^r(\Omega)$.
On the other hand, $(f_j)_j$ is a Cauchy sequence with respect to the $H^1$ norm, so that $(Tf_j)_j$ is a Cauchy sequence with respect to the $L^p$ norm.
Since the $L^p$- and the $L^r$-limit must coincide, $Tf_j \to Tf$ in $L^p(\Omega)$ as well, which proves the $H^1(\R^n) \to L^p(\Omega)$ continuity.
\end{proof}

\section{\texorpdfstring{$\BMO$}{BMO}, the dual space of \texorpdfstring{$H^1$}{H1}}
The abbreviation $\BMO$ stands for ``bounded mean oscillation''.
As we shall see, it does not matter much how exactly we compute the ``mean''.
However, for the moment we have to distinguish between various possibilities.
\begin{definition}
\index{BMO@$\BMO$}
Let $1 \leq q < \infty$.
The \emph{$q$-mean oscillation} of a function $f$ over a ball $B$ is given by
\[
\mo_q(f, B) = \left( \frac{1}{|B|} \int_B \left| f(x) - \frac{1}{|B|} \int_B f \right|^q \right)^{1/q}.
\]
\end{definition}
To shorten the notation we will write $\fint_{B}$ in place of $\frac{1}{|B|} \int_{B}$.

\begin{definition}
The space $\BMO_q$ is the quotient of the space of functions for which
\[
|| f ||_{\BMO_q} = \sup_B \mo_q(f,B)
\]
is finite (where the supremum is taken over all balls $B \subset \R^n$) modulo the constant functions.
\end{definition}

Given a function in $\BMO_q$, we can easily construct from it other functions with bounded mean oscillation.
\begin{proposition}
\label{prop:BMO-arbitrary-constants}
Let $B \mapsto c_B$ be a complex-valued function on the space of all balls in $\R^n$ and $f$ a measurable function on $\R^n$.
Assume that
\[
A = \sup_B \left( \fint_B |f - c_B|^q \right)^{1/q} < \infty.
\]
Then $f \in \BMO_q$ and $||f||_{\BMO_q} \leq 2 A$.
\end{proposition}
\begin{proof}
By Hölder's inequality we have
\[
\left| \fint_B f - c_B \right|
\leq
\fint_B |f - c_B|
\leq
\left( \fint_B |f - c_B|^q \right)^{1/q}
\]
and thus by the Minkowski inequality
\[
\mo_q (f,B)
=
\left( \fint_B \left|f - c_{B} + c_{B} - \fint_B f \right|^{q} \right)^{1/q}
\leq
2 \left( \fint_B \left|f - c_B \right|^q \right)^{1/q}.
\qedhere
\]
\end{proof}

Thus we see that replacing balls by cubes (or other comparable shapes) in the definition of $\BMO_{q}$ we obtain the same space with a comparable norm.

\begin{proposition}
Let $b \in \BMO_q$ and identify it with an arbitrary representative of the equivalence class.
Let $P$ be a contraction on $\C$.
Then $P \circ b \in \BMO_q$ and $|| P \circ b ||_{\BMO_q} \leq 2 || b ||_{\BMO_q}$.
\end{proposition}
\begin{proof}
Since $P$ is a contraction, we have
\[
\left( \fint_B \left| P(b(x)) - P\left(\fint_B b\right) \right|^q \dif x \right)^{1/q}
\leq
\left( \fint_B \left| b(x) - \fint_B b \right|^q \dif x \right)^{1/q}
\leq
||b||_{\BMO_q},
\]
for every ball $B$ and the result follows by Proposition~\ref{prop:BMO-arbitrary-constants}.
\end{proof}

\begin{corollary}
Let $b \in \BMO_q$ and fix a representative for it.
Define a family of functions parameterized by $\alpha > 0$ by composition with projections onto the disc of radius $\alpha$ as
\marginpar{
\begin{tikzpicture}[scale=0.65]
\draw[fill] (0,0) circle (1pt) node[anchor=north west] {$0$};
\draw (0,0) circle (1);
\draw[|->] (30:2) node[right] {$b(x)$} -- (30:1) node[anchor=north west] {$b_\alpha(x)$};
\end{tikzpicture}
}
\[
b_\alpha(x) = P_\alpha(b(x)) =
\begin{cases}
b(x), & |b(x)| < \alpha,\\
\alpha b(x)/|b(x)| & \text{otherwise}.
\end{cases}
\]
Then $b_\alpha \in \BMO_q$ and $|| b_\alpha ||_{\BMO_q} \leq 2 || b ||_{\BMO_q}$.
Furthermore, $|b| \in \BMO_q$ and $|| |b| ||_{\BMO_q} \leq 2 || b ||_{\BMO_q}$.
\end{corollary}

Now we come to the proof of duality between $H^1$ and $\BMO$.
\begin{proposition}
\index{Hardy space@Hardy space $H^{1}$!dual}
\label{prop:H1-dual-BMO}
For every $1 < q \leq \infty$, the dual space of $H^{1,q}_{\text{at}}$ is $\BMO_{q'}$.
The dual pairing is given by
\begin{equation}
\label{eq:H1-BMO-duality}
\< \sum_j \lambda_j a_j , b \> = \sum_j \lambda_j \int a_j(x) b(x) \dif x,
\quad
f = \sum_j \lambda_j a_j \in H^{1,q}_{\text{at}},\,
b \in \BMO_{q'},
\end{equation}
and the value of the series is independent of the choice of the atomic decomposition and the representative of the equivalence class $b$.

In particular all the spaces $\BMO_{q}$ for $1 \leq q < \infty$ are identical up to norm equivalence.
\end{proposition}
\begin{proof}
Since $(1,q)$-atoms have mean zero, the value of the integral $\int a_j b$ does not depend on the choice of the representative of $b$.
For the same reason
\begin{multline*}
\left| \int a_j b \right|
=
\left| \int_{B_j} a_j \left(b - \fint_{B_j} b \right) \right|\\
\leq
|| a_j ||_q |B_j|^{1/q'} \left( \fint_{B_j} \left(b - \fint_{B_j} b \right)^{q'} \right)^{1/q'}
\leq
|| b ||_{\BMO_{q'}}.
\end{multline*}
Here $B_j$ is the ball containing the support of $a_j$.
Hence the series $\sum_j \lambda_j \int a_j b$ converges absolutely.

To show that its value does not depend on the atomic decomposition it is sufficient to show that if $\sum_j \lambda_j a_j = 0$, then also $\sum_j \lambda_j \int a_j b = 0$.
If $b$ is bounded, this follows immediately by the dominated convergence theorem since $\sum_j |\lambda_j| |a_j| \in L^1$.
Otherwise consider the bounded functions $b_\alpha$.
By the definition of the $\BMO_{q'}$-norm, $b \in L^{q'}(B_j)$ for every $j$ and $b_\alpha \to b$ as $\alpha \to \infty$ in $L^{q'}(B_j)$ by the monotone convergence theorem.
Since $a_j \in L^q(B_j)$, this implies
\[
\int a_j b_\alpha \to \int a_j b.
\]
Since $\int a_j b_\alpha \leq ||b_\alpha||_{\BMO_{q'}} \leq 2 ||b||_{\BMO_{q'}}$, we may apply the monotone convergence theorem to the series and we obtain
\[
0 = \sum_j \lambda_j \int a_j b_\alpha \to \sum_j \lambda_j \int a_j b \text{ as } \alpha\to\infty.
\]
Therefore every $b \in \BMO_{q'}$ defines a unique bounded linear form on $H^{1,q}_{\text{at}}$.
For the converse, restrict the problem to the case $1 < q < \infty$ first.

Let $\ell \in (H^{1,q}_{\text{at}})'$.
For every ball $B$ we can define $\ell_B \in (L^q(B))' = L^{q'}(B)$ by
\[
\ell_B (f) = \ell(f - \fint_B f) \text{ for } f \in L^q(B),
\]
since $(f - \fint_B f)/(2 ||f||_q^{-1+1/q})$ is a $(1,q)$-atom.
Hence $|| \ell_B ||_{q'} \leq 2 |B|^{-1+1/q} ||\ell||$ and we have that
\[
\ell(a) = \int_B a \ell_B,
\quad
\ell(f - \fint_B f) = \ell_B (f) = \int_B f \ell_B - \fint_B f \int \ell_B
\]
for every atom $a$ supported in $B$ and $f \in L^q(B)$, respectively.
The latter identity implies that
\[
\ell(f - \chi_B \fint_B f)
=
\int_B f \ell_B - \fint_B f \int_B \ell_B
=
\int_B f \ell_C - \fint_B f \int_B \ell_C
\]
for every ball $C \supset B$ and every $f \in L^q(B) \subset L^q(C)$.
This may be written as
\[
0
=
\int_B f \left(\ell_B - \ell_C - \fint_B (\ell_B - \ell_C) \right).
\]
Since $f$ is arbitrary, we have that
\[
\ell_B - \left. \ell_C \right|_B = \fint_B (\ell_B - \ell_C),
\]
i.e.\ that the two functions differ by a constant.
Since any two balls are contained in a greater one, we see that
\[
b(x) := \ell_B(x) - \fint_B \ell_B, \quad x \in B
\]
is well-defined.
Furthermore, for every ball $B$,
\[
\left( \int_B \left( b(x) + \fint_B\ell_B \right)^{q'} \dif x \right)^{1/q'}
\leq
2 |B|^{-1+1/q} ||\ell||,
\]
so that $||b||_{\BMO_q} \leq 4 ||\ell||$ by Proposition~\ref{prop:BMO-arbitrary-constants}.
Since the series $f = \sum_j \lambda_j a_j$ converges in $H^{1,q}_{\text{at}}$, we have
\[
\ell(f)
=
\lim_{N \to \infty} \sum_{j\leq N} \lambda_j \ell(a_j)
=
\lim_{N \to \infty} \sum_{j\leq N} \lambda_j \int a_j b.
\]
Hence $\BMO_{q'} = (H^q)'$ with equivalent norms for every $1 < q < \infty$.
By the first part of the proof we have also $\BMO_1 \subset (H^1)'$ with continuous inclusion.
On the other hand, $\BMO_{q'} \subset \BMO_1$ since contractivity of this inclusion follows by Hölder's inequality from the definition of $\BMO_{q'}$.
Summarizing, we have
\[
\BMO_1 \subset (H^{1,\infty}_{\text{at}})' = (H^1)' = (H^{1,q}_{\text{at}})' = \BMO_{q'} \subset \BMO_1.
\qedhere
\]
\end{proof}
Henceforth we will call $\BMO$ the dual space of $H^1$ and keep in mind that $\BMO = \BMO_q$ for every $1\leq q < \infty$.

The best-known example of a function which lies in $\BMO(\R^{n})$ but not in $L^{\infty}(\R^{n})$ is probably $f(x)=\log|x|$.
The fact that it lies in $\BMO$ follows immediately from the next lemma taking $c_{Q} = \max_{Q} f$.
\begin{lemma}
Let $f$ be a measurable function on $\R^{n}$ and let $b > 0$, $B \geq 0$ be constants such that for every cube $Q \subset \R^{n}$ there exists a $c_{Q} \in \R$ such that for every $s > 0$
\begin{equation}
\label{eq:john-nirenberg-c}
| Q \cap \{ |f-c_{Q}| > s \} | \leq B e^{- b s} | Q |.
\end{equation}
Then $f \in \BMO(\R^{n})$.
\end{lemma}
The converse, namely that $f \in \BMO$ implies (\ref{eq:john-nirenberg-c}) with $c_{Q} = \fint_{Q} f$, is known as the John-Nirenberg inequality.
\begin{proof}
Fix a cube $Q \subset \R^{n}$.
Then
\begin{align*}
\int_{Q} |f-c_{Q}|
&= \int_{s=0}^\infty |Q \cap \{ |f-c_Q| > s \} | \dif s\\
&\leq \int_{s=0}^\infty B e^{-b s} | Q | \dif s\\
&= B/b |Q|.
\end{align*}
Proposition~\ref{prop:BMO-arbitrary-constants} implies $||f||_\BMO \leq 2B/b$.
\end{proof}

\section{\texorpdfstring{$\VMO$}{VMO}, a predual of \texorpdfstring{$H^1$}{H1}}
We now consider functions for which the oscillation over ever smaller balls not merely stays bounded but converges to zero, hence the name $\VMO$, or ``vanishing mean oscillation''.
For technical reasons we shall use a less direct definition.
\begin{definition}
\index{VMO@$\VMO$}
The space $\VMO$ is the closure of the space $C_{c}$ of continuous functions with compact support in the space $\BMO$.
\end{definition}
We shall need a functional analytic lemma.
\begin{lemma}[Goldstine]
\label{lem:X1-dense-in-X''1}
Let $X$ be a Banach space.
Then, under the canonical identification of $X$ with a subspace of $X''$, the unit ball $B_{X}$ is $\sigma(X'',X')$-dense in $B_{X''}$.
\end{lemma}
\begin{proof}
Given a $\psi \in B_{X''}$ and a finite dimensional space $W \subset X'$ it is a matter of linear algebra to find an $x \in X$ such that $\psi|_{W} = x|_{W}$.
Let now $Y=\cap_{\phi \in W} \ker \phi$.
Assume that $x + Y \cap B_{X} = \emptyset$.
Then $\dist(x,Y) > 1$ and by the Hahn-Banach theorem there exists a $\phi \in X'$ such that $\phi|_{Y}=0$, $\phi(x)>1$ and $||\phi||_{X'}=1$.
Then $\phi \in W$ and
\[
1 < \phi(x) = \psi(\phi) \leq ||\psi||_{X''} ||\phi||_{X'} \leq 1,
\]
a contradiction.
Therefore the affine subspace $x+Y$ contains an element with norm at most $1$.
This element coincides with $\psi$ on $W$.
By definition of the $\sigma(X'',X')$ topology we are done.
\end{proof}

\begin{proposition}[{\cite[Theorem 4.1]{MR0447954}}]
\index{Hardy space@Hardy space $H^{1}$!predual}
\label{prop:VMO-dual-H1}
The Hardy space $H^{1}$ is the Banach space dual of $\VMO$.
The duality is given by (\ref{eq:H1-BMO-duality}).
\end{proposition}
\begin{proof}
Proposition~\ref{prop:H1-dual-BMO} provides the canonical maps
\[
H^{1} \to \BMO' \to \VMO'.
\]
Since $C_{c}$ separates the points of $L^{1}$ and $H^{1}$ is a subspace of $L^{1}$, the space $\VMO$ separates the points of $H^{1}$.
Therefore the canonical contraction $\iota : H^{1} \to \VMO'$ is injective.

By Lemma~\ref{lem:X1-dense-in-X''1} the unit ball $B_{H^{1}}$ is $\sigma(\BMO',\BMO)$-dense in $B_{\BMO'}$.
Hence the Hahn-Banach theorem implies that the set $\iota(B_{H^{1}})$ is $\sigma(\VMO', \VMO)$-dense in $B_{\VMO'}$.
If we can prove that the completion of $\iota(B_{H^{1}})$ in the $\sigma(\VMO', \VMO)$ topology is a subset of $\iota(H^{1})$, then we obtain the set-theoretic equality $\iota(H^{1}) = \VMO'$, and the norm equivalence follows from the closed graph theorem.

Since $(C_{c}(\R^{n}), ||\cdot||_{\infty})$ is separable and the inclusion into $\BMO$ is continuous, the space $\VMO$ is separable.
Therefore the $\sigma(\VMO', \VMO)$ topology is metrizable and it suffices to show that every sequence $(f_{k})_{k \in N} \subset B_{H^{1}}$ contains a $\sigma(H^{1},\VMO)$-convergent subsequence.
By Proposition~\ref{prop:infinite-atomic} every $f_{k}$ can be written as
\[
f_{k} = \sum_{l} \lambda_{k,l} a_{k,l},
\]
where $a_{k,l}$ are $(1,\infty)$-atoms and $\sum_{l} |\lambda_{k,l}| \leq C$.
Spending a multiplicative constant in the last inequality we may assume that every $a_{k,l}$ is supported in $3Q$, where $Q$ is some dyadic cube.
Combining the atoms supported in the same cube we may equally assume that $\supp a_{k,l} = 3 Q_{l}$, where $(Q_{l})_{l \in\N}$ is a fixed enumeration of the dyadic cubes.

Passing to a subsequence we may assume that $a_{k,l} \to a_{l}$ in the $\sigma(L^{\infty}, L^{1})$ topology and $\lambda_{k,l} \to \lambda_{l}$ as $k \to \infty$.
Then $a_{l}$ are $(1,\infty)$-atoms, $\sum_{l} |\lambda_{l}| \leq C$ and $f := \sum_{l} \lambda_{l} a_{l} \in H^{1}$.
Let now $v \in C_{c}$.
Then $||v||_{1} < \infty$ and $v$ is uniformly continuous.
The former property implies that $\<a,v\> \to 0$ as $|\supp a| \to \infty$ (and thus $||a||_{\infty} \to 0$) uniformly for all atoms $a$, while the latter shows that $\<a,v\> \to 0$ as $|\supp a| \to 0$, again uniformly for all atoms $a$.
Furthermore there are only finitely many ``medium-sized'' dyadic cubes $Q$ with a given upper and a lower bound on the edge length such that $3Q$ intersects $\supp v$.
Upon an appropriate choice of the bounds we obtain
\[
\<f_{k},v\>
=
\sum_{l : Q_{l} \text{ small}} \lambda_{k,l} \underbrace{\<a_{k,l}, v\>}_{\leq \epsilon}
+
\sum_{l : Q_{l} \text{ mid-size}} \lambda_{k,l} \underbrace{\<a_{k,l}, v\>}_{\to \<a_{l}, v\>}
+
\sum_{l : Q_{l} \text{ large}} \lambda_{k,l} \underbrace{\<a_{k,l}, v\>}_{\leq \epsilon},
\]
where the middle sum is finite.
Therefore $\<f_{k},v\> \to \<f,v\>$.
By a $3\epsilon$ argument we see that $f_{k} \to f$ in the $\sigma(H^{1},\VMO)$ topology.
\end{proof}

\section{The sharp function and the inverse \texorpdfstring{$L^p$}{Lp} inequality}
We now turn to interpolation between $H^1$ and $L^p$ spaces.
We take the detour \cite{MR0447953} across dual spaces and use Fefferman's sharp function $f^\sharp$ to reduce statements about $\BMO$ to statements about $L^\infty$.
To make use of results in term of the sharp function we will need an estimate for $f$ in terms of $f^\sharp$.
\begin{definition}
\index{Fefferman sharp function}
The \emph{sharp function} is defined by
\[
f^\sharp(x) = \sup_{x \in Q} \mo_1 (f, Q).
\]
\end{definition}

It is clear from definitions that $f \in \BMO$ if and only if $f^\sharp \in L^\infty$ and
\[
||f||_\BMO = ||f||_{\BMO_1} = ||f^\sharp||_\infty.
\]
Moreover, $f^\sharp$ is bounded by $2Mf$, where $Mf$ is the usual maximal function.
Hence the classical maximal inequality gives the bound
\[
|| f^\sharp ||_p \leq 2 || Mf ||_p \leq C_p || f ||_p
\]
for every $1 < p \leq \infty$.
Interestingly, the inverse inequality is true as well under an additional qualitative assumption.

\begin{proposition}
\label{prop:inverse-inequality-for-the-sharp-function}
Let $1 < p < \infty$, $1 \leq p_0 \leq p$ and assume that $f \in L^{p_0}$, $f^\sharp \in L^p$.
Then
\[
||M f||_p \leq C_p ||f^\sharp||_p.
\]
The bound $C_p$ is independent of $p_0$ and $||f||_{p_0}$.
\end{proposition}
\begin{proof}
We use the Calderón-Zygmund decomposition of $f$.
For each $\alpha>0$ we find some cubes $\{ Q^\alpha_j \}$ such that
\begin{equation}
\label{eq:calderon-zygmund-mean-value}
\alpha < \fint_{Q^\alpha_j} |f| \leq 2^n \alpha,
\end{equation}
\[
|f|<\alpha
\quad
\text{a.e.\ in } \R^n \setminus \cup_j Q^\alpha_j
\]
and $\{Q^\beta_j\}$ are sub-cubes of $\{Q^\alpha_k\}$ whenever $\beta>\alpha$.
Essentially, we do so by dividing $\R^n$ into cubes on which the average of $|f|$ is small and then subdividing these cubes dyadically until we encounter some cubes on which the average is large.
Doing so for all $\alpha$ at once, however, requires careful formulation.
See Section~\ref{sec:calderon-zygmund} for details.

\paragraph*{Comparison of scales}
Let $A$ be a constant that will be optimized later and $\mu(\alpha) = \sum_j |Q^\alpha_j|$.
We claim that
\begin{equation}
\label{eq:sharp-scale-comparison}
\mu(\alpha)
\leq
\left|\{ f^\sharp > \alpha/A \}\right|
+
\frac{2}{A} \mu(2^{-n-1} \alpha).
\end{equation}
Let $Q_0 = Q^{2^{-n-1} \alpha}_k$ and consider the sub-cubes $\{ Q^\alpha_j \}_{j \in J_0} \subset Q_0$.
Then either $Q_0 \subset \{ f^\sharp > \alpha/A \}$ and hence
\[
Q^\alpha_j \subset \{ f^\sharp > \alpha/A \}
\text{ for all } j \in J_0
\]
or there exists an $x \in Q_0$ such that $f^\sharp(x) \leq \alpha/A$.
Then
\[
\frac{\alpha}{A} \geq f^\sharp(x) \geq \fint_{Q_0} \left| f - \fint_{Q_0} f \right|.
\]
By (\ref{eq:calderon-zygmund-mean-value}) we have
$\left| \fint_{Q_0} f \right| \leq \fint_{Q_0} |f| \leq 2^n(2^{-n-1}\alpha) = \alpha/2$
and
$\fint_{Q^\alpha_j} |f| > \alpha$ and hence the estimate
\begin{align*}
\frac\alpha2 |Q^\alpha_j|
&\leq
|Q^\alpha_j| \left| \fint_{Q^\alpha_j} |f| - \left| \fint_{Q_0} f \right| \right|\\
&=
\left| \int_{Q^\alpha_j} \left( |f| - \left| \fint_{Q_0} f \right| \right) \right|\\
&\leq
\int_{Q^\alpha_j} \left| |f| - \left|\fint_{Q_0} f \right| \right|\\
&\leq
\int_{Q^\alpha_j} \left| f - \fint_{Q_0} f \right|.
\end{align*}
Summing over $J_0$ we obtain
\[
\sum_{j \in J_0} |Q^\alpha_j|
\leq
\frac{2}{\alpha} \int_{Q_0} \left| f - \fint_{Q_0} f \right|
\leq
\frac2A |Q_0|
\]
and summation over $Q_0$ proves the claim.

\paragraph*{Distribution function of the maximal function}
Let $\lambda$ denote the distribution function of the maximal function
\[
M f (x) = \sup_{x \in Q} \fint_Q |f|.
\]
Since $x \in Q^\alpha_j$ implies $Mf(x) \geq \alpha$ by construction of $Q^{\alpha}_{j}$, we have that
\[
\lambda(\alpha) \geq \mu(\alpha).
\]
To obtain a converse estimate consider the cubes $3 Q^\alpha_j$ which have the same centers as $Q^\alpha_j$ but the triple edge length.
Let $x \not\in \union_j 3 Q^\alpha_j$ be arbitrary.
Then, whenever $Q$ is a cube containing $x$ such that $Q \intersection Q^\alpha_j \neq \emptyset$, the cube $2 Q$ contains $Q^\alpha_j$.
Since $|f| \leq \alpha$ outside of $\union_{j} Q^\alpha_j$, we have
\[
\fint_Q |f|
\leq
\alpha
+
\frac{1}{|Q|} \sum_{j : Q \intersection Q^\alpha_j \neq \emptyset} \int_{Q^\alpha_j} |f|
\leq
\alpha
+
\frac{2^n}{|2Q|} \sum_{j : Q \intersection Q^\alpha_j \neq \emptyset} |Q^\alpha_j| 2^n \alpha
\leq
(1 + 4^n) \alpha,
\]
that is, $Mf < (1+4^n \alpha)$ outside of $\union_j 3 Q^\alpha_j$, so that
\[
\lambda((1+4^n) \alpha) \leq 3^n \mu(\alpha).
\]

\paragraph*{$L^p$ norm of the maximal function in terms of its distribution function}
\begin{align*}
|| Mf ||_p^p
&=
p \int_{a=0}^\infty a^{p-1} \lambda(a) \dif a\\
&\leq
3^n p \int_{a=0}^\infty a^{p-1} \mu(a/(1+4^n)) \dif a\\
&=
3^n (1+4^n)^p \sup_{N>0} p \int_{a=0}^N a^{p-1} \mu(a) \dif a.
\end{align*}
By (\ref{eq:sharp-scale-comparison}) we can estimate the last term by
\begin{align*}
p &\int_{a=0}^N a^{p-1} \mu(a) \dif a\\
&\leq
p \int_{a=0}^N a^{p-1}
\left|\{ f^\sharp > a/A \}\right|
\dif a
+
p \int_{a=0}^N a^{p-1}
\frac{2}{A} \mu(2^{-n-1} a)
\dif a\\
&\leq
A^p p \int_{a=0}^{N/A} a^{p-1}
\left|\{ f^\sharp > a \}\right|
\dif a
+
\frac{2}{A} 2^{(n+1)p} p \int_{a=0}^{2^{-n-1} N} a^{p-1}
\mu(a)
\dif a\\
&\leq
A^p || f^\sharp ||_p^p
+
\frac{2 \cdot 2^{(n+1)p}}{A}
p \int_{a=0}^{N} a^{p-1} \mu(a) \dif a.
\end{align*}
Note that the latter integral is bounded by
\[
p \int_{a=0}^N a^{p-1} |\{|f| > a\}| \dif a
\leq
N^{p-p_0} p \int_{a=0}^N a^{p_0-1} |\{|f| > a\}| \dif a
\leq
N^{p-p_0} || f ||_{p_0} < \infty
\]
and hence both sides of the inequality are finite.
Specifying to $A=4 \cdot 2^{(n+1)p}$ we obtain
\[
p \int_{a=0}^N a^{p-1} \mu(a) \dif a
\leq
2 \cdot 4^p \cdot 2^{(n+1)p^2} || f^\sharp ||_p^p
\]
uniformly in $N$ and hence
\[
|| Mf ||_p^p
\leq
3^n (1+4^n)^p
\cdot 2 \cdot 4^p \cdot 2^{(n+1)p^2} || f^\sharp ||_p^p.
\qedhere
\]
\end{proof}
Since $|f| \leq Mf$ a.e., Proposition~\ref{prop:inverse-inequality-for-the-sharp-function} also tells us that $||f||_p \leq C_p ||f^\sharp||_{p}$ for $1<p<\infty$.

\section{Interpolation between \texorpdfstring{$L^{p}$}{Lp} and \texorpdfstring{$H^{1}$}{H1}}
We shall now see how the Hardy space theory can be used to obtain $L^p$ estimates.
\begin{theorem}
\index{complex interpolation!between VMO and Lp@between $\VMO$ and $L^p$}
\label{th:Lp-VMO-interpolation}
Let $1 < p < \infty$, $0 < \theta < 1$ and $1/p_\theta = (1-\theta)/p$.
Then, up to norm equivalence,
\[
[ L^p, \VMO ]_\theta = L^{p_\theta}.
\]
\end{theorem}
\begin{proof}
Let $f \in L^{p_\theta}$ be a Schwartz function and assume without loss of generality that $||f||_{p_\theta} = 1$.
For every $\epsilon > 0$ there exists an analytic representative $f_z \in \mathcal{F}(L^p, L^\infty, \Schwartz)$ with norm less or equal $1+\epsilon$.
Since $(\Schwartz, ||\cdot||_{\infty}) \hookrightarrow \VMO$, we have that $f_z \in \mathcal{F}(L^p, \BMO)$ and its norm is bounded by a constant.
By density of $\Schwartz$ in $L^{p}$ we obtain the inclusion
\[
L^{p_\theta} \hookrightarrow [ L^p, \VMO ]_\theta.
\]
To obtain the converse consider an arbitrary function $f_z \in \mathcal{F}(L^p, \VMO, L^{p} \intersection \VMO)$ with norm $1$.
The latter space is dense in $\mathcal{F}(L^p, \VMO)$ by Proposition~\ref{prop:F-X0-X1-D-dense} and is a subspace of $\mathcal{F}(L^p, \BMO)$.
To rewrite the $\BMO$ boundedness property in linear terms, consider measurable functions $Q$ on $\R^n$ such that $Q(x)$ is a cube containing $x$ and $\eta$ on $\R^n \times \R^n$ such that $|\eta(x,y)| = 1$.
Let also
\[
U_{z,Q,\eta}(x) = \fint_{Q(x)} \left( f_z(y) - \fint_{Q(x)} f \right) \eta(x,y) \dif y.
\]
Then $f_z^\sharp (x) = \sup_{Q, \eta} |U_{z,Q,\eta}(x)|$.
Therefore
\begin{align*}
|| U_{iy,Q,\eta} ||_p &\leq || f_{iy}^\sharp ||_p \leq C_p || f_{iy} ||_p \leq C_p \text{ and}\\
|| U_{1+iy,Q,\eta} ||_\infty &\leq || f_{1+iy}^\sharp ||_\infty = || f_{1+iy} ||_\BMO \leq 1.
\end{align*}
Since $U_{z,Q,\eta}$ is an analytic function of $z$, this implies that
\[
||U_{\theta,Q,\eta}||_{p_\theta} \leq C_p^{1-\theta},
\]
and the bound is independent of $Q$ and $\eta$.
Choose $Q$ and $\eta$ such that $2 |U_{\theta, Q, \eta}| \geq f_\theta^\sharp$.
The resulting function is in $L^{p_\theta}$ and dominates $U_{\theta, Q, \eta}$ for all $Q$ and $\eta$.
Therefore the dominated convergence theorem applies and
\[
||f_{\theta}^\sharp||_{p_\theta} \leq C_p^{1-\theta}.
\]
On the other hand, $f_\theta \in L^p$ and $p < p_\theta$.
Proposition~\ref{prop:inverse-inequality-for-the-sharp-function} thus implies
\[
|| f_\theta ||_{p_\theta} \leq C || f_\theta^\sharp ||_{p_\theta} \leq C.
\]
Therefore we have
\[
[ L^p, \VMO ]_\theta \hookrightarrow L^{p_\theta}.
\qedhere
\]
\end{proof}

Since $L^{p}$ is reflexive for $1 < p < \infty$, by Proposition~\ref{prop:VMO-dual-H1} and Theorem~\ref{thm:complex-interpolation-dual} we immediately obtain
\begin{corollary}
\label{cor:Lp-H1-interpolation}
Let $1 < p < \infty$ and $0 < \theta < 1$ and $1/p_\theta = (1-\theta)/p$.
Then
\[
[L^{p'}, H^{1}]_{\theta} = L^{p_{\theta}'}
\text{ and }
[L^{p}, \BMO]_{\theta} = L^{p_{\theta}}.
\]
\end{corollary}

An important property of the interpolation pair $(L^{p}, H^{1})$ is the fact that the space of Schwartz functions with mean zero is dense in $L^{p} \intersection H^{1}$.
Indeed, by Proposition~\ref{prop:infinite-atomic} we can approximate every function in $L^{p} \intersection H^{1}$ by a finite linear combination of atoms simultaneously in the $L^{p}$- and the $H^{1}$-norm.
Then we use a smooth approximated identity to make that linear combination smooth.

We now give a direct proof of the fact that $[H^{1}, L^{p}]_{\theta} = L^{p_{\theta}}$ modeled on \cite{MR671315} which does not rely on the duality with $\BMO$.
Together with the preceding observation it enables us to use complex interpolation for analytic families of operators initially defined on Schwartz functions with mean zero.
\begin{proposition}
\index{complex interpolation!between H1 and Lp@between $H^1$ and $L^{p}$}
\label{prop:h1-lp-interpolation-by-schwartz-functions}
Let $f \in H^{1}(\R^{n}) \cap L^{p}(\R^{n})$ be such that $||f||_{p_{\theta}} = 1$, where
\[
1 < p < \infty, \quad
0 < \theta < 1, \text{ and }
\frac{1}{p_{\theta}} = \frac{1-\theta}{1} + \frac{\theta}{p}.
\]
Then there exists an analytic function $f_{z} \in \mathcal{F} (H^{1}, L^{p})$ such that $f_{\theta} = f$ and
\[
||f||_{\mathcal{F}(H^{1},L^{p})} \leq C.
\]
If in addition $f \in \Schwartz(\R^{n})$, then we can choose $f_{z} \in \mathcal{F}(H^{1}, L^{p}, \Schwartz \cap H^{1})$.
\end{proposition}
\begin{proof}
As usual we look for a function which is merely uniformly bounded (and does not exhibit decay) first.

Let $f = \sum_{m,j} b_{j}^{(m)}$ be the atomic decomposition given by Proposition~\ref{prop:infinite-atomic}.
Since it converges both in $H^{1}$ and in $L^{p}$, we can write $f = \sum_{(m,j) \in \mathcal{I}} b_{j}^{(m)} + \tilde f$, where the sum is finite and $||\tilde f||_{H^{1} \cap L^{p}} < 1$.
Let now
\[
f_{z} := \sum_{(m,j) \in \mathcal{I}} (2^{m})^{p_{\theta}/p_{z}} \frac{b_{j}^{(m)}}{2^{m}} + \tilde f.
\]
This function is clearly analytic and bounded on $\bar S$.
Since $\frac{b_{j}^{(m)}}{2^{m} |\supp b_{j}^{m}|}$ is a bounded multiple of a $(1,\infty)$-atom,
\begin{align*}
||f_{0+iy}||_{H^{1}}
&\leq
C \sum_{(m,j) \in \mathcal{I}} (2^{m})^{p_{\theta}} |\supp b_{j}^{m}| + ||\tilde f||_{H^{1}}\\
&\leq
CN \sum_{m} (2^{m})^{p_{\theta}} |O^{(m)}| + ||\tilde f||_{H^{1}}\\
&\leq
C || \GM f ||_{p_{\theta}}^{p_{\theta}} + 1\\
&\leq
C || f ||_{p_{\theta}}^{p_{\theta}} + 1\\
&\leq
C.
\end{align*}
Here we have used the bounded intersection property for $\supp b_{j}^{(m)}$ and (\ref{eq:GM-M}) together with the Hardy-Littlewood maximal inequality (Theorem~\ref{th:hardy-littlewood-maximal-inequality}).
On the other hand,
\begin{align*}
||f_{1+iy} - \tilde f||_{L^{p}}^{p}
&\leq
\int_{\R^{n}} \left| \sum_{(m,j) \in \mathcal{I}} (2^{m})^{p_{\theta}/p} \frac{|b_{j}^{(m)}|}{2^{m}} \right|^{p}\\
&\leq
\int_{\R^{n}} \left| \sum_{m} (2^{m})^{p_{\theta}/p-1} C N 2^{m} \chi_{O^{(m)}} \right|^{p}\\
&=
C \int_{\R^{n}} \left| \sum_{m} \chi_{O^{(m)} \setminus O^{(m-1)}} \sum_{l=-\infty}^{m} (2^{l})^{p_{\theta}/p} \right|^{p}\\
&=
C \int_{\R^{n}} \sum_{m} \chi_{O^{(m)} \setminus O^{(m-1)}} \left| \frac{(2^{m})^{p_{\theta}/p}}{1 - 2^{-p_{\theta} / p}} \right|^{p}\\
&=
C \int_{\R^{n}} \sum_{m} \chi_{O^{(m)} \setminus O^{(m-1)}} (2^{m})^{p_{\theta}}\\
&\leq
C || \GM f ||_{p_{\theta}}^{p_{\theta}}\\
&\leq
C.
\end{align*}
Therefore $e^{z^{2}-\theta^{2}} f_{z}$ has the required property.

If $f$ is a Schwartz function, then we convolve every $b_{j}^{(m)}$ for $(j,m) \in \mathcal{I}$ with a compactly supported positive  smooth function $\phi$ with $\int\phi = 1$.
If the support of $\phi$ is small enough, the bounded intersection property for $\supp b_{j}^{(m)}$ is preserved by conclusion (\ref{atomic:bdd-int}) of Proposition~\ref{prop:infinite-atomic} and $||b_{j}^{(m)}-\phi * b_{j}^{(m)}||_{p}$ is small.
Since $H^{1} = H^{1,p}_{\text{at}}$, the latter assertion implies that $||b_{j}^{(m)}-\phi * b_{j}^{(m)}||_{H^{1}}$ is small too.
Hence the function $e^{z^{2}-\theta^{2}} f_{z}$, where
\[
f_{z} := \sum_{(m,j) \in \mathcal{I}} (2^{m})^{p_{\theta}/p_{z}} \frac{\phi * b_{j}^{(m)}}{2^{m}} + \tilde f + \sum_{(m,j) \in \mathcal{I}} (b_{j}^{(m)} - \phi * b_{j}^{(m)}),
\]
has the required property.
\end{proof}
By density of $H^{1} \cap L^{p}$ in $L^{p_{\theta}}$ we obtain a continuous inclusion $L^{p_{\theta}} \hookrightarrow [H^{1}, L^{p}]_{\theta}$.
The continuity of the converse mapping $[H^{1}, L^{p}]_{\theta} \hookrightarrow L^{p_{\theta}} = [L^{1}, L^{p}]_{\theta}$ is clear since $H^{1}$ is a subspace of $L^{1}$.

\section{Whitney decomposition}
\label{sec:whitney}
\begin{proposition}
\index{Whitney decomposition}
\label{prop:whitney}
Let $O \subset \R^n$ be an open set with non-empty complement and $c, c'$ be positive numbers satisfying
\[
c + \frac{2c}{c'} \leq 1; \quad c' < 1.
\]
Then there exist $x_j \in O$ such that
\begin{enumerate}
\renewcommand{\theenumi}{W\arabic{enumi}}
\renewcommand{\labelenumi}{\theenumi.}
\item the balls $B(x_j, c d_j)$ are pairwise disjoint,
\label{whitney:disjoint}
\item the balls $B(x_j, c' d_j)$ cover $O$,
\label{whitney:covering}
\end{enumerate}
Here, $d_j = \dist (x_j, \complement O)$.
\end{proposition}
\begin{proof}
Let $\{x_j\}_{j \in J}$ be a maximal collection of points for which (\ref{whitney:disjoint}) is satisfied.
For notational convenience, $0 \not\in J$.

Assume that there exists a point $x_0 \in O \setminus \union_j B(x_j, c' d_j)$.
Then, for every $j \in J$, $d := \dist(x_0, x_j) \geq c' d_j$ since $x_0 \not\in B(x_j, c' d_j)$ and $d_0 \leq d + d_j$ by the triangle inequality.
This combines to
\[
c d_0 + c d_j \leq c (d + 2 d_j) \leq c d (1 + \frac{2}{c'}) \leq d
\]
by the assumption.
Therefore $B(x_0, c d_0)$ and $B(x_j, c d_j)$ are disjoint, which contradicts the maximality of $\{x_j\}_{j \in J}$.
Hence (\ref{whitney:covering}) is proved.
\end{proof}

\begin{proposition}
\label{prop:whitney-comparable-radii}
Let $c$ and $c'$ be as above, $c' < c'' < 1$ and $O^{(+)} \subseteq O \subset \R^n$ be open sets with non-empty complement.
Let $\{ x_j \}_{j \in J}$, $\{ x_j \}_{j \in J^{(+)}}$ be their Whitney decompositions.
Then,
\begin{enumerate}
\setcounter{enumi}{2}
\renewcommand{\theenumi}{W\arabic{enumi}}
\renewcommand{\labelenumi}{\theenumi.}
\item \label{whitney:comparable-radii}
for each pair of $j\in J$, $k \in J^{(+)}$ such that $B(x_j, c'' d_j) \cap B(x_k, c'' d_k^{(+)}) \neq \emptyset$, we have that
\[
d_k^{(+)} < \frac{1 + c''}{1 - c''} d_j.
\]
\end{enumerate}
\end{proposition}
\begin{proof}
Let $j$ and $k$ satisfy the assumptions.
Since $O^{(+)} \subseteq O$, by the triangle inequality and by the assumption that the two balls intersect, we have
\[
d_k^{(+)} \leq d_k \leq \dist(x_j, x_k) + d_j < c'' (d_j + d_k^{(+)}) + d_j.
\]
Rearranging this inequality we obtain the claim.
\end{proof}

\begin{proposition}
In the setting of Proposition~\ref{prop:whitney} for every $c''$ such that $c' < c'' < 1$ there exists an $N \in \N$ such that
\begin{enumerate}
\setcounter{enumi}{3}
\renewcommand{\theenumi}{W\arabic{enumi}}
\renewcommand{\labelenumi}{\theenumi.}
\item the balls $B(x_j, c'' d_j)$ have the bounded intersection property, that is, each $x \in O$ is contained in at most $N$ balls.
\label{whitney:bounded-intersection}
\end{enumerate}
\end{proposition}
\begin{proof}
Let $x_0 \in O$ and $J_0 \subset J$ be the set of indices $j$ such that the corresponding balls $B(x_j, c'' d_j)$ contain $x_0$.
By Proposition~\ref{prop:whitney-comparable-radii} with $O^{(+)} = O$ we have that $d_j < \frac{1 + c''}{1 - c''} d_0$ for every $j \in J_0$.
Therefore $\dist(x_0, x_j) < c'' d_j < c'' \frac{1 + c''}{1 - c''} d_0$ and $B(x_j, c d_j) \subset B(x_0, (c + c'') \frac{1 + c''}{1 - c''} d_0)$.
The volume of the latter ball is
\[
V = \frac{\Omega_n}{n} \left((c + c'') \frac{1 + c''}{1 - c''} \right)^n d_0^n.
\]
On the other hand, the balls $B(x_j, c d_j)$ are disjoint.
Again by Proposition~\ref{prop:whitney-comparable-radii} we have that $d_0 < \frac{1 + c''}{1 - c''} d_j$.
Therefore the volume of each such ball is bounded from below by
\[
|B(x_j, c d_j)| = \frac{\Omega_n}{n} c^n d_j^n
\geq
\left(c \frac{1 - c''}{1 + c''} \right)^n d_0^n =: W.
\]
Since their total volume is bounded from below by $|J_0| W$ and from above by $V$, we obtain that
\[
|J_0|
\leq
V / W
=
\left(\frac{c + c''}{c} \left(\frac{1 + c''}{1 - c''}\right)^2 \right)^n.
\]
This gives (\ref{whitney:bounded-intersection}) with $N$ equal to the integer part of the latter number.
\end{proof}

\section{Calderón-Zygmund decomposition}
\label{sec:calderon-zygmund}
\index{Calderon-Zygmund decomposition@Calderón-Zygmund decomposition}
Throughout this section, $f \in L^{p_0}(\R^n)$.
Let $\{ Q_{m,v} \}_{m \in \Z, v \in \Z^n}$ be the mesh of dyadic cubes in $\R^n$, i.e.\ the set of cubes with edges parallel to the axes, side length $2^m$ and a vertex in $2^m v$.
For each $\alpha > 0$ let $m_\alpha$ be the smallest integer such that
\[
\fint_{Q_{m_\alpha,v}} |f|
\leq
||f||_{p_0} |Q_{m_\alpha, v}|^{-1/p_0}
=
||f||_{p_0} 2^{-m_\alpha n / p_0}
\leq
\alpha
\text{ for all } v.
\]
For each $\alpha$, the set of cubes $S_\alpha = \{Q_{m,v}\}_{m \leq m_\alpha}$ is partially ordered by inclusion.
Let $S'_\alpha \subset S_\alpha$ be the subset of cubes $Q'$ with the property
\[
\fint_{Q'} |f| > \alpha
\]
and $S''_\alpha \subset S'_\alpha$ the subset of cubes which are maximal with respect to inclusion.
Since each increasing chain in $S_\alpha$ is finite, every $Q' \in S'_\alpha$ is contained in some cube $Q'' \in S''_\alpha$.
Since the top-level cubes $Q_{m_\alpha, v} \not\in S'_\alpha$ by the definition of $m_\alpha$ and by maximality, every $Q_{m,v} \in S''_\alpha$ is contained in a cube $Q_{m+1,v'} \in S_\alpha \setminus S'_\alpha$.
Therefore
\[
\fint_{Q_{m,v}} |f| \leq 2^n \fint_{Q_{m+1,v'}} |f| \leq 2^n \alpha.
\]
Write $\{Q^\alpha_j\} = S''_\alpha$.
Since for every $x \not\in \cup_j Q^\alpha_j$ there exists a sequence of cubes converging to $x$ such that the average of $|f|$ over them is bounded by $\alpha$, the Lebesgue density theorem implies that $|f(x)| \leq \alpha$ for almost every such $x$.

Furthermore, if $\beta > \alpha$, then $m_\beta < m_\alpha$ and hence $S_\beta \subset S_\alpha$ as well as $S'_\beta \subset S'_\alpha$.
Therefore each $Q^\beta_j$ is a sub-cube of some $Q^\alpha_k$.


%% file: chapter_k_plane.tex
\chapter{The \texorpdfstring{$k$}{k}-plane transform}
\label{chap:k-plane}
Analogously to the Radon transform we consider the integrals of a function on $\R^n$ over affine subspaces of an arbitrary dimension $k$.
Most of the results in this chapter are due to Christ \cite{MR763948}.

The case $n=3$, $k=1$ furnishes an idealized model of tomography.
In this model, the waves we use to scan the object (e.g.\ a brain) are assumed to travel along straight lines without refraction.
Absorption and scattering are supposed to be the only measurable phenomena, and both are thought of as isotropic and linear (in the intensity of the wave).
With these assumptions, the relative decrease of the intensity of the wave is
\[
\log \frac{\text{intensity after}}{\text{intensity before}} = - \int_{\text{taken path}} f,
\]
where $f$ is some material-specific absorption coefficient, which may vary in space.
Therefore, sending a wave with known intensity through the object and measuring its intensity on the opposite side, we can measure line integrals of $f$.
If we could do so with arbitrary precision for every direction and use an arbitrary amount of computing time, we would be able to compute $f$ easily (under reasonable assumptions on its regularity).
The practical problems of doing so include the facts that you neither want to use more radiation then strictly necessary (because it might destroy the object) nor to wait arbitrarily long for the results.
However we will not be concerned with these problems regarding the inverse transform.

Before proceeding, let us establish the notation.
Let $G_{n,k}$ denote the Grassmann manifold of $k$-planes passing through the origin in $\R^n$ and $M_{n,k}$ the manifold of all affine $k$-planes in $\R^n$.
The spaces $G_{n,k}$ and $M_{n,k}$ are homogeneous for $O(n)$ and the group of Euclidean transformations, respectively, and are equipped with the unique (up to multiplication by a constant) measures $\nu_{n,k}$ and $\mu_{n,k}$, respectively, invariant under the actions of these groups.
$G_{n,k}(x)$ is a translated copy of $G_{n,k}$, i.e.\ the manifold of all $k$-planes passing through $x$.

We will use the mixed norms for functions on $M_{n,k}$ given by
\[
|| g ||_{q;r}
=
\left(
\int_{G_{n,k}} \left(
\int_{\pi^\perp}
|g(\pi + x)|^r
\dif \Leb[\pi^\perp](x)
\right)^{q/r}
\dif \nu_{n,k}(\pi)
\right)^{1/q}.
\]
Here, $\Leb[\pi^\perp]$ is the $n-k$-dimensional Lebesgue measure on the subspace $\pi^\perp \subset \R^n$.
Observe that if $q=r$, then $||g||_{q; r}$ is just the usual norm on $L^r(M_{n,k})$.

The space $M_{n,k}$ has a natural structure of a fiber bundle over $G_{n,k}$.
To simplify the notation, we will often use a measurable trivialization of this bundle and identify $M_{n,k} = G_{n,k} \times \R^{n-k}$ as measure spaces.

The $k$-plane transform of a test function on $\R^n$ is given by
\index{k-plane transform@$k$-plane transform}
\[
\Radon_{n,k} f (\pi) = \int_\pi f(x) \dif\Leb[k](x), \quad \pi \in M_{n,k}.
\]
Again, we will be interested in estimates of the form
\begin{equation}
\label{eq:k-plane-p-qr-estimate}
|| \Radon_{n,k} f ||_{q; r} \leq C || f ||_p.
\end{equation}
By Fubini's theorem, we have
\begin{equation}
\label{eq:k-plane-fubini-estimate}
|| \Radon_{n,k} f ||_{\infty; 1} \leq C || f ||_1.
\end{equation}
The remaining part of this chapter is devoted to obtaining a second end-point estimate of the form (\ref{eq:k-plane-p-qr-estimate}) with $q=r=n+1$.

\section{Measure equivalences}
In evaluating integrals over $M_{n,k}$, we will find it useful to transform them into integrals over $\R^n$ as in \cite[p.~497]{MR748958}.
Let $\det (x_0, \dots, x_k)$ denote the $k$-dimensional volume of the convex hull of $x_0, \dots, x_k$.
\begin{lemma}
\label{lem:measures-on-Rn-k+1}
The following measures on $(\R^n)^{k+1}$ are equivalent up to a constant:
\[
\dif\Leb[\pi](x_0) \dots \dif\Leb[\pi](x_k) \dif\mu_{n,k}(\pi)
=
C \det (x_0, \dots, x_k)^{k-n} \dif\Leb[n](x_0) \dots \dif\Leb[n](x_k).
\]
Here, $\Leb[\pi]$ denotes the $k$-dimensional Lebesgue measure on the affine hyperplane $\pi$.
\end{lemma}
\begin{proof}
Observe first that, disregarding sets of measure zero on both sides, there is a natural bijective correspondence between $\{(\pi \in M_{n,k}, x_0 \in \pi, \dots, x_k \in \pi)\}$, the $k+1$-th power of the tautological fiber bundle over $M_{n,k}$, and $\{(x_0,\dots,x_k)\} = \R^{n \times (k+1)}$.
With this identification, the measures are absolutely continuous with respect to each other, so that
\[
\dif\Leb[\pi](x_0) \dots \dif\Leb[\pi](x_k) \dif\mu_{n,k}(\pi)
= C
J_{n,k}(x_0, \dots, x_k)
\dif\Leb[n](x_0) \dots \dif\Leb[n](x_k).
\]
The problem is now to find $J_{n,k}$.
In case $k=0$ there is nothing to prove.

If $k=1$, observe that the measures on both sides are invariant under the action of Euclidean transformations.
Thus, $J_{n,1}$ must be a Euclidean invariant of the simplex $(x_0,x_1)$, so that it is a function of $|x_0 - x_1|$.
By homogeneity with respect to rescaling of the coordinates, one sees that in fact $J_{n,1} = C |x_0 - x_1|^{1-n}$, so that the lemma is proved for $k=0,1$.

For $k>1$ this approach does not work because there are multiple Euclidean invariants for a $k$-simplex.
So we use induction on $k$.

If $\nu_{n,k,x_0}$ denotes the measure on $G_{n,k}(x_0)$, then
\begin{equation}
\label{eq:measure-on-taut-fiber-bundle}
\dif\Leb[\pi](x_0) \dif\mu_{n,k}(\pi)
= C
\dif\nu_{n,k,x_0}(\pi) \dif x_0
\end{equation}
since both are measures on the tautological fiber bundle over $M_{n,k}$ invariant under Euclidean transformations.
Furthermore, if $\mu_{k,k-1,\pi}$ denotes the measure on the manifold of affine hyperplanes contained in $\pi$, then
\begin{equation}
\label{eq:measure-on-planes-not-through-pt}
\dist(x_0,\theta)^{(k-1)-k} \dif\mu_{k,k-1,\pi}(\theta) \dif\nu_{n,k,x_0}(\pi)
= C
\dist(x_0,\theta)^{(k-1)-n} \dif\mu_{n,k-1}(\theta),
\end{equation}
since both are measures on $\tilde M_{n,k-1}(x_0) = M_{n,k-1} \setminus \{ \text{planes through } x_0 \}$ invariant under rotations around $x_0$ and dilations with base point $x_0$, which together generate a group acting transitively on $\tilde M_{n,k-1}(x_0)$.

Now, using first the induction hypothesis with $k,k-1$ instead of $n,k$, we obtain
\begin{align*}
&\dif\Leb[\pi](x_0) \dots \dif\Leb[\pi](x_k) \dif\mu_{n,k}(\pi)\\
&= C
\dif\Leb[\pi](x_0) \det(x_1,\dots,x_k)^{k-(k-1)} \dif\Leb[\theta](x_1) \dots \dif\Leb[\theta](x_k) \dif\mu_{k,k-1,\pi}(\theta) \dif\mu_{n,k}(\pi)\\
&= C
\det(x_1,\dots,x_k) \dif\Leb[\theta](x_1) \dots \dif\Leb[\theta](x_k) \dif\mu_{k,k-1,\pi}(\theta) \dif\nu_{n,k,x_0}(\pi) \dif x_0
\quad\text{ by }(\ref{eq:measure-on-taut-fiber-bundle})
\\
&= C
\det(x_1,\dots,x_k) \dif\Leb[\theta](x_1) \dots \dif\Leb[\theta](x_k) \dist(x_0,\theta)^{k-n} \dif\mu_{n,k-1}(\theta) \dif x_0
\quad\text{ by }(\ref{eq:measure-on-planes-not-through-pt})
\\
&= C
\det(x_1,\dots,x_k)^{k-n} \dif x_1 \dots \dif x_k \dist(x_0,\pi(x_1, \dots, x_k))^{k-n} \dif x_0
\quad\text{ by induction}
\\
&= C
\det(x_0,x_1,\dots,x_k)^{k-n} \dif x_1 \dots \dif x_k \dif x_0.
&\qedhere
\end{align*}
\end{proof}

A similar result for measures on the projective space $G_{n,1}$ is easily deduced by a change of coordinates.
Let $\Det(\omega_1, \dots, \omega_k)$ for $\omega_1, \dots, \omega_k \in G_{n,1}$ denote $\det(0, x_1, \dots, x_k)$ whenever $x_j \in S^{n-1} \intersection \omega_j$.
Note that this quantity does not depend on the choice of $x_1, \dots, x_k$.

\begin{lemma}
\label{lem:measures-on-Gn1-k}
The following measures on $(G_{n,1})^{k}$ are equivalent up to a constant:
\[
\dif\nu_{k,1}(\omega_1) \dots \dif\nu_{k,1}(\omega_k) \dif\nu_{n,k}(\pi)
=
C \Det (\omega_1, \dots, \omega_k)^{k-n} \dif\nu_{n,1}(\omega_1) \dots \dif\nu_{n,1}(\omega_k).
\]
\end{lemma}
\begin{proof}
Using (\ref{eq:measure-on-taut-fiber-bundle}) and the measurable coordinates $(\omega, r) \in G_{n,1}(x_0) \times \R \cong G_{n,1} \times \R$ centered at $x_0$ on $\R^n$ the assertion of Lemma~\ref{lem:measures-on-Rn-k+1} becomes
\begin{align*}
\dif\nu_{k,1}(\omega_1) |r_1|^{k-1} \dif r_1 &\dots \dif\nu_{k,1}(\omega_k) |r_k|^{k-1} \dif r_k \dif\nu_{n,k, x_0}(\pi) \dif\Leb[n](x_0)\\
&=
C \dif\Leb[\pi](x_0) \dots \dif\Leb[\pi](x_k) \dif\mu_{n,k}(\pi)\\
&=
C \det (x_0, \dots, x_k)^{k-n} \dif\Leb[n](x_0) \dots \dif\Leb[n](x_k)\\
&=
C |r_1|^{k-n} \dots |r_k|^{k-n} \Det(\omega_1, \dots, \omega_k)^{k-n}\\
&\qquad \cdot \dif\nu_{n,1}(\omega_1) |r_1|^{n-1} \dif r_1 \dots \dif\nu_{n,1}(\omega_k) |r_k|^{n-1} \dif r_k \dif\Leb[n](x_0).
\end{align*}
Eliminating the integration over $x_0$ and $r_1, \dots r_k$ yields the claim.
\end{proof}

\section{Estimates by rearrangement}
We consider the multilinear functional
\[
A_{n,k}(f_0, \dots, f_n)
=
\int_{\pi \in M_{n,k}} \Prod_{j=0}^n \Radon_{n,k} f_j(\pi) \dif\mu_{n,k}(\pi)
\]
for functions $f_0, \dots, f_n$ on $\R^n$.
For a positive function $f$,
\[
A_{n,k}(f,\dots,f) = ||\Radon_{n,k} f||_{n+1; n+1}^{n+1}.
\]
The multilinear nature of $A_{n,k}$ allows one to apply rearrangement techniques.
Note that by Lemma~\ref{lem:measures-on-Rn-k+1} we have
\begin{align*}
A_{n,k}(f_0, \dots, f_n)
&=
\int_{\pi \in M_{n,k}} \Prod_{j=0}^n \int_\pi f_j(x_j) \dif\Leb[k](x_j) \dif\mu_{n,k}(\pi)\\
&=
\int_{x_0, \dots, x_k \in \R^n} f_0(x_0) \dots f_k(x_k)
\prod_{j=k+1}^n \left( \int_{x_j \in \pi(x_0, \dots, x_k)} f_j(x_j) \dif\Leb[k](x_j) \right)\\
&\qquad \cdot \det(x_0, \dots, x_k)^{k-n} \dif x_k \dots \dif x_0
\end{align*}
Let us show that the latter integral does not decrease under the action of the Steiner symmetrization on the functions $f_j$.
Note that the formula is invariant under the action of $O(n)$, so that we may choose the coordinates in such a way as to consider the symmetrization along the $e_n$ axis.
Write $x_j=(z_j,t_j) \in \R^{n-1} \times \R$.
Disregarding a set of measure zero, we may assume that $z_0, \dots, z_k$ are in general position, i.e.\ they span a $k$-dimensional affine subspace.
Henceforth let $z_0, \dots, z_k$ be fixed.
Using the remaining freedom in the choice of coordinates, we may assume that $z_0,\dots,z_k$ lie in the coordinate plane $\R^k \subset \R^{n-1}$.
This ensures that the $k$-plane $\pi(x_0, \dots, x_k)$ is not parallel to the basis vector $e_n$.
It is now sufficient to show that the integral
\[
\tilde A =
\int\limits_{\R^{k+1}} f_0(x_0) \dots f_k(x_k)
\prod_{j=k+1}^n \left( \int_{\mathrlap{\pi(x_0, \dots, x_k)}} f_j(x_j) \dif\Leb[k](x_j) \right)
\det(x_0, \dots, x_k)^{k-n} \dif t_k \dots \dif t_0.
\]
does not decrease under Steiner symmetrization.
This integral only depends on the restriction of $f_j$'s to the $k+1$-dimensional subspace $V=\R^k \times \{0\} \times \R \subset \R^n$, and the following manipulations are restricted to $V$.

In our coordinates, $\pi(x_0, \dots, x_k)$ is just a graph of a linear function, so that integration over it reduces to integration over $\R^k$:
\[
\int\limits_{\pi(x_0, \dots, x_k)} f_j(x_j) \dif\Leb[k](x)
=
\int\limits_{\R^k \subset \R^{n-1}} f_j(z_j,t_j(x_0, \dots, x_k, z_j)) \frac{\det (x_0, \dots, x_k)}{\det (z_0, \dots, z_k)} \dif z_j,
\]
where $t_j$ is a linear function of $t_0, \dots, t_k$ and the last fraction relates the volumes in the graph and in $\R^k$.
The quantity we are interested in becomes
\[
\tilde A
=
\int\limits_{\R^{k+1}} f_0(x_0) \dots f_k(x_k)
\prod_{j=k+1}^n \left(
\int\limits_{\R^k \subset \R^{n-1}} \frac{f_j(z_j,t_j(x_0, \dots, x_k, z_j))}{\det (z_0, \dots, z_k)} \dif z_j
\right)
\dif t_k \dots \dif t_0.
\]
Pulling out the integration over $z_j$'s, it is sufficient to show that
\[
\int\limits_{\R^{k+1}} f_0(x_0) \dots f_k(x_k)
\prod_{j=k+1}^n f_j(z_j,t_j(x_0, \dots, x_k, z_j))
\dif t_k \dots \dif t_0
\]
does not decrease under Steiner symmetrization.
Since the $t_j$'s are linear in $t_0, \dots, t_k$, this is given by Theorem~\ref{thm:steiner-sym-increases-integrals}.

We have proved that $A$ does not decrease under Steiner symmetrization of its arguments.
A much more useful result would be if it would hold not just for one-dimensional, but for $n$-dimensional radial non-increasing rearrangement of the arguments.
\begin{proposition}
\label{prop:k-plane-rearrangement}
Let $f_0, \dots, f_n$ be positive functions.
Then
\[
A(f_0, \dots, f_n) \leq A(f_0^{**}, \dots, f_n^{**}).
\]
Here, $f^{**}$ denotes the non-increasing radial rearrangement of $f$ in $n$ dimensions.
\end{proposition}
\begin{proof}
By the monotone convergence theorem and a decomposition similar to (\ref{eq:decomposition-characteristic}), it is sufficient to consider simple functions.
By multilinearity of $A$ the functions $f_j$ may be assumed to be characteristic functions of sets $E_j$.
Again by monotone convergence, the sets may be assumed to be bounded by some $R$.
By Proposition~\ref{prop:steiner-approx}, there is a sequence of directions such that the successive Steiner symmetrizations of $E_j$ converge to their radial rearrangements $B(0,r_j)$ in the Hausdorff distance, in particular $E_i^{*u_1 \dots *u_l} \subset B(0,r_j+o(1))$.
Together with the preceding discussion this shows that
\[
A(f_0, \dots, f_n)
\leq
A(f_0^{*u_1 \dots *u_l}, \dots, f_n^{*u_1 \dots *u_l})
\leq
A(\chi_{B(0,r_0+o(1))}, \dots, \chi_{B(0,r_n+o(1))}).
\]
Letting $l \to \infty$, one sees that the right-hand side of the inequality converges to $A(f_0^{**}, \dots, f_n^{**})$.
\end{proof}

The problem is now reduced to finding an estimate for the $k$-plane transform of a radial non-increasing function.
Our result is somewhat stronger then needed in that it gives an estimate in terms of a Lorentz space norm.
\begin{proposition}
\label{prop:k-plane-estimate-radial}
Let $f$ be a positive non-increasing radial function on $\R^n$.
Then
\[
A_{n,k}(f, \dots, f)
\leq C
|| f ||_{(n+1)/k+1, n+1}.
\]
\end{proposition}
\begin{proof}
For convenience, $f$ is seen both as function on $\R^n$ and on $(0,\infty)$.
Using radial symmetry of $f$ we can make the following simplifications.
\begin{align*}
A_{n,k}(f, \dots, f)
&= \int_{\pi \in G_{n,k}} \int_{x \perp \pi} |\Radon_{k,n} f(\pi, x)|^{n+1} \dif x \dif \pi\\
&= \int_{x \perp \pi} |\Radon_{k,n} f(\pi, x)|^{n+1} \dif x \quad\text{for an arbitrary } \pi\\
&= C \int_{r=0}^\infty r^{n-k-1} \left| \int_{\rho=0}^\infty \rho^{k-1} f(\sqrt{r^2+\rho^2}) \dif\rho \right|^{n+1} \dif r\\
&= C
\int_{r=0}^\infty r^{n-k-1} \left|
\int_{s=r}^\infty f(s) s (s^2-r^2)^{k/2-1} \dif s\right|^{n+1} \dif r.
\end{align*}
In the last line we have substituted $s=\sqrt{r^2+\rho^2}$.
If $k \geq 2$, we have the estimate $(s^2-r^2)^{k/2-1} \leq s^{k-2}$, so that
\begin{align*}
A_{n,k}(f, \dots, f)
&\leq C
\int_{r=0}^\infty r^{n-k-1} \left|
\int_{s=r}^\infty f(s) s^{k-1} \dif s\right|^{n+1} \dif r\\
&\leq C
\int_{r=0}^\infty r^{k n+n-1} f(r)^{n+1} \dif r
\end{align*}
by Hardy's inequality (\ref{ineq:hardy-positive}).
We now use the relation $f^*(\Omega_n |x|^n/n) = f^{**}(x)$ between the linear and the radial non-increasing rearrangement of $f$.
Here the argument of $f^*$ is the volume of the ball of radius $|x|$ in $\R^n$.
From the fact that $f$ is both radial and non-increasing we infer $f^*(s) = f(C s^{1/n})$.
The change of variable $r = C s^{1/n}$ in the last integral yields
\begin{align*}
A_{n,k}(f, \dots, f)
&\leq
C \int_{r=0}^\infty r^{k n+n} f(r)^{n+1} \frac{\dif r}{r} \\
&=
C \int_{s=0}^\infty s^{(k n+n)/n} f^*(s)^{n+1} \frac{\dif s}{s} \\
&=
C \int_{s=0}^\infty \left( s^{(k+1)/(n+1)} f^*(s) \right)^{n+1} \frac{\dif s}{s} \\
&=
C ||f||_{(n+1)/(k+1),n+1}^{n+1},
\end{align*}
which is the desired estimate.

In the remaining case $k=1$ the integral reduces to
\begin{multline*}
A_{n,k}(f, \dots, f)
\leq
C \int_0^\infty r^{n-2} \left( \int_r^\infty |f(s)| (s^2 - r^2)^{-1/2} s \dif s \right)^{n+1} \dif r\\
\leq
C \int_0^\infty r^{n-2} \left( \int_r^{2 r} \underbrace{|f(s)|}_{\leq f(r)} \underbrace{(s^2 - r^2)^{-1/2}}_{\leq (s-r)^{-1/2} (2r)^{-1/2}} \underbrace{s}_{\leq 2r} \dif s \right)^{n+1} \dif r \\
+
C \int_0^\infty r^{n-2} \left( \int_{2r}^\infty |f(s)| \underbrace{(s^2 - r^2)^{-1/2}}_{\leq (s^2 / 4)^{-1/2}} s \dif s \right)^{n+1} \dif r.
\end{multline*}
Here the last term may be estimated as before, while the first is dominated by
\begin{multline*}
C \int_0^\infty r^{n-2} \left( f(r) (2r)^{-1/2} 2r \int_r^{2 r} (s-r)^{-1/2} \dif s \right)^{n+1} \dif r\\
=
C \int_0^\infty r^{n-2} \left( f(r) r \right)^{n+1} \dif r
=
C \int_0^\infty r^{2 n-1} f(r)^{n+1} \dif r,
\end{multline*}
which again may be estimated as before.
\end{proof}

\begin{theorem}
\label{th:k-plane}
Let $f$ be a test function on $\R^n$.
Then, for every $1 \leq p \leq (n+1)/(k+1)$,
\[
|| \Radon_{n,k} f ||_{q; r}
\leq
C || f ||_{p},
\]
where
$q \leq (n-k) p'$, $n/p - (n-k)/r = k$.
\end{theorem}
\begin{proof}
For $q=(n-k)p'$ we have the endpoint estimate $p=(n+1)/(k+1)$ by Proposition~\ref{prop:k-plane-rearrangement}, Proposition~\ref{prop:k-plane-estimate-radial} and Corollary~\ref{cor:lorentz-scale}, while the case $p=1$ is covered by Fubini's theorem as mentioned in (\ref{eq:k-plane-fubini-estimate}).
The result follows by the interpolation Theorem~\ref{th:stein}.

The estimate for $q<(n-k)p'$ follows by the Hölder inequality.
\end{proof}

\section{Estimates by induction on \texorpdfstring{$k$}{k}}
In order to sharpen the results of the previous section we have to consider another integral operator.
For a function $F$ defined on $G_{n,1}$ set
\[
S_{n,k} F (\pi) = \int_{\omega \subset \pi} F(\omega) \dif\nu_{k,1}(\omega),
\quad
\pi \in G_{n,k}.
\]
Analogously to what was done for the Radon transform, we will consider the multilinear form
\[
B_{n,k}(F_1, \dots, F_n)
=
\int_{\pi \in G_{n,k}} \Prod_{j=1}^n S_{n,k} F_j (\pi) \dif\nu_{n,k}(\pi).
\]
By Lemma~\ref{lem:measures-on-Gn1-k} we have
\begin{align*}
B_{n,k}(F_1, \dots, F_n)
&=
\int_{\pi \in G_{n,k}} \Prod_{j=1}^n \left( \int_{\omega \subset \pi} F_j(\omega) \dif\nu_{k,1}(\omega)\right) \dif\nu_{n,k}(\pi)\\
&=
\int_{\omega_1, \dots, \omega_k \in G_{n,1}} \Prod_{j=1}^k F_j(\omega_j) \Prod_{j=k+1}^n \left( \int_{\omega_j \subset \pi(\omega_1, \dots, \omega_k)} F_j(\omega_j) \dif\nu_{k,1}(\omega_j)\right)\\
&\qquad \cdot \Det(\omega_1, \dots, \omega_k)^{k-n}
\dif\nu_{n,1}(\omega_1) \dots \dif\nu_{n,1}(\omega_k).
\end{align*}
For $F\geq 0$, $B_{n,k}(F,\dots, F) = ||S_{n,k} f||_n^n$, and the next theorem shows that $S_{n,k} : L^{n/k,n}(G_{n,1}) \to L^n(G_{n,k})$ is a bounded operator.

\begin{theorem}
\label{th:ind}
For every $1 \leq k \leq n$, we have that
\begin{enumerate}
\item\label{ind-sphere} $|B_{n,k}(F_1,\dots,F_n)| \leq C \prod_{j=1}^n ||F_j||_{n/k,n}$ for arbitrary $F_j \in L^{n/k,n}(G_{n,1})$ and
\item\label{ind-radon} $|A_{n,k}(f_0,\dots,f_n)| \leq C \prod_{j=0}^n ||f_j||_{\frac{n+1}{k+1},n+1}$ for arbitrary $f_j \in L^{\frac{n+1}{k+1},n+1}(\R^n)$.
\end{enumerate}
\end{theorem}
\begin{proof}
We use induction on $k$.
The basis is (\ref{ind-sphere}, $k=1$), which is trivial since $S_{n,1}$ is the identity operator.
The induction steps are (\ref{ind-sphere}, $k$) $\implies$ (\ref{ind-radon}, $k$) $\implies$ (\ref{ind-sphere}, $k+1$).

\paragraph*{(\ref{ind-sphere}, $k$) $\implies$ (\ref{ind-radon}, $k$)}
By the symmetry and positivity of $A$ and real interpolation (Proposition~\ref{prop:interpolation-multilinear}) it is sufficient to show the stronger estimate
\[
A_{n,k}(f_0, \dots, f_n) \leq C ||f_0||_1 \Prod_{j=1}^n ||f_j||_{n/k, 1}.
\]
for positive functions $f_0,\dots,f_n$.
Since
\[
A_{n,k}(f_0, \dots, f_n) = \int_{x_0 \in \R^n} f_0(x_0) K(x_0) \dif x_0
\]
with
\begin{multline*}
K(x_0)
=
\int_{x_1, \dots, x_k \in \R^n} f_1(x_1) \dots f_k(x_k)
\prod_{j=k+1}^n \left( \int_{x_j \in \pi(x_0, \dots, x_k)} f_j(x_j) \dif\Leb[k](x_j) \right)\\
\cdot\det(x_0, \dots, x_k)^{k-n} \dif x_1 \dots \dif x_1
\end{multline*}
and by duality it is sufficient to show $||K||_\infty \leq C \prod_{j=1}^n ||f_j||_{n/k, 1}$.
By translation invariance we only need to consider $x_0=0$.
By definition of $\Det$,
\[
\det(0, x_1, \dots, x_k) = \Det(\omega_1, \dots, \omega_k) \Prod_{j=1}^k |x_j|,
\]
where $\omega_j \in G_{n,1}$ is the line passing through $x_j$.
With this notation we also have
\[
\pi(0, x_1, \dots, x_k) = \pi(\omega_1, \dots, \omega_k).
\]
Hence the change to radial coordinates, induction hypothesis (\ref{ind-sphere}, $k$) and Proposition~\ref{prop:Rn-Lp1-sphere-Lp} yield
\begin{align*}
K(0)
&=
C \int_{\omega_1, \dots, \omega_k \in G_{n,1}} \Prod_{j=1}^k \left(\int_{t_j \in \R} |t_j|^{k-n} f_j(t_j \omega_j) |t_j|^{n-1} \dif t_j \right)\\
&\qquad \prod_{j=k+1}^n \left( \int_{\omega_j \subset \pi(\omega_1, \dots, \omega_k)} \int_{t_j \in \R} f_j(t_j \omega_j) |t_j|^{k-1} \dif t_j \dif\nu_{k,1}(\omega_j) \right)\\
&\qquad \cdot \Det(\omega_1, \dots, \omega_k)^{k-n} \dif \omega_1 \dots \dif \omega_k\\
&=
C \int_{\omega_1, \dots, \omega_k \in G_{n,1}} \Prod_{j=1}^k \left( U_{n/k} f_j (\omega_j) \right)\\
&\qquad \prod_{j=k+1}^n \left( \int_{\omega_j \subset \pi(\omega_1, \dots, \omega_k)} U_{n/k} f_j (\omega_j)\dif\nu_{k,1}(\omega_j) \right)\\
&\qquad \cdot \Det(\omega_1, \dots, \omega_k)^{k-n} \dif \omega_1 \dots \dif \omega_k\\
&=
C B_{n,k}(U_{n/k} f_1, \dots, U_{n/k} f_n)
\leq C \Prod_{j=1}^n ||U_{n/k} f_j||_{n/k, n}^n
\leq C \Prod_{j=1}^n ||f_j||_{n/k, 1}^n.
\end{align*}

\paragraph*{(\ref{ind-radon}, $k$) $\implies$ (\ref{ind-sphere}, $k+1$)}
By multilinearity of $B$ and compactness of $G_{n,1}$ it is sufficient to obtain an estimate in the case that each $F_j$ is supported on a small ball in $G_{n,1} = S^{n-1}/(\pm\id)$ (this manifold inherits a metric structure from $S^{n-1}$).
The size of the balls is chosen as to ensure that the aperture of the dashed cone in Figure~\ref{fig:ind-coor} is bounded from above, so that we may apply centrographic projection to change to flat coordinates.
To see that such a choice of the radii is possible observe that every (generic) collection of directions $\{ \omega_j \in \supp F_j \}_{j=1, \dots, k+1}$ divides $S^k = S^n \intersection \pi(\omega_1, \dots, \omega_{k+1})$ into $2^{k+1}$ sectors, at least one of which is contained in the positive sector of some coordinate system given by an orthonormal basis, as is easily seen by induction.
Since such a positive sector is contained in a cone of aperture $2 \arccos(1/\sqrt{k+1}) < \pi$, the sets $\supp F_j$ may be represented by subsets of the sphere contained in a cone with slightly greater aperture.

\begin{figure}
\centering
\begin{tikzpicture}[scale=1.9]
\coordinate (C) at (0,0); 
\coordinate (PL) at (-3,-1); 
\coordinate (PR) at (3,-1); 
\coordinate (O1U) at (230:-0.5); 
\coordinate (O1L) at (230:1.5); 
\coordinate (O2U) at (300:-0.5); 
\coordinate (O2L) at (300:1.5); 

\fill (0,0) circle (1pt) node[anchor=north] {$0$};
\draw (180:1) arc (180:360:1) node[right] {$\tilde S$}; 
\draw (PL) -- (PR); 

\draw[thick,|-|] (220:1) node[anchor=north east] {$\supp F_1$} arc (220:240:1);
\draw (O1L) -- (O1U) node[right] {$\omega_1$};
\coordinate (I1) at (intersection of PL--PR and O1L--O1U);
\fill (I1) circle (1pt) node[anchor=north west] {$x_1$};

\draw[thick,|-|] (290:1) arc (290:310:1) node[anchor=west] {$\supp F_2$};
\draw (O2L) -- (O2U) node[left] {$\omega_2$};
\coordinate (I2) at (intersection of PL--PR and O2L--O2U);
\fill (I2) circle (1pt) node[anchor=north east] {$x_2$};

\draw[dashed] ([xshift=0.5cm] PL) -- (0,0) -- ([xshift=-0.5cm] PR);
\end{tikzpicture}
\caption{Coordinates for the induction step (\ref{ind-radon}, $k$) $\implies$ (\ref{ind-sphere}, $k+1$)}
\label{fig:ind-coor}
\end{figure}
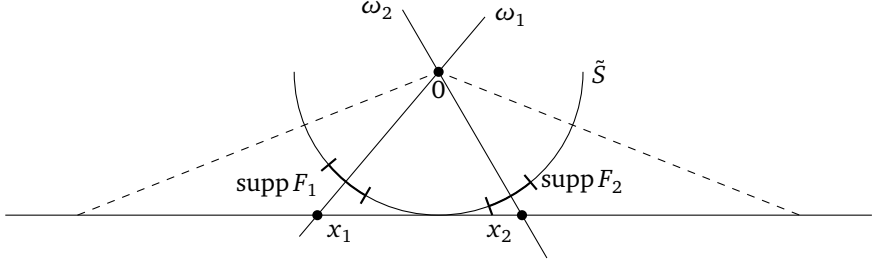

The Jacobian of the change of coordinates in Figure~\ref{fig:ind-coor} is bounded.
In particular, every Lorentz norm of the function $f_j(x_j) = F(\omega_j)$ is comparable with the corresponding norm of $F$.
Furthermore we have
\[
\Det(\omega_1, \dots, \omega_{k+1}) \leq C \det(x_1, \dots, x_{k+1}).
\]
Since $\omega_j \subset \pi(\omega_1, \dots, \omega_{k+1})$ if and only if $x_j \in \pi(x_1, \dots, x_{k+1})$ we conclude that
\begin{align*}
B_{n,k+1}(F_1, \dots, F_n)
&\leq
C \int_{x_1, \dots, x_{k+1} \in \R^{n-1}} \Prod_{j=1}^{k+1} f_j(x_j) \Prod_{j=k+2}^n \left( \int_{x_j \in \pi(x_1, \dots, x_{k+1})} f_j(x_j) \dif x_j \right)\\
&\quad \det(x_1, \dots, x_{k+1})^{k+1-n}
\dif x_1 \dots \dif x_{k+1}\\
&= A_{n-1,k}(f_1, \dots, f_n)\\
&\leq \Prod_{j=1}^n ||f_j||_{n/(k+1), n}^n\\
&\leq \Prod_{j=1}^n ||F_j||_{n/(k+1), n}^n.
\qedhere
\end{align*}
\end{proof}

\section{Estimates using the Hardy space}
Once again we will use the Plancherel theorem to obtain an end-point estimate.
At the other end-point we will however use a genuine Hardy space estimate instead of the usual $L^1$ estimate provided by Fubini's theorem.
\begin{theorem}[\cite{MR782573}]
\label{th:radon-riesz-bound}
Let $1 < p \leq 2$.
Then
\[
|| (I^{k/p'})\inv \Radon_{n,k} f ||_{p'; p} \leq C || f ||_p,
\]
where the Riesz potential acts on $\pi^\perp$ for every $\pi \in G_{n,k}$.
\end{theorem}
\begin{proof}
Consider the family of operators
\[
T_z f
=
(I^{z k/2})\inv \Radon_{n,k} f,
\quad z \in \bar S
\]
defined on the space of Schwartz functions $f$ with mean zero (which is dense in $H^{1} \cap L^{p}$).

The function $T_{z} f$ is holomorphic with values in  $L^{\infty}(G_{n,k},L^{2}(\R^{n-k}))$ for every $f$, because $\Fourier(T_{z}f)(\pi,\xi)=|\xi|^{zk/2} \Fourier(\Radon_{n,k} f)(\pi,\xi)$ and $\Fourier(\Radon_{n,k} f)(\pi,\cdot)$ is bounded in $\Schwartz(\R^{n-k})$ uniformly in $\pi$.
By Theorem~\ref{thm:hol-lcvs} the family $T_{z}$ is locally analytic.
Note also that
\[
L^{\infty}(G_{n,k},L^{2}(\R^{n-k})) \subset L^{\infty}(G_{n,k},L^{1}(\R^{n-k})) + L^{2}(G_{n,k},L^{2}(\R^{n-k}))
\]
with continuous inclusion.

At the end point $\Re z=1$ we have by the Plancherel theorem and by the uniqueness of the rotation-invariant probability measure on $S^{n-1}$ that
\begin{align*}
|| (I^{z k/2})\inv \Radon_{n,k} f ||_{2; 2}
&=
\left( \int_{G_{n,k}} \int_{\R^{n-k}}
| (I^{z k/2})\inv \Radon_{n,k} f(\pi, y) |^2
\dif y \dif\pi \right)^{1/2}\\
&=
\left( \int_{G_{n,k}} \int_{\xi \perp \pi}
|\xi|^k |\Fourier f(\xi) |^2
\dif \xi \dif\pi \right)^{1/2}\\
&=
C \left( \int_{G_{n,k}} \int_{r=0}^\infty r^{(n-k)-1} \int_{\pi^\perp \cap S^{n-1}}
r^k |\Fourier f(r \sigma) |^2
\dif \sigma \dif r \dif\pi \right)^{1/2}\\
&=
C \left( \int_{r=0}^\infty r^{n-1} \int_{G_{n,k}} \int_{\pi^\perp \cap S^{n-1}}
|\Fourier f(r \sigma) |^2
\dif \sigma \dif\pi \dif r \right)^{1/2}\\
&=
C \left( \int_{r=0}^\infty r^{n-1} \int_{S^{n-1}}
|\Fourier f(r \sigma) |^2
\dif \sigma \dif r \right)^{1/2}\\
&= C || \Fourier f ||_2\\
&= C || f ||_2.
\end{align*}
\begin{figure}
\centering
\begin{tikzpicture}
\draw (-1.5,0) -- (1.5,0) node[below] {$\pi^\perp$};
\draw (0,-1.5) -- (0,1.5) node[left] {$\pi$};

\begin{scope}
\clip[draw] (0,0) circle (1);
\draw[thick] (-1,0) -- (1,0);
\draw (0.3,-1) -- (0.3,1);
\end{scope}

\draw[->] (-2,-1) node[left] {$\supp T a$} -- (-0.5,0);
\draw[->] (-2,1) node[left] {$\supp a$} -- (135:1);
\draw[->] (2,1) node[right] {$A$} -- (0.3,0.5);
\end{tikzpicture}
\caption{Restriction of the $k$-plane transform to a fiber of $M_{n,k}$ applied to an $H^1$ atom}
\label{fig:k-plane-restriction}
\end{figure}
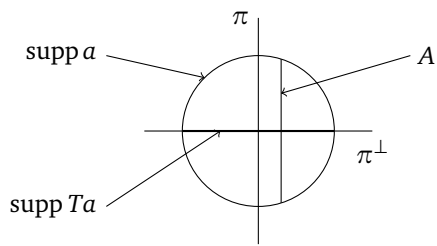
At the other end point, $\Re z = 0$, we are going to obtain a Hardy space estimate.
Let $a$ be an $H^1$ atom on $\R^n$ supported without loss of generality in $B(0,R)$ and bounded by $C R^{-n}$.
For every $\pi \in G_{n,k}$, $Ta := \Radon_{n,k} a(\pi, \cdot)$ has support in $B(0,R)$ and
\[
|T a| \leq \sup |a| \sup A \leq C R^{-n} C R^k = C R^{-(n-k)} \text{ a.e.}
\]
by Fubini's theorem, see Figure~\ref{fig:k-plane-restriction}.
Thus $Ta$ is a bounded multiple of an $H^1$ atom.
Furthermore, if $\{ \lambda_j \}$ is a summable sequence, then $\sum_j \lambda_j T a_j \to T(\sum_j \lambda_j a_j)$ in $L^1$ by Fubini's theorem.
Since the bound on $Ta$ is uniform in $\pi$, we have that
\[
\Radon_{n,k} : H^1(\R^n) \to L^\infty(G_{n,k}, H^1(\R^{n-k}))
\]
is a bounded operator.
Now let $a$ be an atom of $H^1(\R^{m})$, once again without loss of generality supported on $B(0,R)$.
We will show that the Riesz potential $(I^{z k/2})\inv a = I^{-z k/2} a$ has bounded $L^1$ norm.
Write for simplicity $-z k/2 = i \gamma$, $\gamma \in \R$ and split the $L^1$ norm as
\begin{align*}
\int_{\R^{m}} |I^{i \gamma} a(x)| \dif x
&=
\int_{|x| \leq 2 R} |I^{i \gamma} a(x)| \dif x
+
\int_{|x| > 2 R} |I^{i \gamma} a(x)| \dif x.
\end{align*}
The first integral may be estimated using the Hölder inequality and fact that $I^{i \gamma}$ is a contraction on $L^2$ in the following way
\begin{multline*}
\int_{|x| \leq 2 R} |I^{i \gamma} a(x)| \dif x
\leq
|B(0, 2R)|^{1/2} || I^{i \gamma} a ||_{L^2(B(0,2R))}\\
\leq
C R^{m/2} || I^{i \gamma} a ||_{L^2(\R^m)}
\leq
C R^{m/2} || a ||_{2}
\leq
C.
\end{multline*}
The estimate for the second integral requires the cancellation property of the atom $a$.
If $|x| > 2 R$, we have
\begin{align*}
|I^{i \gamma} a(x)|
&=
2^{-i \gamma} \pi^{-n/2} \frac{\Gamma(n/2 - i \gamma/2)}{\Gamma(i \gamma/2)}
\left| \int a(x-y) |y|^{-n+i\gamma} \dif y \right|\\
&=
C_\gamma \left| \int_{B(x,R)} a(x-y) \left( |y|^{-n+i\gamma} - |x|^{-n+i\gamma} \right) \dif y \right|\\
&\leq
C_\gamma \int_{B(x,R)} |a(x-y)| R \sup_{y' \in B(x,R)} \left| \nabla |y'|^{-n+i\gamma} \right| \dif y\\
&=
C_\gamma R ||a||_1 \sup_{y \in B(x,R)} \left| \nabla |y|^{-n+i\gamma} \right|\\
&\leq
C_\gamma R |x|^{-n-1}.
\end{align*}
Inserting this estimate into the integral, we obtain
\[
\int_{|x| > 2 R} |I^{i \gamma} a(x)| \dif x
\leq
C_\gamma R \int_{|x| > 2 R} |x|^{-n-1} \dif x
=
C_\gamma,
\]
where $C_\gamma$ has admissible growth in $\gamma$.
Since $I^{i \gamma}$ is bounded on atoms and is a continuous operator on $L^2$, Theorem~\ref{th:H1-Lp-cont} implies that it is continuous on $H^1$.
Summarizing, we have shown that
\begin{align*}
|| (I^{z k/2})\inv \Radon_{n,k} f ||_{\infty; 1}
&\leq
C_\gamma || f ||_{H^1}
& (\Re z = 0),\\
|| (I^{z k/2})\inv \Radon_{n,k} f ||_{2; 2}
&\leq
C || f ||_{2}
& (\Re z = 1).
\end{align*}
The interpolation theorem~\ref{th:stein} together with proposition~\ref{prop:h1-lp-interpolation-by-schwartz-functions} and corollary~\ref{cor:complex-interpolation-LP} concludes the proof.
\end{proof}

To extract an $L^p$ bound for $\Radon_{n,k}$ from this result we now need an estimate of the form
\[
|| I^{k/p'} F ||_{q; r} \leq || F ||_{p'; p}
\]
for functions $F$ on $M_{n,k}$.
We will use a special property of $F = (I^{k/p'})\inv \Radon_{n,k} f$, namely
\begin{equation}
\label{eq:radon-f-fourier-compatibility-condition}
\Fourier F(\pi, \xi) = \Fourier F(\theta, \xi), \text{ whenever } \pi\perp\xi, \theta\perp\xi,
\end{equation}
which is clear since both sides equal $|\xi|^{-\gamma} \Fourier f(\xi)$.
We will consider a dense subset of smooth functions in $L^{p'}(L^p)$.
\begin{proposition}
Let $F \in C^\infty(M_{n,k})$ have rapid decay in $\xi$ uniformly in $\pi$, satisfy (\ref{eq:radon-f-fourier-compatibility-condition}), and be such that $\Fourier F(\pi, \xi) = 0$ for all $\xi$ in a neighborhood of $0$ and all $\pi$.
Let also $\gamma < n-k$.
Then
\[
I^\gamma F(\pi, y)
=
C \int_{\omega \perp \pi} \int_{\mathrlap{\alpha \perp \omega}} I^{\gamma+1-(n-k)} \Radon_{n-k,n-k-1} F(\alpha, \omega^\perp, \<y,\omega\>) \dif\nu_{n-1,k}(\alpha) \dif\nu_{n-k,1}(\omega).
\]
Here $\Radon$ and $\omega^\perp$ are taken with respect to the complementary subspace of $\alpha$, we use the measurable coordinates $(\omega, t) \in G_{n-k,1} \times \R \cong G_{n-k,n-k-1} \times \R$ on $\alpha^\perp$ and identify each $\omega$ with a point in $S^{n-1} \intersection \omega$.
\end{proposition}
\begin{proof}
The conditions on $\gamma$ and $F$ imply $\Fourier I^\tau F=|\cdot|^{-\tau} \Fourier F$ when it is needed in the following computation.
\begin{align*}
I^\gamma &F(\pi, y)\\
&=
\int_{\xi \perp \pi} e^{i \<y,\xi\>} \Fourier I^\gamma F(\theta, \xi) \dif\Leb[n-k](\xi)\\
&=
C \int_{\omega \perp \pi} \int_{t \in\R} e^{i \<y,t \omega\>} |t|^{-\gamma} \Fourier F(\theta, t\omega) |t|^{n-k-1} \dif t \dif\nu_{n-k,1}(\omega)\\
&=
C \int_{\omega \perp \pi} \int_{t \in\R} e^{i \<y,t \omega\>} |t|^{n-k-1-\gamma} \int_{\alpha \perp \omega} \Fourier F(\alpha, t\omega) \dif\nu_{n-1,k}(\alpha) \dif t \dif\nu_{n-k,1}(\omega)\\
&=
C \int_{\omega \perp \pi} \int_{\alpha \perp \omega} \int_{t \in\R} e^{i \<y,t \omega\>} |t|^{n-k-1-\gamma} \Fourier F(\alpha, t\omega) \dif t \dif\nu_{n-1,k}(\alpha) \dif\nu_{n-k,1}(\omega)\\
&=
C \int_{\omega \perp \pi} \int_{\alpha \perp \omega} \int_{t \in\R} e^{i \<y,t \omega\>} \Fourier I^{\gamma+1-(n-k)} F(\alpha, t\omega) \dif t \dif\nu_{n-1,k}(\alpha) \dif\nu_{n-k,1}(\omega)\\
&=
C \int_{\omega \perp \pi} \int_{\mathrlap{\alpha \perp \omega}} I^{\gamma+1-(n-k)} \Radon_{n-k,n-k-1} F(\alpha, \omega^\perp, \<y,\omega\>) \dif\nu_{n-1,k}(\alpha) \dif\nu_{n-k,1}(\omega).
\qedhere
\end{align*}
\end{proof}
Now we estimate the operators this one is composed of.
\begin{proposition}
\label{prop:IR-bound}
Let $1 < p \leq 2$, $F \in L^{p'}(L^p)(M_{n,k})$, $n-1 < kp < n$ and
\[
h(\omega,t) = \int_{\alpha \perp \omega} I^{k/p'+1-(n-k)} \Radon_{n-k,n-k-1} F(\alpha, \omega^\perp, t) \dif\nu_{n-1,k}(\alpha),
\quad
\omega\in G_{n,1}, t\in\R.
\]
Then $||h||_{p';Q} \leq C ||F||_{p';p}$, where $Q\inv = n/p - k$.
\end{proposition}
\begin{proof}
By Theorem~\ref{th:radon-riesz-bound} we have
\[
|| (I^{(n-k-1)/p'})\inv \Radon_{n-k,n-k-1} F(\alpha,\cdot) ||_{p'; p} \leq C || F(\alpha,\cdot) ||_p.
\]
By the weak-type Young inequality $I^{k/p'+1-(n-k)+(n-k-1)/p'}=I^{k-(n-1)/p} : L^p \to L^Q$ with $1+Q\inv = p\inv - (k-(n-1)/p-1)$ since the condition on $kp$ implies $0 < Q\inv < 1$ and $0 < -(k-(n-1)/p-1) < 1$.
Hence
\[
|| F_\alpha ||_{p'; Q} \leq C || F(\alpha,\cdot) ||_p,
\quad
F_\alpha(\cdot) = I^{k/p'+1-(n-k)} \Radon_{n-k,n-k-1} F(\alpha,\cdot).
\]
By convexity of the $L^Q$ norm,
\begin{align*}
||h||_{p',Q}
&=
\left( \int \left( \int \left| \int_{\alpha \perp \omega} F_\alpha(\omega^\perp, t) \dif\nu_{n-1,k}(\alpha) \right|^Q \dif t \right)^{p'/Q} \dif\nu_{n,1}(\omega) \right)^{1/p'}\\
&\leq
\left( \int \left( \int_{\alpha \perp \omega} \left( \int |F_\alpha(\omega^\perp, t)|^Q \dif t \right)^{1/Q} \dif\nu_{n-1,k}(\alpha) \right)^{p'} \dif\nu_{n,1}(\omega) \right)^{1/p'}\\
&\leq
\left( \int \int_{\alpha \perp \omega} \left( \int |F_\alpha(\omega^\perp, t)|^Q \dif t \right)^{p'/Q} \dif\nu_{n-1,k}(\alpha) \dif\nu_{n,1}(\omega) \right)^{1/p'}\\
&=
C \left( \int_{\alpha} \int_{\omega\perp\alpha} \left( \int |F_\alpha(\omega^\perp, t)|^Q \dif t \right)^{p'/Q} \dif\nu_{n-1,1}(\omega) \dif\nu_{n,k}(\alpha) \right)^{1/p'}\\
&=
C \left( \int_{\alpha} ||F_\alpha||_{p';Q}^{p'} \dif\nu_{n,k}(\alpha) \right)^{1/p'}\\
&\leq
C \left( \int_{\alpha} ||F(\alpha,\cdot)||_{p}^{p'} \dif\nu_{n,k}(\alpha) \right)^{1/p'}\\
&=
C ||F||_{p'; p}.
\qedhere
\end{align*}
\end{proof}

\begin{proposition}
\label{prop:U-Lqr-Lp-bound}
Let $X$ be a Banach space, $h \in L^{q'}(\pi \in G_{n,k}, L^r(\omega \in G_{n-k,1, \pi^\perp}))$, where $n/p - (n-k)/r=k$ and $q \leq (n-k) p'$ and define
\[
U h(\omega) = \int_{\pi\perp\omega} h(\pi, \omega) \dif\nu_{n-1,k}(\pi),
\quad \omega\in G_{n,1}.
\]
Then $||U h||_{p; X} \leq C ||h||_{q'; r; X}$.
\end{proposition}
\begin{proof}
Let $\phi \in L^{p'}(G_{n,1})$ be a test function.
Then
\begin{align*}
\<||Uh||_X, \phi\>
&=
\int_\omega \phi(\omega) \left|\left| \int_{\pi\perp\omega} h(\pi, \omega) \dif\nu_{n-1,k}(\pi) \right|\right|_{X} \dif\nu_{n,1}(\omega)\\
&\leq
C \int_\pi \int_{\omega\perp\pi} \phi(\omega) || h(\pi, \omega) ||_X \dif\nu_{n-k,1}(\omega) \dif\nu_{n,k}(\pi)\\
&\leq
C \int_\pi
\left( \int_{\omega\perp\pi} |\phi(\omega)|^{r'} \dif\nu_{n-k,1}(\omega) \right)^{1/r'}\\
&\quad\cdot \left( \int_{\omega\perp\pi} ||h(\pi, \omega)||_X^r \dif\nu_{n-k,1}(\omega) \right)^{1/r} \dif\nu_{n,k}(\pi).
\end{align*}
The term in the former pair of parentheses equals
\[
\left( S_{n,n-k} |\phi|^{r'} (\pi^\perp) \right)^{1/r'}.
\]
The assumptions on the exponents, Theorem~\ref{th:ind} and (\ref{ind-sphere}) imply that $S_{n,n-k} : L^{n/(n-k)}(G_{n,1}) = L^{p'/r'}(G_{n,1}) \to L^{n}(G_{n,n-k}) \hookrightarrow L^{q}(G_{n,n-k}) \cong L^{q}(G_{n,k})$ is bounded, so that
\begin{align*}
\<||Uh||_X, \phi\>
&\leq
C \int_{\mathrlap{\pi}}
\left( S_{n,n-k} |\phi|^{r'} (\pi^\perp) \right)^{1/r'}
\left( \int_{\mathrlap{\omega\perp\pi}} ||h(\pi, \omega)||_X^r \dif\nu_{n-k,1}(\omega) \right)^{1/r} \dif\nu_{n,k}(\pi)\\
&\leq
C || \left( S_{n,n-k} |\phi|^{r'} \right)^{1/r'} ||_q ||h||_{q'; r; X}\\
&\leq
C || \phi ||_{p'} ||h||_{q'; r; X}.
\qedhere
\end{align*}
\end{proof}

\begin{proposition}
Let $h \in L^{p'}(G_{n,1}, L^Q(\R))$ and define
\[
Th(\pi, y) = \int_{\omega\perp\pi} h(\omega, \<\omega, y\>) \dif\nu_{n-k,1}(\omega).
\]
Under the same assumptions on the exponents as in Propositions~\ref{prop:IR-bound} and~\ref{prop:U-Lqr-Lp-bound}, we have $||Th||_{q;r} \leq C ||h||_{p'; Q}$.
\end{proposition}
\begin{proof}
Let $g \in L^{q'}(L^{r'})$ and calculate
\begin{align*}
\<Th, g\>
&=
\int_\pi \int_{y\perp\pi} g(\pi, y) \int_{\omega\perp\pi} h(\omega,\<\omega,y\>) \dif\nu_{n-k,1}(\omega) \dif\Leb[n-k](y) \dif\nu_{n,k}(\pi)\\
&=
C \int_{\omega} \int_{\pi\perp\omega} \int_{t\in\R} \int_{\mathrlap{y\perp\pi, \<y,\omega\>=t}} g(\pi, y) h(\omega,t) \dif\Leb[n-k](y) \dif t \dif\nu_{n-1,k}(\pi) \dif\nu_{n,1}(\omega).
\end{align*}
It is sufficient to give a bound for
\begin{align*}
T^* g(\omega, t)
&=
C \int_{\pi\perp\omega} \int_{y\perp\pi, \<y,\omega\>=t} g(\pi, y) \dif\Leb[n-k](y) \dif\nu_{n-1,k}(\pi)\\
&=
C (U \Radon_{n-k,n-k-1} g)(\omega,t).
\end{align*}
Since $p>(n-1)/k \geq n/(k+1)$ implies $r' < (n-k-1)/(n-k)$ we have $\Radon_{n-k,n-k-1} : L^{r'} \to L^r(L^{Q'})$ by Theorem~\ref{th:radon-lp-lqr-estimate}, while $U : L^{q'}(L^r(L^{Q'})) \to L^p(L^{Q'})$ by Proposition~\ref{prop:U-Lqr-Lp-bound}.
The assertion follows by duality.
\end{proof}
Applying in order the preceding propositions we obtain that
\[
|| \Radon_{n,k} f ||_{q; r}
\leq
C || f ||_{p},
\]
under the assumptions, in order of appearance,
\[
1 < p \leq 2,
\quad
\frac{k}{p'} < n-k,
\quad
n-1 < kp < n,
\quad
\frac{n}{p}-\frac{n-k}{r}=k,
\quad
q \leq (n-k)p'.
\]
The first two are implied by the third if $k \geq n/2$.
On the other hand, if $k < n/2$, then $(n+1)/(k+1) \geq 2$ and the assertion is contained in the results of previous sections.

\section{The complex \texorpdfstring{$k$}{k}-plane transform}
\index{k-plane transform@$k$-plane transform!complex}
We will now take a brief look at the analogue of $\Radon_{n,k}$ given by integration over $k$-dimensional complex affine subspaces of $\C^n$.
We denote this map by $\Radon_{n,k,\C}$ and continue using the notation $G_{n,k}$, $M_{n,k}$, etc., but all subspaces are now assumed to be complex.
The volume function $\det$ will be substituted by
\[
\det_\C(0, x_1, \dots, x_k) = \det(0, x_1, \dots, x_k, i x_1, \dots, i x_k).
\]
It extends to non-zero values of the first argument by translation invariance.
Note that $x_1, \dots, x_k$
are $\C$-linearly independent if and only if $x_1, \dots, x_k, i x_1, \dots, i x_k$ are $\R$-linearly independent.
With this notation the following analogue of Lemma~\ref{lem:measures-on-Rn-k+1} holds.
\begin{lemma}
\label{lem:measures-on-Cn-k+1}
The following measures on $(\C^n)^{k+1}$ are equivalent up to a constant:
\[
\dif\Leb[\pi](x_0) \dots \dif\Leb[\pi](x_k) \dif\mu_{n,k}(\pi)
=
C \det_\C (x_0, \dots, x_k)^{k-n} \dif\Leb[n](x_0) \dots \dif\Leb[n](x_k).
\]
Here, $\Leb[\pi]$ denotes the $2k$-dimensional Lebesgue measure on the complex affine hyperplane $\pi$.
\end{lemma}
\begin{proof}
Proceed as in the proof of Lemma~\ref{lem:measures-on-Rn-k+1}.
The induction basis $k=1$ holds since the equivalence class of a pair $(x_0,x_1)$ under complex affine transformations is still uniquely determined by $|x_0-x_1|^2 = \det_\C(x_0,x_1)$, because the unitary group acts transitively on the unit sphere in $\C^n$.

In the induction step (\ref{eq:measure-on-planes-not-through-pt}) is replaced by
\begin{equation}
\label{eq:measure-on-complex-planes-not-through-pt}
\dist(x_0,\theta)^{2((k-1)-k)} \dif\mu_{k,k-1,\pi}(\theta) \dif\nu_{n,k,x_0}(\pi)
= C
\dist(x_0,\theta)^{2((k-1)-n)} \dif\mu_{n,k-1}(\theta)
\end{equation}
because of the change in homogeneity.

The important observation is now
\[
\det_\C(x_0\dots,x_{k-1}) \dist(x_k,\pi(x_0, \dots, x_{k-1}))^{2}
= C
\det_\C(x_0,\dots,x_{k-1},x_k).
\]
It holds since, assuming by translation invariance $x_0 = 0$,
\begin{multline*}
\dist(x_k,\pi(x_0, \dots, x_{k-1}))
=
\dist(i x_k,\pi(x_0, \dots, x_{k-1}))\\
=
\dist(i x_k,\pi(x_0, \dots, x_{k-1})+\R x_k),
\end{multline*}
because $\pi(x_0, \dots, x_{k-1})$ is a complex subspace and $x_k \perp i x_k$.
The inductive argument applies without further changes.
\end{proof}
The proof of Proposition~\ref{prop:k-plane-rearrangement} remains the same in the complex setting.
Furthermore we have the following substitute for Proposition~\ref{prop:k-plane-estimate-radial}.
\begin{proposition}
\label{prop:complex-k-plane-estimate-radial}
Let $f$ be a positive non-increasing radial function on $\C^n$.
Then
\[
|| \Radon_{n,k,\C} f ||_{n+1; n+1}
\leq C
|| f ||_{(n+1)/k+1, n+1}.
\]
\end{proposition}
\begin{proof}
By arguments analogous to those in real setting we have
\begin{align*}
|| \Radon_{n,k,\C} f ||_{n+1; n+1}^{n+1}
&= C
\int_{r=0}^\infty r^{2n-2k-1} \left|
\int_{s=r}^\infty f(s) s (s^2-r^2)^{k-1} \dif s\right|^{n+1} \dif r\\
&\leq C
\int_{r=0}^\infty r^{2 k n + 2 n} f(r)^{n+1} \frac{\dif r}{r}
\end{align*}
since the real dimension of a proper complex affine subspace is always $\geq 2$.
In the complex case the linear and the radial non-increasing rearrangements are related by
\[
f^*(\Omega_{2n} |x|^{2n}/{2n}) = f^{**}(x),
\]
so that $f^*(s) = f(C s^{1/(2 n)})$.
The change of variable $r = C s^{1/(2n)}$ now yields
\begin{align*}
|| \Radon_{n,k,\C} f ||_{n+1; n+1}^{n+1}
&\leq
C \int_{s=0}^\infty \left( s^{(k+1)/(n+1)} f^*(s) \right)^{n+1} \frac{\dif s}{s} \\
&=
C ||f||_{(n+1)/(k+1),n+1}^{n+1}.
\qedhere
\end{align*}
\end{proof}
The same estimate is true for arbitrary $f$.
Note that in the real case the Theorem~\ref{th:ind} and real interpolation provide the sharper estimate
\[
|| \Radon_{2n,2k} f ||_{2n+2; n+1}
\leq
C ||f||_{(n+1)/(k+1),n+1}.
\]
This is consistent with the fact that the complex Grassmannian $G_{n,k}^\C$ is a submanifold of the real Grassmannian $G_{2n,2k}^\R$.


%% file: chapter_curvature.tex
\chapter{Convolution kernels supported on submanifolds}
In section~\ref{subsec:radon-heisenberg} we have seen that the Radon transform in odd dimension is equivalent to the convolution with the Lebesgue measure on a hyperplane in the Heisenberg group.
Here we present a related $L^p$ continuity result which applies to more general Lie groups, submanifolds and measures supported thereon.

The submanifold in question must satisfy a curvature condition.
We restrict our attention to a variant of it from \cite{MR937632}.
For the most general results obtained by similar methods see \cite{MR1726701}.

Note that the convolution kernels treated in this chapter are compactly supported, so that the theory is not directly applicable to the Radon transform.

\section{Sobolev spaces}
For $k \in \N$ and $1 \leq p \leq \infty$ the Sobolev space $W^{k,p}(\R^n)$ is defined as the space of all distributions $g$ such that
\[
||g||_{W^{k,p}} := \sum_{j=0}^k ||\nabla^j g||_{L^p} < \infty.
\]
For the sake of completeness we recall some important facts from the theory of Sobolev spaces.
First of all, $C^\infty_c(\R^n)$ is dense in $W^{k,p}(\R^n)$ for every $k \in \N$ and every $1\leq p < \infty$.
Furthermore, $W^{k,p}$ may be thought of as a closed subspace of $(L^p)^{k+1}$ by considering the function and its derivatives separately.
Therefore the Lebesgue convergence theorems apply (under the assumption of domination or monotonicity on each derivative).

We will need a special case of the Sobolev embedding theorem \cite[Theorem 4.12]{MR2424078}.
The proof given here is taken from lecture notes by T.~Tao \cite{tao-blog}.
We begin by showing the Loomis-Whitney inequality.
While it admits a striking generalization, see e.g.\ \cite{MR2377493}, we content ourselves with the classical formulation.
\begin{proposition}[Loomis-Whitney]
\index{Loomis-Whitney inequality}
\label{prop:loomis-whitney}
Let $n \geq 2$ and $f_1, \dots, f_n \in L^{n-1}(\R^{n-1})$.
Then
\[
\int_{\R^n} \Prod_{j=1}^n |f_j (\hat x^j_{1,\dots,n})| \dif x_1 \dots \dif x_n \leq \Prod_{j=1}^n ||f_j||_{n-1}.
\]
Here $\hat x^j_{1,\dots,n} = (x_1, \dots, \hat x_j, \dots, x_n)$ denotes the omission of the $j$-th coordinate.
\end{proposition}
\begin{proof}
The case $n=2$ is just the Fubini theorem.
We proceed by induction on $n$.
Two applications of the Hölder inequality followed by the induction hypothesis show that
\begin{align*}
\int_{\R^{n+1}} &\Prod_{j=1}^{n+1} |f_j (\hat x^j_{1,\dots,n+1})| \dif x_{1,\dots,n+1}\\
&=
\int_{\R^{n}} f_{n+1} (x_{1,\dots,n}) \int_\R \Prod_{j=1}^{n} |f_j (\hat x^j_{1,\dots,n+1})| \dif x_{n+1} \dif x_{1,\dots,n}\\
&\leq
||f_{n+1}||_n \left( \int_{\R^{n}} \left( \int_\R \Prod_{j=1}^{n} |f_j (\hat x^j_{1,\dots,n+1})| \dif x_{n+1} \right)^{\frac{n}{n-1}} \dif x_{1,\dots,n}\right)^{\frac{n-1}{n}}\\
&\leq
||f_{n+1}||_n \left( \int_{\R^{n}} \Prod_{j=1}^{n} \left( \int_\R |f_j (\hat x^j_{1,\dots,n+1})|^n \dif x_{n+1} \right)^{\frac1n \frac{n}{n-1}} \dif x_{1,\dots,n}\right)^{\frac{n-1}{n}}\\
&\leq
||f_{n+1}||_n \Prod_{j=1}^{n} \left( \int_{\R^{n-1}} \left( \int_\R |f_j (\hat x^j_{1,\dots,n+1})|^n \dif x_{n+1} \right)^{\frac1n \frac{n}{n-1} (n-1)} \dif \hat x^j_{1,\dots,n}\right)^{\frac{1}{n-1} \frac{n-1}{n}}\\
&=
\Prod_{j=1}^{n+1} ||f_{j}||_n.
\qedhere
\end{align*}
\end{proof}

Next lemma comes in handy in density arguments.
\begin{lemma}
\label{lem:adj-dual-unit-ball-sstar-closed}
Let $Y$ be a locally convex space, $X$ a Banach space and $\iota : Y \to X$ a continuous injection with dense image.
Then the image of the closed unit ball $B_{X'}$ under the adjoint map $\iota'$ is $\sigma^*$-closed in $Y'$.
\end{lemma}
\begin{proof}
Let $(\phi_n) \subset B_{X'}$ be such that $\iota'(\phi_n) \to \psi$ in the $\sigma(Y',Y)$-topology on $Y'$.
This just means that $\phi_n(\iota(y)) \to \psi(y)$ for every $y \in Y$.
Therefore
\[
|\psi(y)| \leq \limsup_n ||\phi_n||_{X'} ||\iota(y)||_X \leq ||\iota(y)||_X,
\]
so that $\psi$, viewed as a linear form on $\iota(Y)$, admits a continuous extension with norm less or equal $1$.
\end{proof}

\begin{theorem}
\index{Sobolev embedding of $W^{1,1}$}
\label{th:sobolev-embedding-W11}
Let $1 \leq q \leq \frac{n}{n-1}$.
Then $W^{1,1}(\R^n) \hookrightarrow L^q(\R^n)$ with continuous embedding.
\end{theorem}
\begin{proof}
By interpolation it is sufficient to consider $q=\frac{n}{n-1}$.
Let $f \in C^\infty_c(\R^n)$.
By the fundamental theorem of calculus we obtain the estimate
\[
|f(x_{1,\dots,n})| \leq \int_{\R} |\nabla f(x_{1,\dots,n})| \dif x_j =: f_j(\hat x^j_{1,\dots,n}).
\]
By definition of $f_j$ and from the Loomis-Whitney inequality (Proposition~\ref{prop:loomis-whitney}) we infer that
\begin{multline*}
\int_{\R^n} |f(x_{1,\dots,n})|^{\frac{n}{n-1}} \dif x_{1,\dots,n}
\leq
\int_{\R^n} \Prod_{j=1}^n f_j(\hat x^j_{1,\dots,n})^{\frac{1}{n-1}} \dif x_{1,\dots,n}\\
\leq
\Prod_{j=1}^n \left( \int_{\R^{n-1}} f_j(\hat x^j_{1,\dots,n}) \dif \hat x^j_{1,\dots,n} \right)^{\frac{1}{n-1}}
=
||\nabla f||_1^{\frac{n}{n-1}}.
\end{multline*}

Let now $f \in W^{1,1}$ be arbitrary.
There exists a sequence $(f_n) \subset C^\infty_c$ such that $f_n \to f$ in $W^{1,1}$.
Since the inclusion $C^\infty_c \hookrightarrow L^\infty \hookrightarrow (W^{1,1})'$ is continuous, we also have the convergence in the sense of distributions.
By Lemma~\ref{lem:adj-dual-unit-ball-sstar-closed} with $X = L^{q'}$ and $Y=C^\infty_c$, the unit ball of $L^q$ is closed as a subset of the space of distributions, so that by the above $f \in L^q$ and we have a bound on $||f||_q$ in terms of $||f||_{W^{1,1}}$.
\end{proof}

\section{Transport of measure}
The image of a measure given by a differentiable density under a measurable map need not, in general, be regular in any sense.
However we do obtain some regularity by imposing additional constraints on the map.
\begin{proposition}
\label{prop:measure-transport-non-singular}
Let $D \subset \R^m$ be an open set with compact closure, $n \leq m$ and $\Phi : \bar D \to \R^n$ be a smooth map whose differential has full rank at every point of $D$.
Then there exists a constant $C$ (which depends polynomially on $\sup_D ||\dif \Phi||$) such that, for every $\psi \in C_c^1(D)$, the measure $\Phi_*(\phi \dif\Leb[m])$ is absolutely continuous with respect to the usual Lebesgue measure, and its density $\rho$ satisfies
\[
|| \rho ||_{W^{1,1}(\R^n)} \leq C ||\psi||_{C^1} \int_D J^{-2},
\]
where $J^2 = \sum_k J_k^2$ and $J_k$ are the minors of $\dif \Phi$ of order $m$.
\end{proposition}
\begin{proof}
The absolute continuity and the bound $||\rho||_{L^1} \leq C ||\psi||_{C^0}$ are clear locally and therefore globally by the monotone convergence theorem.

Let $L_x = \{ \Phi = x \}$ denote the level sets of $\Phi$.
Since $\Phi$ is a submersion and by the implicit function theorem, each $L_x$ is a (smooth) submanifold of $D$.
Let $NL_x$ denote the normal vector bundle to and $\mu_x$ the Riemannian volume on $L_x$ with respect to the Euclidean scalar product.
Then
\[
\rho(x) = \int_{L_x} \psi(y) \left( \det \left. \dif \Phi \right|_{NL_x}(y) \right)\inv \dif \mu_x(y).
\]
Fix $1 \leq j \leq n$ and define a vector field on $D$ by
\[
X(y) := \left. \dif \Phi \right|_{NL_{\Phi(y)}}\inv \partial_j.
\]
By Cramer's rule this vector field is smooth and bounded.
Denote the flow generated by $X$ by $\Psi_t$.
Then, for a fixed $x$, $\Psi_t : \tilde L_x \to \tilde L_{x+te_j}$ is bijective for small $t$, where $\tilde L_{x'}$ are some manifolds such that $L_{x'} \intersection \supp \psi \subset \tilde L_{x'} \subset L_{x'}$.
Assume for the moment that every $\tilde L_x$ can be covered by one chart and let $\{v_1, \dots, v_{m-n}\} \subset \Gamma T L_x$ be a frame of tangent vector fields.
For small $t$ the density of the transported measure satisfies
\begin{multline*}
\rho(x + t e_j)
=
\int_{y \in \tilde L_x} \psi(\Psi_t(y)) \left( \det \left. \dif \Phi \right|_{NL_{x+te_j}}(\Psi_t(y)) \right)\inv\\
\sqrt{\det\left(\< \dif \Psi_t v_k(y), \dif \Psi_t v_l(y) \>\right)_{kl}} w^1 \wedge \dots \wedge w^{m-n},
\end{multline*}
where the differential forms $w^j$ are dual to $v_j$ and vanish on $NL_x$.
We compute the derivative of the square root first.
Let $M_{kl} = \< v_k(y), v_l(y) \>$.
\begin{align*}
\tdif{}{t} & \left.\sqrt{\det\left(\< \dif \Psi_t v_k(y), \dif \Psi_t v_l(y) \>\right)_{kl}} \right|_{t=0}\\
&=
\frac12 \frac{1}{\sqrt{\dots}} \tdif{}{t}\left. \det\left(\< \dif\Psi_t v_k(y), \dif\Psi_t v_l(y) \>\right)_{kl} \right|_{t=0}\\
&=
\frac12 \sqrt{\dots}
\sum_{k,l} \left(\< v_k(y), v_l(y) \>\right)\inv_{lk}
\tdif{}{t}\left. \< \dif\Psi_t v_k(y), \dif\Psi_t v_l(y) \> \right|_{t=0}\\
&=
\frac12 \sqrt{\dots}
\sum_{k,l} M\inv_{lk}
\left( \< (\nabla X)|_{TL_{\Phi(y)}} v_k(y), v_l(y) \> + \< v_k(y), (\nabla X)|_{TL_{\Phi(y)}} v_l(y) \> \right)\\
&=
\frac12 \sqrt{\dots}
\tr \left( M\inv
\left( (\nabla X)|_{TL_{\Phi(y)}} M + M (\nabla X)|_{TL_{\Phi(y)}} \right) \right)\\
&=
\sqrt{\det\left(\< v_k(y), v_l(y) \>\right)_{kl}}
\tr ( (\nabla X)|_{TL_{\Phi(y)}} )
\end{align*}
The remaining terms are compositions of functions with $\Psi_t$, so that their derivatives are given by $X$ applied to the respective functions.
Altogether, the differentiation gives
\begin{multline*}
\partial_j \rho(x)
=
\int_{y \in \tilde L_x} \left[
(X \psi)(y) \left( \det \left. \dif \Phi \right|_{NL_{x}}(y) \right)\inv \sqrt{\det\left(\< v_k(y), v_l(y) \>\right)_{kl}}\right.\\
+
\psi(y) X \left(\left( \det \left. \dif \Phi \right|_{NL_{x}} \right)\inv \right)(y) \sqrt{\det\left(\< v_k(y), v_l(y) \>\right)_{kl}}\\
+
\left. \psi(y) \left( \det \left. \dif \Phi \right|_{NL_{x}}(y) \right)\inv \sqrt{\det\left(\< v_k(y), v_l(y) \>\right)_{kl}} \tr(\left. \nabla X \right|_{TL_x})
\right]\\
w^1 \wedge \dots \wedge w^{m-n},
\end{multline*}
or, more succinctly,
\[
\partial_j \rho(x)
=
\int_{y \in L_x}
F(y)
\left( \det \left. \dif \Phi \right|_{NL_{x}}(y) \right)\inv
\dif \mu_x(y),
\]
where
\[
F
=
X \psi
+
\psi \left( \det \left. \dif \Phi \right|_{NL_{x}} \right) X \left(\left( \det \left. \dif \Phi \right|_{NL_{x}} \right)\inv \right)
+
\psi \tr(\left. \nabla X \right|_{TL_x}).
\]
From the form of $F$ we see that
\[
|F|
\leq
P(\dif \Phi) ||\nabla \psi|| \left| \det \left. \dif \Phi \right|_{NL_{x}} \right|^{-1}
+
Q(\dif \Phi) |\psi| \left| \det \left. \dif \Phi \right|_{NL_{x}} \right|^{-2}
\]
for some polynomials $P$ and $Q$.
Hence
\[
\int_{\R^m} |\partial_j \rho(x)|
\leq
\int_D |F|
\leq
C ||\psi||_{C_1} \int_D J^{-2},
\]
where $C$ is polynomial in $\sup_D ||\dif\Phi||$ (for any operator norm).
The general case follows from the fact that $\supp \psi$ can be covered by countably many charts which locally trivialize the leaves $L_x$ and the vector fields $X=X_j$.
\end{proof}

\begin{proposition}
\label{prop:measure-transport}
Let $B \subset \R^m$ be the open unit ball, $n \leq m$ and $\Phi : \bar B \to \R^n$ be an analytic map whose differential has full rank almost everywhere.
Then there exists a $0<\theta<1$ such that, for every $\phi \in C_c^1(D)$, the measure $\Phi_*(\phi \dif\Leb[m])$ is absolutely continuous with respect to the usual Lebesgue measure, and its density $\rho$ lies in the real interpolation space $[ L^1, W^{1,1} ]_{\theta, \infty}$.
\end{proposition}
We remark that, with the notation from \cite{MR928802}, the space $[ L^1, W^{1,1} ]_{\theta, \infty}$ is the Besov space $B_{\theta,\infty}^1$.
However we are only interested in the existence of an embedding
\[
[ L^1, W^{1,1} ]_{\theta, \infty} \hookrightarrow L^\delta
\]
for some $\delta > 1$.
This is an immediate consequence of the Sobolev embedding theorem~\ref{th:sobolev-embedding-W11} and the Marcinkiewicz interpolation theorem~\ref{th:marcinkiewicz}.
\begin{proof}[Proof of Proposition~\ref{prop:measure-transport}]
The idea of the proof is to use the $K$-method and to decompose $\rho$ by the absolute value of $J^2 = \sum_k J_k^2$, where $J_k$ are the minors of $\dif \Phi$ of order $m$.

Since $\Phi$ is analytic, the zeroes of $J^2$ have finite order.
By the Weierstraß preparation theorem, $|\{ |J^2| < a \}| = O(a^{1/k})$ as $a \to 0$ locally near every zero of $J^2$ for some $k \in \N$.
By compactness of $\bar B$, we even have that $|\{ |J^2| < a \}| = O(a^{1/k})$ as $a \to 0$ on all of $\bar B$ for some $k \in \N$.
The same argument shows that $\dist(\{|J^2| < a\}, \{|J^2| > 2a\}) > Ca$ as $a\to 0$.

Let $\psi_a : \bar B \to [0,1]$ be a smooth function which is $1$ if $|J^2|>2a$, $0$ if $|J^2|<a$ and such that $||\psi_a||_{C^1} \leq C a^{-1}$.
Consider the decomposition
\[
\Phi_* ( \phi \dif\Leb[m] )
=
\Phi_* ( (1-\psi_a) \phi \dif\Leb[m] )
+
\Phi_* ( \psi_a \phi \dif\Leb[m] ).
\]
By the above, the $L^1$-norm of former summand is $O( ||\phi||_\infty a^{1/k})$.
On the other hand, by Proposition~\ref{prop:measure-transport-non-singular} the $W^{1,1}$-norm of the latter summand is $O(||\phi||_{C^1} a^{-2})$, so that
\[
K(\rho,t; L^1, W^{1,1}) \leq C a^{1/k} + C t a^{-2}.
\]
Taking $a = t^{k/(2k+1)}$ we obtain
\[
K(\rho,t; L^1, W^{1,1}) \leq C t^{1/(2k+1)}.
\]
Therefore $\rho \in [L^1, W^{1,1}]_{1/(2k+1), \infty}$ by definition of the $K$-method.
\end{proof}

\section{\texorpdfstring{$L^p$}{Lp} improvement}
Now we deduce an $L^p$-improvement result \cite[Theorem 1.1]{MR1021141} for a convolution operator whose kernel is a differentiable measure supported on a curved submanifold.
The idea is to use the curvature condition to convolve multiple copies of the measure into an absolutely continuous measure and to deduce some regularity for it from Proposition~\ref{prop:measure-transport}.

Let $G$ be a connected, simply connected Lie group and $W \subset G$ be a connected analytic submanifold.
Assume also that $WW\inv$ generates $G$ in the sense that $WW\inv$ is not contained in a proper closed subgroup of $G$.
An example of such a submanifold is the hyperplane $\{ t = 0 \}$ in the Heisenberg group.
By \cite[Proposition 1.1]{MR937632}, there exists a $k$ such that the map $\Pi : (WW\inv)^k \to G$, $(w_1, \dots, w_k) \mapsto w_1 \cdot\dots\cdot w_k$ has full rank on a dense subset of $(WW\inv)^k$.

\begin{theorem}
Let $G$ and $W$ be as above with the additional assumption that $G$ is unimodular, $\phi$ be a positive $C^1$ function with compact support on $W$ and $\sigma$ be a smooth volume measure on $W$.
Then there exists a $p<2$ such that
\[
||f * (\phi \dif\sigma)||_2 \leq C ||f||_p
\]
for every smooth function $f$ with compact support, where the norms are taken with respect to the Haar measure $\nu$ on $G$.
\end{theorem}
\begin{proof}
By compactness of $\supp f$ we may assume that $(W,\psi)$ is a coordinate neighborhood, where $\psi : B \to W$ is some analytic map and that $(W W\inv)^{2^l}$ is contained in a coordinate neighborhood $(B_{\mathfrak{g}},\exp)$ of the identity given by the exponential map, where $l$ is a natural number such that $2^l \geq k$ and $B_{\mathfrak{g}}$ is a ball in the Lie algebra $\mathfrak{g}$ of $G$.

For a measure $\mu$ on $G$ let $T_\mu$ denote the operator $T_\mu f=f * \mu$, i.e.
\[
T_\mu f(x) = \int_G f(xy\inv) \dif\mu(y).
\]
Then
\begin{multline*}
\< T_\mu f,g \>
=
\int\int f(xy\inv) \dif\mu(y) \bar g(x) \dif\nu(x)
=
\int\int f(x) \bar g(x y) \dif\nu(x) \dif\mu(y)\\
=
\int f(x) \int \bar g(x y) \dif\mu(y) \dif\nu(x)
=
\int f(x) T_{\mu^*} \bar g(x) \dif\nu(x)
=
\< f, T_{\mu^*} g \>,
\end{multline*}
where $\mu^*(E) = \mu(E\inv) = \iota_* \mu(E)$ ($\iota$ denotes the inversion $\iota(x) = x\inv$).
Observe that
\[
T_{\mu^*}f(x)
=
\int_B f(x \psi(a)) \dif(\psi\inv_* \iota\inv_* \mu^*)(a)
=
\int_B f(x \psi(a)) \dif(\psi\inv_* \mu)(a)
\]
and, analogously,
\[
T_{\mu}f(x)
=
\int_B f(x \psi(a)\inv) \dif(\psi\inv_* \mu)(a).
\]
Therefore $(T_{\mu^*} T_\mu)^{2^l}$ is the convolution operator with the measure
\[
(\mu * \mu^*)^{*2^l} = \exp_* \Phi_* (\psi\inv_* \mu)^{\times 2^{l+1}},
\]
where $\Phi$ is the map
\[
\Phi(x_1, \dots, x_{2^{l+1}}) = \exp\inv (\psi(x_1) \psi(x_2)\inv \psi(x_3) \dots \psi(x_{2^{l+1}})\inv).
\]
Since $\sigma$ is a smooth volume measure and $\psi$ is analytic, $\psi\inv_* \mu$ is continuous with respect to the Lebesgue measure and its density is once continuously differentiable.
We already know that $\Phi$ is a submersion on a dense set.
It cannot fail to be a submersion on a set of positive measure, since then it would not be a submersion anywhere by the Lebesgue density theorem, the Weierstrass preparation theorem and analyticity.
Therefore $\Phi$ satisfies the hypothesis of Proposition~\ref{prop:measure-transport}.
By that Proposition, $\Phi_* (\psi\inv_* \mu)^{\times 2^{l+1}}$ is in $L^\delta$ for some $\delta > 1$ with respect to the Lebesgue measure.
Since the Haar measure $\nu$ is a volume measure as well, its restriction to $\exp(B_{\mathfrak{g}})$ is equivalent to the Lebesgue measure on $B$ under $\exp$ up to some function which is bounded from above and below.
Therefore $(\mu * \mu^*)^{* 2^l} \in L^\delta$ as well.
This implies that
\[
T_l := (T_\mu^* T_\mu)^{2^l} = (T_{\mu^*} T_\mu)^{2^l} : L^p \to L^{p'}
\]
is continuous for some $p<2$.
We perform a descending induction on $l$ in order to show that the same is also true for $l=0$.
Indeed, assume that $T_{l+1} = T_l^2 : L^p \to L^{p'}$ is continuous.
Since $T_l$ is self-adjoint, we obtain that it is continuous as an operator from $L^p$ to $L^2$.
Indeed,
\[
||T_l f||_2^2 = \<T_l^* T_l f, f\> \leq ||T_{l+1}||_{L^p \to L^{p'}} ||f||_p^2.
\]
Still by self-adjointness, $T_l$ is also continuous from $L^2$ to $L^{p'}$.
By the Interpolation Theorem~\ref{th:stein} it follows that it is continuous from $L^q$ to $L^{q'}$, where $q$ is given by $2/q = 1/p+1/2$.

Since $T_\mu^* T_\mu : L^p \to L^{p'}$, we have in fact that $T_\mu : L^p \to L^2$.
\end{proof}


%% file: back_matter.tex
\selectlanguage{ngerman}
\chapter{Zusammenfassung}
Wir behandeln die abstrakte Formulierung und einige Anwendungen der Interpolationstheorie.

Dabei geht es um \emph{Methoden} einen Raum ``zwischen'' zwei gegebenen Banach\-räumen zu definieren.
Die Methoden sollten die \emph{Interpolationseigenschaft} besitzen: stetige Operatoren, die auf den ursprünglichen Räumen auf eine verträgliche Weise definiert sind, sollten stetige Operatoren auf Interpolationsräumen induzieren.
Dies ist mit der Hoffnung verbunden dass sich die Operatoren auf den ursprünglichen Räumen leichter untersuchen lassen.

Diese Theorie wird angewendet um Integraloperatoren der Form $Tf(y)=\int_{M} K(y,x) f(x)$ auf $L^p$-Stetigkeit zu untersuchen.
Hier ist $M$ eine Mannigfaltigkeit und der Träger der Distribution $K(y, \cdot)$ eine Untermannigfaltigkeit positiver Kodimension ist, beispielsweise eine affine Gerade in $\R^{n}$, $n\geq 2$.

Wir fassen die wichtigsten Eigenschaften der reellen $K$-Methode von Peetre zusammen und behandeln die komplexe Methode von \Calderon{} samt dem komplex-analytischen Unterbau.
Wir beweisen eine Version des Interpolationssatzes von Stein für nicht überall definierte Operatoren und zeigen dass diese bei der Untersuchung von $[L^{p_{0}}(\R^{n}),L^{p_{1}}(\R^{n})]_{\theta}$ angewendet werden kann.

Danach gehen wir auf Riesz-Transformationen ein.
Diese analytische Familie von Operatoren verallgemeinert die Ableitungsoperatoren.

Im vierten Kapitel wird die Wirkung der klassischen Radontransformation auf Schwartzfunktionen untersucht und die $L^{p}$-Stetigkeit charakterisiert.
Wir überprüfen dabei die technischen Annahmen der Originalarbeiten.

Wir untersuchen dann, in welchem Sinne man eine messbare Menge $T \subset \R^{n}$ in einen Ball umordnen kann und beweisen die Brunn-Minkowski-Ungleichung sowie eine Umordnungsungleichung von Brascamp, Lieb und Luttinger.

Kapitel~\ref{chap:hardy} ist dem Hardyraum $H^{1}$ gewidmet.
Dieser dient bei der Interpolation häufig als Ersatz für $L^{1}$.
Wir konstruieren die schärfste Version der atomaren Zerlegung in $H^1$ und erläutern den Zusammenhang mit der Stetigkeit der auf $H^1$ definierten Operatoren.
Den klassischen Beweis der Dualität $H^{1} = \VMO'$ geben wir in vereinfachter Form wieder.
Unser wichtigstes Ergebnis ist die Möglichkeit, zwischen $H^{1}(\R^{n})$ und $L^{p}(\R^{n})$ mittels Schwartzfunktionen zu interpolieren.

Der siebte Abschnitt enthält einige Anwendungen der Interpolation und der Umordnungsungleichungen auf die $k$-Ebenentransformation.

Schließlich betrachten wir Faltungsoperatoren über Liegruppen, die Träger deren Kerne Untermannigfaltigkeiten sind.
Das zentrale Maßtransportlemma wird in der Sprache der Interpolationstheorie bewiesen.
\selectlanguage{american}
